\documentclass[reqno]{amsart}
\usepackage[left=1.25in, right=1.25in]{geometry}
\usepackage[all]{xy}
\usepackage{graphicx}
\usepackage{indentfirst}
\usepackage{bm}
\usepackage{amsthm}
\usepackage{mathrsfs}
\usepackage{latexsym}
\usepackage{amsmath}
\usepackage{amssymb}
\usepackage{hyperref}
\usepackage{microtype}
\usepackage{color}
\usepackage{amsaddr}
\usepackage{tikz}
\allowdisplaybreaks

\newcommand{\gra}[1]{\raisebox{-.3cm}{\includegraphics[height=.7cm]{YB#1.pdf}}}
\newcommand{\graa}[1]{\raisebox{-.6cm}{\includegraphics[height=1.5cm]{YB#1.pdf}}}
\newcommand{\grb}[1]{\raisebox{-.8cm}{\includegraphics[height=2cm]{YB#1.pdf}}}
\newcommand{\grc}[1]{\raisebox{-1.3cm}{\includegraphics[height=3cm]{YB#1.pdf}}}
\newcommand{\grd}[1]{\raisebox{-1.8cm}{\includegraphics[height=4cm]{YB#1.pdf}}}
\newcommand{\gre}[1]{\raisebox{-2.3cm}{\includegraphics[height=5cm]{YB#1.pdf}}}
\newcommand{\grf}[1]{\raisebox{-3.9cm}{\includegraphics[height=9cm]{YB#1.pdf}}}

\newcommand\numberthis{\addtocounter{equation}{1}\tag{\theequation}}

\begin{document}

\theoremstyle{plain}
\newtheorem{theorem}{Theorem~}[section]
\newtheorem{main}{Main Theorem~}
\newtheorem{lemma}[theorem]{Lemma~}
\newtheorem{proposition}[theorem]{Proposition~}
\newtheorem{corollary}[theorem]{Corollary~}
\newtheorem{definition}[theorem]{Definition~}
\newtheorem{notation}[theorem]{Notation~}
\newtheorem{example}[theorem]{Example~}
\newtheorem*{remark}{Remark~}
\newtheorem*{question}{Question}
\newtheorem*{claim}{Claim}
\newtheorem*{ac}{Acknowledgement}
\newtheorem*{conjecture}{Conjecture~}
\renewcommand{\proofname}{\bf Proof}

\newpage
\title{Yang-Baxter relation planar algebras}
\author{Zhengwei Liu}
\address{Harvard University, Cambridge, MA 02138}
\date{}

\begin{abstract}
We give a new type of Schur-Weyl duality for the representations of a family of quantum subgroups and their centralizer algebra.
We define and classify singly-generated, Yang-Baxter relation planar algebras.
We present the skein theoretic construction of a new parameterized planar algebra.
We construct infinitely many new subfactors and unitary fusion categories, and compute their trace formula as a closed-form expression, in terms of Young diagrams.
\end{abstract}
\maketitle

\section{Introduction}
There are three different aspects to this paper.
(1)~We discover a new parameterized algebra which is the centralizer algebra for a family of quantum subgroups.
We study the representations of the centralizer algebra and these quantum subgroups with a flavor reminiscent of Schur-Weyl duality.
(2)~We define and classify {\it singly-generated, Yang-Baxter relation planar algebras}.
This can be interpreted as an initial step toward Bisch and Jones suggested skein theoretic classification.
(3)~We use skein theory to construct our new parameterized algebra as a planar algebra.
In this construction, we overcome the three fundamental problems in skein theory: evaluation, consistency, and positivity.

We now outline these aspects, their connections with other areas, and our main results one by one.

\subsection{Schur-Weyl duality}
Schur studied the representations of $SU(N)$, for all $N$, by studing the representations of its centralizer algebra, which is the symmetric group algebra. This correspondence is well known as Schur-Weyl duality \cite{Sch27,Wey46}. The centralizer algebra and Schur-Weyl duality had been well understood in the last century both for Lie groups \cite{Bra37,Wey46,Wen88}, and for quantum groups \cite{Jim85,Dri86,TemLie71,Jon87,BirWen,Mur87}. These centralizer algebras are known as Brauer algebras, Hecke algebras, and Birman-Murakami-Wenzl (BMW) algebras.

Quantum subgroups were investigated soon afterward. For quantum $SU(2)$, the classification of its subgroups was clarified in the early 90's, namely the $A_n$, $D_{2n}$, $E_6$, $E_8$ classification. This is a quantum version of the McKay correspondence. See \S 3.1 in the survey paper \cite{JMS} and further references cited in that review. Quantum subgroups and their representations (or modules) were defined in \cite{Ocn00,Xu98,Ost03} using different terminologies. In the first two papers, Ocneanu and Xu constructed representations of subgroups of some small rank quantum groups.

Here we extend the Schur-Weyl duality to two families of subgroups of quantum groups, whose rank approaches infinity.
We give the first example of a parameterized centralizer algebra $\mathscr{C}_{\bullet}=\mathscr{C}(q)_{\bullet}$ defined over the field $\mathbb{C}(q)$ for a family of subgroups of quantum groups. We construct all irreducible representations of the centralizer algebra $\mathscr{C}_{\bullet}$ and of these quantum subgroups. We compute a closed-form of the quantum dimensions of these representations. Our study of a family of quantum subgroups by its parameterized centralizer algebra follows the philosophy of Schur-Weyl duality. In this way, we can study analytic properties of the parameterized centralizer algebra, and capture information that would not be accessible if only worked at roots of unity.

To illustrate these points, we summarize a few main results for Schur-Weyl duality in Figure \ref{Figure:SW}.
We conjecture that one can define a parameterized centralizer algebra for quantum subgroups from any family of conformal inclusions.

\begin{figure}[h]
\begin{tabular}[Schur-Weyl duality]{|p{2.8in}|p{2.8in}|}

  \hline
  \begin{center} {\it Centralizer algebras} \end{center} & \begin{center}  {\it Lie groups, quantum groups or subgroups} \end{center} \\ \hline
  the symmetric group algebra  & $SU(N)$ \\ \hline
  the Brauer algebra  & $O(N)$ or $Sp(2N)$ \\ \hline
  the Temperley-Lieb algebra  & quantum $SU(2)$  \\  \hline
  the type A Hecke algebra  & quantum $SU(N)$  \\ \hline
  the Birman-Murakami-Wenzl algebra  & quantum $O(N)$ or $Sp(2N)$  \\ \hline
  the new algebra $\mathscr{C}_{\bullet}$ in this paper & quantum subgroups of $SU(N)_{N\pm2}$   \\ \hline
  \begin{center} {\large ?} \end{center} & quantum subgroups arising from families of conformal inclusions  \\ \hline
\end{tabular}
\caption{Schur-Weyl duality, with ``?" indicating the conjectured, parameterized centralizer algebras.}\label{Figure:SW}
\end{figure}

\subsection{Yang-Baxter relation planar algebras}
Subfactor theory provides a framework to generalize quantum groups, quantum subgroups, representations and centralizer algebras.
(Representations are given by Hilbert space bimodules over factors. Morphisms are bimodule maps. Connes introduced the tensor functor, called Connes fusion \cite{Pop86}.)
The centralizer algebra of a subfactor is called the {\it standard invariant}. A deep theorem of Popa showed that the standard invariant is a complete invariant of strongly amenable subfactors of the hyperfinite factor of type II$_1$ \cite{Pop94}. One can consider this result as the Tannaka-Krein duality for subfactors.

Jones introduced planar algebras for subfactors \cite{JonPA} as an axiomatization of the standard invariant. In the planar algebra framework, one can study the standard invariant by skein theory, in analogy with the presentation theory of geometric groups.
Here we have topological relations for the generators in addition to the algebraic ones.
Bisch and Jones suggested to classify planar algebras by skein theory \cite{BisJon03}.

In this paper, we introduce a new type of relation, which we call the Yang-Baxter relation. We are motivated by the Yang-Baxter equation \cite{Yan67,Bax07}, which is significant in mathematics and physics.
The Yang-Baxter relation generalizes the Yang-Baxter equation and the star triangle equation \cite{Ons44}.

Haagerup constructed the first exotic subfactor, well known as the Haagerup subfactor \cite{Haa94}.
It is considered to be exotic, since it remains an open question, can it be constructed from conformal field theory?
The Yang-Baxter relation provides skein theory for exotic subfactors, such as the Haagerup factor, while the Yang-Baxter equation does not \cite{LiuPen}.

Using the Yang-Baxter relation, we define {\it Yang-Baxter relation planar algebras}.
We give a complete classification of {\it singly-generated}, Yang-Baxter relation planar algebras.
The Yang-Baxter relation is a {\it global} skein relation for planar algebras.
It generalizes the {\it local} skein relations appearing in the small-dimension classification \cite{BisJon00,BisJon03,BJL}.
The critical dimension of these {\it local} skein relations is no longer an upper bound in our classification.
We discover our new algebra $\mathscr{C}_{\bullet}$ from this classification.

\subsection{Skein theoretic construction}
Planar algebras that appeared in previous small-dimension classification were all known, so one do not need to construct them for the classification. The skein theory of these planar algebras is still interesting.
In our classification result, we find an unexpected family of generators and relations, different from those of known planar algebras.
The bulk of this paper involves the construction of this new family.

When Jones introduced planar algebras, there were three fundamental problems for the skein theoretic construction: evaluation, consistency, and positivity \cite{JonPA}-the last being the most difficult. Overcoming these difficulties were thought to be mission impossible at the time. A few methods were explored to avoid these problems. One powerful method is the embedding theory \cite{JonPen,MorWal}, which was successfully used in the construction of the extended Haagerup subfactor planar algebra \cite{BMPS}. However, all these methods fail with a parameterized family.

Here we give new methods to overcome the three fundamental problems in a straight forward manner.
These methods benefit from results in analysis, algebra, and topology. Our methods also work for other cases.

For example, one problem for proving positivity is to determine the semisimple quotient of the centralizer algebra for non-generic case, i.e., the parameter is a root of unity.
Actually, Brauer and Weyl were interested in determining the semisimple quotient of the Brauer algebra. Wenzl resolved this problem in \cite{Wen88}. His proof utilized the representation theory for Lie groups.
We can apply our method to determine the semisimple quotient of the Brauer algebra and the BMW algebra without using the representation theory of Lie groups and quantum groups. Our method also works for Bisch-Jones planar algebras \cite{BisJonFC} \footnote{Bisch-Jones planar algebras are called Fuss-Catalan planar algebras in \cite{BisJonFC}}.
Beliakova and Blanchet gave some ideas on this problem for the BMW algebra in Section 8 in \cite{BelBla}.
Our method is slightly different from theirs.

\subsection{Connections with other areas}
The centralizer algebras for quantum groups are closed related to link invariants, the Jones polynomial \cite{Jon85}, the HOMFLY-PT polynomial \cite{Homfly,PT88} and the Kauffman polynomial \cite{Kau90}.
The topological interpretation of these invariants was pointed out by Witten in topological quantum field theory (TQFT) \cite{Ati88,Wit88}, and formalized by Reshetikhin and Turaev \cite{ResTur91}.

From the new algebra $\mathscr{C}_{\bullet}$, we construct unitary fusion categories \cite{ENO} $\mathscr{C}^{N,k,l}$, for $N,k,l\in \mathbb{N}^{+}$, therefore 3D TQFT by Turaev-Viro construction \cite{TurVir92}. It is not clear that whether the 3D TQFT can be parameterized by three parameters.

The unitary fusion categories $\mathscr{C}^{N,1,0}$ and $\mathscr{C}^{N,1,1}$ are representation categories of subgroups of quantum groups (or module categories of commutative algebras in modular tensor categories) \cite{Ocn00,Ost03}.
When $N=3,4$, they are isomorphic to the bimodule categories constructed by Xu from conformal inclusions \cite{Xu98} by Ocneanu's classification result \cite{Ocn00}.
We conjecture that this is true for all $N$.

We also obtain many new subfactors: the usual subfactors from the Bernoulli shift, the Goodman-Harpe-Jones construction
\cite{GHJ}, and others from the dihedral group symmetry of certain sub lattices of the Young lattice.
In particular, we obtain a sequence of subfactors which is an extension of the $\mathbb{Z}_4$ near-group subfactor \cite{Izu93}.
It remains open that whether near-group subfactors live in an infinite family and whether they can be constructed from conformal field theory.
Our result answers this question positively for the $\mathbb{Z}_4$ near-group subfactor.

The dihedral group symmetry was studied by Suter in algebra \cite{Sut02} and in geometry \cite{Sut12}. We expect to see more relations between these ideas.

\subsection{Definitions}
For a finite dimensional Hilbert space $V$, a one parameter family of (unitary) matrices $R(\cdot)$ on $V\otimes V$ satisfy the Yang-Baxter equation, if for any $s,t$,
$$(1\otimes R(t))(R(s+t)\otimes 1)(1\otimes R(s))= (R(s)\otimes 1)(1\otimes R(s+t))(R(t)\otimes 1).$$
By taking the limit at infinity, one obtain a solution of the parameter independent Yang-Baxter equation,
$$(1\otimes R)(R\otimes 1)(1\otimes R)= (R\otimes 1)(1\otimes R)(R\otimes 1).$$
It can be interpreted as the Reidemester move of type III:
$$\gra{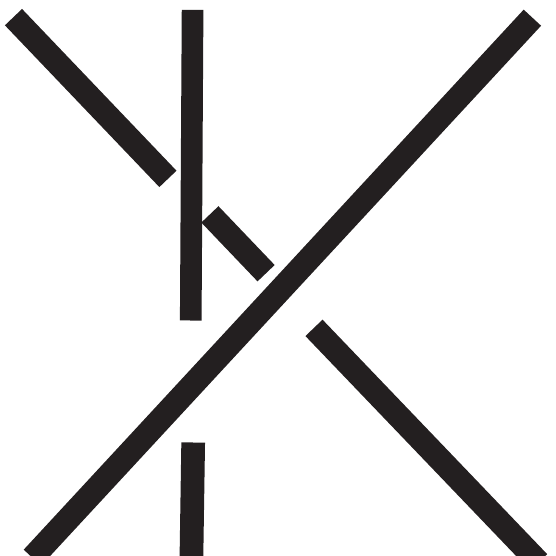}=\gra{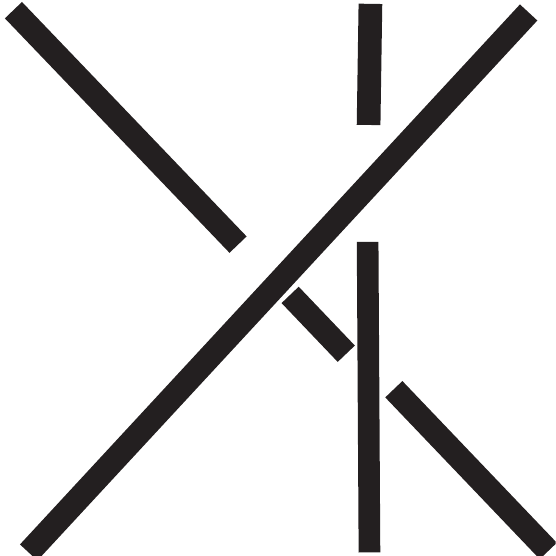}.$$

We generalize $V$ as a Hilbert space bimodule over von-Neumann algebras.
Inspired by the Yang-Baxter equation, we introduce the Yang-Baxter relation for bimodule maps on $\hom(V\otimes V)$.
\begin{definition}
We say a triple of bimodule maps $R_i, R_j, R_k\in \hom(V\otimes V)$ has a Yang-Baxter relation, if
$$(1\otimes R_i)(R_j\otimes 1)(1\otimes R_k)= \sum_{\substack{i',j',k'}} c_{i,j,k}^{i',j',k'} (R_{k'}\otimes 1)(1\otimes R_{j'})(R_{i'}\otimes 1),$$
for some \emph{duality coefficients} $c_{i,j,k}^{i',j',k'}$ and $R_{i'}, R_{j'}, R_{k'}\in \hom(V\otimes V)$.
We say the space of bimodule maps $\hom(V\otimes V)$ has a Yang-Baxter relation, if any triple of bimodule maps has.
\end{definition}
The diagrammatic interpretation of the Yang-Baxter relation is
$$\grb{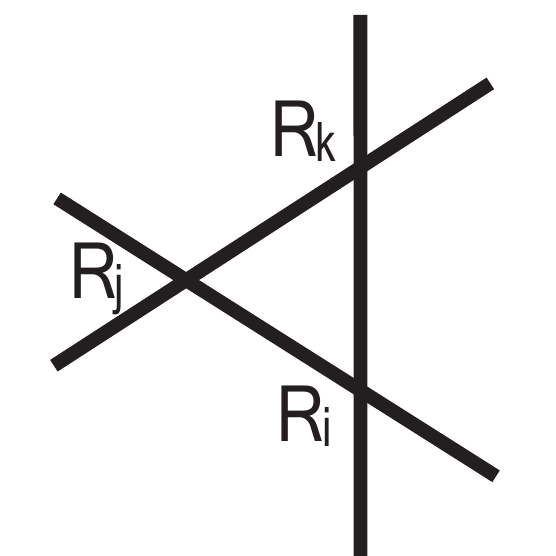}= \sum_{\substack{i',j',k'}} c_{i,j,k}^{i',j',k'} \grb{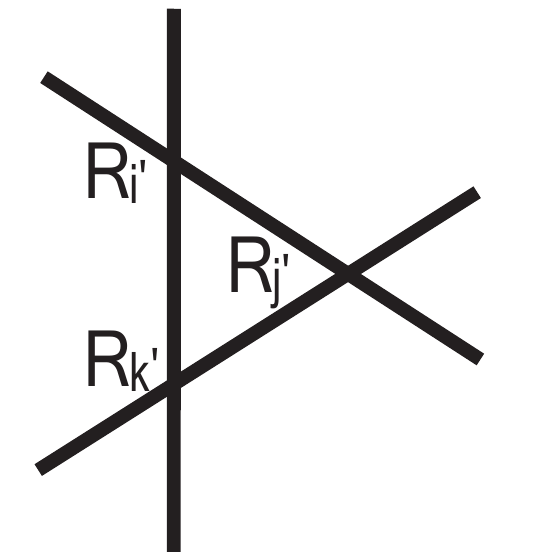}.$$

A planar algebra $\mathscr{P}$ consists of graded vector spaces $\{\mathscr{P}_{m,\pm}\}_{m\in \mathbb{N}}$ whose elements can be combined naturally in multilinear operations indexed by {\it planar tangles}. (See \cite{JonPA} and an example in Figure \ref{Figure:tangle}.)

\begin{figure}[h]
$$\grc{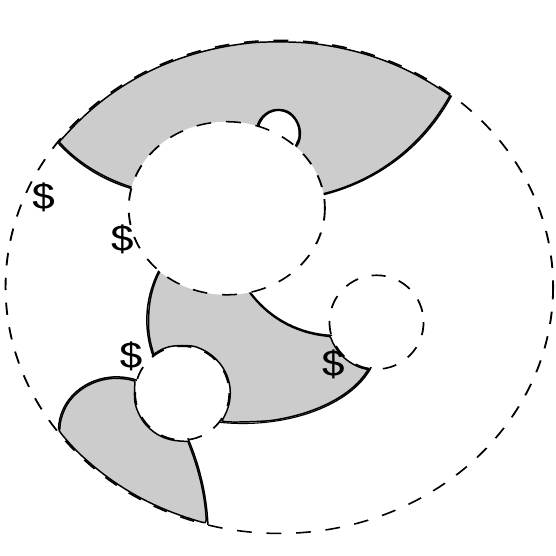}.$$
\caption{An example of planar tangles}\label{Figure:tangle}
\end{figure}

The vector in $\mathscr{P}_{m,\pm}$ can be considered as a diagram with $2m$-boundary points, called a $m$-box. Acting by a multiplication tangle, $\mathscr{P}_{m,\pm}$ forms an algebra, which is usually realized as bimodule maps on the $m$th (alternating) tensor power of a Hilbert space bimodule by the reconstruction theorem \cite{Ocn88,Pop95,GJS}.

In this paper, we study the planar algebras which can be presented by 2-box generators and Yang-Baxter relations. Furthermore, if the planar algebra has a \emph{positive definite Markov trace}, then we call it a Yang-Baxter relation planar algebra. The positivity is desired for many reasons.

\subsection{Main results}
We prove that the partition function of a planar algebra can be \emph{evaluated} by the intrinsic 2-box structure and the Yang-Baxter relation. (See Theorem \ref{evaluable}.) Therefore one can expect a classification result based on this type of skein theory.

The 2-box space of a planar algebra always contains two Temperley-Lieb diagrams $\gra{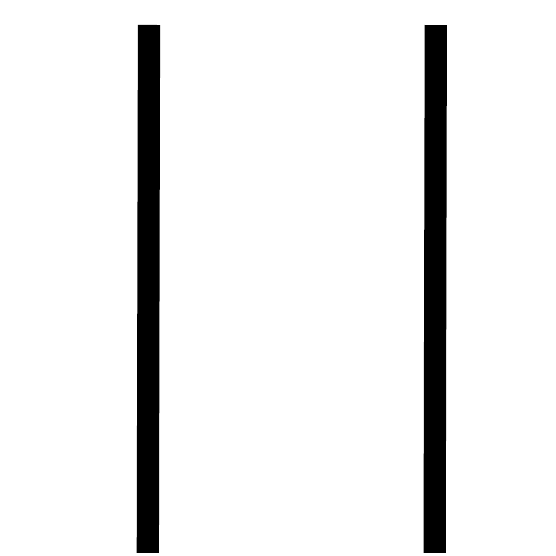}$, $\gra{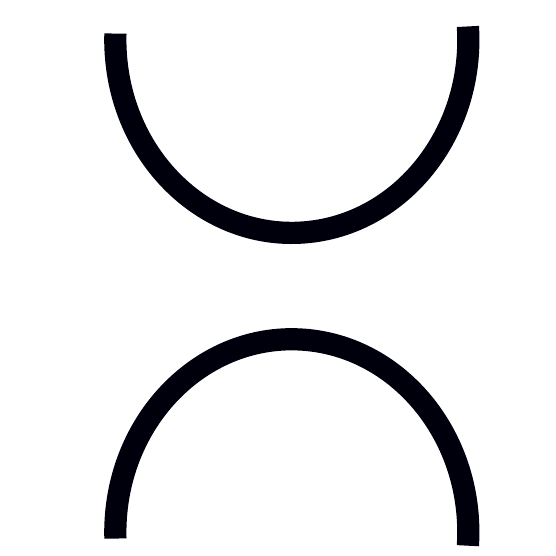}$.
We give a classification of Yang-Baxter relation planar algebras with one more generator in the 2-box space, so that the 2-box space is three dimensional. (See Theorems \ref{classification1}, \ref{class-1}.) We call them singly-generated Yang-Baxter relation planar algebras.
From this classification, we obtain an unexpected set of $q$-parameterized relations in addition to the previously known examples. To complete the classification, one needs to work more to prove the \emph{consistency} of the relations and to construct the corresponding $q$-parameterized planar algebras $\mathscr{C}_{\bullet}$ by skein theory. Furthermore, one needs to determine all values of $q$, for which the planar algebra has a positive semi-definite Markov trace. Then one obtains a Yang-Baxter relation planar algebra as the quotient of the planar algebra by the null space of the Markov trace.

We give a new method to reduce the algorithmic complexity, and prove consistency. (See Theorem \ref{consistency}.) Then we obtain the $q$-parameterized planar algebra $\mathscr{C}_{\bullet}$.

The positivity is proved by the following three steps:

(1) We construct all matrix units of $\mathscr{C}_{\bullet}$ over the field $\mathbb{C}(q)$. Its principal graph is Young's lattice. Consequently, the irreducible representations of its $m$-box algebra are indexed by Young diagrams with $m'$ cells, $m'\leq m$ with the same parity. (See Theorem \ref{P=YL}.)

(2) We compute the trace formula of minimal idempotents. (See Theorem \ref{trace formula} which we restate here.)
\begin{theorem}[Trace formula]\label{main theorem trace formula}
The quantum dimension of an irreducible representation indexed by a Young diagram $\lambda$ is given by
$$<\lambda>=\prod_{c\in\lambda} \frac{i(q^{h(c)}+q^{-h(c)})}{q^{h(c)}-q^{-h(c)}},$$
where $h(c)$ is the hook length of a cell $c$ in $\lambda$.
\end{theorem}
(3) We still have to resolve a special situation: if $q$ is a root of unity, then the planar algebra that we obtain is not semisimple over the filed $\mathbb{C}$.
How can one determine the semisimple quotient?
Here we have neither the presumed semisimple quotients nor the trace formula. We resolve this problem by constructing the matrix units and computing the trace in a specific (and very delicate) order.
Consequently, we show that the planar algebra has a positive Markov trace if and only if $q=e^{\frac{2\pi i}{2N+2}}$, for $N=1,2,3,\cdots$.
(See Theorem \ref{Positivity}.)

Now we accomplish the goal and construct a sequence of Yang-Baxter relation planar algebras $\mathscr{C}^{N}_{\bullet}$.
Thereby we achieve the classification result:
\newpage
\begin{theorem}[Classification]\label{main theorem classification}
Yang-Baxter relation planar algebras generated by a 3 dimensional 2-box space are either:
\begin{itemize}
\item[(1)] Bisch-Jones planar algebras;

\item[(2)] Birman-Murakami-Wenzl planar algebras;

\item[(3)] or the new sequence $\mathscr{C}^{N}_{\bullet}$, $N\geq2$, $N\in \mathbb{N}$.
\end{itemize}
\end{theorem}

The principal graph of $\mathscr{C}^{N}_{\bullet}$ is the sub lattice of the Young lattice consisting of Young diagrams whose $(1,1)$ cell has hook length at most $N$. (See Theorem \ref{Positivity} and Figure \ref{Figure:Pinciplegraph}.)
We prove that the principal graph of $\mathscr{C}^{N}_{\bullet}$ has an dihedral group $D_{2(N+1)}$ symmetry, and construct more new subfactors by this symmetry. (See Proposition \ref{Prop:symmetry}, Theorem \ref{Thm:subfactors1} and \ref{Thm:subfactors2}.)
We also construct unitary fusion categories $\mathscr{C}^{N,k,l}$ and compute their branching formulas. (See Theorems \ref{Thm:UFC} and Figure \ref{Figure:Branchingformula}.)
In particular, $\mathscr{C}^{N,1,0}$ and $\mathscr{C}^{N,1,1}$ are representation categories of subgroups of $SU(N)_{N+2}$ and $SU(N+2)_N$. (See Theorems \ref{Cor:quantumsubgroup1} and \ref{Cor:quantumsubgroup2}.)
We give an algebraic presentation of the algebra $\mathscr{C}_{\bullet}$ in Appendix \ref{Appendix:algebraic presentation}. (See Theorem \ref{Thm:algebraic presentation}.)

\begin{figure}[h]
$$\grc{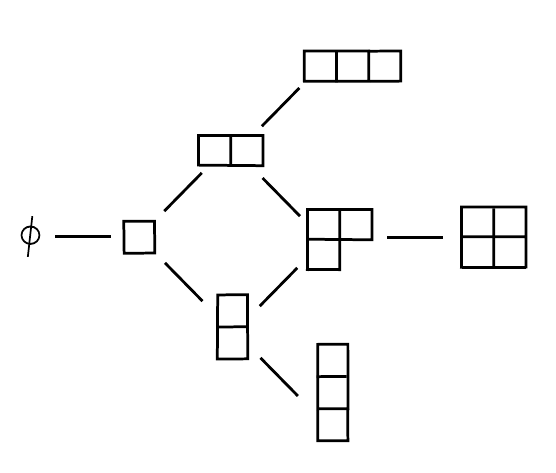}.$$
\caption{The principal graph of $\mathscr{C}^{3}_{\bullet}$} \label{Figure:Pinciplegraph}
\end{figure}

\begin{figure}[h]
$$\grf{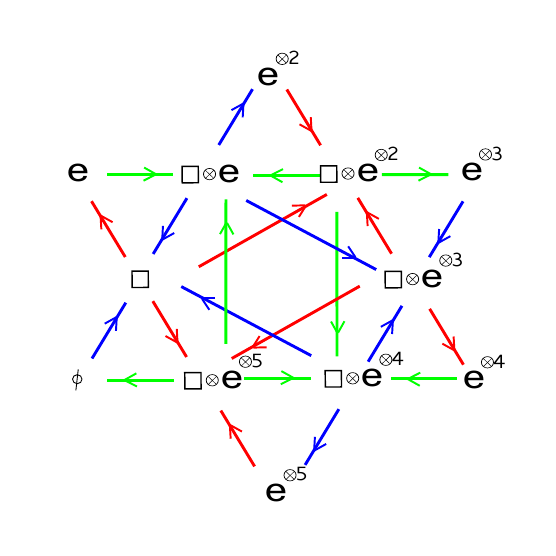}.$$
  \caption{The branching formula for $\mathscr{C}^{3,1,0}$, the representation category of the quantum subgroup of $SU(3)_5$} \label{Figure:Branchingformula}
\end{figure}

\subsection{Organization of the paper}
The paper is organized as follows.
In Section \ref{section preliminary}, we define the Yang-Baxter relation and briefly review subfactor theory and Hecke algebras.
In Section \ref{YBR}, we introduce Yang-Baxter relation planar algebras, an evaluation algorithm and some examples. In Section \ref{classification}, we give a classification for singly-generated Yang-Baxter relation planar algebras. In the classification, we find a $q$-parameterized family of generators and relations. In Section \ref{subsection consistency}, we prove the consistency of the relations and construct a planar algebra $\mathscr{C}_{\bullet}$ over the field $\mathbb{C}(q)$. In Section \ref{subsection matrix units}, we construct the matrix units of $\mathscr{C}_{\bullet}$, therefore its representations. In Section \ref{subsection trace formula}, we compute the trace formula, i.e. proving Theorem \ref{main theorem trace formula}. In Section \ref{subsection positivity}, we study $\mathscr{C}_{\bullet}$ over the field $\mathbb{C}$ and prove the positivity: find all $q\in \mathbb{C}$, such that the quotient of $\mathscr{C}_{\bullet}$ is a subfactor planar algebra. Then we obtain a sequence of subfactor planar algebras $\mathscr{C}^{N}_{\bullet}$ and complete the classification, i.e. proving Theorem \ref{main theorem classification}. In Section \ref{section dihedral}, we prove the dihedral group symmetry for $\mathscr{C}^{N}_{\bullet}$ and construct more subfactors. In Section \ref{section quantum subgroups}, we construct unitary fusion categories $\mathscr{C}^{N,k,l}$ involving representation categories of two families of quantum subgroups. In the Appendix we give an algebraic presentation for the centralizer algebra $\mathscr{C}_{\bullet}$, and prove some technical results, which are parts of the proofs of the main theorems.

\begin{ac}
The author would like to thank Dietmar Bisch and Vaughan Jones for their guidance for this paper, to thank Arthur Jaffe for many helpful comments, and to thank Noah Snyder, Hans Wenzl, Feng Xu for helpful discussions about quantum groups, BMW algebras and conformal inclusions.
The author thanks the FIM of the ETH Zurich for hospitality.
The author was supported by NSF Grant DMS-1001560, DOD-DARPA Grant HR0011-12-1-0009 and a grant from Templeton Religion Trust.
\end{ac}

\newpage

\section{Preliminaries}\label{section preliminary}

\subsection{Yang-Baxter equations and relations}
The Yang-Baxter equation plays an important role in lattice models \cite{Yan67,Bax07}.
It is given by
$$(1\otimes R(t))(R(s+t)\otimes 1)(1\otimes R(s))= (R(s)\otimes 1)(1\otimes R(s+t))(R(t)\otimes 1),$$
where $R(\cdot)$ is one parameter family of (unitary) matrixes on $V\otimes V$ for a finite dimensional Hilbert space $V$.

The parameter independent Yang-Baxter equation
$$(1\otimes R)(R\otimes 1)(1\otimes R)= (R\otimes 1)(1\otimes R)(R\otimes 1).$$
can be interpreted as the Reidemester move of type III:
$$\gra{yb2}=\gra{yb1}.$$
We refer the reader to \cite{Per06} for a review and interesting examples.

One can find out more solutions when $V$ is a Hilbert space bimodule over von-Neumann algebras. The morphisms are given by bimodule maps. The tensor functor of bimodule categories is known as the Connes fusion \cite{Pop86}.
When the vou-Neumann algebras are the ground field $\mathbb{C}$, the Connes fusion of bimodules is the usual tensor of Hilbert spaces.

Inspired by the Yang-Baxter equation, we introduce the Yang-Baxter relation for bimodule maps on the tensor of bimodules.

\begin{definition}
We say a triple of bimodule maps $R_i, R_j, R_k\in \hom(H\otimes H)$ has a Yang-Baxter relation, if
$$(1\otimes R_i)(R_j\otimes 1)(1\otimes R_k)= \sum_{\substack{i',j',k'}} c_{i,j,k}^{i',j',k'} (R_{k'}\otimes 1)(1\otimes R_{j'})(R_{i'}\otimes 1).$$
for some scalar $c_{i,j,k}^{i',j',k'}$ and $R_{i'}, R_{j'}, R_{k'}\in \hom(H\otimes H)$.
We say the algebra $\hom(H\otimes H)$ has a Yang-Baxter relation, if any triple of bimodule maps $R_i, R_j, R_k\in \hom(H\otimes H)$ has.
\end{definition}
By linearity, the Yang-Baxter relation is defined by the coefficients $c_{i,j,k}^{i',j',k'}$ for a basis $\{R_{i}\}$ of $\hom(H\otimes H)$, and we call $c_{i,j,k}^{i',j',k'}$ {\it duality coefficients}.

The diagrammatic interpretation of the Yang-Baxter relation is
$$\grb{par4}= \sum_{\substack{i',j',k'}} c_{i,j,k}^{i',j',k'} \grb{par5}.$$
This provides a generalization of the Yang-Baxter equation and the star-triangle equation \cite{Ons44}.

We can also define the Yang-Baxter relation for bimodule maps on the tensor of bimodules $H_1$, $H_2$, $H_3$. We say a triple of bimodule maps $R_i\in \hom(H_1\otimes H_2, H_2\otimes H_1)$, $R_j \in \hom(H_1\otimes H_3, H_3\otimes H_1)$, $R_k \in \hom(H_2\otimes H_3, H_3 \otimes H_2)$ has a Yang-Baxter relations, if
$$(1\otimes R_i)(R_j\otimes 1)(1\otimes R_k)= \sum_{\substack{i',j',k'}} c_{i,j,k}^{i',j',k'} (R_{k'}\otimes 1)(1\otimes R_{j'})(R_{i'}\otimes 1).$$
for some duality coefficient $c_{i,j,k}^{i',j',k'}$ and $R_{i'}\in \hom(H_1\otimes H_2, H_2\otimes H_1)$, $R_{j'} \in \hom(H_1\otimes H_3, H_3\otimes H_1)$, $R_{k'} \in \hom(H_2\otimes H_3, H_3 \otimes H_2)$.

\begin{remark}
The Yang-Baxter relation can be defined on a monoidal category over a field, if the positivity (or unitary) condition is not required.
\end{remark}

Inspired by the star-triangle equation for checkerboard lattice models \cite{Ons44}, we also study Yang-Baxter relations with alternating shading which is intrinsic in subfactor theory.

\subsection{Subfactors}
Modern subfactor theory was initiated by Jones \cite{Jon83} and developed by many others to study quantum symmetries \cite{EvaKaw98}.
We hope that a brief review will be helpful for understanding the motivation and terminology in this paper.

Suppose $\mathcal{N}\subset \mathcal{M}$ is a subfactor with finite index.
Then the standard representation $L^2(\mathcal{M})$ forms an irreducible $(\mathcal{N},\mathcal{M})$ bimodule, denoted by $X$. Its conjugate $\overline{X}$ is an $(\mathcal{M},\mathcal{N})$ bimodule.
The bimodule tensor products $X\otimes\overline{X}\otimes\cdots\otimes \overline{X}$, $X\otimes\overline{X}\otimes\cdots\otimes X$, $\overline{X}\otimes X\otimes\cdots\otimes X$ and $\overline{X}\otimes X\otimes\cdots\otimes\overline{X}$ are decomposed into irreducible bimodules over $(\mathcal{N},\mathcal{N})$, $(\mathcal{N},\mathcal{M})$, $(\mathcal{M},\mathcal{N})$ and $(\mathcal{M},\mathcal{M})$ respectively, where $\otimes$ is the Connes fusion of bimodules.
\begin{definition}\label{def principal graph}
The principal graph of a subfactor $\mathcal{N}\subset \mathcal{M}$ is an induction-restriction graph. Its vertices are equivalence classes of irreducible bimodules over $(\mathcal{N},\mathcal{N})$ and $(\mathcal{N},\mathcal{M})$ appeared in the above tensor powers.
The number of edges between two vertices corresponding to an $(\mathcal{N},\mathcal{N})$ bimodule $Y$ and an $(\mathcal{N},\mathcal{M})$ bimodule $Z$ is the multiplicity of $Z$ in $Y\otimes X$, (or $Y$ in $Z\otimes \overline{X}$ by Frobenius reciprocity) .
\end{definition}

We call a subfactor \emph{finite depth}, if its principal graph is a finite graph. In this case, the statistical dimensions of the bimodules form a vector on the vertices which is the unique Perron-Frobenius eigenvector of the adjacent matrix.

The centralizer algebra of a subfactor is called the \emph{standard invariant}:
$$\begin{array}{ccccccccc}
\mathbb{C}&\subset&\hom(X)&\subset&\hom(X\otimes \overline{X})&\subset&\hom(X\otimes \overline{X} \otimes X)&\subset&\cdots\\
&&\cup&&\cup&&\cup&&\\
&&\mathrm{C}&\subset&\hom(\overline(X))&\subset&\hom(\overline(X) \otimes X)&\subset&\cdots\\
\end{array}$$

(See \cite{Xu98S} for examples from quantum groups.)

A deep theorem of Popa showed that the standard invariant is a complete invariant of strongly amenable subfactors of the hyperfinite factor of type II$_1$ \cite{Pop94}. This involves the finite depth case.

The axiomatizations of the standard invariant was given by Ocneanu's paragroups for finite depth case \cite{Ocn88}; by Popa's standard $\lambda$-lattices \cite{Pop95} in general.

In this paper, we use Jones' axiomatization: subfactor planar algebras \cite{JonPA}. From planar algebras perspective, one can study subfactors and bimodule categories by skein theory.

\subsection{Planar Algebras}
We refer the reader to Section 2 in \cite{Jon12} for the definitions of planar tangles, (subfactor) planar algebras.
For reader's convenience, we give a brief review here.
\begin{definition}[Planar tangles]
A shaded planar tangle has

$\bullet$ finite ``input" discs

$\bullet$ one ``output" disc

$\bullet$ non-intersecting strings

$\bullet$ alternating shading

$\bullet$ a distinguished interval of each disc marked by $\$$
\end{definition}

We define unshaded planar tangles by ignoring shading.

\begin{definition}[Composition of tangles]
For two planar tangles
$T=\grb{tangle1}$,
$S=\gra{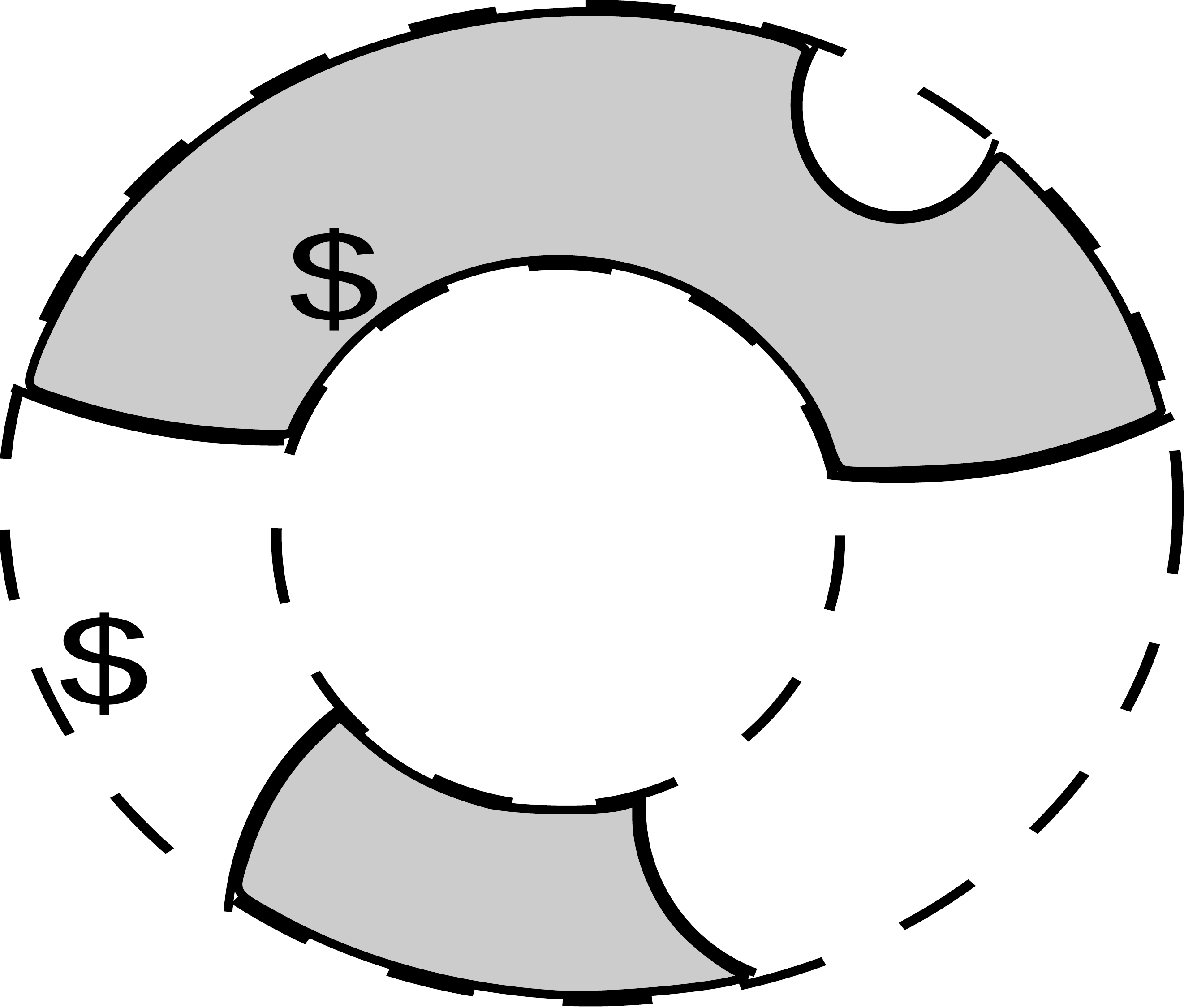}$, such that the output disc $S$ is identical to $i$th input disc (the top one) of $T$, we define the composition as
$$\grc{tangle1}\circ_i\gra{tangle2}=\delta \grc{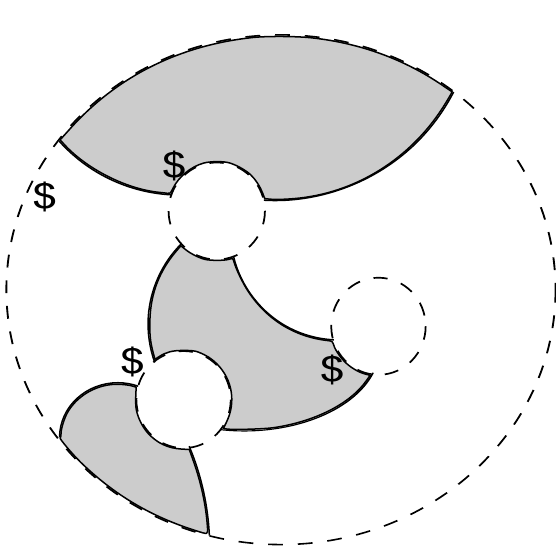},$$
where $\delta$ is the circle parameter in the ground field.
\end{definition}

\begin{definition}
A shaded \underline{general planar algebra} $\mathscr{P}_{\bullet}=\{\mathscr{P}_{m,\pm}\}_{m\in \mathbb{N}_0}$, $\mathbb{N}_0=\{0,1,2,\cdots\}$, is a family of $\mathbb{Z}_2$ graded vector spaces with multilinear maps of $\mathscr{P}_{\bullet}$ indexed by (shaded) planar tangles subject to

$\bullet$ Isotopy invariance

$\bullet$ Naturality

\begin{center}
\xymatrix{
    \textcolor{red}{\mathscr{P}_{2,-}}\otimes\mathscr{P}_{2,+}\otimes\mathscr{P}_{1,-} \ar[dr]_{\delta\grb{tangle4}} \ar[r]^{\gra{tangle2}}
                & \textcolor{red}{\mathscr{P}_{3,+}}\otimes\mathscr{P}_{2,+}\otimes\mathscr{P}_{1,-} \ar[d]^{\quad\quad\quad\quad\grb{tangle1}}  \\
                & \mathscr{P}_{2,+}     }
\end{center}

If $\mathscr{P}_{0,\pm}$ is one dimensional, and $\mathscr{P}_{\bullet}$ is unital, i.e., the empty diagram is in $\mathscr{P}_{0,\pm}$, then $\mathscr{P}_{\bullet}$ is called a \underline{planar algebra}.
\end{definition}

\begin{definition}
An unshaded (general) planar algebra $\mathscr{P}_{\bullet}=\{\mathscr{P}_{m}\}_{m\in \mathbb{N}_0}$ is a family of vector spaces with multilinear maps of $\mathscr{P}_{\bullet}$ indexed by unshaded planar tangles subject to isotopy invariance
and naturality.
\end{definition}

\begin{definition}
  A planar algebra is called a planar *-algebra, if an involution $*$ is defined on each $\mathscr{P}_{m,\pm}$ as an anti-linear map, and it is compatible with the vertical reflection of planar tangles.
\end{definition}

\begin{definition}
A subfactor planar algebra is an evaluable spherical planar *-algebra over $\mathbb{C}$ with a positive definite Markov trace.

$\bullet$ Evaluable: $\dim(\mathscr{P}_{0,\pm})\cong\mathbb{C}$, $\dim(\mathscr{P}_{m,\pm})<\infty$

$\bullet$ Spherical: $\gra{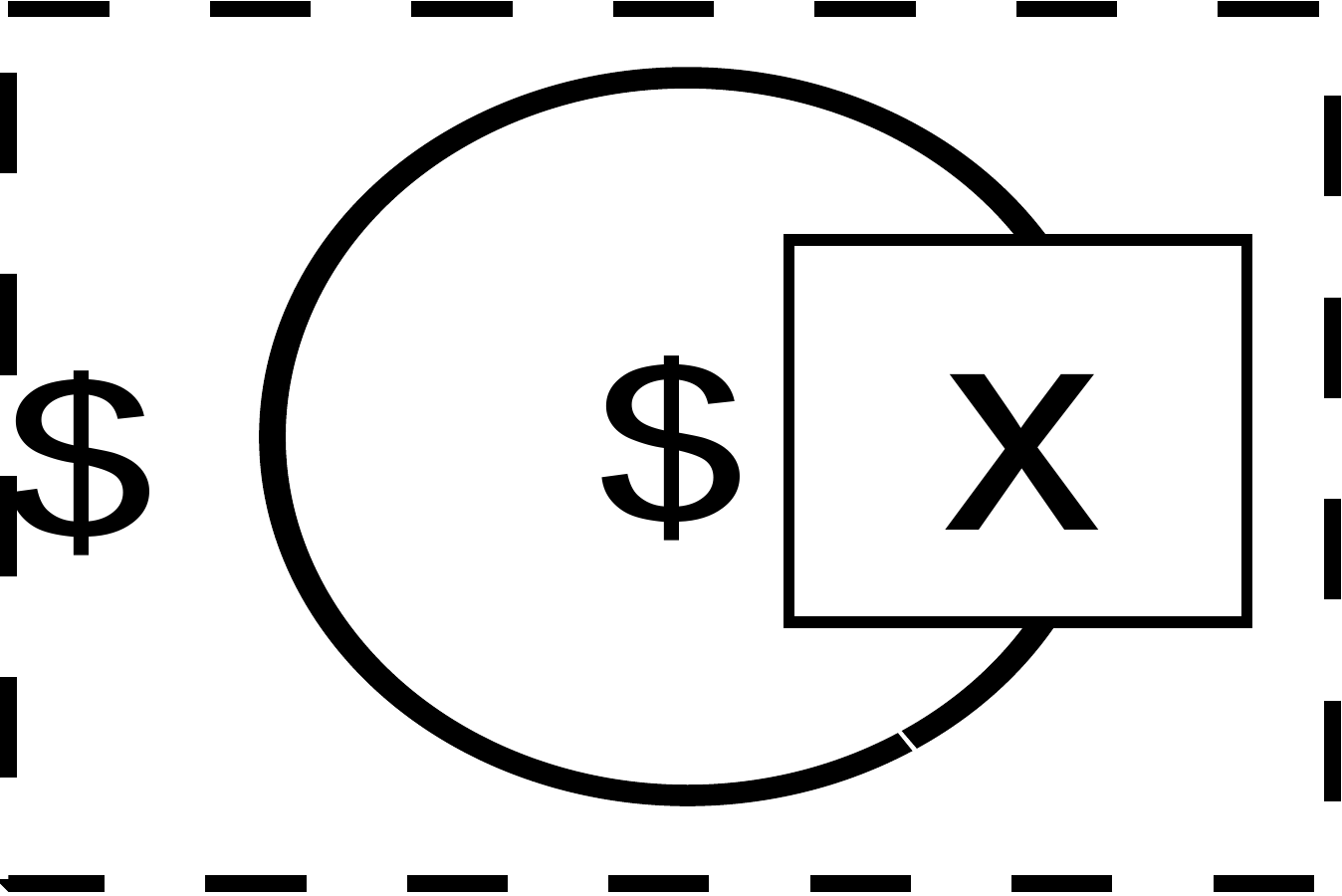}=\gra{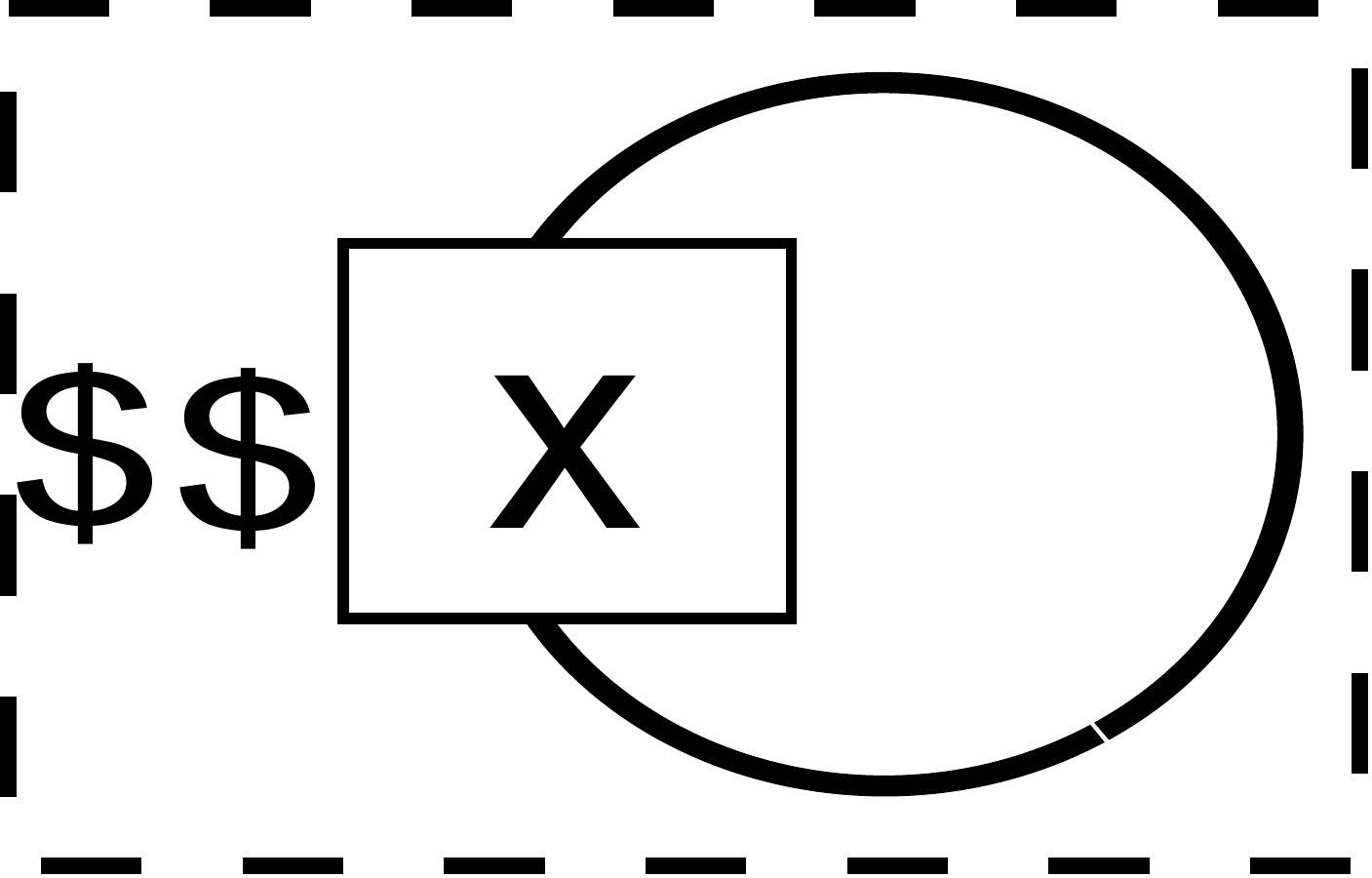}$

$\bullet$ the Markov trace $tr(z^*y)=\grb{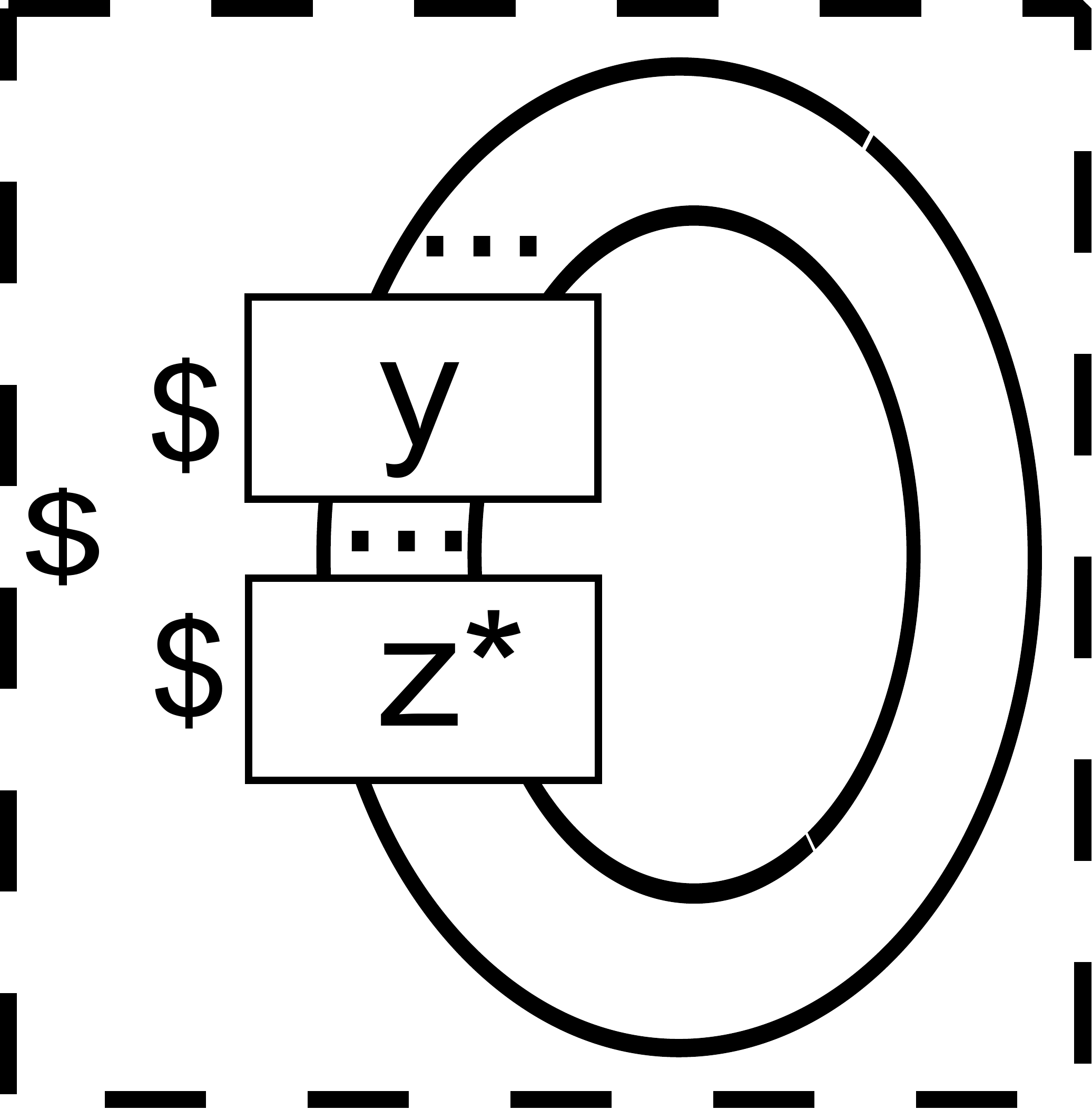}$ is positive definite.
\end{definition}

\begin{theorem}[Theorem 4.21, 4.31 \cite{JonPA}]
The standard invariant of a finite index extremal subfactor is a subfactor planar algebra. The converse statement is true.
\end{theorem}

Therefore the vectors in $\mathscr{P}_{m,\pm}$ can be realized as bimodule maps.

Let us give some planar tangles and notations.
\begin{notation}\label{Notation:diagrams}
Usually we draw the input and output discs of a planar tangle as rectangles with the same number of boundary points on the top and the bottom, and the $\$$ sign on the left. Under this convention, we can omit the $\$$ signs and the output disc of a planar tangle.
\enumerate{
\item For two vectors $x\in\mathscr{P}_{m,\pm}$ , $y\in\mathscr{P}_{m',\pm}$, their tensor product $x\otimes y$ is $\grb{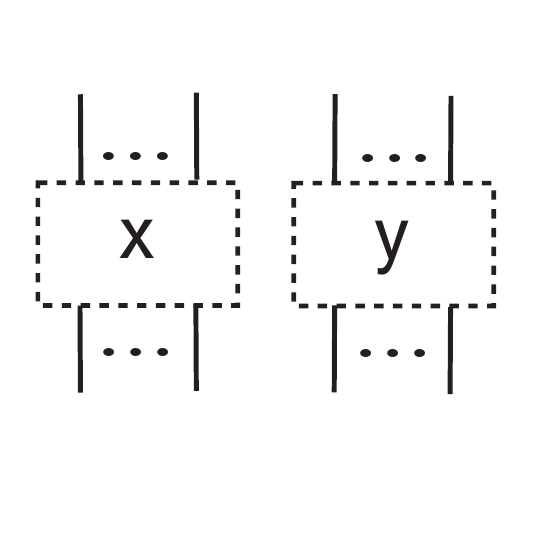}$.
\item the $k$th tensor power of $x$ is denoted by $x^{\otimes k}$.
\item When $m'=m$, their multiplication $xy$ is
$\grb{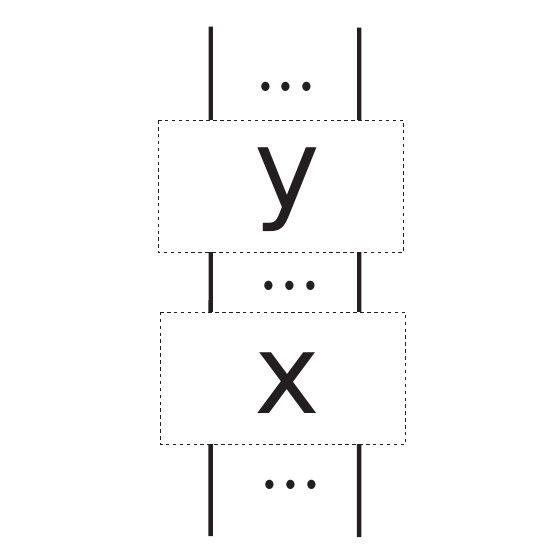}$.
\item The contragredient of $x$, i.e. the $180^\circ$ rotation of $x$, is denoted by $\overline{x}$.
\item The right (or left) side inclusion $\cdot \otimes 1$ (or $1 \otimes \cdot$) is adding a through string to the right (or left). We identify an $m$-box $x$ as the $(m+1)$-box $x\otimes 1$, if there is no confusion.
\item The Fourier transform $\mathcal{F}$ is given by the 1-click rotation $\graa{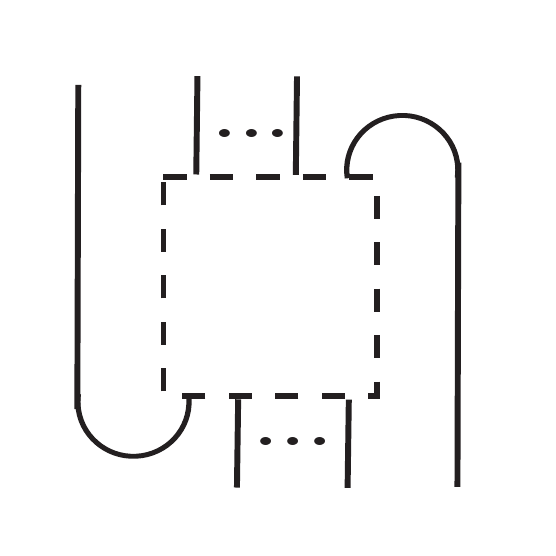}$.
\item A shift is a composition of left and right side inclusions, i.e.
\graa{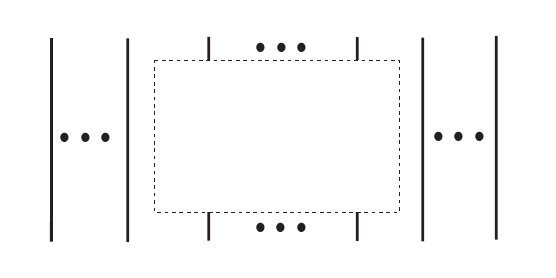}

\item The right side conditional expectation is adding a right cap to a (rectangle) diagram, i.e.
    \graa{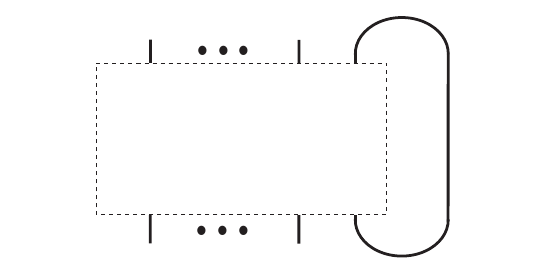}.
\item The Markov trace on the $m$-box space is defined by adding $m$ right caps to $m$-boxes, denoted by $tr_m$. We write $tr$ for $tr_2$.

\item The circle parameter of the planar algebra is denoted by $\delta$.

\item A cap is the diagram $\cap$; a cup is the diagram $\cup$. The multiplication of $\cap$ and $\cup$ is given by the 2-box \gra{22}. The multiplication of $\cup$ and $\cap$ is $\delta$.
\item We write a labeled 2-box as a crossing with the label located at the position of the $\$$.
}
\end{notation}

With the above notation, one can construct a unitary fusion category \cite{ENO} from a (finite depth) unshaded spherical planar algebra, see Section 4.1 in \cite{MPSD2n}. We will use this identification in Section \ref{section quantum subgroups}.

\begin{definition}
A diagram is called a standard multiplication form, if it is obtained by adding right caps to the multiplication of $n$-box shift tangles partially labeled by \gra{22}, e.g.
$$\gre{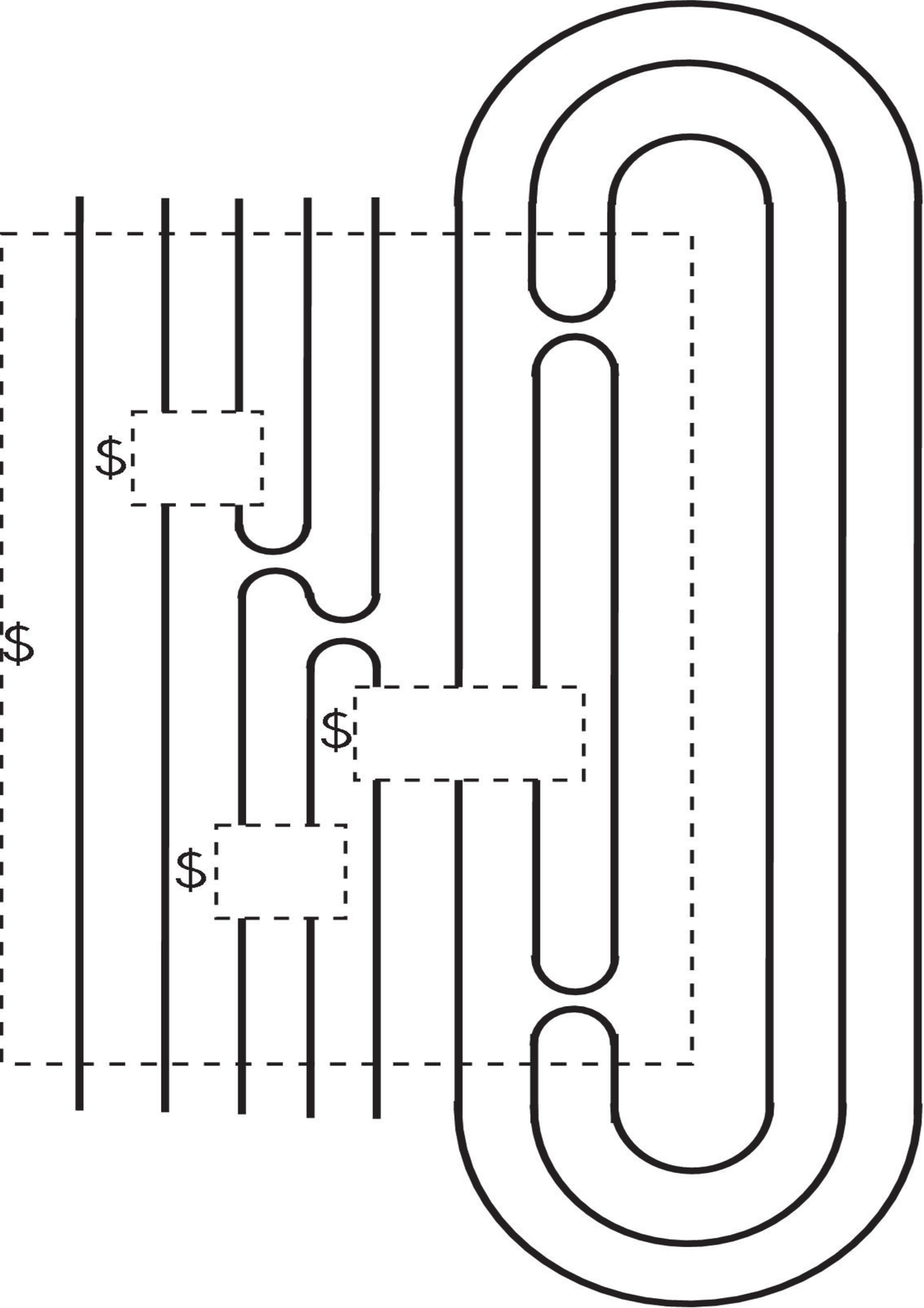}.$$
\end{definition}

\begin{proposition}\label{stdform}
Any (shaded) planar tangle is isotopic to a standard multiplication form by adding some closed circles.
\end{proposition}

\begin{proof}
Suppose the number of boundary points of the output disc of the planar tangle is $2m$. Now let us construct a diagram from the planar tangle as follows:
\begin{itemize}
\item[(1)] Draw each output disc and input disc of the planar tangle as a rectangle with the same number of boundary points on the top and the bottom, and a $\$$ sign on the left.
\item[(2)] Cut the tangle into pieces by pairs of ``horizontal" lines around input discs, such that the tangle between each pair of lines is a shift tangle. Note that the planar tangle admits alternating shading, so intersection points on each line has the same parity as $m$.

\item[(3)] Add circles over each pair of lines like
\grb{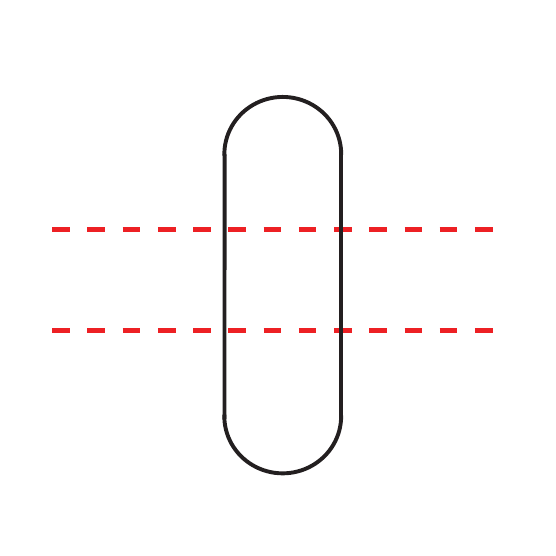},
such that every horizontal line passes through $n$ intersection points, for a fixed $n\geq m$.
\item[(4)] Add $\displaystyle \frac{n-m}{2}$ circles over the output disc: make up $\frac{n-m}{2}$ cups on the right top and caps on the right bottom and add $\displaystyle \frac{n-m}{2}$ pairs of right caps.
\end{itemize}
The final diagram is isotopic to the disjoint union of original tangle and some closed circles. Moreover, this diagram is a multiplication of $n$-box shift tangles and $n$-box Temperley-Lieb diagrams. Note that each $n$-box Temperley-Lieb diagram is a $n$-box shift tangle labeled by $\gra{22}$. Therefore the diagram is a standard multiplication form.
\end{proof}

Suppose $T$ is a (shaded) planar tangle. Let us draw each output disc and input disc of $T$ as a rectangle with the same number of boundary points on the top and the bottom, and a $\$$ sign on the left. Moreover, the strings and the discs intersect orthogonally. Let the strings move along shaded regions counterclockwise, and the total winding angle is denoted by $\theta$. By Proposition \ref{stdform}, $\theta=2m\pi$, for some $m\in\mathbb{Z}$. Let us define $g(T)=e^{\frac{i\theta}{2}}=(-1)^m$. If we switch the shading of $T$, then the winding number is $-2m\pi$ and $g(T)$ does not change.

More precisely, if $S$ is a standard form isotopic to the disjoint union of $T$ and $k_1$ closed circles. Then the total winding angle of $S$ is the sum of those of $T$ and $k_1$ closed circles. Let the number of labels \gra{22} and right caps in $S$ be $k_1$ and $k_2$, respective. Note that the winding angle of a closed circle, a right cap or \gra{22}, is $\pm2\pi$. They all contribute $-1$ to $g(T)$, thus $g(T)=(-1)^{k_1+k_2+k_3}$.

\begin{definition}[Gauge transformation]
Given a (shaded) planar algebra $\mathscr{P}_{\bullet}$ with the circle parameter $\delta$, we obtain a (shaded) planar algebra $\mathscr{P}^-_{\bullet}$ with the circle parameter $-\delta$ by changing the action $Z_T$ of a planar tangle $T$ to $Z^-_T:=g(T)Z_T$, called the gauge transformation of $\mathscr{P}_{\bullet}$.
Furthermore, if $\mathscr{P}_{\bullet}$ is unshaded, then $\mathscr{P}^-_{\bullet}$ is unshaded.
\end{definition}

\subsection{Skein theory}\label{subsection skein theory}
We refer the reader to \cite{JonPA} for the skein theory of planar algebras (in Section 1) and many interesting examples (in Section 2).

Given a generating set $S$, one intermediately obtains a planar algebra $\mathscr{U}(S)$, called the \emph{universal planar algebra} (Definition 1.10 \cite{JonPA}). Its vector space $\mathscr{U}(S)_{m,\pm}$ is the linear span of labeled $2m$-tangles, i.e., planar tangles with $2m$ boundary points whose input discs are labeled by elements in the generating set. Planar tangles act on $\mathscr{U}(S)$ in the natural way.
The partition function $Z$ is a linear functional on the 0-box space.
The kernel of the partition function $\bigcup_{m,\pm}\{x\in \mathscr{U}(S)_{m,\pm} |Z(tr_m(xy))=0, \forall y\in \mathscr{U}(S)_{m,\pm}\}$ is an ideal of $\mathscr{U}(S)$.
Modulo this ideal, the action of planar tangles is well-defined on the quotient of the universal planar algebra. If the partition function is multiplicative, then the quotient is unital. If $S$ have an involution $*$, then $\mathscr{U}(S)$ is a planar *-algebra. The partition function should satisfy the condition $Z(B^*)=\overline{Z(B)}$, for any 0-box $B$, so that the involution is well defined on the quotient.

The main difficulty is defining a \underline{positive} partition function for the universal planar algebra, i.e., the Markov trace is positive semidefinite with respect to an involution $*$.
The strategy of skein theory is to derive a partition function by a proper set of relations of the universal planar algebra.
We will encounter three fundamental problems:

\begin{itemize}
\item[(1)] Evaluation: Enough relations are required, such that the $0$-box space is reduced to a scalar in the ground field. (Usually, we also require that the $n$-box space is reduced to be finite dimensional.) Then we may define the partition function of the 0-box as this scalar.

\item[(2)] Consistency: Different evaluation processes a $0$-box give the same scalar. Then the partition function is well defined.

\item[(3)] Positivity: The partition function derived from the evaluation is positive semi-definite with respect to an involution. Then the quotient of the universal planar algebra by the null space of the partition function is a subfactor planar algebra.
\end{itemize}

\begin{remark}
If the consistency holds, then kernel of the partition function contains the ideal generated by the relations, but not necessarily equal. It is hard to find out the extra relations in the kernel.
In other words, the planar algebra defined by generators modulo relations may be non-semisimple. One needs to determined the relevant semisimple quotient. We refer the reader to \cite{Wen88} for the case of Brauer's centralizer algebras.
\end{remark}

Now let us talk about one example for the three problems. (Example 2.2 \cite{JonPA}.)

If the generating set is given by 1-boxes with the same shading, then one can introduce relations to reduce 1-boxes to the ground field:
$\gra{lefttrace}=\gra{righttrace}=f(x),$
where $f$ a linear functional on the linear span these 1-boxes.
Furthermore, one can introduce relations to reduce two adjacent 1-box generators $x_i$, $x_j$ to a linear sum of 1-boxes generators, i.e., a multiplication $x_ix_j=\sum_k c_{ij}^k x_k$, for some formal parameters $c_{ij}^k$.

These relations are enough to reduce any 0-box to a polynomial over the parameters $f(x)$, $c_{ij}^k$.

Consistency is equivalent to a set of equations over the parameters which means that the multiplication is associative and $f$ is a trace.

Positivity was proved in Theorem 3.16 in \cite{JonPA}. Positivity is already hard for Temperley-Lieb-Jones planar algebra which has neither generators nor relations. This was achieved in Jones' remarkable rigidity result \cite{Jon83}:
Positivity holds iff the circle parameter $\delta$ belongs to $$\{2\cos\frac{\pi}{n}, n=3,4,\cdots\}\cup[2,\infty].$$

In general, for a given set of generators, one need to find out a type of relations based on formal parameters such that any 0-box can be evaluated. Then the planar algebra is completely determined by these parameters. One can ask for a \underline{classification} of subfactor planar algebras \underline{by} this type of \underline{skein theory}. The consistency is the constraint of these parameters. The positivity is a stronger constraint, but rarely used in the skein theoretic classification. See \cite{Liuex,BJL,JLW16} for an application of the positivity in the study of subfactors from the point of view of harmonic analysis.

\begin{definition}
We define the complexity for a universal planar algebra to be a map from labeled tangles to the partially ordered set $\mathbb{N}^m$, for some $m\in \mathbb{N}$. The complexity is called local, if the order is strictly preserved under the action of labeled annular tangles.
\end{definition}

\begin{definition}
Given a finite set of generators and relations of a planar algebra and a complexity, an evaluation algorithm is called local or $n$-local, for $n\in\mathbb{N}$, if the complexity is local and
\itemize{
\item[(1)] for any $k\leq 2n$, there are finitely many least complex labeled $k$-tangles, (i.e., tangles can not be reduced to linear sums of less complex tangles;)
\item[(2)] any non-least complex labeled $k$-tangle can be reduced to a linear sum of less complex tangles by applying the relations to an at most $n$-box part of the $k$-tangle.

\item[(3)] the empty diagram is the unique least complex labeled 0-tangle.
}

Otherwise the evaluation algorithm is called global.
\end{definition}

For example, the number of labels of a tangle is a local complexity. One can find interesting local evaluation algorithms by the discharge method, see \cite{BJL,MPS15} for 3 or 4-valent graphs.

For an evaluation algorithm, there is a supremum for the dimension of the $m$-box space. We call this supremum the {\it critical dimension}.
The critical dimension for the 3-box space is 14 for the local evaluation algorithm of 4-valent graphs in \cite{BJL}.
One needs an global evaluation algorithm to go beyond the critical dimension of the local evaluation algorithm.
In Section \ref{YBR}, we give a global evaluation algorithm for 4-valent graphs based on the Yang-Baxter relation.
One known global evaluation algorithm is given by Thurston for 6-valent graphs \cite{Thu04}. Another global one is the Jellyfish evaluation algorithm \cite{BMPS}.

Now let us give a general method to prove the consistency for a local evaluation algorithms based on solving polynomial equations. Of course, solving these equations could be hard. The method will not give a proof of the consistency of our global evaluation algorithm, but the idea will be used (see Section \ref{subsection consistency}).

\textbf{Consistency for a local evaluation algorithm:}
By condition (2), let us define the partition function as the average of all complexity reducing evaluations. We can prove the consistency by showing that the relations are in the kernel of the partition function.

We want to prove the following statement by induction based on solving polynomial equations: given a $m$-box relation $R$ and an labeled annular tangle $\Phi$ mapping from the $m$-box space to the $0$-box space, the partition function of $\Phi(R)$ is 0.

If the complexity of $\Phi(R)$ is minimal, then by condition (3), $\Phi$ is empty and $R=0$. So $\Phi(R)=0$.

We assume that the statement is true for any $\Phi'(R')$ whose complexity is less than that of $\Phi(R)$. Let us compute $\Phi(R)$. For a complexity reducing evaluation, if $\Phi(R)$ is reduced by applying the relations to a part in $\Phi$, then this evaluation is zero by induction. If $\Phi(R)$ is reduced by applying the relations to a part in $\Phi(R)$ which overlaps $R$, then the union of this part and $R$ is in a $k$-box space, $k\leq m+n$. The union can be reduced to a linear sum of least complex tangles in the $k$-box space by condition (2). By induction, it is enough to show that the polynomial coefficients of these least complex tangles are zero.

\subsection{The Hecke algebra of type A and the HOMFLY-PT polynomial}\label{hecke}
Let us recall some results about the Hecke algebra of type A which will be used in Section \ref{subsection matrix units}, \ref{subsection trace formula}, \ref{subsection positivity}, \ref{section quantum subgroups}. The HOMFLY-PT polynomial will be used in Section \ref{subsection consistency}.

The HOMFLY-PT polynomial is a link invariant given by a braid $\gra{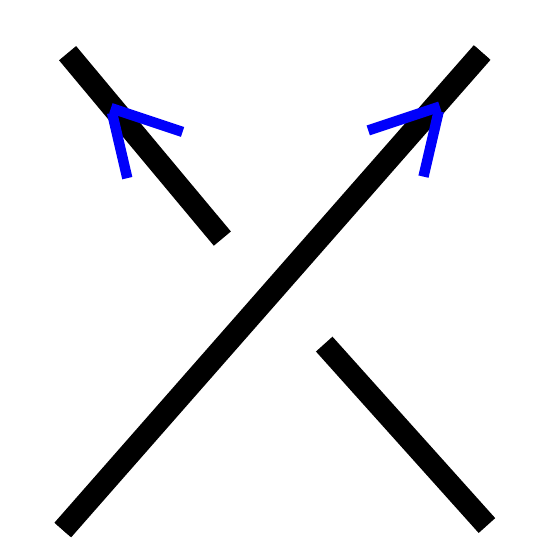}$ satisfying Reidemeister moves I, II, III and the Hecke relation.

\begin{align*}
&\text{the Hecke relation:}    & \gra{b+}-\gra{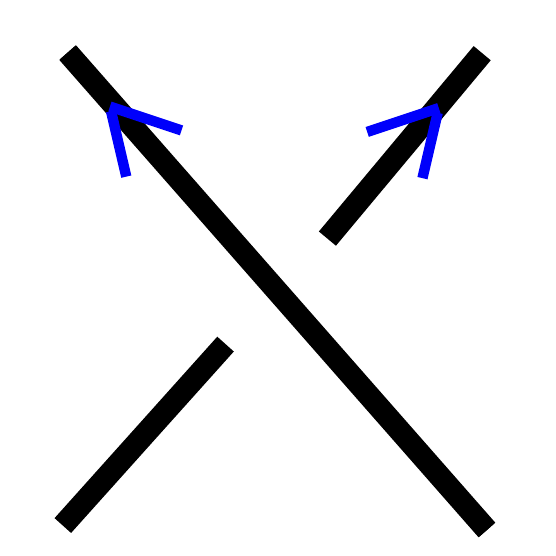}&=(q-q^{-1})\gra{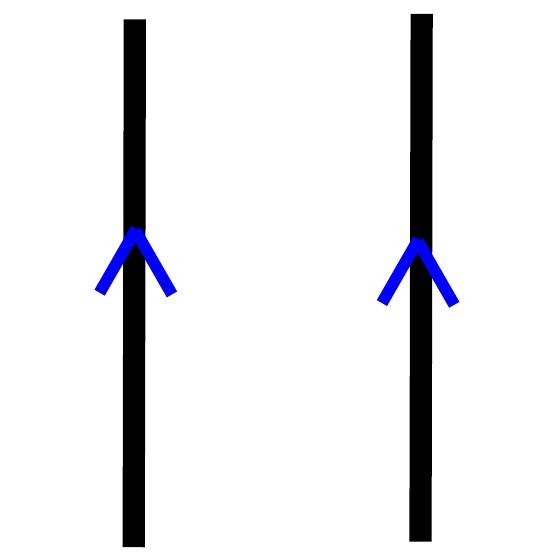}, && \\
&\text{Reidemeister moves I:}   & \gra{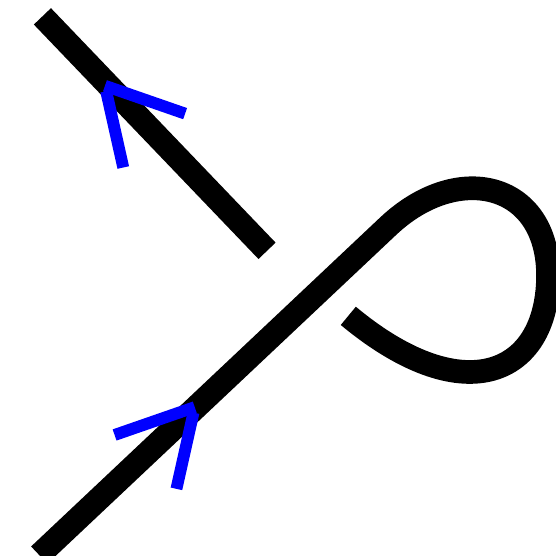}&=r\gra{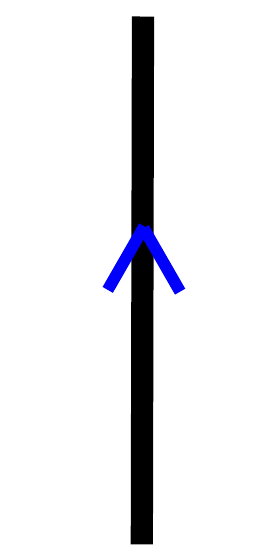}; & \gra{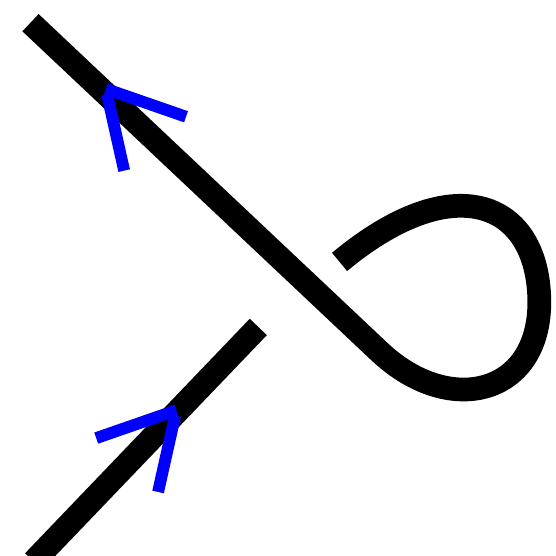}&=r^{-1}\gra{b1id}; \\
                                && \gra{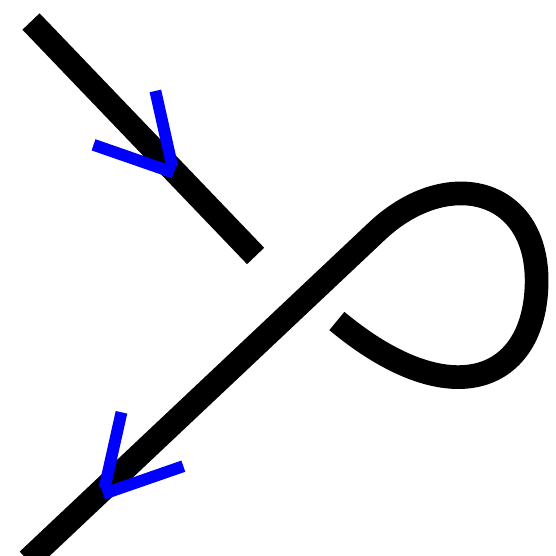}&=r \gra{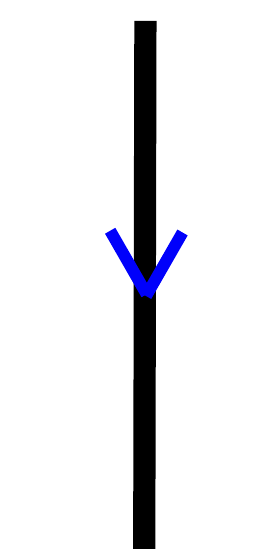}; & \gra{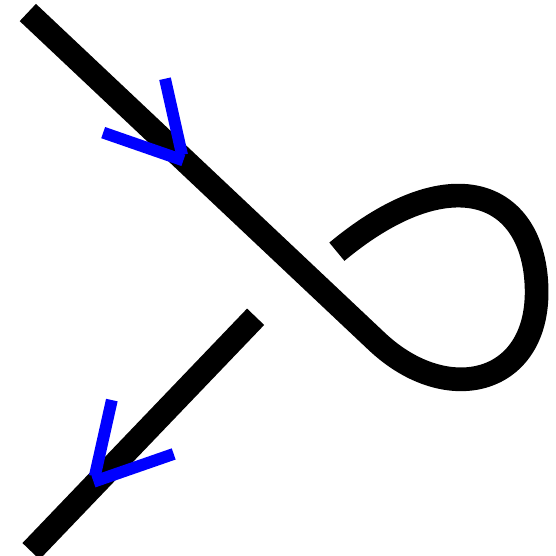}&=r^{-1}\gra{b1id-},\\
&\text{Reidemeister moves II:}  & \gra{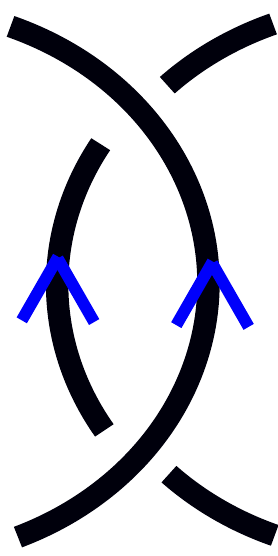}&=\gra{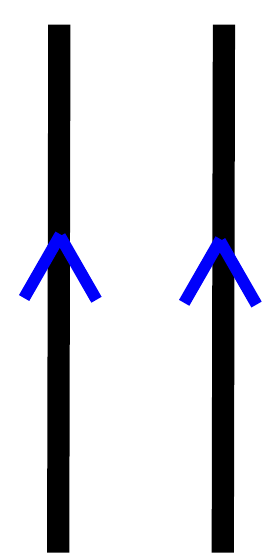}; &  \gra{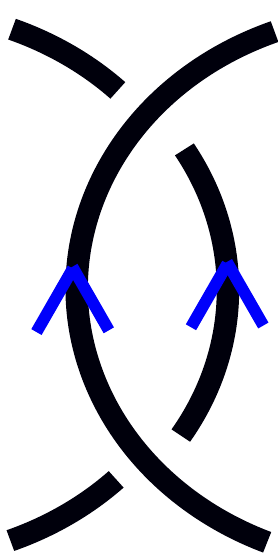}&=\gra{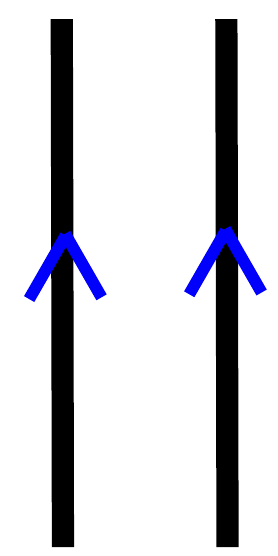}; \\
                                && \gra{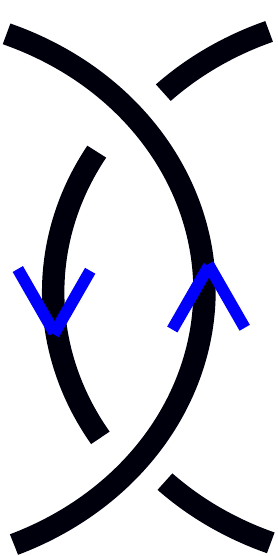}&=\gra{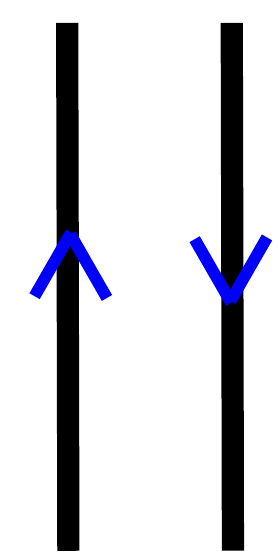}; &  \gra{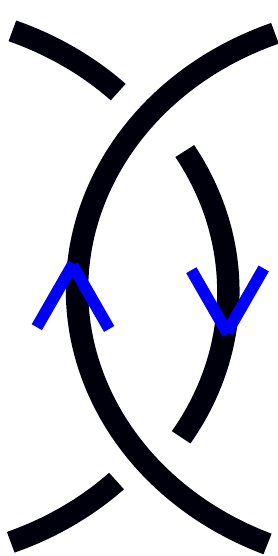}&=\gra{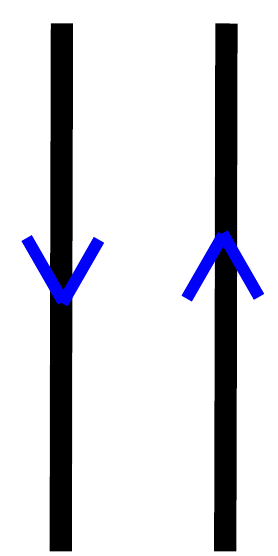},\\
&\text{Reidemeister moves III:}  & \gra{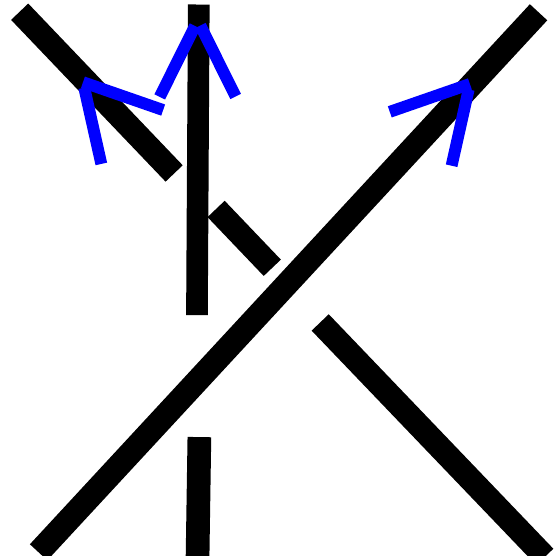}&=\gra{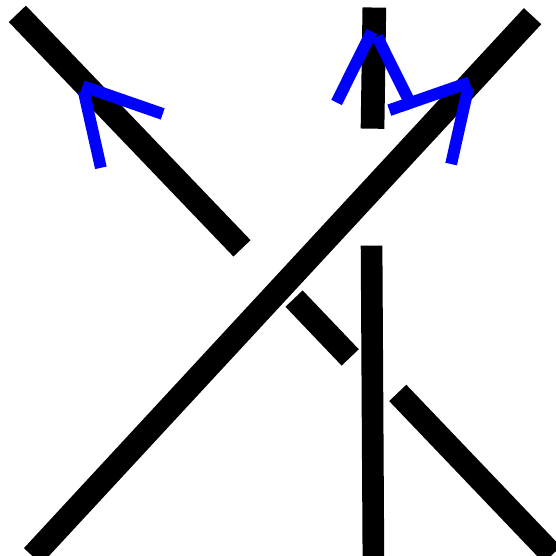}; &  \gra{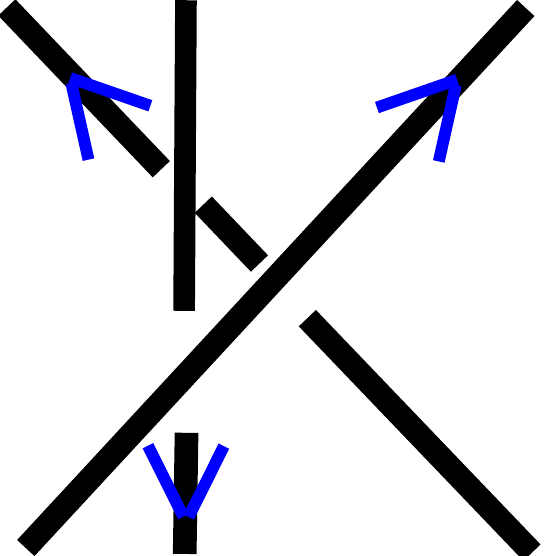}&=\gra{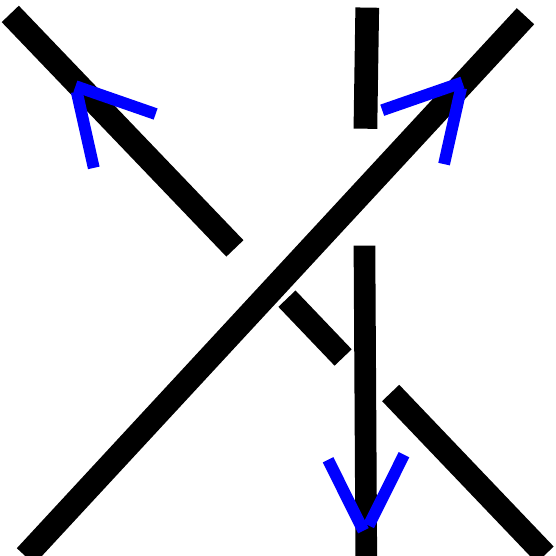},\\
&\text{circle parameter:} && \gra{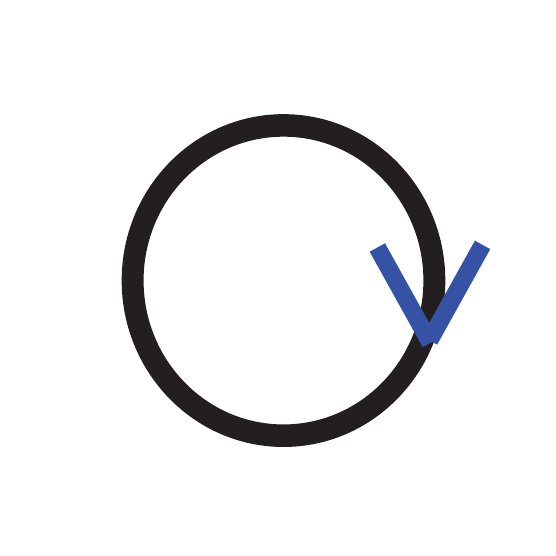}=\gra{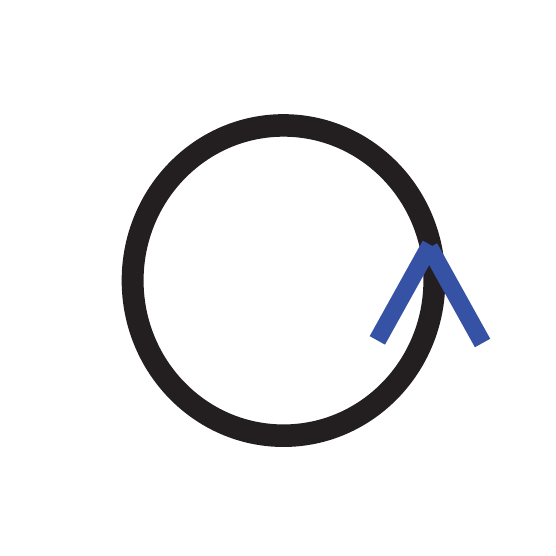}=\frac{r-r^{-1}}{q-q^{-1}}.
\end{align*}

Let $\sigma_i$, $i\geq1$, be the diagram by adding $i-1$ oriented (from bottom to top) through strings on the left of $\gra{b+}$.
The Hecke algebra of type $A$ is a (unital) filtered algebra $H_{\bullet}$. The algebra $H_n$ is generated by $\sigma_i$, $1\leq i \leq n-1$ and $H_n$ is identified as a subalgebra of $H_{n+1}$ by adding an oriented through string on the right.
Over the field $\mathbb{C}(r,q)$, rational functions over $r$ and $q$, the matrix units of $H_{\bullet}$ were constructed in \cite{Yok97,Ais98}.  A skein theoretic proof of the trace formula via the $q$-Murphy operator was given in \cite{Ais97}.

For readers' convenience, let us sketch the construction of the matrix units in \cite{Yok97} with slightly different notations. The ($l$-box) symmetrizer $f^{(l)}$ and antisymmetrizer $g^{(l)}$, for $l\geq1$, are constructed inductively as follows,
\begin{align}
f^{(l)}&=f^{(l-1)}-\frac{[l-1]}{[l]}f^{(l-1)}(q-\sigma_{l-1})f^{(l-1)}; \label{fl}\\
g^{(l)}&=g^{(l-1)}-\frac{[l-1]}{[l]}g^{(l-1)}(q^{-1}+\sigma_{l-1})g^{(l-1)},\label{gl}
\end{align}
where $f^{(1)}=g^{(1)}=1$.

\begin{proposition}
\begin{align}
f^{(l)}&=1\otimes f^{(l-1)}-\frac{[l-1]}{[l]}(1\otimes f^{(l-1)})(q-\sigma_1)(1\otimes f^{(l-1)}); \label{fl1}\\
g^{(l)}&=1\otimes g^{(l-1)}-\frac{[l-1]}{[l]}(1\otimes g^{(l-1)})(q^{-1}+\sigma_1)(1\otimes g^{(l-1)}).  \label{gll}
\end{align}
\end{proposition}

\begin{proof}
Take $u_1=1$ and $\displaystyle u_l=\prod_{i=1}^{l-1} \sigma_i$, for $l\geq 2$. Then $u_l$ is a unitary.
By Reidemeister moves of $\gra{b+}$, we have
\begin{align}
  u_l(f_{l-1})f_l^*&=1\otimes f_{l-1}. \label{equbraidf2}
\end{align}

Since $f_l$ is a central minimal idempotent in the Hecke algebra $H_{l}$, we have $u_lf_l=\chi_l f_l$, for some $\chi$, $|\chi_l|=1$. Then $f_lu_l^*=\overline{\chi_l}f_l$.
Thus
\begin{align*}
  &u_lf_{l-1}\sigma_{l-1}f_{l-1}u_l^*\\
  =&u_lf_{l-1}u_{l-1}\sigma_{l-1}u_{l-1}^*f_{l-1}u_l^*\\
  =&(1\otimes f_{l-1})\sigma_1 (1\otimes f_{l-1}) && \text{by Reidemeister moves of \gra{b+}}. \numberthis \label{equbraidf1}
\end{align*}
Therefore
\begin{align*}
  f^{(l)}&=u_lf_lu_l^*\\
         &=u_l\left(f^{(l-1)}-\frac{[l-1]}{[l]}f^{(l-1)}(q-\sigma_{l-1})f^{(l-1)}\right)u_l^* && \text{by Equation (\ref{fl}),}\\
         &=1\otimes f^{(l-1)}-\frac{[l-1]}{[l]}(1\otimes f^{(l-1)})(q-\sigma_1)(1\otimes f^{(l-1)}) && \text{by Equations (\ref{equbraidf1}), (\ref{equbraidf2}}). \numberthis \label{fll}
\end{align*}

We can prove Equation (\ref{gll}) in a similar way.
\end{proof}

Given a Young diagram $\lambda$, we can construct an idempotent by inserting the symmetrizers in each row on the top and the bottom and the antisymmetrizers in each column in the middle as follows.
For example, $\lambda=\graa{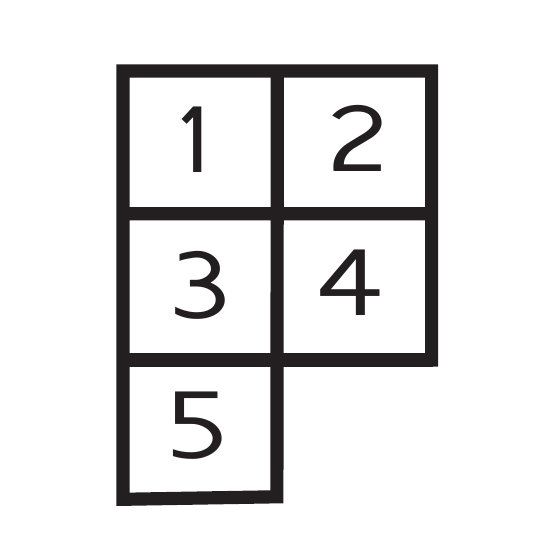}$, take
\begin{align}
\dot{y}_{\lambda}&=\grc{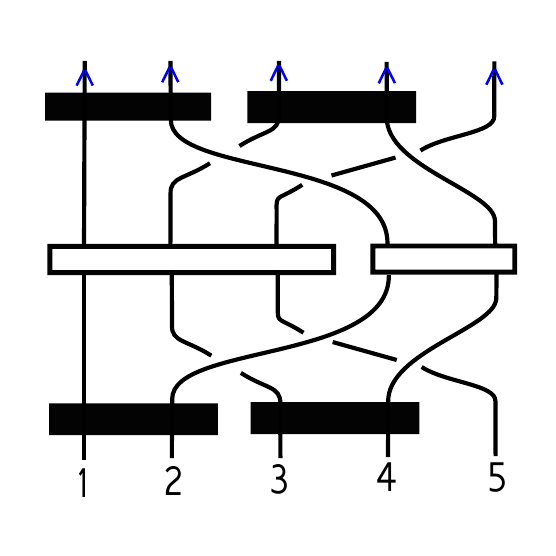}, \label{equyoungidempotent}
\end{align}
where the black boxes and white boxes indicate symmetrizers and antisymmetrizers respectively.
Then $\dot{y}_{\lambda}^2=m_{\lambda}\dot{y}_{\lambda}$. The coefficent $m_{\lambda}$ was computed in Proposition 2.2 in \cite{Yok97}. Over $\mathbb{C}(q,r)$, $m_{\lambda}$ is non-zero. We can normalize $\dot{y}_{\lambda}$ to $y_{\lambda}$ by
$\displaystyle y_{\lambda}=\frac{1}{m_{\lambda}}\dot{y}_{\lambda}$. Then $y_{\lambda}$ is an idempotent. Moreover, $\{y_{\lambda} ~|~ |\lambda|=n\}$ are inequivalent minimal idempotents in $H_n$.

\begin{notation}\label{Notation:Young diagrams}
  If a Young diagram $\lambda$ is obtained by adding one cell to a Young diagram $\mu$, then we say $\lambda>\mu$ or $\mu<\lambda$.
\end{notation}

For $\lambda>\mu$, the morphisms $\dot{\rho}_{\mu<\lambda}$ from $y_{\mu} \otimes 1$ to $y_\lambda$ and $\dot{\rho}_{\lambda>\rho}$ from $y_\lambda$ to $y_{\mu} \otimes 1$ were constructed in Lemma 2.10 in \cite{Yok97}.
Moreover, $(\dot{\rho}_{\mu<\lambda}\dot{\rho}_{\lambda>\rho})^2=m_{[\mu|\lambda|\mu]}\dot{\rho}_{\mu<\lambda}\dot{\rho}_{\lambda>\rho}$ and the coefficient $m_{[\mu|\lambda|\mu]}$ was also computed there. Over $\mathbb{C}(q,r)$, $m_{[\mu|\lambda|\mu]}$ is non-zero. We normalize $\dot{\rho}_{\mu<\lambda}$ and $\dot{\rho}_{\lambda>\rho}$ by
$\displaystyle \rho_{\mu<\lambda}=\frac{1}{m_{[\mu|\lambda|\mu]}}\dot{\rho}_{\mu<\lambda}$ and $\rho_{\lambda>\rho}=\dot{\rho}_{\lambda>\rho}$. Then
$\rho_{\mu<\lambda}\rho_{\lambda>\rho}$ is an idempotent and $\rho_{\lambda>\rho}\rho_{\mu<\lambda}=y_\lambda$. The branching formula is proved in Proposition 2.11 in \cite{Yok97},
\begin{align}\label{branching formula}
y_{\mu}\otimes 1&=\sum_{\lambda>\mu} \rho_{\mu<\lambda}\rho_{\lambda>\mu}.
\end{align}
Therefore the Bratteli diagram of $H_{\bullet}$ over $\mathbb{C}_{q,r}$ is the Young's lattice, denoted by $YL$.

For each length $n$ path $t$ in $YL$ from $\emptyset$ to $\lambda$, $|\lambda|=n$, $n\geq1$, i.e., a standard tableau $t$ of the Young diagram $\lambda$,
take $t'$ to be the first length $(n-1)$ sub path of $t$ from $\emptyset$ to $\mu$.
There are two elements $P^+_t$, $P^-_t$ in $H_n$ defined by the following inductive process,
\begin{align*}
P^{\pm}_\emptyset&=\emptyset, \\
P^+_t&=(P^+_{t'}\otimes 1)\rho_{\mu < \lambda} ,\\
P^-_t&= \rho_{\lambda>\mu} (P^-_{t'}\otimes 1).
\end{align*}
The matrix units of $H_{n}$ are given by $P^+_tP^-_\tau$, for all Young diagrams $\lambda$, $|\lambda|=n$, and all pairs of length $n$ paths $(t,\tau)$ in $YL$ from $\emptyset$ to $\lambda$.
Moreover, the multiplication of these matrix units coincides with the multiplication of loops, i.e.,
\begin{align*}
P^+_tP^-_\tau P^+_sP^-_\sigma=\delta_{\tau s} P^+_tP^-_\sigma,
\end{align*}
where $\delta_{\tau s}$ is the Kronecker delta.

Furthermore, when $|q|=|r|=1$, $H_{\bullet}$ admits an involution $*$, which is an anti-linear anti-isomorphism mapping $\gra{b-}$ to $\gra{b-}$, ($q$ to $q^{-1}$ and $r^{-1}$ to $r^{-1}$) over the field $\mathbb{C}$.
The symmetrizer $f^{(l)}$ and antisymmetrizer $g^{(l)}$ can be constructed by Equation (\ref{fl}) and (\ref{gl}) inductively whenever $[l]\neq 0$.
Note that $[l]^*=[l]$. By the Hecke relation of $\gra{b+}$, we have $(q-\sigma_i)^*=q-\sigma_i$.
So $(f^{(l)})^*=f^{(l)}$ and $(g^{(l)})^*=g^{(l)}$ by the inductive construction.
Then $y_{\lambda}$ can be constructed if the required symmetrizers and antisymmetrizers are well-defined and $m_{\lambda}\neq0$.
For $\lambda>\mu$, $\dot{\rho}_{\lambda>\rho}$ and $\dot{\rho}_{\mu<\lambda}$ can be constructed if $y_{\lambda}$ and $y_{\mu}$ are well-defined.
If $m_{[\mu|\lambda|\mu]}> 0$, then we have (different) normalization
$\displaystyle \rho'_{\mu<\lambda}=\sqrt{\frac{1}{m_{[\mu|\lambda|\mu]}}}\dot{\rho}_{\mu<\lambda}$ and $\displaystyle \rho'_{\lambda>\rho}=\sqrt{\frac{1}{m_{[\mu|\lambda|\mu]}}}\dot{\rho}_{\lambda>\rho}$.
By this normalization process (which is permitted over $\mathbb{C}$, but not over $\mathbb{C}(q,r)$), we have
$(\rho_{\mu<\lambda}')^*=\rho_{\lambda>\rho}'.$
Similarly we can define the matrix unit $P^+_tP^-_\tau$ for a loop $t\tau^{-1}$ when the morphisms along the paths $t$ and $\tau$ are defined.
Then $(P^+_tP^-_\tau)^*=P^+_\tau P^-_t$.

We will construct a $q$-parameterized planar algebra $\mathscr{C}_{\bullet}$ in Section \ref{subsection consistency}.
We will use the matrix units of $H_{\bullet}$ to construct the matrix units of $\mathscr{C}_{\bullet}$ in Section \ref{subsection matrix units}.

We will determine the semisimple quotient of $\mathscr{C}_{\bullet}$ for the case $q=e^{\frac{i\pi}{2N+2}}, r=q^N$ in Section \ref{subsection positivity}.
For all Young diagrams whose (1,1) cell has hook length at most $N+1$, one can check that all the corresponding coefficients $[l]$, $m_{\lambda}$, $m_{[\mu|\lambda|\mu]}$ are positive.
So the corresponding minimal idempotents $y_{\lambda}$ and morphisms $\rho_{\mu<\lambda}$, $\rho_{\lambda>\rho}$ are well-defined.
We will construct the null space ideal of $\mathscr{C}_{\bullet}$ and matrix units of the semisimple quotient modulo the ideal by these well-defined matrix units of $H_{\bullet}$.
These quotients are subfactor planar algebras in (3) of Theorem \ref{main theorem classification}.

\section{Yang-Baxter relation planar algebras}\label{YBR}

\subsection{General theory}
\begin{definition}
For a planar algebra $\mathscr{P}_{\bullet}$ over a field $k$, if
for any 2-boxes $R_i, R_j, R_k$ with compatible shading, we have
$$(1\otimes R_i)(R_j\otimes 1)(1\otimes R_k)= \sum_{\substack{i',j',k'}} c_{i,j,k}^{i',j',k'} (R_{k'}\otimes 1)(1\otimes R_{j'})(R_{i'}\otimes 1),$$
for some $c_{i,j,k}^{i',j',k'}\in k$ and 2-boxes $R_{i'}, R_{j'}, R_{k'}$, then its 2-box space is said to have a Yang-Baxter relation.
\end{definition}

\begin{definition}
We call $\mathscr{P}$ a Yang-Baxter relation planar algebra, if $k=\mathbb{C}$ and $\mathscr{P}$ is a subfactor planar algebra generated by its 2-boxes with a Yang-Baxter relation.
\end{definition}

\begin{remark}
There are two different kinds of equations due to the two choices of shadings. This behavior is the same as the star-triangle equation for checkerboard lattice models \cite{Ons44}.
\end{remark}
The diagrammatic interpretation of the Yang-Baxter relation is
$$\grb{par4}= \sum_{\substack{i',j',k'}} c_{i,j,k}^{i',j',k'} \grb{par5}.$$
It can be considered as a generalization of the Reidemester move of type III.

The generalized Reidemester moves of type I, II are intrinsic in the action of planar tangles as follows.
Given a basis $s_i$ and $x_j$ for 1-boxes and 2-boxes respectively, we have the following relations.

Move I: For any 2-box $x_j$,
$$
\raisebox{-.47cm}{
\tikz{
\draw (0,0)--(2/3,2/3);
\draw (2/3,2/3) arc (135:-135: 1.414/6);
\draw (0,1)--(2/3,1/3);
\node at (1/6,1/2) {$x_j$};
}}
=\sum_{\substack{i}} b_{j}^{i}
\raisebox{-.47cm}{
\tikz{
\draw (0,0)--(0,1);
\node at (-1/3,1/2) {$s_i$};
}},
$$
for some constants $b_{j}^{i}$.

Move I': For any 1-box $s_i$,
$$
\raisebox{-.47cm}{
\tikz{
\draw (0,0) arc (180:-180: 1/2);
\node at (-1/3,0) {$s_i$};
}}
=b_i,
$$
for some constant $b_i$.

Move II: For any 2-boxes $x_j$, $x_k$ with compatible shading, we have
$$
\raisebox{-.95cm}{
\tikz{
\draw (0,0)--(1,1)--(0,2);
\draw (1,0)--(0,1)--(1,2);
\node at (1/6,1/2) {$x_j$};
\node at (1/6,1+1/2) {$x_k$};
}}
=\sum_{\substack{l}} b_{jk}^l
\raisebox{-.47cm}{
\tikz{
\draw (0,0)--(1,1);
\draw (1,0)--(0,1);
\node at (1/6,1/2) {$x_l$};
}},
$$
for and constants $b_{jk}^l$.

Move II': For any 1-box $s_i$ and 2-box $x_j$,
$$
\raisebox{-.95cm}{
\tikz{
\draw (0,0)--(1,1)--(1,2);
\draw (1,0)--(0,1)--(0,2);
\node at (1/6,1/2) {$x_j$};
\node at (-1/3,1+1/2) {$s_j$};
}}
=\sum_{\substack{k}} b(1)_{ij}^k
\raisebox{-.47cm}{
\tikz{
\draw (0,0)--(1,1);
\draw (1,0)--(0,1);
\node at (1/6,1/2) {$x_k$};
}},
$$
for and constants $b(1)_{ij}^k$. Similarly, we have other relations and constants $b(l)_{ij}^k$, $l=1,2,3,4$, when $s_i$ and $x_j$ are connected at four different boundary points.

\begin{notation}
Given bases of the 1-box space and the 2-box space, we call the coefficients arising from above moves I, I', II, II', the structure constants of 2-boxes.
\end{notation}

A Yang-Baxter relation planar algebra is not determined by the structure constants of 2-boxes, but it is determined by the structure constants of 2-boxes and the duality coefficients of the Yang-Baxter relation.

\begin{theorem}[Evaluation]\label{evaluable}
If a planar algebra is generated by its 2-box space with a Yang-Baxter relation, then it is evaluable by the type I, II, III moves of 2-boxes.
Consequently, the planar algebra is determined by the structure constants of 2-boxes and the duality coefficients of the Yang-Baxter relation.
\end{theorem}

We use the standard multiplication form to describe the complexity of $m$-tangles and evaluate $m$-tangles by generalized type I, II, III moves of 2-boxes.
\begin{proof}
Note that any vector is a finite linear sum of labeled tangles. By Proposition \ref{stdform}, we may assume that these tangles are standard multiplication forms.
For each diagram, when we ignore the right caps and view the Temperley-Lieb-Jones 2-boxes as generators, it is a multiplication of shifts of 2-box generators.
Similar to Alexander's argument \cite{Ale23},
applying type II and III moves,
the multiplication part could be replaced by a linear sum of multiplications of shifts of generators with only one generator on the right most. If there is a cap on the right in the standard multiplication form, then it acts on the rightmost generator. By type I move, the cap is reduced. Repeating this process, we reduce all the right caps. Therefore the vector is reduced to a linear sum of multiplications of shifts of generators.
By the above process, we can assume that there is at most one generator on the rightmost. By induction, we have that the $m$-box space is reduced to be finite dimensional, and the $0$-box space is at most one dimensional, i.e., the planar algebra is evaluable.
\end{proof}

From the above proof, we have
\begin{proposition}\label{algebragenerated}
If a planar algebra is generated by its 2-box space with a Yang-Baxter relation, then its $m$-box space, $m\geq1$, is generated by shifts of 1-boxes and 2-boxes as an algebra.
\end{proposition}

\subsection{Examples}\label{subsection example}
We give some examples of Yang-Baxter relation planar algebras.

(1) The first ``exotic" subfactor was constructed by Haagerup in \cite{Haa94}.
Its even part is an unshaded subfactor planar algebra with 4 dimensional 2-box space. By a direct computation, one can show that there is no non-zero solution of the parameter-independent Yang-Baxter equation in its 2-box space.
We prove that the planar algebra of the even part of the Haagerup subfactor is a Yang-Baxter relation planar algebra in \cite{LiuPen}.
The proof also works for generalized Haagerup subfactors $3^G$ \cite{Izu93}, for ordered groups $G$.

\begin{remark}
Peters gave a skein relation for the planar algebra of the Haagerup subfactor \cite{Pet10}.
\end{remark}

(2) Another example is the Birman-Murakami-Wenzl (BMW) planar algebra \cite{BirWen,Mur87}. It is generated by a three dimensional 2-boxes $\text{span}\{\gra{21},\gra{22},\gra{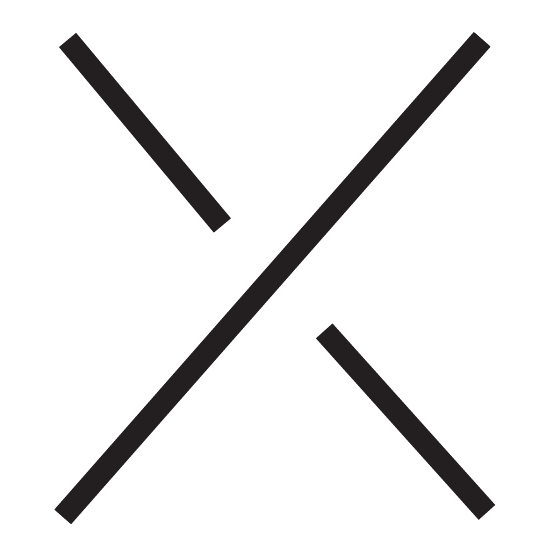}\}$.
The element $\gra{24}$ is the universal $R$ matrix for quantum groups $O(N)$ and $Sp(2N)$. The generator $\gra{24}$ satisfies the following relations,

\begin{align*}
&\text{the BMW relation:}    && \gra{24}-\gra{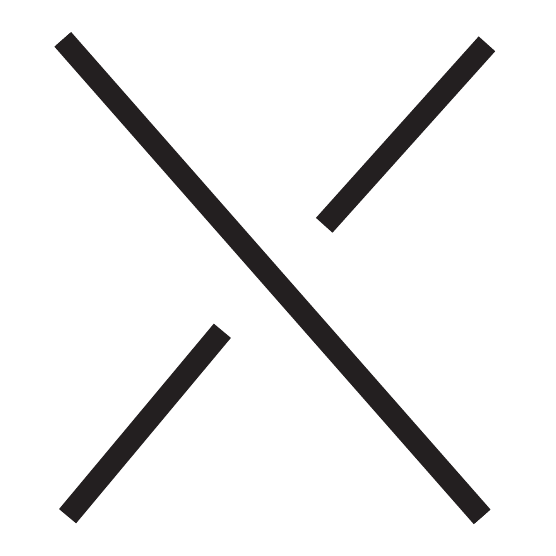}=(q-q^{-1})(\gra{21}-\gra{22}), \\
&\text{Reidemeister moves I:}   && \gra{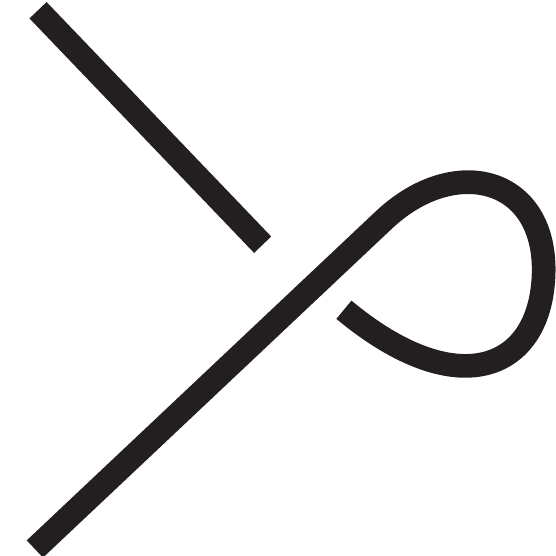}=r\gra{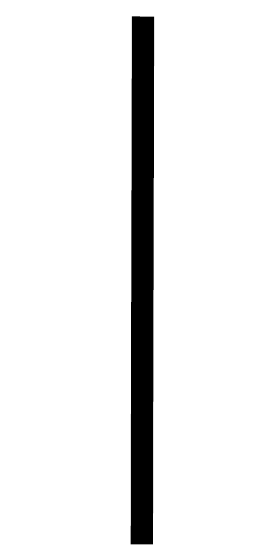}; \quad \gra{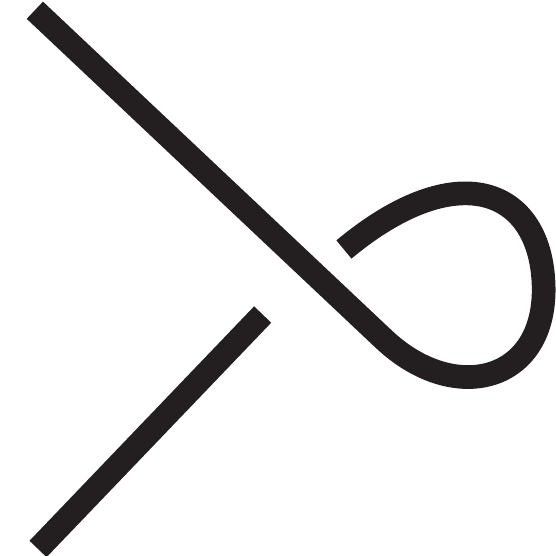}=r^{-1}\gra{1id},\\
&\text{Reidemeister moves II:}  && \gra{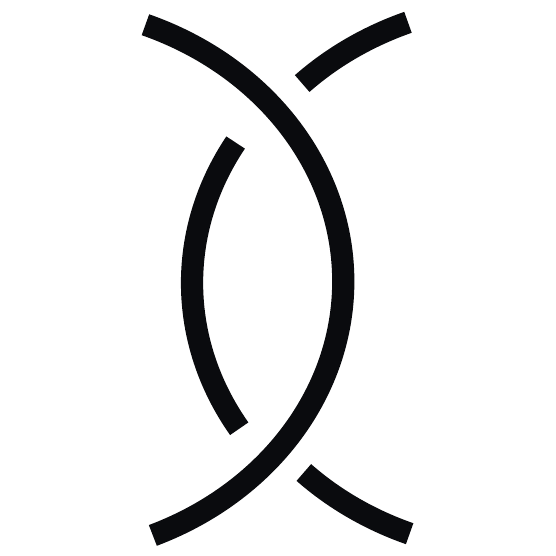}=\gra{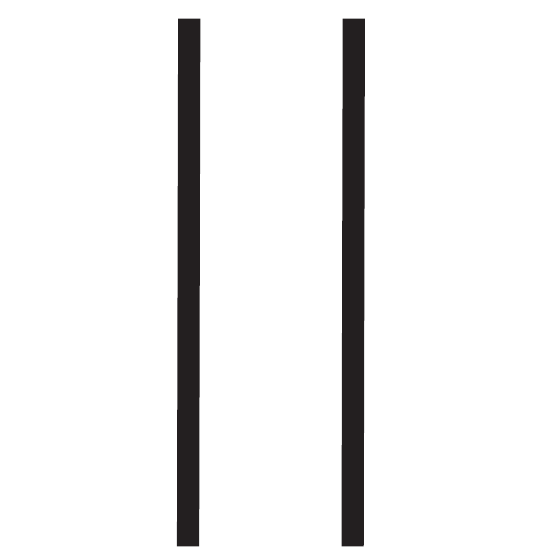},    \\
&\text{Reidemeister moves III:}  && \gra{yb2}=\gra{yb1},  \\
&\text{circle parameter:} && \gra{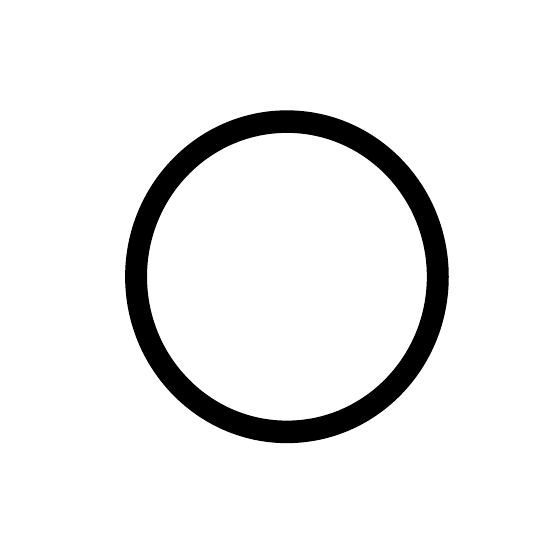}=\frac{r-r^{-1}}{q-q^{-1}}+1.
\end{align*}

\begin{remark}
For quantum group $Sp(2N)$,
the corresponding planar algebra has a negative circle parameter. One needs to apply the gauge transformation to get a planar algebra with a positive circle parameter.
\end{remark}

(3) The Bisch-Jones planar algebras \cite{BisJonFC} are Yang-Baxter relation planar algebras.

(4) One complex conjugate pair of Yang-Baxter relation planar algebras were discovered in \cite{LMP13}. They have the following principal graph. The two depth-two vertices are dual to each other.

$$\grc{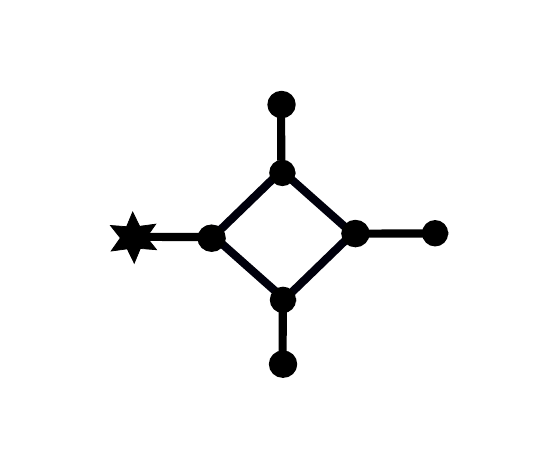}.$$

(5) All depth two subfactors planar algebras are Yang-Baxter relation planar algebras.

Given a unitary fusion category \cite{ENO}, take the direct sum of its simple objects and dual objects, denoted by $X$. Then $\mathscr{S}_n=\hom(X^n)$ forms an unshaded depth two subfactor planar algebra whose even part recovers the original fusion category.

(6) The tensor product of Yang-Baxter relation planar algebras is a Yang-Baxter relation planar algebras.

(7) A non-Yang-Baxter relation planar algebra is given by the group subgroup subfactor planar algebra $S_2\times S_3 \subset S_5$. Its 2-box space has a \emph{one way Yang-Baxter relation} \cite{Ren}.

\section{Classifications}\label{classification}
Bisch and Jones suggested the skein theoretic classification of subfactor planar algebras \cite{BisJon03}.
Thanks to Theorem \ref{evaluable}, a Yang-Baxter relation planar algebra is determined by the structure constants of 2-boxes and duality coefficients of the Yang-Baxter relation. Therefore one can ask for a classification of Yang-Baxter relation planar algebras.

The 2-box space of a planar algebra always contains the two Temperley-Lieb diagrams $\gra{21}$ and $\gra{22}$. Let us classify Yang-Baxter relation planar algebras with one more 2-box generator, so that the 2-box space is three dimensional.
We call them {\it singly-generated, Yang-Baxter relation planar algebras.}

The evaluation algorithm for Yang-Baxter relation planar algebras is global. As an advantage, the critical dimension of the local evaluation algorithm in \cite{BisJon00,BisJon03,BJL} is no longer an restriction.
We have to pay the cost while proving the consistency for the Yang-Baxter relation of the unexpected solution in Section \ref{non-self}.

Suppose $\mathscr{P}_{\bullet}$ is a non-degenerate planar algebra generated by three dimensional 2-boxes with a Yang-Baxter relation.
By Proposition \ref{algebragenerated}, we have $\dim(\mathscr{P}_{3,+})\leq 15$.
Actually any 3-box is reduced to a linear sum of the following 15 diagrams,
\begin{align*}
&\gra{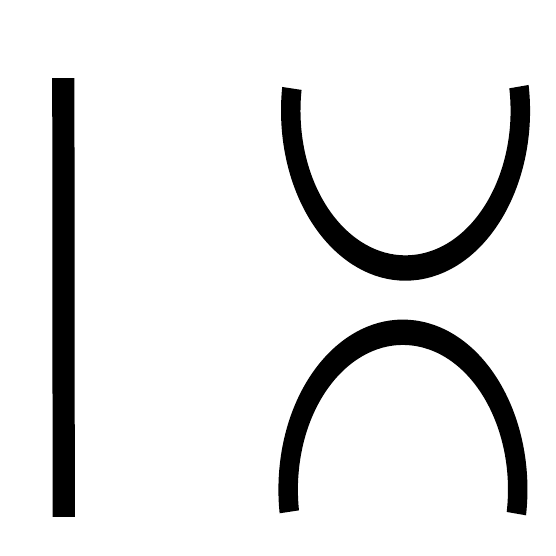}, \gra{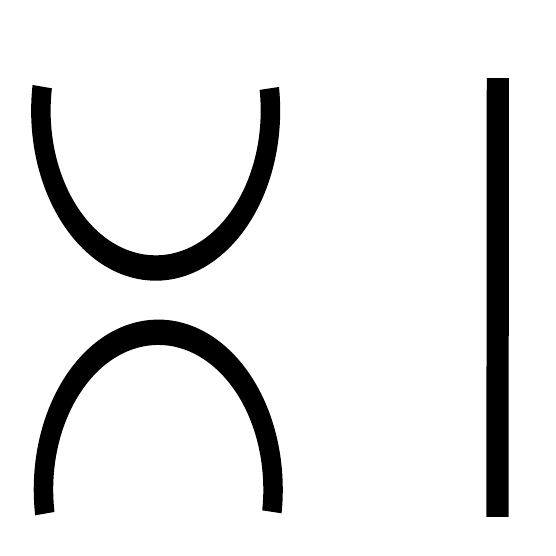}, \gra{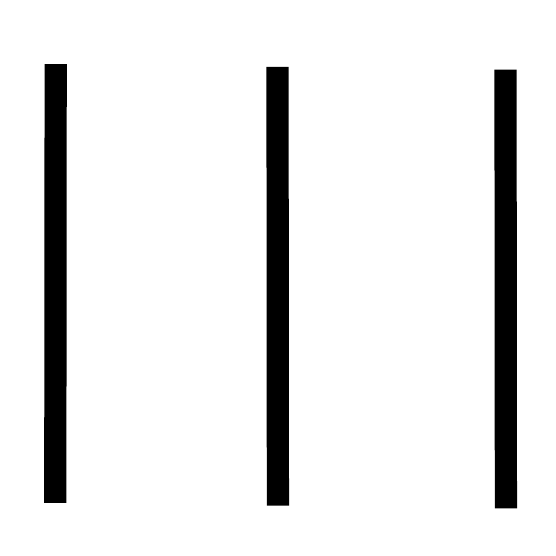}, \gra{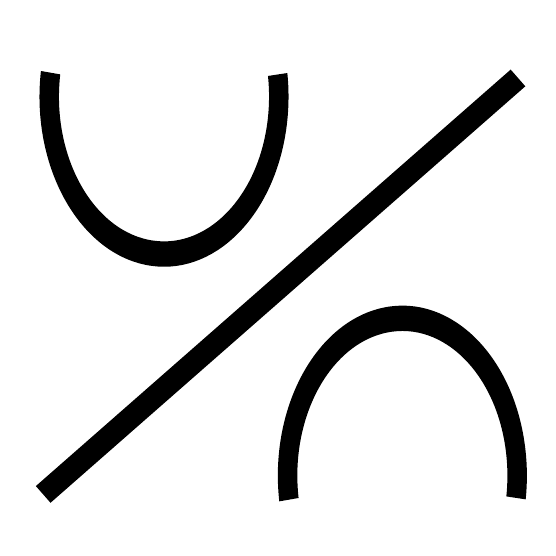}, \gra{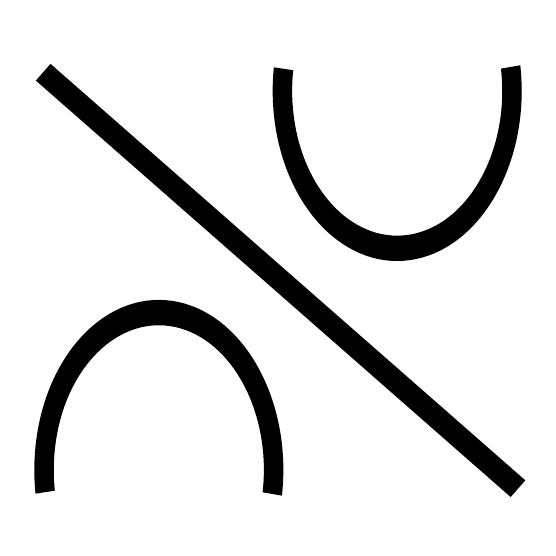};\\
&\gra{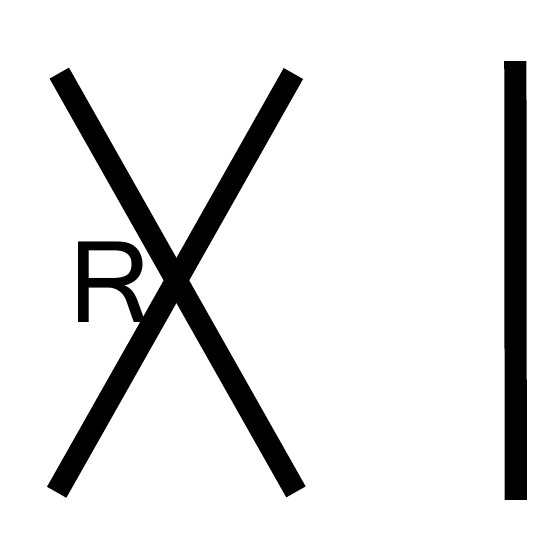}, \gra{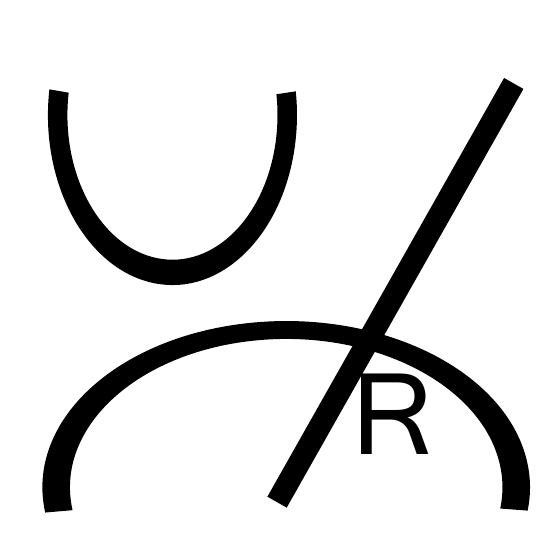}, \gra{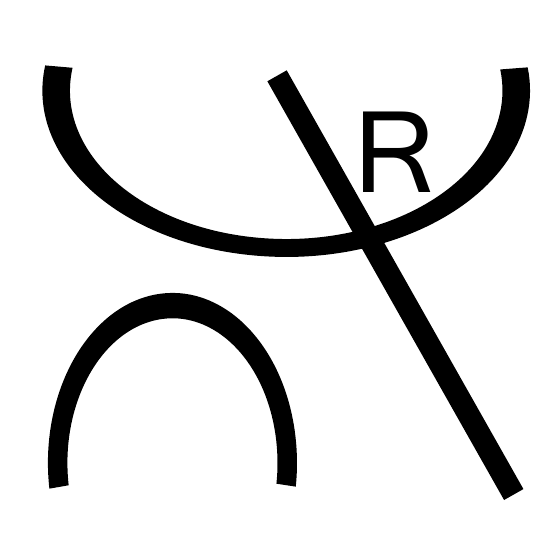}, \gra{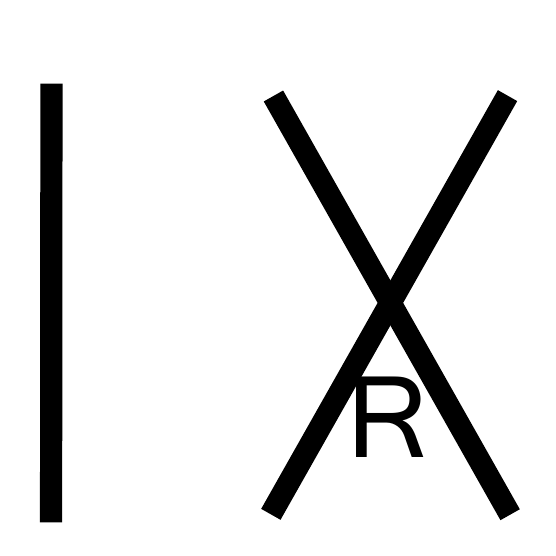}, \gra{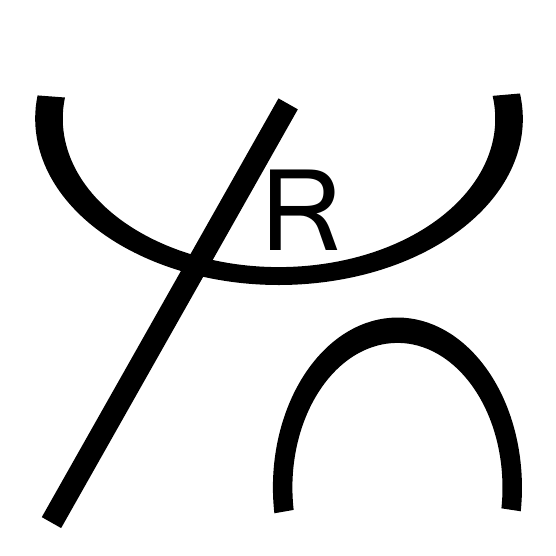}, \gra{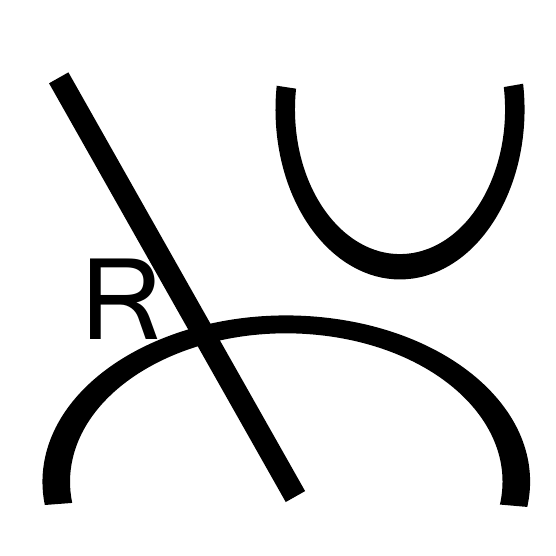};\\
&\gra{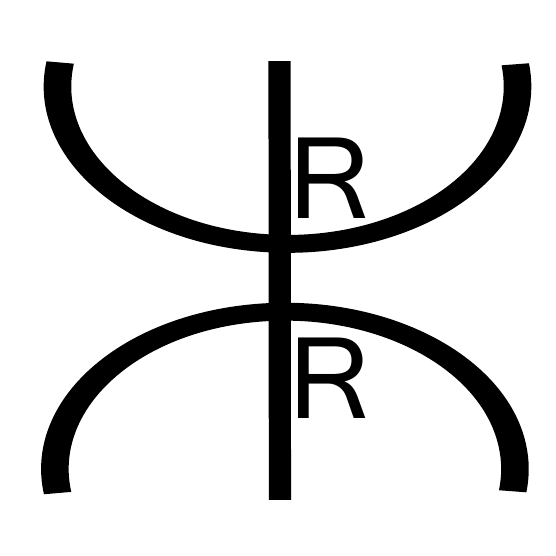}, \gra{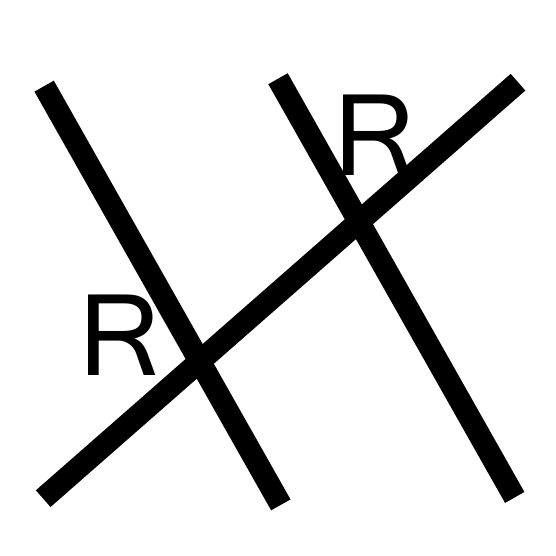}, \gra{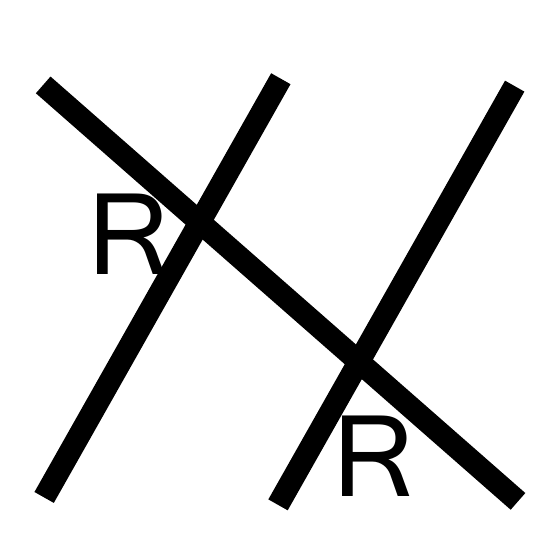};\\
&\gra{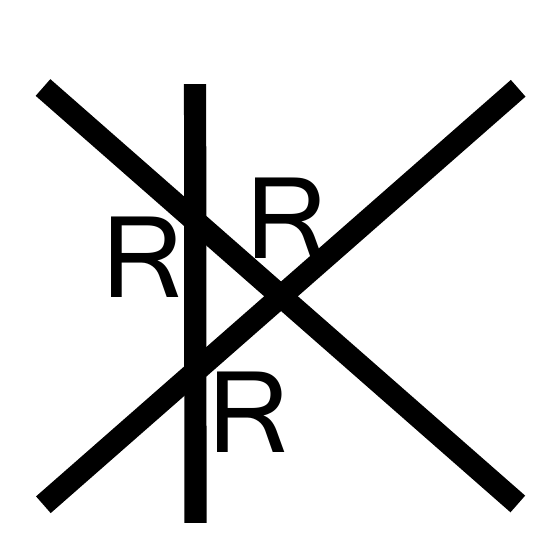}.
\end{align*}
When $\dim(\mathscr{P}_{3,+})\leq 14$, the planar algebra has a local evaluation algorithm. The subfactor planar algebras were classified in \cite{BJL}. Thus we only consider the case $\dim(\mathscr{P}_{3,+})=15$.
In this case, the first 15 diagrams forms a basis of the 3-box space.
Applying the same argument in the dual space $\mathscr{P}_{3,-}$, we have that the first 14 diagrams and $\gra{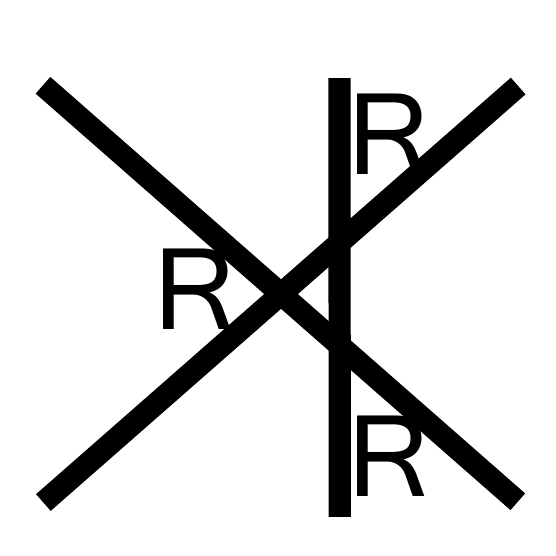}$ also form a basis of the 3-box space.
From this duality, one can prove that the 2-box space of an unshaded planar algebra with at most 15 dimensional 3-boxes has an Yang-Baxter relation.

Now we assume $\dim(\mathscr{P}_{3,+})=15$.
Let us set up the structure constants of 2-boxes and the duality coefficients of the Yang-Baxter relation.
Suppose $\mathscr{P}_{\bullet}$ is a non-degenerate planar algebra generated by a 2-box with a Yang-Baxter relation and $\dim(\mathscr{P}_{2,\pm})=15$.
Then $\delta\neq 0,\pm1$, otherwise the 5 Temperley-Lieb-Jones 3-boxes are linearly dependent and $\dim(\mathscr{P}_{2,\pm})<15$.
Let $\displaystyle e=\frac{1}{\delta}\gra{22}$, $P$, $Q$ be the three minimal idempotents of $\mathscr{P}_{2,+}$.
Let $x,y$ be the solution of
$$\left\{
\begin{aligned}
xtr(P)+ytr(Q)&=0 \\
xy&=-1
\end{aligned}
\right.$$
Take
\begin{align}\label{rxpyq}
R=xP+yQ.
\end{align}
Then $R$ is determined up to a $\pm$ sign.
Moreover, $R$ is uncappable, i.e.,
$$\raisebox{-.47cm}{
\tikz{
\draw (0,0)--(2/3,2/3);
\draw (2/3,2/3) arc (135:-135: 1.414/6);
\draw (0,1)--(2/3,1/3);
\node at (-1/6+1/2,1/2) {\tiny{$R$}};
}}=
\raisebox{-.47cm}{
\tikz{
\draw (0,0)--(2/3,2/3);
\draw (2/3,2/3) arc (135:-135: 1.414/6);
\draw (0,1)--(2/3,1/3);
\node at (1/2,1/2+1/6) {\tiny{$R$}};
}}=
\raisebox{-.47cm}{
\tikz{
\draw (0,0)--(2/3,2/3);
\draw (2/3,2/3) arc (135:-135: 1.414/6);
\draw (0,1)--(2/3,1/3);
\node at (1/6+1/2,1/2) {\tiny{$R$}};
}}=
\raisebox{-.47cm}{
\tikz{
\draw (0,0)--(2/3,2/3);
\draw (2/3,2/3) arc (135:-135: 1.414/6);
\draw (0,1)--(2/3,1/3);
\node at (1/2,1/2-1/6) {\tiny{$R$}};
}}=0.$$
And $R^2=aR+id-e$, where $a=x+y$, i.e.,
$$
\raisebox{-.95cm}{
\tikz{
\draw (0,0)--(1,1)--(0,2);
\draw (1,0)--(0,1)--(1,2);
\node at (-1/6+1/2,1/2) {\tiny{$R$}};
\node at (-1/6+1/2,1+1/2) {\tiny{$R$}};
}}
=a
\raisebox{-.47cm}{
\tikz{
\draw (0,0)--(1,1);
\draw (1,0)--(0,1);
\node at (-1/6+1/2,1/2) {\tiny{$R$}};
}}
+
\raisebox{-.47cm}{
\tikz{
\draw (0,0)--(0,1);
\draw (2/3,0)--(2/3,1);
}}
-\frac{1}{\delta}
\raisebox{-.47cm}{
\tikz{
\draw (0,0) arc (180:0:1/3);
\draw (0,1) arc (180:360:1/3);
}}
.$$

Recall that the Fourier transform $\mathcal{F}$ is given by the 1-click rotation.
By isotopy, we have $tr(\mathcal{F}(R)\mathcal{F}^3(R))=tr(R^2)$. Note that $tr(R^2)=tr(id-e)=\delta^2-1\neq0$,
so $\mathcal{F}(R)\mathcal{F}^3(R)=a'\mathcal{F}(R)+id-e$, for some $a'\in \mathbb{C}$, i.e.,
$$
\raisebox{-.95cm}{
\tikz{
\draw (0,0)--(1,1)--(0,2);
\draw (1,0)--(0,1)--(1,2);
\node at (1/2,+1/6+1/2) {\tiny{$R$}};
\node at (1/2,-1/6+1+1/2) {\tiny{$R$}};
}}
=a
\raisebox{-.47cm}{
\tikz{
\draw (0,0)--(1,1);
\draw (1,0)--(0,1);
\node at (1/2,1/6+1/2) {\tiny{$R$}};
}}
+
\raisebox{-.47cm}{
\tikz{
\draw (0,0)--(0,1);
\draw (2/3,0)--(2/3,1);
}}
-\frac{1}{\delta}
\raisebox{-.47cm}{
\tikz{
\draw (0,0) arc (180:0:1/3);
\draw (0,1) arc (180:360:1/3);
}}
.$$
We will deal with the two cases for $\overline{R}=\pm R$ in the next two subsections.

We have reduced the structure constants of 2-boxes to $\delta,a,a',\pm1$.
Let us define the duality coefficients for the Yang-Baxter relation for the basis $\{\gra{21}.\gra{22},\gra{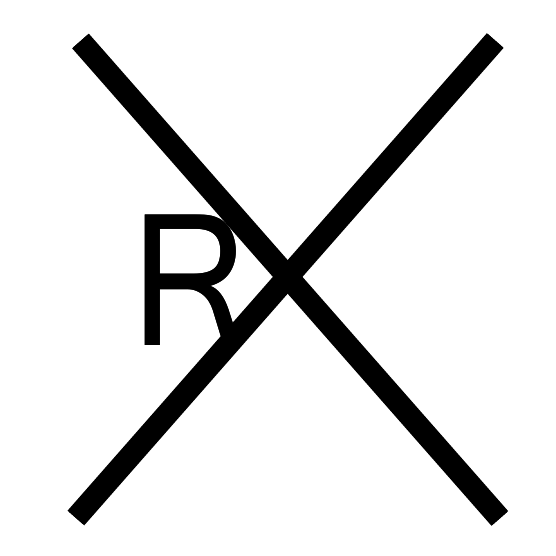}\}$. If the triple are not $R,R,R$, then the Yang-Baxter relation is determined by the structure constants of 2-boxes.
We set up the duality coefficients for the only non-trivial Yang-Baxter relation as follows.
\begin{align*}
\gra{15}&=A\gra{1}+B\gra{2}+C(\gra{3}+\gra{4}+\gra{5})\\
&+D(\gra{6}+\gra{7}+\gra{8})+E(\gra{9}+\gra{10}+\gra{11})\\
&+F(\gra{12}+\gra{13}+\gra{14})+G\gra{16},\\
\end{align*}

\subsection{The generator is self-contragredient}
In this section, we deal with the case $R=\overline{R}$.
We eliminate the structure constants and duality coefficients by solving the polynomial equations derived from the consistency condition.
These formal parameters reduce to two parameters. Then we identify this two-parameter family of planar algebras as BMW planar algebras with a singularity which is identified as unshaded Bisch-Jones planar algebras.
These unshaded Bisch-Jones planar algebras are \emph{limits} of BMW planar algebras. The limit is defined by the convergence of the structure constants and duality coefficients of these Yang-Baxter relation planar algebras.

When $R=\overline{R}$, we have $\mathcal{F}(R)^2=a'\mathcal{F}(R)+id-e$.
So $R*R=a'R+\delta e-\frac{1}{\delta}id$.
\begin{theorem}[Classification1]\label{classification1}
Suppose $\mathscr{P}_{\bullet}$ is a non-degenerate planar algebra generated by $R$ in $\mathscr{P}_{2,+}$ with a Yang-Baxter relation, such that
$\dim(\mathscr{P}_{3,\pm})=15$,
$R$ is uncappable,
$R=\overline{R}$,
$R^2=aR+id-e$,
$\mathcal{F}(R)^2=a'\mathcal{F}(R)+id-e$,
and
\begin{align}\label{YBR1}
\gra{15}&=A\gra{1}+B\gra{2}+C(\gra{3}+\gra{4}+\gra{5})\\
&+D(\gra{6}+\gra{7}+\gra{8})+E(\gra{9}+\gra{10}+\gra{11})\\
&+F(\gra{12}+\gra{13}+\gra{14})+G\gra{16},\\
\end{align}
then

$$\left\{
  \begin{aligned}
    G&=\pm1\\
    A&=G\frac{a}{\delta}\\
    B&=-A\\
    C&=0\\
    (G&\delta^2-2\delta)D=1-Ga^2\delta\\
    E&=-GD\\
    F&=0\\
    a'&=Ga\\
  \end{aligned}
\right.
$$
\end{theorem}

\begin{proof}
See Appendix \ref{Appendix:classification1}
\end{proof}

\begin{corollary}
The planar algebra $\mathscr{P}_{\bullet}$ is unshaded with the relation $G\mathcal{F}(R)=R$.
\end{corollary}

\begin{proof}
Note that $G\mathcal{F}(R)$ in $\mathscr{P}_{2,-}$ satisfies the same type I, II, III moves as $R$ by switching shading. Thus the identification between $R$ and $G\mathcal{F}(R)$ extends to a symmetric self-duality, i.e., a planar algebra isomorphism $\phi_{\pm}$ from $\mathscr{P}_{m,\pm}$ to the dual $\mathscr{P}_{m,\mp}$, such that $\phi_{\pm}\phi_{\mp}$ is the identity, (similar to the case in Theorem 3.10 \cite{LMP13}). Therefore $\mathscr{P}_{\bullet}$ is unshaded by identifying $G\mathcal{F}(R)$ as $R$.
\end{proof}

Note that the parameterized BMW planar algebra is generated by a self-contragredient braid satisfying type I, II, III Reidemester moves and the BMW relation (example (2) in Section \ref{subsection example}). Let us find such a braid which satisfies these relations in $\mathscr{P}_{\bullet}$. Then $\mathscr{P}_{\bullet}$ is BMW.

Let $z_1$, $z_2$ be the solution of
\begin{align}
\left\{\label{sumproduct}
\begin{aligned}
z_1+z_2G&=-a \\
z_1z_2G&=-E
\end{aligned}
\right.
\end{align}
For $a_3\neq 0$, take $a_1=z_1a_3$, $a_2=z_2a_3$;
\begin{align*}
R_U&=a_1\gra{21}+a_2\gra{22}+a_3\gra{23};
\end{align*}

\begin{lemma}[bi-invertible]\label{biunitary}
The element $R_U$ satisfies
$$\mathcal{F}(R_U)R_U=G(1-E)a_3^2\gra{21}.$$
\end{lemma}

\begin{proof}
By Equation (\ref{sumproduct}), $E=-GD$ and $(G\delta^2-2\delta)D=1-Ga^2\delta$, we have
\begin{align*}
\mathcal{F}(R_U)R_U
=&(a_1a_2+a_3^2G)\gra{21}+(a_1^2+a_2^2+a_1a_2\delta-a_3^2G\frac{1}{\delta})\gra{22}+(a_1Ga_3+a_2a_3+a_3^2Ga)\gra{23}\\
=&(a_1a_2+a_3^2G)\gra{21}+((-a)^2+(\delta-2G)(-E)-\frac{G}{\delta})a_3^2\gra{22}+(-aG+Ga)\gra{23}\\
=&(a_1a_2+a_3^2G)\gra{21}
\end{align*}
\end{proof}

\begin{lemma}[Yang-Baxter equation]\label{solution of YBE 1} We identify $R_U$ in the 3-box space as $R_U\otimes 1$, then
$$R_U(1\otimes R_U)R_U=(1\otimes R_U)R_U(1\otimes R_U).$$
\end{lemma}

\begin{proof}
See Appendix \ref{Appendix:solution of YBE 1}
\end{proof}

\begin{theorem}
The relations for $R$ in Theorem \ref{classification1} are consistent.
The planar algebra given by the generator and relations is BMW when $E\neq1$; Bisch-Jones when $E=1$.
\end{theorem}

(The dimension of the 3-box space of Bisch-Jones planar algebras is at most 12. All BMW subfactor planar algebras are listed in Section 2.4 in \cite{BJL}, based on the results in \cite{Wen90,Saw95}.)

\begin{proof}

When $E\neq 1$, let us take $a_3$ to be a square root of $\frac{1}{G(1-E)}$.
Then $\mathcal{F}(R_U)R_U=id$ and $R_U(1\otimes R_U)R_U=(1\otimes R_U)R_U(1\otimes R_U).$ In this case, when $G=1$, we have $R_U-\mathcal{F}(R_U)=(a_1-a_2)(\gra{21}-\gra{22})$, so $\mathscr{P}_{\bullet}$ is BMW from $O(N)$.
When $G=-1$, we have $R_U+\mathcal{F}(R_U)=(a_1+a_2)(\gra{21}+\gra{22})$, so $\mathscr{P}_{\bullet}$ is BMW from $Sp(2N)$.
Consequently the relation for $R$ is consistent.

When $E=1$, recall that
$(G\delta^2-2\delta)D=1-Ga^2\delta$ and $E=-GD$, we have
$$\delta^2-(2+a^2) \delta G+1=0.$$
Recall that $R=xP+yQ$ (\ref{rxpyq}) and
\begin{align*}
\left\{
\begin{aligned}
x+y&=a \\
xy&=-1,
\end{aligned}
\right.
\end{align*}
so $$(\delta-x^2G)(\delta-y^2G)=0.$$
Without loss of generality, we assume that $y^2=G\delta$. Then $xG\delta=-y$. Note that
\begin{align}\label{q1}
\left\{
\begin{aligned}
xtr(P)+ytr(Q)&=0\\
tr(P)+tr(Q)&=\delta^2-1,
\end{aligned}
\right.
\end{align}
so
$$\left\{
\begin{aligned}
tr(P)&=G\delta-1\\
tr(Q)&=\delta^2-G\delta
\end{aligned}
\right.$$

Recall that $z_1$, $z_2$ are the solution of
\begin{align*}
\left\{
\begin{aligned}
z_1+z_2G&=-a \\
z_1z_2G&=-1
\end{aligned}
\right..
\end{align*}
Let us take $z_1=-x$, $z_2=-Gy$.
Then
\begin{align*}
R_U=&(-x-Gy\delta)e+(-x+x)P+(-x+y)Q \\
=&(y-x)Q.
\end{align*}
Note that $R_U\neq 0$, so $y-x\neq0$.
By Lemma \ref{biunitary}, we have
\begin{align}
F(Q)Q=0 \label{q2}
\end{align}
By Lemma \ref{solution of YBE 1}, we have
\begin{align}
Q(1\otimes Q)Q=(1\otimes Q)Q(1\otimes Q). \label{q3}
\end{align}

Observe that the type I, II, III moves of $Q$ are determined by Equation (\ref{q1}), (\ref{q2}), (\ref{q3}). Moreover,
the relation is the same as that of the 2-box $id\otimes(id-e)$ in the Bisch-Jones planar algebra with parameters $(\delta_a,\delta_a)$, where $\delta_a$ is a square root of $\delta$.
Therefore $\mathscr{P}_{\bullet}$ is Bisch-Jones and $\dim(\mathscr{P}_3)\leq 12$.
\end{proof}

\begin{remark}
The Bisch-Jones planar algebra with parameters $(\delta_a,\delta_a)$ is unshaded. It is a limit of BMW planar algebras as in the above proof.
\end{remark}

Recall that $R$ is determined up to a $\pm$ sign. However, the duality coefficients $D$, $E$ and $G$ in the Yang-Baxter relation \ref{YBR1} are independent of the choice of $\pm$. So they are invariants of the planar algebra.
Moreover, the condition $E=1$ distinguishes BMW planar algebras and Bisch-Jones planar algebras. Furthermore, the value of $G=\pm 1$ distinguishes $O(N)$ and $Sp(2N)$ for BMW; distinguishes the two unshaded Bisch-Jones planar algebras.

When $\delta\neq2G$, we have $\displaystyle E=\frac{a^2\delta-1}{G\delta^2-2\delta}$. Then the planar algebra $\mathscr{P}_{\bullet}$ is uniquely determined by $a$, $\delta$, $G$. Note that $a$, $\delta$ are derived from the traces of the one 1-box and two 2-box minimal idempotents. Thus we can distinguish BMW planar algebras and Bisch-Jones planar algebras by the trace.

When $\delta=2G$, we have $a^2=\frac{1}{2}$. Up to the choice of $\pm R$, $a$ is unique. In this case $E$ is a free parameter. When $\delta=2$, the planar algebra parameterized by $E$ is BMW planar algebras subject to $r=q$.
We cannot distinguish BMW planar algebras and Bisch-Jones planar algebras by $\delta$ and $a$ in this case. The extended $D$ subfactor planar algebra is both BMW and Bisch-Jones.
The case $\delta=-2$ reduces to the case $\delta=2$. Precisely, by the gauge transformation, the trace on $m$-boxes $tr_m$ is replaced by $(-1)^m tr_m$. In particular, we can change $\delta,a$ to $-\delta,a$.

\subsection{The generator is non-self-contragredient}\label{non-self}
In this section, we deal with the case $R=-\overline{R}$. We eliminate the structure constants and duality coefficients by solving the polynomial equations derived from the consistency condition.
It is unexpected that the circle parameter $\delta$ survives after solving several equations.
Only two subfactor planar algebra were known in this family, the group subfactor planar algebra $\mathbb{Z}_3$ and the example (4) in Section \ref{subsection example}.
One has to construct these (subfactor) planar algebras by skein theory.
We prove the consistency of this parameterized relation in Section \ref{subsection consistency} and construct a new parameterized planar algebra $\mathscr{C}_{\bullet}$. We find out all values of the parameter for which $\mathscr{C}_{\bullet}$ has positivity and constructed a sequence of subfactor planar algebras in Section \ref{subsection positivity}, based on the results in Sections \ref{subsection matrix units} and \ref{subsection trace formula}.
The algebraic presentation of $\mathscr{C}_{\bullet}$ is given in the Appendix \ref{Appendix:algebraic presentation}.

When $R=-\overline{R}$,
we have $R^2=\overline{R^2}=a\overline{R}+id-e=-aR+id-e$.
So $a=0$ and $R^2=id-e$.
Similarly we have $a'=0$ and $\mathcal{F}(R)^2=-id+e$.
So $\displaystyle R*R=-\delta e+\frac{1}{\delta}id$.

\begin{theorem}[{\bf Classification 2}]\label{class-1}
Suppose $\mathscr{P}_{\bullet}$ is a non-degenerate planar algebra generated by $R$ in $\mathscr{P}_{2,+}$ with a Yang-Baxter relation, $\dim(\mathscr{P}_{3,\pm})=15$,
$R$ is uncappable,
$R=-\overline{R}$
$R^2=id-e$,
$\mathcal{F}(R)^2=-id+e$, and
\begin{align*}
\gra{15}&=A\gra{1}+B\gra{2}+C(\gra{3}+\gra{4}+\gra{5})\\
&+D(\gra{6}+\gra{7}+\gra{8})+E(\gra{9}+\gra{10}+\gra{11})\\
&+F(\gra{12}+\gra{13}+\gra{14})+G\gra{16}.\\
\end{align*}
Then
$$\left\{
  \begin{aligned}
    G&=\pm i\\
    A&=0\\
    B&=0\\
    C&=0\\
    D&=-\frac{1}{G\delta^2}\\
    E&=-\frac{1}{\delta^2}\\
    F&=0\\
  \end{aligned}
\right.$$
\end{theorem}

\begin{proof}
See Appendix \ref{Appendix:class-1}.
\end{proof}

\begin{corollary}\label{Cor:FourierRelation}
The planar algebra $\mathscr{P}_{\bullet}$ is unshaded with the relation $G\mathcal{F}(R)=R$.
\end{corollary}

\begin{proof}
Note that $G\mathcal{F}(R)$ in $\mathscr{P}_{2,-}$ satisfies the same type I, II, III moves as $R$.  Thus the identification between $R$ and  $G\mathcal{F}(R)$ extends to a symmetric self-duality.
Therefore $\mathscr{P}_{\bullet}$ is unshaded by identifying $G\mathcal{F}(R)$ as $R$.
\end{proof}

\section{Consistency}\label{subsection consistency}
In this section, we are going to construct the one-parameter family of unshaded planar algebras whose generator and relations are given in Theorem \ref{class-1} and Corollary \ref{Cor:FourierRelation}. We only work for the case $G=i$. The case $G=-i$ is given by its complex conjugate.

First we construct the planar algebra $\mathscr{C}_{\bullet}$ in terms of $\delta$, and $\delta \in \mathbb{R}$.
Then we replace $\delta$ by $q$, and the planar algebra can be defined over the field $\mathbb{C}(q)$.
(The relation between $\delta$ and $q$ is given by $\displaystyle  q=\frac{i+\delta }{\sqrt{1+\delta^2}}$ and
$\displaystyle \delta=\frac{i(q+q^{-1})}{q-q^{-1}}$.)

\begin{definition}\label{Def:Centralizer algebra}
Let us define $\mathscr{C}_{\bullet}$ to be the unshaded general planar algebra generated by a 2-box $R=\gra{23}$ with the following relations:
$\mathcal{F}(R)=-iR$;
$R$ is uncappable;
$R^2=id-e$;
and
\begin{align}\label{YBrelation}
\gra{15}=\frac{i}{\delta^2}(\gra{6}+\gra{7}+\gra{8})-\frac{1}{\delta^2}(\gra{9}+\gra{10}+\gra{11})+i\gra{16}.
\end{align}
\end{definition}

To show the consistency, it is enough to find a partition function for the universal planar algebra generated by a 2-box $R$, such that (type I, II moves and) the Yang-Baxter relation are in the kernel of the partition function. As we mentioned, the evaluation algorithm of the Yang-Baxter relation is global. We can not defined the partition function as the average of all complexity reducing evaluations. This is different from the case for local evaluation algorithms. We will define the partition function globally on closed diagrams by induction on the number of labels $R$.
To introduce the globally defined partition function, let us solve the Yang-Baxter equation for $\mathscr{C}_{\bullet}$ first.

\begin{lemma}\label{solution of YBE -1}
Take $\tilde{A}\in\mathscr{C}_{2}, \tilde{B}\in\mathscr{C}_{2},$
\begin{align*}
\tilde{A}&=a_1\gra{21}+a_2\gra{22}+a_3\gra{23}, &a_3&\neq0; \\
\tilde{B}&=b_1\gra{21}+b_2\gra{22}+b_3\mathcal{F}(\gra{23}),&b_3&\neq 0,
\end{align*}
and $A=\tilde{A}\otimes 1$,  $B=1\otimes \tilde{B}$.
If $\dim(\mathscr{C}_3)=15$, then $ABA=BAB$ if and only if
$$ a_1=b_1, a_2=b_2, b_3=ia_3, a_1^2=-\frac{a_3^2}{\delta^2}, a_2^2=\frac{a_3^2}{\delta^2}.$$
\end{lemma}

\begin{proof}
See Appendix \ref{Appendix:solution of YBE -1}
\end{proof}

Observe that the 2-box solution of the Yang-Baxter equation in Lemma \ref{solution of YBE -1} satisfies the Hecke relation in Section \ref{hecke}.
We will define the partition function on the universal planar algebra generated by $R$ using the HOMFLY-PT polynomial.
Note that the braid generator $\gra{b+}$ of the Hecke algebra and the Jones projection $\displaystyle \frac{1}{\delta} \gra{22}$ have incompatible orientations on the boundary. So the braid and the Jones projection cannot be interpreted as diagrams simultaneously. These make the definition of the partition function complicated.
If one wants to prove that the Yang-Baxter relation is in the kernel of the partition function directly, the proof will be incredibly tedious.

We give a new method to \underline{reduce the algorithmic complexity} by constructing several intermediate quotients from the universal planar generated by $R$ to the quotient $\mathscr{C}_{\bullet}$. We prove that the relations of the generator $R$ are in the kernel of the partition function on these quotients step by step. This method helps us to \underline{utilize repeating data} in the proof.
One should keep in mind that the 2-box solution of the Yang-Baxter equation in $\mathscr{C}_{\bullet}$ no longer satisfies the Yang-Baxter equation on these intermediate quotients.

Now let us assume $\delta\in \mathbb{R}$ and define these intermediate quotients from the universal planar generated by the 2-box to the quotient $\mathscr{C}_{\bullet}$.
\begin{definition}
Let $\mathscr{C}_{\bullet}'$ be the universal planar algebra generated by a single 2-box $R$.
\end{definition}

\begin{definition}
Let $Ann_i^j(n)$ be the set of annular tangles labeled by $n$ copies of $R$ from $\mathscr{C}_{i}'$ to $\mathscr{C}_{j}'$.
\end{definition}

\begin{definition}
Let $\mathscr{C}_{\bullet}''$ be the general planar algebra generated by a single 2-box $R$ such that
$$\gra{r0}=\delta, \quad \mathcal{F}(R)=-iR.$$
\end{definition}

\begin{definition}
Let us define $\gra{b1id}=\gra{1id}$,
\begin{align}
\gra{b+}&=\frac{i}{\sqrt{1+\delta^2}}\gra{21}+\frac{1}{\sqrt{1+\delta^2}}\gra{22}+\frac{\delta}{\sqrt{1+\delta^2}}\gra{23},\label{braid+}\\
\gra{b-}&=-\frac{i}{\sqrt{1+\delta^2}}\gra{21}+\frac{1}{\sqrt{1+\delta^2}}\gra{22}+\frac{\delta}{\sqrt{1+\delta^2}}\gra{23}. \label{braid-}
\end{align}

\end{definition}

\begin{notation}
Take $\mathcal{D}=\displaystyle \frac{\delta}{\sqrt{1+\delta^2}}$, $\displaystyle r=\frac{\delta i+1}{\sqrt{1+\delta^2}}$, $\displaystyle q=\frac{i+\delta }{\sqrt{1+\delta^2}}$, we have $|r|=|q|=1$.
\end{notation}

\begin{definition}
Let us define
\begin{align*}
R_1&=\gra{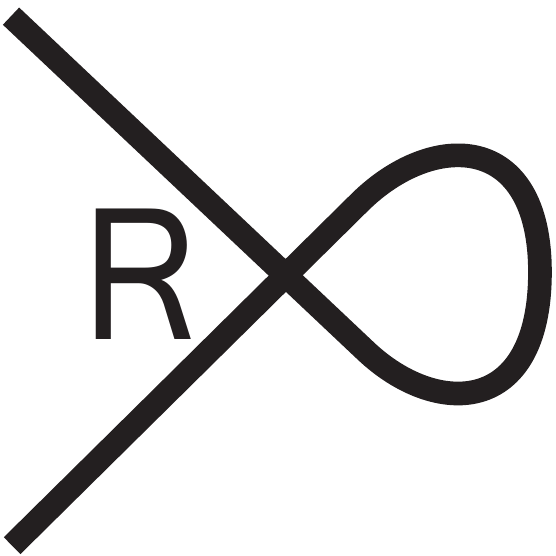},\\
R_2&=\gra{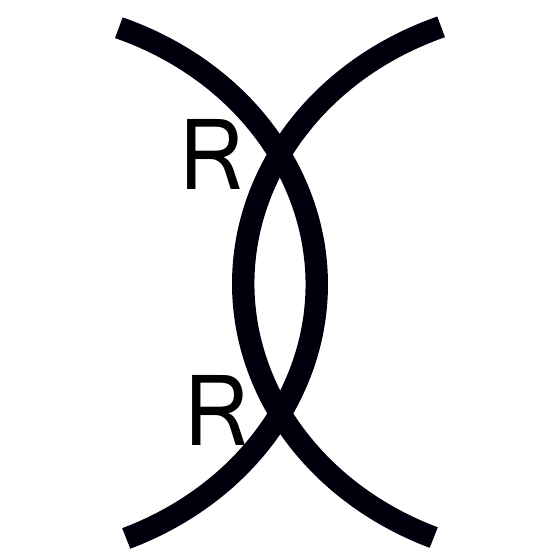}-(\gra{21}-\frac{1}{\delta}\gra{22}), \\
R_3&=\gra{15}-(\frac{i}{\delta^2}(\gra{6}+\gra{7}+\gra{8})-\frac{1}{\delta^2}(\gra{9}+\gra{10}+\gra{11})+i\gra{16}),
\end{align*}
then $\mathcal{F}(R_3)=-R_3$ in $\mathscr{C}_{\bullet}''$.
\end{definition}

\begin{notation}
Let us define the general planar algebras
$\mathscr{C}_{\bullet}'''=\mathscr{C}_{\bullet}''/\{R_1\},$
$\mathscr{C}_{\bullet}''''=\mathscr{C}_{\bullet}'''/\{R_2\}.$
Then $\mathscr{C}_{\bullet}=\mathscr{C}_{\bullet}''''/\{R_3\}.$
\end{notation}

On these intermediate quotients, we have the following relations for \gra{r21}.

\begin{lemma}\label{Move I}
The following relations hold in $\mathscr{C}_{\bullet}''$,
\begin{align*}
&\text{the Fourier relation:}   & \gra{b+}&=i\gra{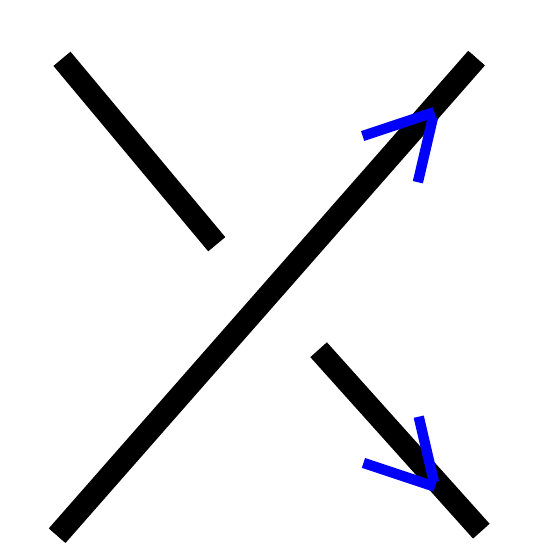}, &&\\
&\text{the Hecke relation:}    &  \gra{b+}-\gra{b-}&=(q-q^{-1})\gra{b2id}, &&\\
&\text{Reidemeister moves I:}   & \gra{homfly11}-r\gra{b1id}&=\displaystyle \mathcal{D}R_1; & \gra{homfly12}-r^{-1}\gra{b1id}&=\displaystyle \mathcal{D}R_1; \\
                        && \gra{homfly13}-r \gra{b1id-}&=\displaystyle \mathcal{D}i^2R_1; & \gra{homfly14}-r^{-1}\gra{b1id-}&=\displaystyle \mathcal{D}i^2R_1.
\end{align*}
\end{lemma}

\begin{proof}
They follow from the definitions.
\end{proof}

\begin{lemma}\label{Move II}
The following relations hold in $\mathscr{C}_{\bullet}'''$,
\begin{align*}
&\text{Reidemeister moves II:}   & \gra{homfly21}-\gra{homfly22}&=\displaystyle \mathcal{D}^2R_2; & \gra{homfly25}-\gra{homfly26}&=\displaystyle \mathcal{D}^2R_2; \\
                        && \gra{homfly23}-\gra{homfly24}&=\displaystyle \mathcal{D}^2R_2; & \gra{homfly27}-\gra{homfly28}&=\displaystyle \mathcal{D}^2R_2.
\end{align*}
The other four Reidemeister moves II can be obtained by a 2-click rotation.
\end{lemma}

\begin{proof}
\begin{align*}
\gra{homfly21}-\gra{homfly22}
&=\displaystyle\left(\frac{i}{\sqrt{1+\delta^2}}\gra{21}+\frac{1}{\sqrt{1+\delta^2}}\gra{22}+ \mathcal{D}\gra{23}\right)\times\\
&\times\left(-\frac{i}{\sqrt{1+\delta^2}}\gra{21}+\frac{1}{\sqrt{1+\delta^2}}\gra{22}+ \mathcal{D}\gra{23}\right)-\gra{21}\\
&=\displaystyle \mathcal{D}^2\gra{r21}+\left(\left(\frac{1}{\sqrt{1+\delta^2}}\right)^2-1\right)\gra{21}+\left(\frac{1}{\sqrt{1+\delta^2}}\right)^2\delta\gra{22}\\
&=\displaystyle \mathcal{D}^2\left(\gra{r21}-\gra{21}+\frac{1}{\delta}\gra{22}\right)\\
&= \mathcal{D}^2R_2
\end{align*}
Taking the complex conjugate of the above equation, we have
$$\gra{homfly25}-\gra{homfly26}=\displaystyle \mathcal{D}^2R_2.$$
Applying the Fourier relation in Lemma \ref{Move I}, we have
$$\gra{homfly23}-\gra{homfly24}=\displaystyle \mathcal{D}^2R_2; \quad \gra{homfly27}-\gra{homfly28}=\displaystyle \mathcal{D}^2R_2.$$

\end{proof}

\begin{lemma}\label{Move III}
The following relations hold in $\mathscr{C}_{\bullet}'''$,
\begin{align*}
&\text{Reidemeister moves III:}  & \gra{homfly31}-\gra{homfly32}&=\displaystyle \mathcal{D}^3i^3R_3; & \gra{homfly33}-\gra{homfly34}&=\displaystyle \mathcal{D}^3i^3R_3.
\end{align*}
The other 46 Reidemeister moves III also hold.
\end{lemma}

\begin{remark}
There are 8 different orientations of the three strings, but only 2 up to rotations.
For each orientation, there are 8 choices of the three braids, but only 6 of them admit a Reidemeister move III.
So we have 48 Reidemeister moves III in total.
\end{remark}

\begin{proof}
By the computation in Lemma \ref{solution of YBE -1}, we have
$\gra{homfly31}-\gra{homfly32}=\displaystyle \mathcal{D}^3i^3R_3$.
By the Hecke relation in Lemma \ref{Move I} and the Reidemeister moves II in Lemma \ref{Move II}, we can change the layer of strings and obtain the other 5 Reidemeister moves III with the same boundary orientation, such as
$$\gra{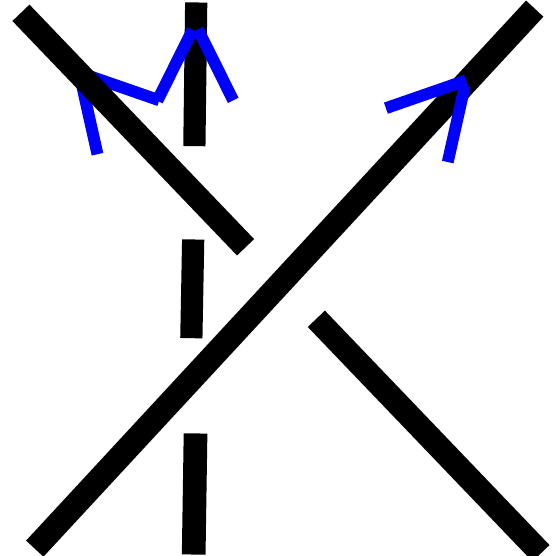}-\gra{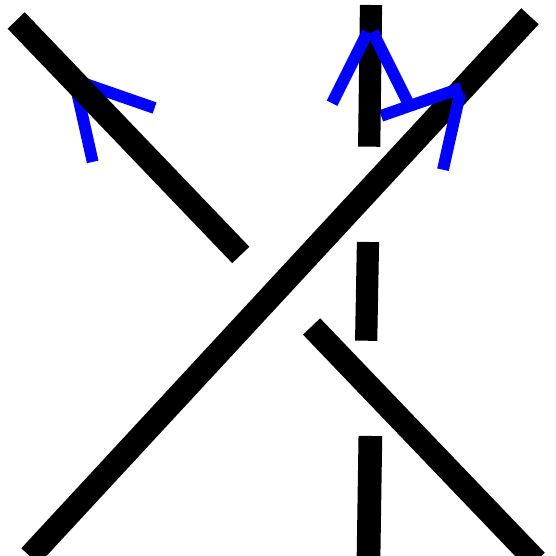}=\displaystyle \mathcal{D}^3i^3R_3.$$
Applying the Fourier relation in Lemma \ref{Move I}, we can switch the orientation of the string at the bottom of a Reidemeister moves III, such as
$$\gra{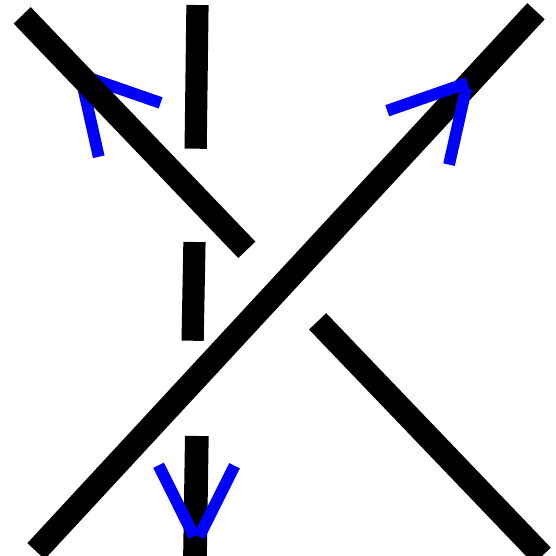}-\gra{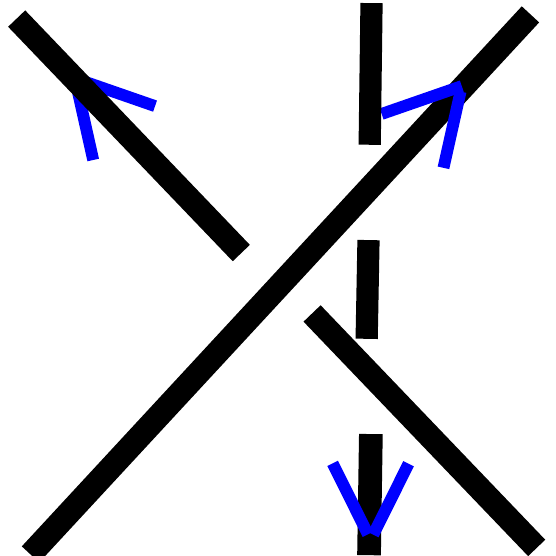}=\displaystyle \mathcal{D}^3i^3R_3.$$
Once again applying the Hecke relation in Lemma \ref{Move I} and the Reidemeister moves II in Lemma \ref{Move II}, we obtain the other 5 Reidemeister moves III with the same boundary orientation but different layers of strings, such as
$$\gra{homfly33}-\gra{homfly34}=\displaystyle \mathcal{D}^3i^3R_3.$$

Note that $\mathcal{F}(R_3)=-R_3$, we can derive the other Reidemeister moves III with different orientations by rotations.
\end{proof}

\begin{proposition}
The following relations hold in $\mathscr{C}_{\bullet}$.
\begin{align*}
&\text{the Hecke relation:}    & \gra{b+}-\gra{b-}&=(q-q^{-1})\gra{b2id}, && \\
&\text{Reidemeister moves I:}   & \gra{homfly11}&=r\gra{b1id}; & \gra{homfly12}&=r^{-1}\gra{b1id}; \\
                                && \gra{homfly13}&=r \gra{b1id-}; & \gra{homfly14}&=r^{-1}\gra{b1id-},\\
&\text{Reidemeister moves II:}  & \gra{homfly21}&=\gra{homfly22}; &  \gra{homfly25}&=\gra{homfly26}; \\
                                && \gra{homfly23}&=\gra{homfly24}; &  \gra{homfly27}&=\gra{homfly28},\\
&\text{Reidemeister moves III:}  & \gra{homfly31}&=\gra{homfly32}; &  \gra{homfly33}&=\gra{homfly34}.
\end{align*}
Other Reidemeister moves II, III with different layers and orientations of strings also hold.
\end{proposition}

\begin{proof}
They follow from Lemmas \ref{Move I}, \ref{Move II}, \ref{Move III}.
\end{proof}

Our purpose is to construct a partition function of $\mathscr{C}_{\bullet}'$, such that it is well-defined on the quotient $\mathscr{C}_{\bullet}$.
By Proposition \ref{Move I}, the restriction of the partition function on link diagrams in $\mathscr{C}_{\bullet}'$ has to be the HOMFLY-PT polynomial.
Due to the relations $\gra{r0}=\delta, \mathcal{F}(R)=-iR$ and linearity, the partition function is uniquely determined by $\delta$.
By this observation, we can define the partition function inductively.
By linearity, we only need to define the partition function on closed diagrams labeled by $R$.

Now let us define a partition function $\zeta$ of $\mathscr{C}_{\bullet}'$ by induction on the number $n$ of labels $R$ in a closed diagram.

When $n=0$, we define $\zeta$ on closed Templey-Lieb digrams to be the evaluation map with respect to the relation $\gra{r0}=\delta$.

Suppose $\zeta$ is defined on any closed diagram with at most $n-1$ labels $R$, $n\geq1$.
Let us define $\zeta(T)$ for a closed diagram $T$ with $n$ labels $R$ by the following process.

Considering $R$ in the diagram $T$ as $\gra{23}$, a crossing with a label $R$ indicating the position of $\$ $.
Then $T$ consists of $k$ immersed circles intersecting at $R$'s. Let $\pm(T)$ be the set of $2^k$ choices of orientations of the $k$ circles.
For an orientation $\sigma\in \pm(T)$, let $T_\sigma$ be the corresponding oriented diagram.
Let $\pm(\sigma)$ be the set of $2^n$ choices of replacing the $n$ copies of the oriented crossing $\gra{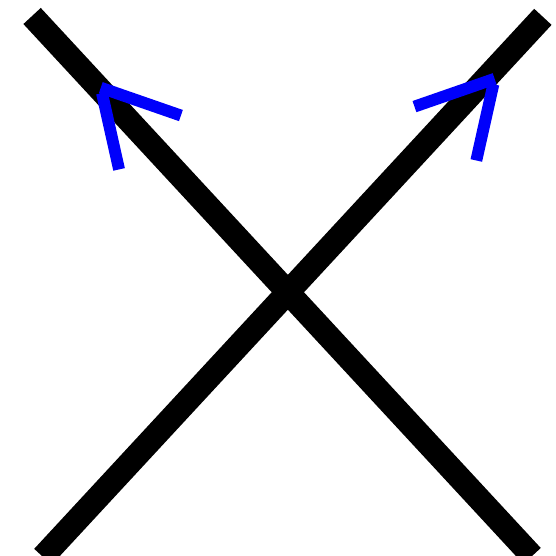}$ of $T(\sigma)$ by a braid $\gra{b+}$ or $\gra{b-}$. For a choice $\gamma \in \pm(\sigma)$, we obtain an oriented link $T_{\sigma,\gamma}$ by replacing the crossings.

Substituting $\gra{b+}$ and $\gra{b-}$ of $T_{\sigma,\gamma}$ by Equations (\ref{braid+}) and (\ref{braid-}), i.e.,
$$\gra{b+}=\frac{i}{\sqrt{1+\delta^2}}\gra{21}+\frac{1}{\sqrt{1+\delta^2}}\gra{22}+ \mathcal{D}\gra{23};$$
$$\gra{b-}=-\frac{i}{\sqrt{1+\delta^2}}\gra{21}+\frac{1}{\sqrt{1+\delta^2}}\gra{22}+ \mathcal{D}\gra{23},$$
we have a decomposition of $T_{\sigma,\gamma}$ as
$$T_{\sigma,\gamma}=\sum_{j=1}^{3^n}T_{\sigma,\gamma}(j),$$
such that each $T_{\sigma,\gamma}(j)$, $2\leq j\leq 3^n$ , is a scalar multiple of a diagram with at most $n-1$ labels $R$, and $T_{\sigma,\gamma}(1)$ is $\displaystyle \mathcal{D}^{n}$ times a diagram with $n$ labels $R$.
Moreover, we can apply the Fourier transform to the $n$ labels $R$ of this diagram $W_\sigma$ times in total, such that this diagram becomes $T$. Note that $W_\sigma$ mod 4 only depends on $\sigma$.

Recall that $Z(T_{\sigma,\gamma}(j))$, for $2\leq j\leq3^n$, have been defined by induction.
Let us define $\zeta_{\sigma,\gamma}(T)$ by the following equality,
\begin{equation}\label{zetadef1}
\text{HOMFLY}_{q,r}(T_{\sigma,\gamma})= \mathcal{D}^{n}i^{W_\sigma}\zeta_{\sigma,\gamma}(T)+\sum_{j=2}^{3^n}\zeta(T_{\sigma,\gamma}(j)).
\end{equation}
Let us define $\zeta(T)$ as
\begin{equation}\label{zetadef2}
\zeta(T)=\frac{1}{2^n 2^k}\sum_{\sigma\in\pm(T)}\sum_{\gamma\in\pm(\sigma)} \zeta_{\sigma,\gamma}(T).
\end{equation}
By induction and a linear extension, we obtain a function $\zeta$ on $\mathscr{C}_{\bullet}'$.

Now let us prove that the function $\zeta$ is a partition function on $\mathscr{C}_{\bullet}$ by passing to the intermediate quotients one by one.

\begin{lemma}
The element $\gra{r0}-\delta$ is in $\text{Ker}(\zeta)$, the kernel of $\zeta$.
Thus the function $\zeta$ is a partition function of $\mathscr{C}_{\bullet}'$ with circle parameter $\delta$.
\end{lemma}

\begin{proof}
Let $T$ be a disjoint union of two closed diagram $T^1$ and $T^2$.

Case 1: $T^1$ and $T^2$ are Temperley-Lieb. Obviously $\zeta(T)=\zeta(T^1)\zeta(T^2)$.

Case 2: $T^1$ (or $T^2$) is Temperley-Lieb.
Note that
$$\text{HOMFLY}_{q,r}(\gra{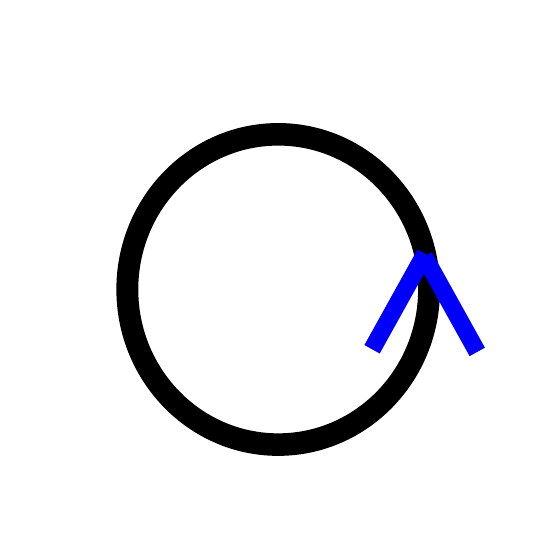})=\text{HOMFLY}_{q,r}(\gra{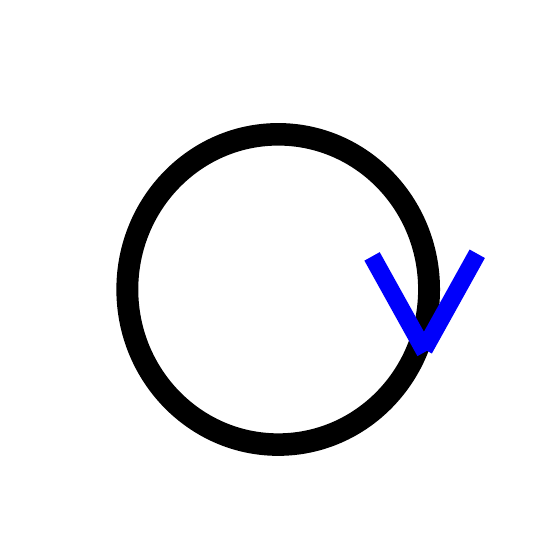})=\frac{r-r^{-1}}{q-q^{-1}}=\delta=\zeta(\gra{r0}),$$
so $\text{HOMFLY}_{q,r}$ coincide with $\zeta$ on closed Temperley-Lieb-Jones diagrams.
By an induction on the number of $R$'s in $T_2$, we have that $\zeta(T)=\zeta(T^1)\zeta(T^2)$.

The general case:
Note that the choices of orientations and braids in the definition of $\zeta$ are independent on disjoint components.
Moreover, the value of the HOMFLY-PT polynomial of the union of two disjoint links is the multiplication of that of the two links.
By an induction on the number of $R$'s in $T_1$ and $T_2$, we have that $\zeta(T)=\zeta(T^1)\zeta(T^2)$.

Recall that $\zeta(\gra{r0})=\delta$, so $\gra{r0}-\delta\in\text{Ker}(\zeta)$.
\end{proof}

\begin{lemma}\label{'to''}
The element $R-i\mathcal{F}(R)$ is in $\text{Ker}(\zeta)$.
Therefore $\zeta$ passes to the quotient $\mathscr{C}_{\bullet}''$.
\end{lemma}

\begin{proof}
For an annular tangle $\Psi \in Ann_2^0(n)$, take $T^0=\Psi(R)$, $T^1=\Psi(\mathcal{F}(R))$.
Then the choices of orientations and braids of $T^0$ coincide with those of $T^1$, i.e., $\pm(T^0)=\pm(T^1)$.
For any $\sigma\in\pm(T^0)$, and $\gamma\in\pm(\sigma)$, by Equation (\ref{zetadef1}), we have
$$\text{HOMFLY}_{q,r}(T^m_{\sigma,\gamma})= \mathcal{D}^{n+1}i^{W^m_\sigma}\zeta_{\sigma,\gamma}(T^m)+\sum_{j=2}^{3^n}\zeta(T^m_{\sigma,\gamma}(j)),$$
for some elements $T^m_{\sigma,\gamma}(j)$ with at most $n-1$ labels $R$, $2\leq j\leq 3^n$, $m=0,1$.
Note that
$$T^0_{\sigma,\gamma}=T^1_{\sigma,\gamma},\quad  T^0_{\sigma,\gamma}(j)=T^1_{\sigma,\gamma}(j), ~\forall~ 2\leq j\leq 3^n, \quad W^0_\sigma+1=W^1_\sigma,$$
so
$$\zeta_{\sigma,\gamma}(T^0)=i\zeta_{\sigma,\gamma}(T^1).$$
By Equation (\ref{zetadef2}), we have
$$\zeta(T^0)=i\zeta(T^1), ~i.e., ~\zeta(\Psi(R-i\mathcal{F}(R)))=0.$$
So $R-i\mathcal{F}(R)\in\text{Ker}(\zeta)$.
\end{proof}

\begin{lemma}\label{''to'''}
The element $R_1$ is in $\text{Ker}(\zeta)$.
Therefore $\zeta$ passes to the quotient $\mathscr{C}_{\bullet}'''$.
\end{lemma}

\begin{proof}
Let us prove that $R_1\in \text{Ker}(\zeta)$ by an inductive argument.

For an annular tangle $\Psi^0 \in Ann_1^0(0)$, take $T^0=\Psi^0(\gra{r1})$.
For any $\sigma\in\pm(T^0)$ and $\gamma\in\pm(\sigma)$, if $\gra{r1}$ is replaced by $\gra{homfly11}$ in $T^0_{\sigma,\gamma}$,
then by Equation (\ref{zetadef1}) and the Reidemester Move I
\begin{equation}\label{Move I 1}
\gra{homfly11}-r\gra{b1id}=\displaystyle \mathcal{D}R_1
\end{equation}
in Lemma \ref{Move I}, we have
$$\text{HOMFLY}_{q,r}(\Psi^0(\gra{homfly11}))= \mathcal{D}\zeta_{\sigma,\gamma}(T^0)+\zeta(\Psi^0(r\gra{b1id})).$$
Note that
$$\text{HOMFLY}_{q,r}(\Psi^0(\gra{homfly11}))=\text{HOMFLY}_{q,r}(\Psi^0(r\gra{b1id}))=\zeta(\Psi^0(r\gra{b1id})).$$
so $\zeta_{\sigma,\gamma}(T^0)=0$.
If $\gra{r1}$ is replaced by $\gra{homfly12}$, $\gra{homfly13}$ or $\gra{homfly12}$, then we still have $\zeta_{\sigma,\gamma}(T)=0$ by applying the corresponding Reidemester Move I in Lemma \ref{Move I} to a similar argument.
Therefore $\zeta(T^0)=0$, i.e., $\zeta(\Psi^0(R_1))=0$ by Equation \ref{zetadef2}.

Suppose
$$\zeta(\Psi^{k}(R_1))=0, ~\forall~ \Psi^{k}\in Ann_1^0(k), ~k<n,$$
for some $n>0$.
For an annular tangle $\Psi^n \in Ann_1^0(n)$, take $T=\Psi^n(\gra{r1})$.
For any $\sigma\in\pm(T)$ and $\gamma\in\pm(\sigma)$, let us define the annular tangle $\Psi^n_{\sigma,\gamma}$ to be the restriction of $T_{\sigma,\gamma}$ on $\Psi^n$.
Replacing the braids of $\Psi^n_{\sigma,\gamma}$ by Equations (\ref{braid+}), (\ref{braid-}),
we have a decomposition of $\Psi^n_{\sigma,\gamma}$ as
$$\Psi^n_{\sigma,\gamma}=\sum_{j=1}^{3^n}\Psi^n_{\sigma,\gamma}(j),$$
such that each $\Psi^n_{\sigma,\gamma}(j)$, $2\leq j\leq 3^n$ , is a scalar multiple of an annular tangle with at most $n-1$ labels $R$, and $\Psi^n_{\sigma,\gamma}(1)$ is $\displaystyle \mathcal{D}^{n}$ times an annular tangle with $n$ labels $R$.

If $\gra{r1}$ is replaced by $\gra{homfly11}$ in $T_{\sigma,\gamma}$,
then by Equation (\ref{zetadef1}) and the Reidemester Move I (\ref{Move I 1}), we have
\begin{equation}\label{equ2}
\text{HOMFLY}_{q,r}(\Psi^n_{\sigma,\gamma}(\gra{homfly11}))= \mathcal{D}^n i^{W_\sigma}\left( \mathcal{D}\zeta_{\sigma,\gamma}(T)+\zeta(\Psi^n(r\gra{b1id}))\right)+\sum_{j=2}^{3^n}\zeta(\Psi^n_{\sigma,\gamma}(j)(\gra{homfly11})).
\end{equation}
On the other hand
\begin{equation}\label{equ3}
\text{HOMFLY}_{q,r}(\Psi^n_{\bar{\sigma},\bar{\gamma}}(\gra{b1id}))= \mathcal{D}^n i^{W_\sigma}(\zeta_{\bar{\sigma},\bar{\gamma}}(\Psi^n(\gra{b1id})))+\sum_{j=2}^{3^n}\zeta(\Psi^n_{\bar{\sigma},\bar{\gamma}}(j)(\gra{b1id})),
\end{equation}
where $\bar{\sigma},\bar{\gamma}$ are the choices of orientations and braids of $\Psi^n(\gra{1id})$ corresponding to $\sigma,\gamma$ for $T$.

By induction and the Reidemester Move I (\ref{Move I 1}), we have
$$\zeta(\Psi^n_{\sigma,\gamma}(j)(\gra{homfly11}))-r\zeta(\Psi^n_{\sigma,\gamma}(j)(\gra{b1id}))= \mathcal{D}\Psi^n_{\sigma,\gamma}(j)(R_1)=0$$
for $2\leq j\leq 3^n$. Moreover,
$$\text{HOMFLY}_{q,r}(\Psi^n_{\sigma,\gamma}(\gra{homfly11}))=\text{HOMFLY}_{q,r}(\Psi^n_{\sigma,\gamma}(\gra{b1id})).$$
So Equation (\ref{equ2})-r(\ref{equ3}) implies
\begin{equation}\label{equ11}
\zeta_{\sigma,\gamma}(T)+r\left(\zeta(\Psi^n(\gra{1id}))-\zeta_{\bar{\sigma},\bar{\gamma}}(\Psi^n(\gra{1id}))\right)=0.
\end{equation}

If $\gra{r1}$ is replaced by $\gra{homfly12}$, $\gra{homfly13}$ or $\gra{homfly12}$,
then we still have Equation (\ref{equ11}) by applying the corresponding Reidemester Move I in Lemma \ref{Move I} to a similar argument.

Note that $\sigma\rightarrow\bar{\sigma}$ is a bijection from $\pm(\Psi^n(\gra{r1}))$ to $\pm(\Psi^n(\gra{1id}))$,
and $\gamma\rightarrow\bar{\gamma}$ is a double cover from $\pm(\sigma)$ to $\pm(\bar{\sigma})$.
Summing over all $\sigma,\gamma$ for Equation (\ref{equ11}), we have
$\zeta(T)=0$, i.e., $\zeta(\Psi^n(R_1))=0$ by Equation (\ref{zetadef2}).

By induction, we have $\zeta(\Psi(R_1))=0$, for any annular tangle $\Psi$. So $R_1\in \text{Ker}(\zeta)$ and $\zeta$ passes to the quotient $\mathscr{C}_{\bullet}'''$.
\end{proof}

\begin{lemma}\label{'''to''''}
The element $R_2$ is in $\text{Ker}(\zeta)$.
Therefore $\zeta$ passes to the quotient $\mathscr{C}_{\bullet}''''$.
\end{lemma}

\begin{proof}
The proof is a similar inductive argument as in the proof of Lemma \ref{''to'''}.

For an annular tangle $\Psi^0 \in Ann_2^0(0)$, take $T^0=\Psi^0(\gra{r21})$.
For any $\sigma\in\pm(T^0)$ and $\gamma\in\pm(\sigma)$, if $\gra{r21}$ is replaced by $\gra{homfly21}$ in $T^0_{\sigma,\gamma}$,
then by Equation (\ref{zetadef1}) and the Reidemester Move II
\begin{equation}\label{Move II 1}
\gra{homfly21}-\gra{homfly22}=\displaystyle \mathcal{D}^2R_2
\end{equation}
in Lemma \ref{Move II}, we have
$$\text{HOMFLY}_{q,r}(\Psi^0(\gra{homfly21}))= \mathcal{D}^2 (\zeta_{\sigma,\gamma}(T^0)+\zeta(\Psi^0(R_2-\gra{r21}))+\zeta(\Psi^0(\gra{homfly22})).$$
Note that
$$\text{HOMFLY}_{q,r}(\Psi^0(\gra{homfly21}))=\text{HOMFLY}_{q,r}(\Psi^0(\gra{homfly22}))=\zeta(\Psi^0(\gra{homfly22})),$$
so
\begin{equation}\label{equ4}
\zeta_{\sigma,\gamma}(T^0)+\zeta(\Psi^0(R_2-\gra{r21}))=0.
\end{equation}

If $\gra{r21}$ is replaced by the other 7 possibilities, then we still have $\zeta(\Psi^0(R_2))=0$ by applying the corresponding Reidemester Move II in Lemma \ref{Move II} to a similar argument.

Summing over all $\sigma,\gamma$, we have $\zeta(\Psi^0(R_2))=0$.

Suppose
$$\zeta(\Psi^{k}(R_2))=0, ~\forall~ \Psi^{k}\in Ann_2^0(k), ~k<n,$$
for some $n>0$.
For an annular tangle $\Psi^n \in Ann_2^0(n)$, take $T=\Psi^n(\gra{r21})$.
For any $\sigma\in\pm(T)$ and $\gamma\in\pm(\sigma)$, let
$$\Psi^n_{\sigma,\gamma}=\sum_{j=1}^{3^n}\Psi^n_{\sigma,\gamma}(j),$$
be the same decomposition as the one in the proof of Lemma \ref{''to'''}.

If $\gra{r21}$ is replaced by $\gra{homfly21}$ in $T_{\sigma,\gamma}$,
then by Equation (\ref{zetadef1}), we have
\begin{equation}\label{equ5}
\text{HOMFLY}_{q,r}(\Psi^n_{\sigma,\gamma}(\gra{homfly21}))= \mathcal{D}^n i^{W_\sigma}\left( \mathcal{D}^2\zeta_{\sigma,\gamma}(T)+\zeta(\Psi^n(
\gra{homfly21}-\mathcal{D}^2\gra{r21}))\right)+\sum_{j=2}^{3^n}\zeta(\Psi^n_{\sigma,\gamma}(j)(\gra{homfly21})).
\end{equation}
On the other hand
\begin{equation}\label{equ6}
\text{HOMFLY}_{q,r}(\Psi^n_{\bar{\sigma},\bar{\gamma}}(\gra{homfly22}))= \mathcal{D}^n i^{W_\sigma}(\zeta_{\bar{\sigma},\bar{\gamma}}(\Psi^n(\gra{homfly22})))+\sum_{j=2}^{3^n}\Psi^n_{\bar{\sigma},\bar{\gamma}}(j)(\gra{homfly22}),
\end{equation}
where $\bar{\sigma},\bar{\gamma}$ are the choices of orientations and braids of $\Psi^n(\gra{homfly22})$ corresponding to $\sigma,\gamma$ for $T$.
By induction and the Reidemester Move II (\ref{Move II 1}), we have
$$\Psi^n_{\sigma,\gamma}(j)(\gra{homfly21})-\Psi^n_{\sigma,\gamma}(j)(\gra{homfly22})= \mathcal{D}^2\Psi^n_{\sigma,\gamma}(j)(R_2)=0$$
for $2\leq j\leq 3^n$. Moreover,
$$\text{HOMFLY}_{q,r}(\Psi^n_{\sigma,\gamma}(\gra{homfly21}))=\text{HOMFLY}_{q,r}(\Psi^n_{\sigma,\gamma}(\gra{homfly22})).$$
So Equation (\ref{equ5})-(\ref{equ6}) implies
\begin{equation}
\mathcal{D}^2\zeta_{\sigma,\gamma}(T)+\zeta(\Psi^n(\gra{homfly21}-\mathcal{D}^2\gra{r21}))-\zeta_{\bar{\sigma},\bar{\gamma}}(\Psi^n(\gra{homfly22}))=0.
\end{equation}
By the Reidemester Move II (\ref{Move II 1}), we have
\begin{equation}\label{equ7}
\mathcal{D}^2\left(\zeta_{\sigma,\gamma}(T)-\zeta(\Psi^n(\gra{r21}))\right)+\left(\zeta(\Psi^n(\gra{21}))-\zeta_{\bar{\sigma},\bar{\gamma}}(\Psi^n(\gra{21}))\right)+\mathcal{D}^2\zeta(\Psi^n(R_2))=0.
\end{equation}

If $\gra{r21}$ is replaced by the other 7 possibilities, then we still have Equation (\ref{equ7}) by applying the corresponding Reidemester Move II in Lemma \ref{Move II} to a similar argument.

Note that $\sigma\rightarrow\bar{\sigma}$ is a bijection from $\pm(\Psi^n(\gra{r21}))$ to $\pm(\Psi^n(\gra{21}))$,
and $\gamma\rightarrow\bar{\gamma}$ is a 4-fold cover from $\pm(\sigma)$ to $\pm(\bar{\sigma})$.
Recall that $T=\Psi^n(\gra{r21})$.
Summing over all $\sigma,\gamma$ for Equation (\ref{equ7}), we have
$\zeta(\Psi^n(R_2))=0$ by Equation (\ref{zetadef2}).

By induction, we have $\zeta(\Psi(R_2))=0$, for any annular tangle $\Psi$. So $R_2\in \text{Ker}(\zeta)$ and $\zeta$ passes to the quotient $\mathscr{C}_{\bullet}''''$.
\end{proof}

\begin{lemma}\label{''''to}
The element $R_3$ is in $\text{Ker}(\zeta)$.
Therefore $\zeta$ passes to the quotient $\mathscr{C}_{\bullet}$.
\end{lemma}

\begin{proof}
The proof is a similar inductive argument as in the proof of Lemma \ref{''to'''}, \ref{'''to''''}.

For an annular tangle $\Psi^0 \in Ann_3^0(0)$, take $T^0=\Psi^0(\gra{15})$.
For any $\sigma\in\pm(T^0)$ and $\gamma\in\pm(\sigma)$,
if $\gra{15}$ is replaced by $\gra{homfly31}$ in $T^0_{\sigma,\gamma}$,
then by Equation (\ref{zetadef1}), we have
$$\text{HOMFLY}_{q,r}(\Psi^0(\gra{homfly31}))= \mathcal{D}^3 i^3\zeta_{\sigma,\gamma}(T^0)+\zeta(\Psi^0(\gra{homfly31}-\mathcal{D}^3i^3\gra{15})).$$

On the other hand, take $S^0=\Psi^0(\gra{16})$ and $\bar{\sigma}\in \pm(S^0),\bar{\gamma}\in \pm(\bar{\sigma})$ such that $S_{\bar{\sigma},\bar{\gamma}}$ is isotopic to $T_{\sigma,\gamma}$ by a Reidemester move III.
Then
by Equation (\ref{zetadef1}), we have
$$\text{HOMFLY}_{q,r}(\Psi^0(\gra{homfly32}))= \mathcal{D}^3 \zeta_{\bar{\sigma},\bar{\gamma}}(S^0)+\zeta(\Psi^0(\gra{homfly32}-\mathcal{D}^3\gra{16})).$$

Note that $\text{HOMFLY}_{q,r}(\Psi^0(\gra{homfly31}))=\text{HOMFLY}_{q,r}(\Psi^0(\gra{homfly32}))$. By the Reidemester Move III
\begin{equation}\label{Move III 1}
\gra{homfly31}-\gra{homfly32}=\displaystyle \mathcal{D}^3i^3R_3
\end{equation}
in Lemma \ref{Move III}, we have
\begin{equation}\label{equ31}
i^3\left(\zeta_{\sigma,\gamma}(T^0)-\zeta(\Psi^0(\gra{15}))\right)-\left(\zeta_{\bar{\sigma},\bar{\gamma}}(S^0)-\zeta(\Psi^0(\gra{16}))\right)+i^3\zeta(\Psi^0(R_3))=0.
\end{equation}

If $\gra{15}$ is replaced other 47 possibilities, then we still have Equation (\ref{equ31}) by applying the corresponding Reidemester Move III in Lemma \ref{Move III} to a similar argument.

Note that $\sigma\rightarrow\bar{\sigma}$ is a bijection from $\pm(T^0)$ to $\pm(S^0)$,
and $\gamma\rightarrow\bar{\gamma}$ is a bijection from $\pm(\sigma)$ to $\pm(\bar{\sigma})$.
Summing over all $\sigma,\gamma$, we have
\begin{equation*}
i^3\left(\zeta(T^0)-\zeta(\Psi^0(\gra{15}))\right)-\left(\zeta(S^0)-\zeta(\Psi^0(\gra{16}))\right)+i^3\zeta(\Psi^0(R_3))=0.
\end{equation*}
Recall that $T^0=\Psi^0(\gra{15})$, $S^0=\Psi^0(\gra{16})$,
so $\zeta(\Psi^0(R_3))=0$.

Suppose
$$\zeta(\Psi^{k}(R_3))=0, ~\forall~ \Psi^{k}\in Ann_3^0(k), ~k<n,$$
for some $n>0$.
For an annular tangle $\Psi^n \in Ann_3^0(0)$, take $T=\Psi^n(\gra{15})$.
For any $\sigma\in\pm(T)$ and $\gamma\in\pm(\sigma)$, let
$$\Psi^n_{\sigma,\gamma}=\sum_{j=1}^{3^n}\Psi^n_{\sigma,\gamma}(j),$$
be the same decomposition as the one in the proof of Lemma \ref{''to'''}.

If $\gra{15}$ is replaced by $\gra{homfly31}$ in $T_{\sigma,\gamma}$,
then by Equation (\ref{zetadef1}), we have
\begin{align*}
&\text{HOMFLY}_{q,r}(\Psi^n_{\sigma,\gamma}(\gra{homfly31}))\\
&=\mathcal{D}^n i^{W_\sigma}\left( \mathcal{D}^3i^3\zeta_{\sigma,\gamma}(T)+\zeta(\Psi^n(
\gra{homfly31}-\mathcal{D}^3i^3\gra{15}))\right)+\sum_{j=2}^{3^n}\zeta(\Psi^n_{\sigma,\gamma}(j)(\gra{homfly31})).
\numberthis \label{equ8}
\end{align*}

On the other hand, take $S=\Psi^n(\gra{16})$, we have
\begin{align*}
&\text{HOMFLY}_{q,r}(\Psi^n_{\bar{\sigma},\bar{\gamma}}(\gra{homfly32}))\\
&= \mathcal{D}^n i^{W_\sigma}\left( \mathcal{D}^3\zeta_{\bar{\sigma},\bar{\gamma}}(S)+\zeta(\Psi^n(
\gra{homfly32}-\mathcal{D}^3\gra{16}))\right)+\sum_{j=2}^{3^n}\zeta(\Psi^n_{\bar{\sigma},\bar{\gamma}}(j)(\gra{homfly32})).  \numberthis \label{equ9}
\end{align*}
where $\bar{\sigma},\bar{\gamma}$ are the corresponding choices of orientations and braids of $\Psi^n(\gra{16})$, such that $\Psi^n_{\sigma,\gamma}=\Psi^n_{\bar{\sigma},\bar{\gamma}}$.

By induction and the Reidemester Move III (\ref{Move III 1}), we have
$$\zeta(\Psi^n_{\sigma,\gamma}(j)(\gra{homfly31}))-\zeta(\Psi^n_{\bar{\sigma},\bar{\gamma}}(j)(\gra{homfly32}))= \mathcal{D}^3i^3\zeta(\Psi^n_{\sigma,\gamma}(j)(R_3))=0,$$
for $2\leq j\leq 3^n$. Moreover,
$$\text{HOMFLY}_{q,r}(\Psi^n_{\sigma,\gamma}(\gra{homfly31}))=\text{HOMFLY}_{q,r}(\Psi^n_{\sigma,\gamma}(\gra{homfly32})).$$
Applying the Reidemester Move III (\ref{Move III 1}) to Equation (\ref{equ8})-(\ref{equ9}), we have
\begin{equation}\label{equ32}
i^3\left(\zeta_{\sigma,\gamma}(T)-\zeta(\Psi^n(\gra{15}))\right)-\left(\zeta_{\bar{\sigma},\bar{\gamma}}(S)-\zeta(\Psi^n(\gra{16}))\right)+i^3\zeta(\Psi^n(R_3))=0.
\end{equation}

If $\gra{15}$ is replaced by the other 47 possibilities, then we still have Equation (\ref{equ32}) by applying the corresponding Reidemester Move III in Lemma \ref{Move III} to a similar argument.

Note that $\sigma\rightarrow\bar{\sigma}$ is a bijection from $\pm(T^0)$ to $\pm(S^0)$,
and $\gamma\rightarrow\bar{\gamma}$ is a bijection from $\pm(\sigma)$ to $\pm(\bar{\sigma})$.
Summing over all $\sigma,\gamma$, we have
\begin{equation*}
i^3\left(\zeta(T)-\zeta(\Psi^n(\gra{15}))\right)-\left(\zeta(S)-\zeta(\Psi^n(\gra{16}))\right)+i^3\zeta(\Psi^n(R_3))=0.
\end{equation*}
Recall that $T=\Psi^n(\gra{15})$, $S=\Psi^n(\gra{16})$,
so $\zeta(\Psi^0(R_3))=0$.

By induction, we have $\zeta(\Psi(R_1))=0$, for any annular tangle $\Psi$. So $R_2\in \text{Ker}(\zeta)$ and $\zeta$ passes to the quotient $\mathscr{C}_{\bullet}$.
\end{proof}

\begin{theorem}[Consistency]\label{consistency}
The general planar algebra $\mathscr{C}_{\bullet}$ is a planar algebra over $\mathbb{C}$ for any $\delta\in\mathbb{R}$.
\end{theorem}

\begin{proof}
The general planar algebra $\mathscr{C}_{\bullet}$ is evaluable by Theorem \ref{evaluable}.
By Lemma \ref{''''to}, the partition function $\zeta$ passes to the quotient $\mathscr{C}_{\bullet}$. So any evaluation of a closed diagram $T$ is $\zeta(T)$.
\end{proof}

Recall that $\displaystyle  q=\frac{i+\delta }{\sqrt{1+\delta^2}}$, so
$\displaystyle \delta=\frac{i(q+q^{-1})}{q-q^{-1}}$, $r=iq^{-1}$.
Therefore the Yang-Baxter relation (\ref{YBrelation}) for $\mathscr{C}_{\bullet}$ is also a relation over the field $\mathbb{C}(q)$.

\begin{corollary}
The general planar algebra $\mathscr{C}_{\bullet}=\mathscr{C}(q)_{\bullet}$ is a planar algebra over $\mathbb{C}(q)$.
\end{corollary}

\begin{proof}
Over the field $\mathbb{C}(q)$, any two evaluations of a closed diagram in $\mathscr{C}_{\bullet}$ are two rational functions over $q$. Moreover, the two rational functions have the same value for $q=\frac{i+\delta }{\sqrt{1+\delta^2}}$, $\delta\in\mathbb{R}$ by Theorem \ref{consistency}, so they are the same.
Therefore the Yang-Baxter relation is consistent over $\mathbb{C}(q)$.
\end{proof}

\section{Representations}\label{subsection matrix units}
We have constructed the planar algebra $\mathscr{C}_{\bullet}$ over the field $\mathbb{C}(q)$.
The next step is to find out all values of $q$, such that the planar algebra $\mathscr{C}_{\bullet}$ has a positive partition function with respect to an involution $*$. Then (the quotient of) $\mathscr{C}_{\bullet}$ is a subfactor planar algebra over the field $\mathbb{C}$. It is easy to figure out the unique possible involution $*$ on $\mathscr{C}_{\bullet}$. It seems impossible to show that the partition function is positive directly.
We prove the positivity by three steps:

(1) We construct the matrix units of $\mathscr{C}_{\bullet}$ over $\mathbb{C}(q)$ in this Section. Its minimal idempotents are indexed by Young diagrams.

(2) We compute the trace formula for the minimal idempotents of $\mathscr{C}_{\bullet}$ over $\mathbb{C}(q)$ in Section \ref{subsection trace formula}. For a minimal idempotent labeled by a Young diagram $\lambda$, its trace $<\lambda>$ is given in Theorem \ref{main theorem trace formula}.

Technically, the construction of the matrix units relies on the trace formula. On the other hand, the computation of the trace formula relies on the construction of the matrix units. The order of constructing matrix units and computing the trace formula is special and very delicate.

(3) The positivity of partition function can only be obtained at $\displaystyle q=e^{\frac{i\pi}{2N+2}}$, for $N\in \mathbb{N}^+$.
When $\displaystyle q=e^{\frac{i\pi}{2N+2}}$, $\mathscr{C}_{\bullet}$ is not semisimple over $\mathbb{C}$. (In this paper, it is semisimple means that it is a direct sum of full matrix algebras over a field $k$.)
We give a method to determine the semisimple quotient without knowing the presumed semi-simple quotients. We construct the semi-simple quotient $\mathscr{C}^{N}_{\bullet}$ by showing that certain matrix units are still well-defined while passing from the field $\mathbb{C}(q)$ to $\mathbb{C}$. Then we prove that $\mathscr{C}_{\bullet}$ has a positive partition function and the semi-simple quotient $\mathscr{C}^{N}_{\bullet}$ is a subfactor planar algebra.

First let us prove that $\mathscr{C}_{\bullet}$ is semisimple over $\mathbb{C}(q)$ and construct its matrix units. Consequently, we obtain all irreducible representations of $\mathscr{C}_{m}$. They are indexed by Young diagrams with $m'$ cells, $m'\leq m$ with the same parity.

Recall that the braid $\gra{b+}$ satisfies the Hecke relation, so $\mathscr{C}_{\bullet}$ has a subalgebra $H_{\bullet}$, the Hecke algebra of type $A$ with parameters $q$, $r$ subject to $r=iq^{-1}$.
Moreover $\mathscr{C}_n/\mathscr{I}_n\cong H_n$, where $\mathscr{I}_n$ is the two sided ideal of $\mathscr{C}_n$ generated by the Jones projection, namely the basic construction ideal.
The Bratteli diagram of $H_{\bullet}$ is Young's Lattice, denoted by $YL$, so the principal graph of (a proper quotient of) $\mathscr{C}_{\bullet}$ is a subgraph of Young's Lattice.
To construct the matrix units of $\mathscr{C}_{\bullet}$, we need to decompose minimal idempotents of $\mathscr{C}_n$ in $\mathscr{C}_{n+1}$. This decomposition can be derived from Wenzl's formula for the basic construction for $\mathscr{C}_{n-1}\subset\mathscr{C}_n$ and the branching formula for $H_{\bullet}$. The basic construction and Wenzl's formula will work, if $\mathscr{C}_n$ is semisimple and the trace $tr_n$ is non-degenerate. To ensure the two conditions, let us take the ground field to be $\mathbb{C}(q)$ first.
We are going to prove that $\mathscr{C}_{\bullet}$ over the field $\mathbb{C}(q)$ is isomorphic to the string algebra of the Young's Lattice starting from the empty Young diagram.

\begin{definition}
The string algebra $YL_{\bullet}$ of $YL$ over the field $\mathbb{C}(q)$ is an inclusion of semisimple algebras $YL_{n}$, $n=0,1,\cdots$.
Moreover, the basis of $YL_{n}$ consists of all length $2n$ loops of $YL$ starting from $\emptyset$. The multiplication of $YL_{n}$ is a linear extension of the multiplication of length $2n$ loops.
The inclusion $\iota: YL_{n}\to YL_{n+1}$ is a linear extension of
\begin{align*}
\iota(t\tau^{-1})=\sum_{s(e)=v} tee^{-1}\tau^{-1},
\end{align*}
where $t$ and $\tau$ are length $n$ paths from $\emptyset$ to some vertex $v$, and $s(E)$ is the source vertex of the edge $E$.
\end{definition}

\begin{definition}
For $n\geq1$, the vertices of $YL$, whose distance to $\emptyset$ is at most $n-1$, and the edges between these vertices form a subgraph of $YL$, denoted by $YL^{n-1}$.
Let $IYL_n$ to be the subspace of $YL_n$ whose basis consisting of all length $2n$ loops of $YL^{n-1}$ starting from $\emptyset$.
Let $HYL_n$ to be the subspace of $YL_n$ whose basis consisting of all length $2n$ loops passing a vertex in $YL\setminus YL^{n-1}$ starting from $\emptyset$.
\end{definition}

\begin{lemma}\label{y=y+h}
The subspace $IYL_n$ is a two sided ideal of $YL_n$,
$YL_n= IYL_n \oplus HYL_n$,
and
$HYL_n\simeq H_n$ as an algebra, for $n\geq1$.
\end{lemma}

\begin{proof}
They follow from the definitions.
\end{proof}

\begin{notation}
The elements $x\otimes 1$, $x\otimes \cap$, $x\otimes \cup$, are adding a string, a cap $\cap$, a cup $\cup$ to the right of $x$ respectively.
\end{notation}

\begin{theorem}[matrix units]\label{P=YL}
Over the field $\mathbb{C}(q)$, $\mathscr{C}_{\bullet}\cong YL_{\bullet}$ as a filtered algebra.
\end{theorem}

(A trace of a semisimple algebra is non-degenerate if and only if the trace of any minimal idempotent is non-zero.)

\begin{proof}
Note that $TL_0$ and $\mathscr{C}_0$ are isomorphic to the ground field $\mathbb{C}(q)$.
We set up $\omega_0: YL_{0}\to \mathscr{C}_{0}$ to be the isomorphism.
Moreover, the minimal idempotent $\emptyset$ is $1$.

We are going to prove the following properties of $\mathscr{C}_m$ inductively for $m\geq1$.

\begin{itemize}
\item[(1)] $\mathscr{C}_m$ is a finite dimensional semisimple algebra and its trace is non-degenerate.

Then the two sided ideal $\mathscr{I}_m$ is a finite dimensional semisimple algebra, so it has a unique maximal idempotent, called the support of $\mathscr{I}_m$.
Moreover, its support is central in $\mathscr{C}_m$.
Let $s_m$ be the complement of the support of $\mathscr{I}_m$.

\item[(2)] $\mathscr{C}_m=\mathscr{I}_m\oplus s_m\mathscr{C}_m$, for a central idempotent $s_m\in\mathscr{C}_m$ orthogonal to $\mathscr{I}_m$ with respect $tr_m$.

Note that $\mathscr{C}_{m}$ has a subalgebra $H_{m}$ generated by the braid $\gra{b+}$.
Moreover, $s_m$ is central and $s_me_i=0$, for any $1\leq i\leq m-1$, so $s_m\mathscr{C}_m=s_mH_m$ by Proposition \ref{algebragenerated}.
For each equivalence class of minimal idempotents of $H_m$ corresponding to the Young diagram $\lambda$, $|\lambda|=m$, we have a minimal idempotent $y_\lambda$ in $H_m$ (see Section \ref{hecke} for the construction of $y_\lambda$).
Thus $s_my_\lambda$ is either a minimal idempotent of $s_mH_m$ or zero.
\item[(3)] For any $|\lambda|=m$, $\tilde{y}_{\lambda}=s_my_\lambda$ is a minimal idempotent in $\mathscr{C}_m$ with a non-zero trace $<\lambda>$.

For a length $m$ path $t$ in $YL$ from $\emptyset$ to $\lambda$,
take $t'$ to be the first length $(m-1)$ sub path of $t$ from $\emptyset$ to $\mu$.
Let us define $\tilde{P}^{\pm}_t$ by induction as follows,
\begin{align*}
P^{\pm}_\emptyset&=\emptyset &\\
\tilde{P}^+_t&=(\tilde{P}^+_{t'}\otimes 1)\rho_{\mu < \lambda} \tilde{y}_{\lambda}, &\text{ when } \mu<\lambda\\
\tilde{P}^+_t&=\frac{<\lambda>}{<\mu>}(\tilde{P}^+_{t'}\otimes 1)(\rho_{\mu>\lambda}\otimes 1) (\tilde{y}_{\lambda}\otimes \cap), &\text{ when } \mu>\lambda\\
\tilde{P}^-_t&=(\tilde{P}^-_{t'}\otimes \cup)(\rho_{\mu < \lambda}\otimes 1) (\tilde{y}_{\lambda}\otimes 1), &\text{ when } \mu<\lambda\\
\tilde{P}^-_t&=\tilde{P}^+_{t'}\rho_{\mu>\lambda} (\tilde{y}_{\lambda}\otimes 1), &\text{ when } \mu>\lambda
\end{align*}
(see Section \ref{hecke} for the construction of $\rho_{\mu < \lambda}$)
\item[(4)] The map $\omega_m: YL_m \to \mathscr{C}_m$ as a linear extension of
\begin{align*}
\omega_m(t\tau^{-1})=\tilde{P}^+_t \tilde{P}^{-}_\tau
\end{align*}
is an algebraic isomorphism.
\item[(5)] $\omega_m(\iota(x))=\omega_{m-1}(x)\otimes 1$, $\forall$ $x\in TL_{m-1}$.
\end{itemize}

When $m=1$, it is easy to check Properties (1)-(5).
Suppose Property (1)-(5) hold for $m=1,2,\cdots, n$, $n\geq 1$, let us prove them for $m=n+1$.

By Property (4),(5), we have an isomorphism $\omega_{n}: YL_{n} \to \mathscr{C}_{n}$, such that $\omega_{n}(\iota(x))=\omega_{n-1}(x)\otimes 1$, for any $x\in YL_{n-1}$.
So $\mathscr{C}_{n-1}\subset \mathscr{C}_{n} \cong YL_{n-1}\subset YL_{n}$ is an inclusion of finite dimensional semisimple algebras.

By Property (1), $\mathscr{C}_{n-1}\subset \mathscr{C}_{n}$ is an inclusion of finite dimensional semisimple algebras with a non-degenerate trace.
We have the basic construction $\mathscr{C}_{n-1}\subset \mathscr{C}_{n}\subset \mathscr{I}_{n+1}$ \cite{GHJ}. Then $\mathscr{I}_{n+1}$ is a finite dimensional semisimple algebra.
Moreover, the trace of $\mathscr{I}_{n+1}$ is determined by the trace of $\mathscr{C}_{n-1}$. Since the trace of any minimal idempotent in $\mathscr{C}_{n-1}$ is non-zero by Property (1), the trace of any minimal idempotent in $\mathscr{I}_{n+1}$ is also non-zero. So the trace of $\mathscr{I}_{n+1}$ is non-degenerate.

Let us define $s_{n+1}$ to be the complement of the support of $\mathscr{I}_{n+1}$, then $\mathscr{C}_{n+1}=\mathscr{I}_{n+1}\oplus s_{n+1}\mathscr{C}_{n+1}$. Property (2) holds for $m=n+1$.

Moreover, we have $s_{n+1}\mathscr{C}_{n+1}=s_{n+1}H_{n+1}$.
For any $|\lambda|=n+1$, the minimal idempotent $\tilde{y}_{\lambda}=s_{n+1}y_\lambda$ in $\mathscr{C}_{n+1}$ has a non-zero trace $<\lambda>$ by Theorem \ref{trace formula}.
(The computation of $<\lambda>$ in Theorem \ref{trace formula} only requires the matrix units of $\mathscr{C}_k$, $k\leq n+1$, which have been constructed by induction.)
Property (3) holds for $m=n+1$.

Furthermore, $s_{n+1}H_{n+1}\cong H_{n+1}$ as a finite dimensional semisimple algebra. Therefore $\mathscr{C}_{n+1}$ is a finite dimensional semisimple algebra. Property (1) holds for $m=n+1$.

By Properties (4) and (5), $\mathscr{C}_{n-1}\subset \mathscr{C}_{n} \cong YL_{n-1}\subset YL_{n}$ is an inclusion of finite dimensional semisimple algebras.
By the basic construction, we can define an isomorphism $\omega_m: IYL_{n+1}\to \mathscr{I}_{n+1}$ with Property (4).
Note that $HYL_{n+1}\cong H_{n+1} \cong s_nH_n=s_n\mathscr{C}_n$, $YL_n=HYL_{n+1}\oplus HYL_{n+1}$ and $\mathscr{C}_n=\mathscr{I}_n\oplus s_n\mathscr{P_n}$, so we can extend $\omega_m$ as an isomorphism $\omega_m: YL_{n+1}\to \mathscr{C}_n$ with Property (4).

Property (5) for $m=n+1$ follows from Wenzl's formula (see Appendix \ref{Appendix:Wenzl's formula} for a proof):
\begin{align}
\tilde{y}_\mu \otimes 1 &=\sum_{\lambda<\mu} \frac{<\lambda>}{<\mu>} (\tilde{y}_\mu\otimes 1)(\rho_{\mu>\lambda} \otimes 1) (\tilde{y}_{\lambda}\otimes \cap) (\tilde{y}_{\lambda}\otimes \cup) (\rho_{\lambda<\mu} \otimes 1) (\tilde{y}_\mu \otimes 1) & \nonumber\\
                        &+\sum_{\lambda>\mu} (\tilde{y}_\mu\otimes 1) \rho_{\mu<\lambda} \tilde{y}_{\lambda} \rho_{\lambda>\mu} (\tilde{y}_\mu\otimes 1),
                        &\forall |\mu|\leq n-1. \label{wenzlformula}
\end{align}

Therefore Properties (1)-(5) hold for all $m$ by induction, and $\mathscr{C}_{\bullet}\cong YL_{\bullet}$ as a filtered algebra
\end{proof}

\begin{corollary}
The dimension of $\mathscr{C}_m$ is given by
$$\dim(\mathscr{C}_n)=(2m-1)!!$$
\end{corollary}

\begin{proof}
Note that $\mathscr{C}_{\bullet}\cong YL_{\bullet}$ is isomorphic to the Brauer algebra as a filtered algebra. Therefore,
$\dim(\mathscr{C}_m)=(2m-1)!!.$
\end{proof}
\section{The trace formula}\label{subsection trace formula}
Recall that the minimal idempotents of $\mathscr{C}_{\bullet}=\mathscr{C}(q)_{\bullet}$ are labeled by Young diagrams. The trace of a minimal idempotent is also called the quantum dimension of the corresponding representation.
In this section, we compute the trace formula for $\mathscr{C}_{\bullet}$ and prove Theorem \ref{main theorem trace formula}.

The $q$-Murphy operator is usually constructed by a braid and used to compute the trace formula for centralizer algebras.
For the BMW planar algebra, this was done by Beliakova and Blanchet in \cite{BelBla} which was inspired by the work of Nazarov in \cite{Naz96}.

In $\mathscr{C}_{\bullet}$, there is no braid. Instead, there is a half-braiding given by the solution of the Yang-Baxter equation in Lemma \ref{solution of YBE -1}. We construct a $q$-Murphy operator for $\mathscr{C}_{\bullet}$ by the half-braiding.

When $\delta\in \mathbb{R}$, we introduced the notations $\mathcal{D}=\displaystyle \frac{\delta}{\sqrt{1+\delta^2}}$, $\displaystyle r=\frac{\delta i+1}{\sqrt{1+\delta^2}}$, $\displaystyle q=\frac{i+\delta }{\sqrt{1+\delta^2}}$, and $|r|=|q|=1$.
Over the field $\mathbb{C}(q)$, let us define $r=iq^{-1}$, $\displaystyle \delta=\frac{i(q+q^{-1})}{q-q^{-1}},$ and $\displaystyle D=\frac{q+q^{-1}}{2}$.

\begin{notation}
Let us define
$$\alpha=\gra{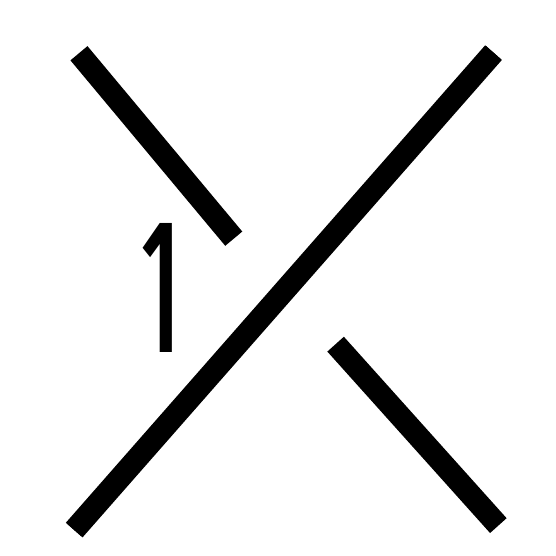}=\frac{q-q^{-1}}{2}\gra{21}+\frac{q-q^{-1}}{2i}\gra{22}+ \mathcal{D}\gra{23};$$
$$\beta=\gra{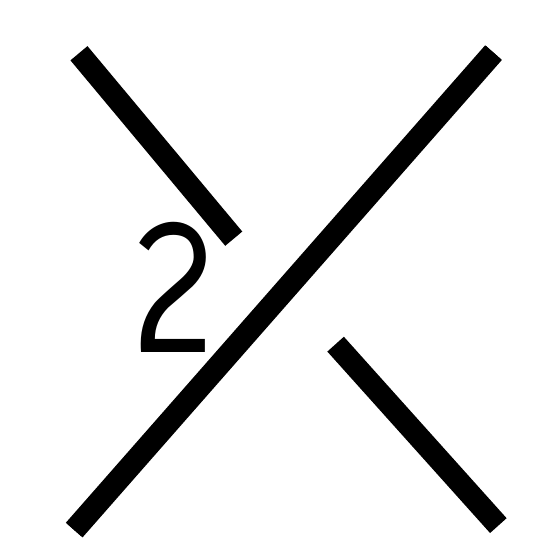}=\frac{q-q^{-1}}{2}\gra{21}-\frac{q-q^{-1}}{2i}\gra{22}+ \mathcal{D}\gra{23}.$$
Then their inverses are given by
$$\alpha^{-1}=\gra{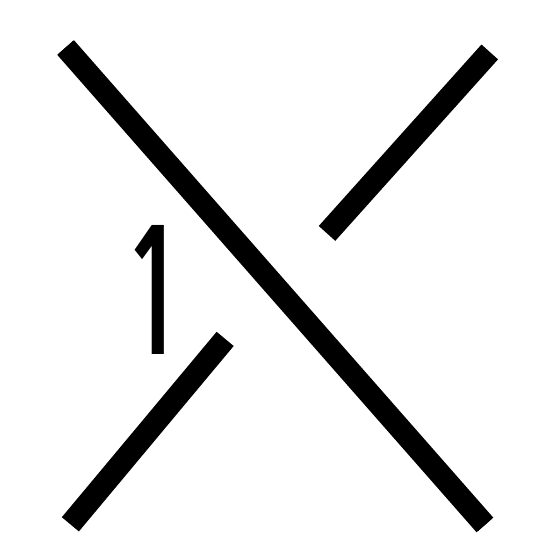}=-\frac{q-q^{-1}}{2}\gra{21}+\frac{q-q^{-1}}{2i}\gra{22}+ \mathcal{D}\gra{23};$$
$$\beta^{-1}=\gra{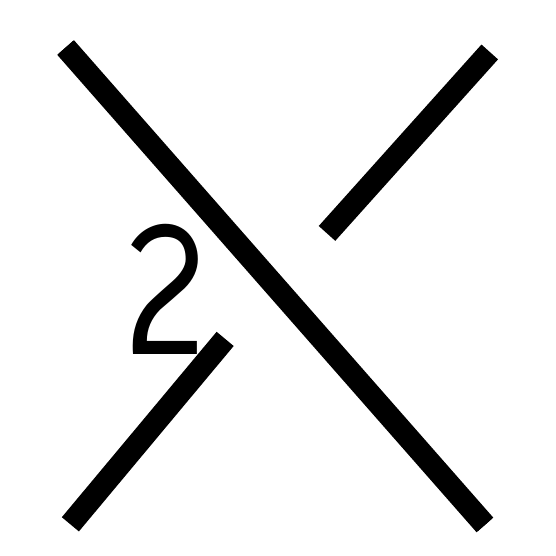}=-\frac{q-q^{-1}}{2}\gra{21}-\frac{q-q^{-1}}{2i}\gra{22}+ \mathcal{D}\gra{23}.$$
\end{notation}

Actually $\gra{b1}=\gra{b+}$. The orientation of $\gra{b+}$ was useful in the proof of the consistency, but it would be confusing in the rest computations. We change the notation to $\gra{b1}$.

\begin{proposition}\label{braidfourier}
In $\mathscr{C}_{\bullet}$, we have
$$\gra{b1}=i\gra{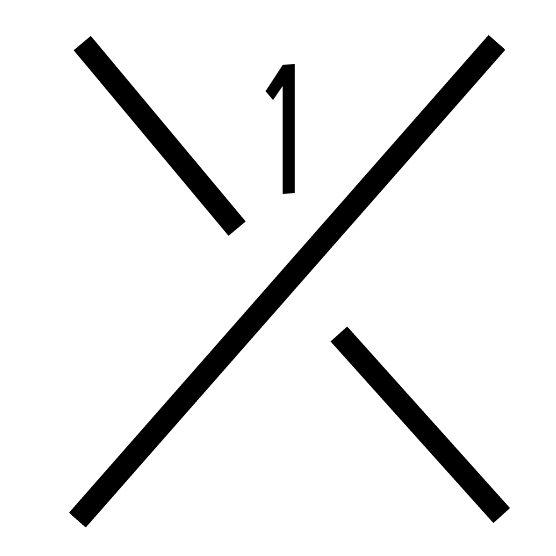}=-\gra{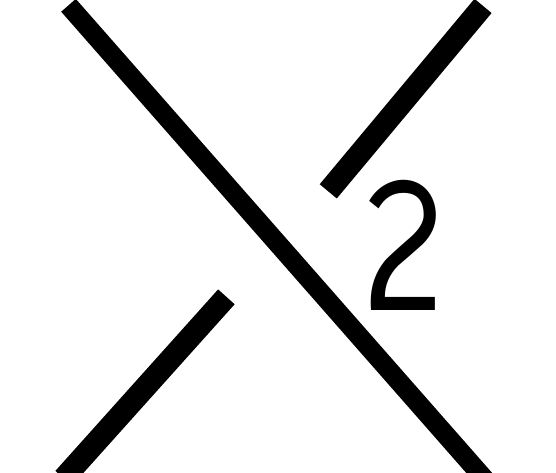}=-i\gra{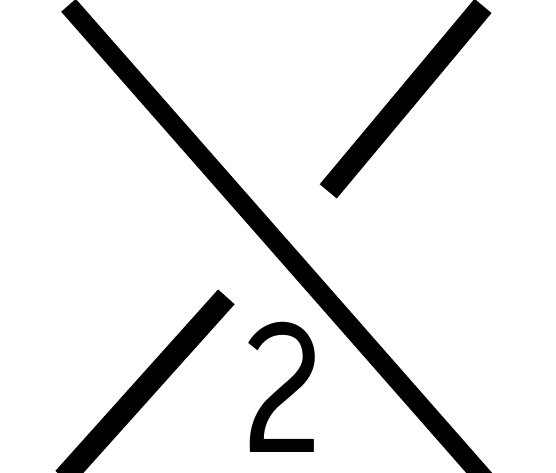}.$$
Equivalently, $$\gra{b2}=i\gra{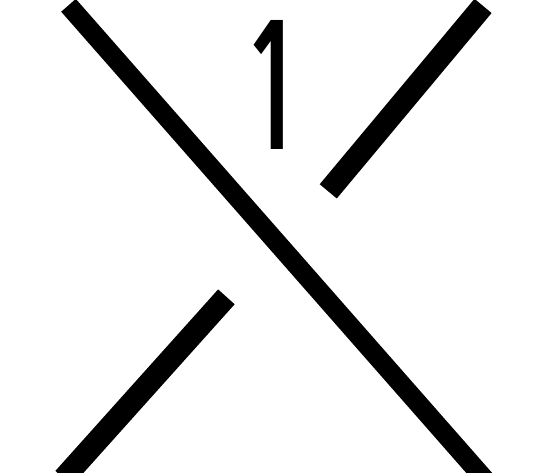}=-\gra{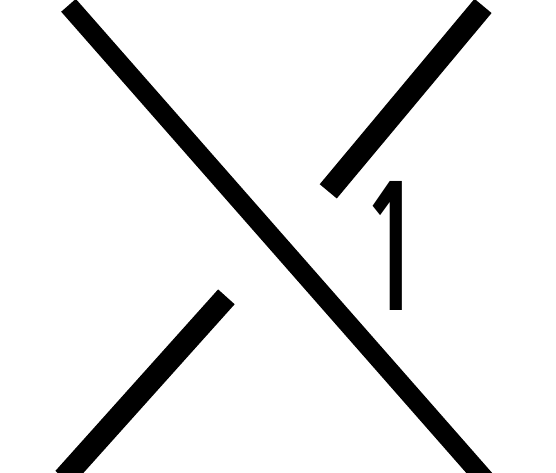}=-i\gra{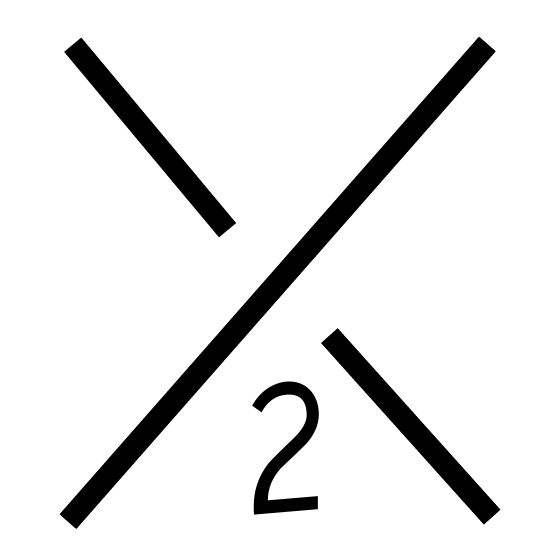}.$$
\end{proposition}

\begin{proof}
They follow from the definitions and the fact that $\mathcal{F}(R)=-iR$.
\end{proof}

\begin{proposition}[half-braidings\footnote{The operator $\gra{b1}$ gives a half-braiding while considering the unshaded planar algebra $\mathscr{C}_{\bullet}$ as a $\mathbb{N}\cup\{0\}$ graded semisimple tensor category.}]\label{flat}
For any element $a\in \mathscr{C}_{\bullet}$, we have
$$\graa{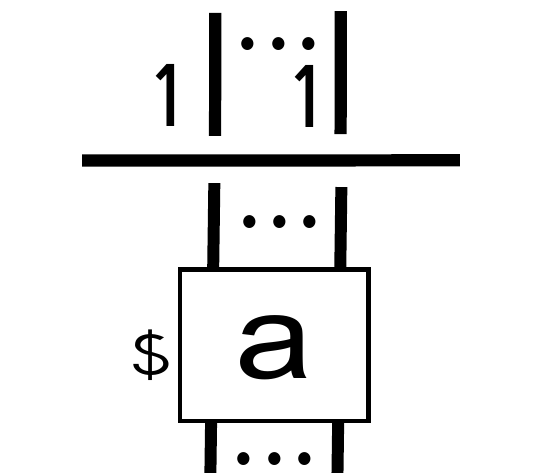}=\graa{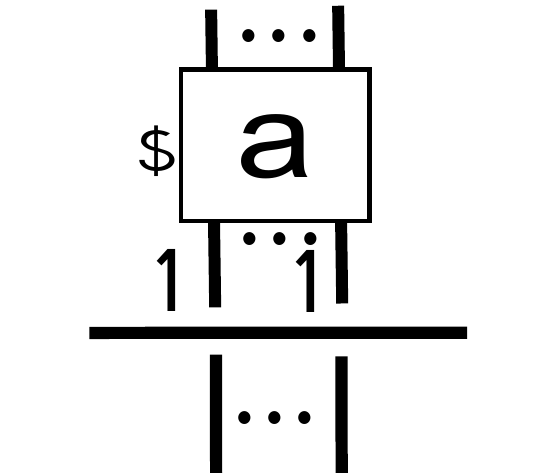}; \quad \graa{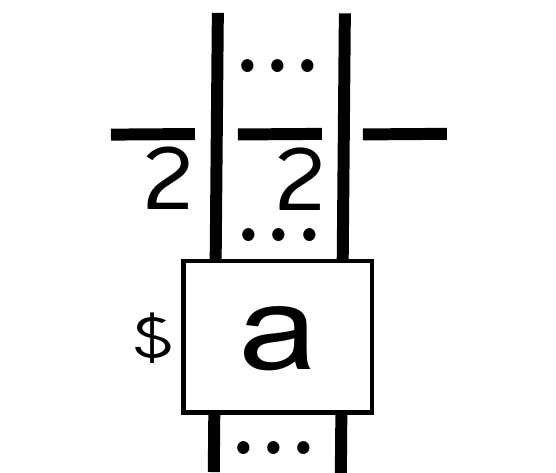}=\graa{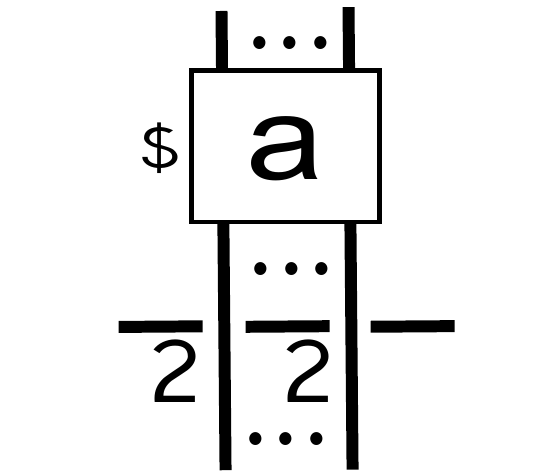}.$$
\end{proposition}

\begin{proof}
By Proposition \ref{braidfourier}, we have
$$\gra{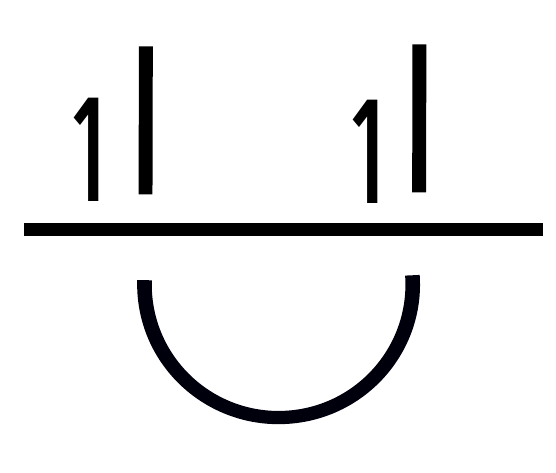}=i\gra{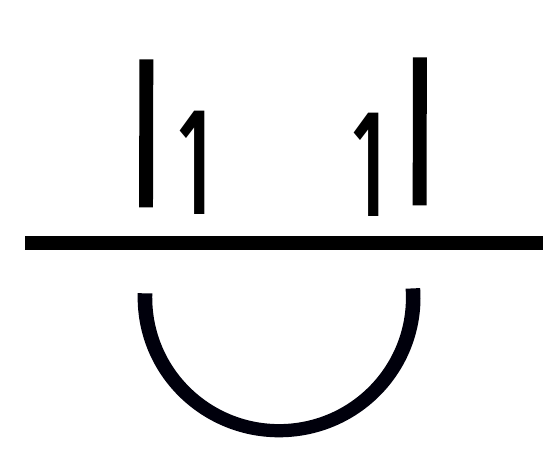}=i\gra{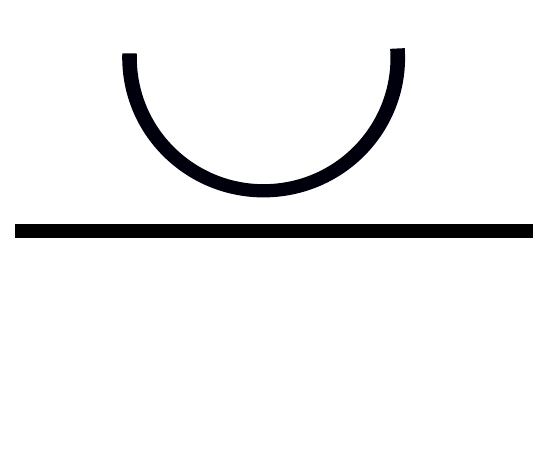};$$
$$\gra{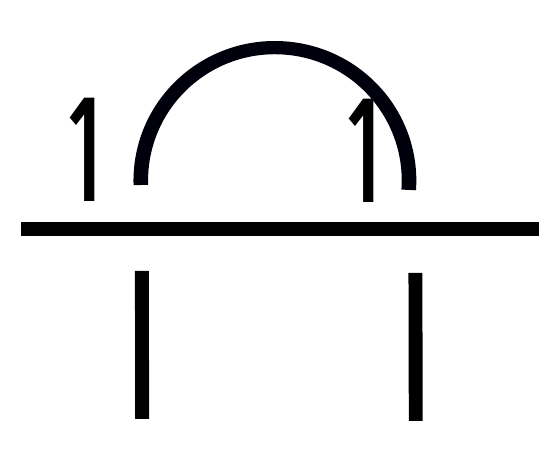}=i\gra{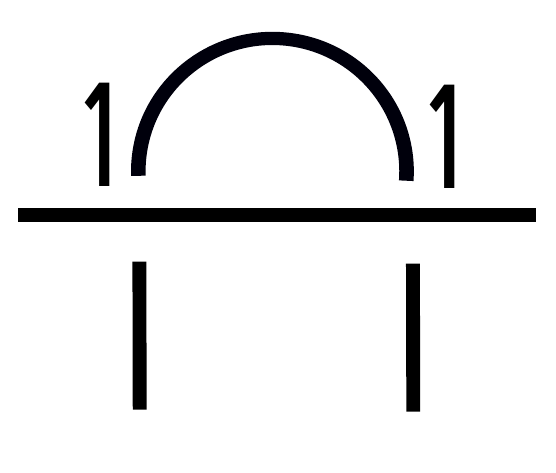}=i\gra{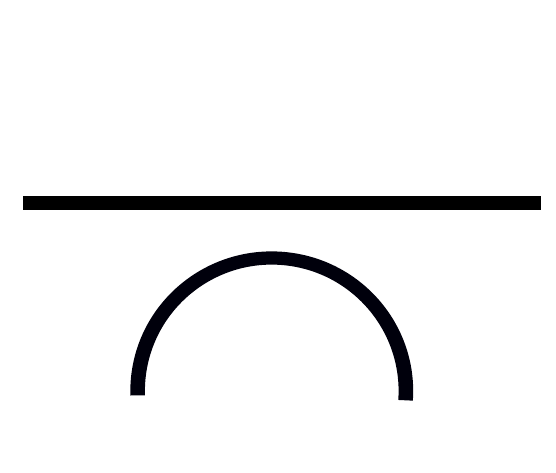}.$$
So the equation $\graa{flat11}=\graa{flat12}$ holds for $a=\gra{22}$.
By Lemma \ref{solution of YBE -1}, it also holds for $a=\gra{b1}$.
So it holds for any element $a$ by Proposition \ref{algebragenerated}.

We can prove the equation $\graa{flat21}=\graa{flat22}$ in a similar way.
\end{proof}

\begin{notation}
Let $\alpha_{n},\beta_{n},h_{n}$ be the diagrams by adding $n-1$ through strings to the left of $\gra{b1},\gra{b2}, \gra{22}$ respectively.
\end{notation}

Recall that $H_{\bullet}$ is the Hecke algebra generated by $\gra{b1}$.
The $n$-box, $n\geq1$, $$\varepsilon_{n}=\grc{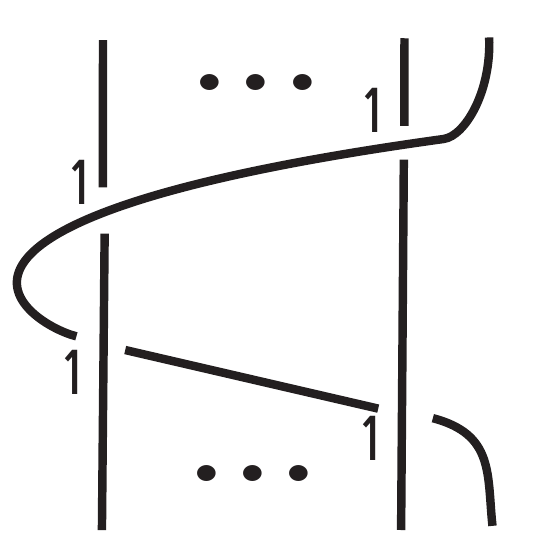}$$ is the $q$-Murphy operator of $H_{\bullet}$.

For $|\mu|=n$, $\lambda>\mu$,
$\rho_{\lambda>\mu}$ is an intertwiner from $\lambda$ to $\mu \otimes 1$, and $y_\mu$, $y_\lambda$ are the minimal idempotents corresponding to $\mu$ and $\lambda$ respectively. (See Section \ref{hecke} for the construction.)
So $\rho_{\lambda>\mu} = y_\lambda\rho_{\lambda>\mu} (y_\mu\otimes 1)$.
Then $\rho_{\lambda>\mu}\varepsilon_{n+1}=y_\lambda\rho_{\lambda>\mu} (y_\mu\otimes 1) \varepsilon_{n+1}=y_\lambda\rho_{\lambda>\mu} \varepsilon_{n+1}(y_\mu\otimes 1)$. It is also an intertwiner from $\lambda$ to $\mu \otimes 1$.
The intertwiner space in the Hecke algebra $H_{\bullet}$ is one dimensional, so $y_\lambda\varepsilon_{n+1}$ is a multiple of $y_\lambda$. The coefficient was known as follows.
\begin{proposition}[\cite{Bla00}, Prop. 1.11]\label{taub}
For $|\mu|=n$, $n\geq0$, $\lambda>\mu$,
$$\rho_{\lambda>\mu}\varepsilon_{n+1}=b_{\lambda-\mu}\rho_{\lambda>\mu},$$
where $b_{\lambda-\mu}=q^{2\text{cn}(\lambda-\mu)}$, and $\text{cn}(\lambda-\mu)=j-i$ is the content of the cell $\lambda-\mu$, which is in the $i$-$th$ row and $j$-$th$ column of $\lambda$.
\end{proposition}

\begin{definition}
Let us define the $q$-Murphy operator $\tau_{n}$, $n\geq1$, for $\mathscr{C}_{\bullet}$ to be the $n$-box
$$\tau_{n}=\grc{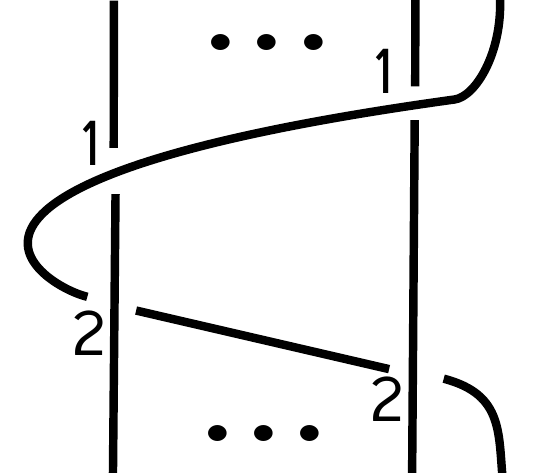}.$$
\end{definition}
One can rewrite the $q$-Murphy operator $\tau_{n}$ in terms of one half-braiding $\gra{b1}$ by Proposition \ref{braidfourier}.

Similar to $\varepsilon_{n}$, the $q$-Murphy operator $\tau_{n}$ acts diagonally on partial matrix units of $\mathscr{C}_{\bullet}$ as follows. (See Theorem \ref{P=YL} for the construction of matrix units.)

\begin{proposition}\label{taubd}
For $|\mu|=n$, $n\geq0$, we have
\begin{align}
\tilde{y}_{\lambda} \rho_{\lambda>\mu} (\tilde{y}_\mu\otimes 1)\tau_{n+1}&=b_{\lambda-\mu}\tilde{y}_{\lambda} \rho_{\lambda>\mu} (\tilde{y}_\mu\otimes 1), &&\text{ for } \lambda>\mu;
\label{eqb}\\
(\tilde{y}_{\lambda}\otimes \cup) (\rho_{\lambda<\mu} \otimes 1) (\tilde{y}_\mu \otimes 1)\tau_{n+1}&=-b_{\mu-\lambda}(\tilde{y}_{\lambda}\otimes \cup) (\rho_{\lambda<\mu} \otimes 1) (\tilde{y}_\mu \otimes 1), &&\text{ for } \lambda<\mu.
\label{eqd}\end{align}
\end{proposition}

\begin{proof}
See Appendix \ref{Appendix:taubd}.
\end{proof}

Let $\Phi_{n+1}: \mathscr{C}_{n+1}\rightarrow \mathscr{C}_{n}$ be the trace preserving conditional expectation, i.e. adding a cap on the right of an $(n+1)$-box.
Then $\Phi_{n+1}(\tau_{n+1}^i)=Z_{n+1}^{(i)}$ defines a central element $Z_{n+1}^{(i)}$ in $\mathscr{C}_{n}$.
We consider the formal power series in $u^{-1}$,
$$Z_{n+1}(u)=\sum_{i\geq0}Z_{n+1}^{(i)}u^{-i}.$$
Then
\begin{align}
Z_{n+1}(u)=\Phi_{n+1}(\frac{u}{u-\tau_{n+1}}).\label{nonphi}
\end{align}

By Theorem \ref{P=YL}, each simple components of $s_n\mathscr{C}_{n}$ is indexed by a Young diagram $\mu$, $|\mu|=n$.
Moreover, $\tilde{y}_{\mu}$ is a minimal idempotent in this component.

\begin{notation}
The trace of the minimal idempotent $tr_n(\tilde{y}_{\mu})$ is denoted by $<\mu>$, i.e. the quantum dimension of the irreducible representation indexed by $\mu$.
\end{notation}

Since $Z_{n+1}^{(i)}$ is central in $\mathscr{C}_{n}$, it is a scalar multiplication on the simple component of $\mathscr{C}_{n}$. Let us define $Z(\mu,u)$ to be the formal power series in $u^{-1}$ by
$$Z_{n+1}(u)\tilde{y}_{\mu}=Z(\mu,u)\tilde{y}_{\mu}.$$

The relation between $Z_{n+1}$ and the trace formula is given by
\begin{lemma}\label{traceZ}
For $|\mu|=n$, $n\geq0$, $\lambda>\mu$,
$$\frac{<\lambda>}{<\mu>}=\text{res}_{u=b_{\lambda-\mu}}\frac{Z(\mu,u)}{u},$$
\end{lemma}

\begin{proof}
See Appendix \ref{Appendix:traceZ}.
\end{proof}

Let us compute $Z_{n}$ recursively.

\begin{lemma}\label{ZZ}
For $n\geq1$,
$$Z_{n+1}-\frac{\delta}{2}=(Z_{n}-\frac{\delta}{2})\frac{(u-\tau_{n})^2(u+q^{-2}\tau_{n})(u+q^2\tau_{n})}{(u+\tau_{n})^2(u-q^{-2}\tau_{n})(u-q^2\tau_{n})}.$$
\end{lemma}

\begin{proof}
See Appendix \ref{Appendix:ZZ}.
\end{proof}

\begin{notation}
For a Young diagram $\mu$, let us define
\begin{align*}
\mu_+&=\{\lambda-\mu~|~\lambda>\mu\};\\
\mu_-&=\{\mu-\lambda~|~\lambda<\mu\}.
\end{align*}
\end{notation}

\begin{lemma}\label{Z=}
For a Young diagram $\mu$, $|\mu|=n$, $n\geq0$,
$$Z(\mu,u)-\frac{\delta}{2}=\frac{\delta}{2}\prod_{c\in\mu_+}\frac{u+b_c}{u-b_c}\prod_{c\in\mu_-}\frac{u-b_c}{u+b_c}.$$
\end{lemma}

\begin{proof}
Note that
$$\displaystyle Z(\emptyset,u)=\sum_{i\geq0} \delta u^{-i}=\frac{\delta u}{u-1},$$
so
$$\displaystyle Z(\emptyset,u)-\frac{\delta}{2}=\frac{\delta}{2}\frac{u+1}{u-1}.$$
The statement is true for $n=0$.

For $|\mu|=n$, $n\geq1$ and $\nu<\mu$, take
$W=\tilde{y}_{\mu}\rho_{\mu>\nu} (\tilde{y}_{\nu} \otimes 1)$.
Then by the definitions of $Z_n$ and $Z_n(\cdot,u)$ and Proposition \ref{taubd}, we have
$$WZ_n=Z_n(\nu,u)W, \quad WZ_{n+1}=Z(\mu,u)W, \quad W\tau_n=b_{\mu-\nu}W.$$
By Lemma \ref{ZZ}, we obtain the recursive formula
\begin{align}
Z_{\mu,u}-\frac{\delta}{2}=(Z_{\nu,u}-\frac{\delta}{2})\frac{(u-b_{\mu-\nu})^2(u+q^{-2}b_{\mu-\nu})(u+q^2b_{\mu-\nu})}{(u+b_{\mu-\nu})^2(u-q^{-2}b_{\mu-\nu})(u-q^2b_{\mu-\nu})}.
\end{align}
Therefore
$$Z(\mu,u)-\frac{\delta}{2}=\frac{\delta}{2}\prod_{c\in\mu_+}\frac{u+b_c}{u-b_c}\prod_{c\in\mu_-}\frac{u-b_c}{u+b_c}.$$
\end{proof}

\begin{theorem}[Trace formula. This is Theorem \ref{main theorem trace formula}]\label{trace formula}
The quantum dimension of an irreducible representation indexed by a Young diagram $\lambda$ is given by
$$<\lambda>=\prod_{c\in\lambda} \frac{i(q^{h(c)}+q^{-h(c)})}{q^{h(c)}-q^{-h(c)}},$$
where $h(c)$ is the hook length of a cell $c$ in $\lambda$.
\end{theorem}

\begin{remark}
When $q=e^{i\theta}$, we have $\delta=\cot(\theta)$ and
$$<\lambda>=\prod_{c\in\lambda} \cot(h(c)\theta).$$
\end{remark}

\begin{proof}
For $|\mu|=n$, $n\geq0$, $\lambda>\mu$,
by Lemmas \ref{traceZ}, \ref{Z=} and Proposition $\ref{taubd}$, we have
\begin{align}
\frac{<\lambda>}{<\mu>}=\delta \prod_{c\in\mu_+,c\neq \lambda-\mu} \frac{b_{\lambda-\mu}+b_c}{b_{\lambda-\mu}-b_c} \prod_{c\in\mu_-} \frac{b_{\lambda-\mu}-b_c}{b_{\lambda-\mu}+b_c}. \label{trace1}
\end{align}

$$\grf{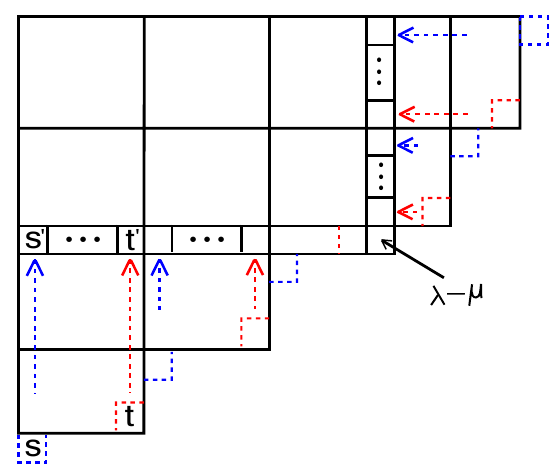}$$

Without loss of generality, let $\lambda$ be the above Young diagram. The cell $\lambda-\mu$ is marked in the diagram.
Let $C$ be the set of cells in $\mu$ located in the same row or column as $\lambda-\mu$.
The cells in $\mu_+$ except $\lambda-\mu$ are marked by dotted boxes outside $\mu$, and $s$ is the leftmost one. The cells in $\mu_-$ are marked by dotted boxes inside $\mu$, and $t$ is the left most one.
The cells in $C$ located in the same column as $s$ and $t$ are denoted by $s'$ and $t'$ respectively. Then
\begin{align*}
\frac{b_{\lambda-\mu}+b_s}{b_{\lambda-\mu}-b_s}&=\frac{q^{h(s')}+q^{-h(s')}}{q^{h(s')}-q^{-h(s')}};\\
\frac{b_{\lambda-\mu}-b_t}{b_{\lambda-\mu}+b_t}&=\frac{q^{h(t')-1}-q^{-(h(t')-1)}}{q^{h(t')-1}+q^{-(h(t')-1)}}.
\end{align*}
So
\begin{align*}
\frac{b_{\lambda-\mu}+b_s}{b_{\lambda-\mu}-b_s}\frac{b_{\lambda-\mu}-b_t}{b_{\lambda-\mu}+b_t}=\prod_{k=h(s')}^{h(t')} \frac{i(q^{k}+q^{-k})}{q^{k}-q^{-k}} \times \left(\frac{i(q^{k-1}+q^{-(k-1)})}{q^{k-1}-q^{-(k-1)}}\right)^{-1}
\end{align*}
Therefore the recursive formula (\ref{trace1}) can be written as
\begin{align*}
\frac{<\lambda>}{<\mu>}=\delta \prod_{c\in C} \frac{i(q^{h(c)}+q^{-h(c)})}{q^{h(c)}-q^{-h(c)}} \times \left(\frac{i(q^{h(c)-1}+q^{-(h(c)-1)})}{q^{h(c)-1}-q^{-(h(c)-1)}}\right)^{-1}.
\end{align*}
Note that $<\emptyset>=1$, $\displaystyle \delta=\frac{i(q+q^{-1})}{q-q^{-1}}$ and $h(\lambda-\mu)=1$, so
\begin{align*}
<\lambda>=\prod_{c\in\lambda} \frac{i(q^{h(c)}+q^{-h(c)})}{q^{h(c)}-q^{-h(c)}}.
\end{align*}
\end{proof}

\section{Positivity}\label{subsection positivity}
We have constructed the matrix units and computed the trace formula of $\mathscr{C}_{\bullet}$ over the field $\mathbb{C}(q)$. In this section, we consider $q$ as a scalar and $\mathscr{C}_{\bullet}$ as a planar algebra over $\mathbb{C}$. We are going to find out all values of $q$, such that (a proper quotient of) $\mathscr{C}_{\bullet}$ is a subfactor planar algebra. We need to be careful while using Wenzl's formula (\ref{wenzlformula}) over the field $\mathbb{C}$, as the formula is only defined for an idempotent with a non-zero trace. When $q$ is not a root of unity, by Theorem \ref{trace formula}, $<\lambda>$ is non-zero for any Young diagram $\lambda$. Therefore we have the following:
\begin{proposition}
When $q$ is not a root of unity, we have $\mathscr{C}_{\bullet}\cong YL_{\bullet}$ as a filtered algebra over the field $\mathbb{C}$.
\end{proposition}

\begin{proof}
Follows from Theorem \ref{P=YL}, \ref{trace formula}.
\end{proof}

When $q$ is a root of unity, $\mathscr{C}_{\bullet}$ is no longer semisimple. We need to consider $(\mathscr{C}/\text{Ker})_{\bullet}$, where $\text{Ker}$ is the kernel of the partition function of $\mathscr{C}_{\bullet}$.
To show $(\mathscr{C}/\text{Ker})_{\bullet}$ is a subfactor planar algebra, we need an involution $*$ which reflects planar tangles vertically and a positive definite Markov trace with respect to the involution. In this case, each $(\mathscr{C}/\text{Ker})_m$ is a $C^*$-algebra.

Recall that $\mathscr{C}_{\bullet}$ is generated by $R$. (See definition \ref{Def:Centralizer algebra}.)
\begin{lemma}
If $(\mathscr{C}/\text{Ker})_{\bullet}$ is a subfactor planar algebra, then $\displaystyle q=e^{\frac{i\pi}{2N+2}}$, for some $N\in \mathbb{N}^+$; and $R=R^*$ for the generator $R$.
\end{lemma}
\begin{proof}
Recall that $R^2=id-e$, so $R^*=R$.

To obtain a subfactor planar algebra, $\delta$ has to be a positive number. Recall that $\displaystyle q=\frac{i+\delta }{\sqrt{1+\delta^2}}$. So $q=e^{i\theta}$, for some $0<\theta<\frac{\pi}{2}$.
When $\displaystyle \frac{\pi}{2N+2}<\theta<\frac{\pi}{2N}$, $N\geq1$, the minimal idempotents $\tilde{y}_{[i]}$, $1\leq i \leq N$, can be constructed inductively as in Theorem \ref{P=YL}, where $[i]$ is the Young diagram with 1 row and $N$ columns. However, by Theorem \ref{trace formula}, $<[N]>=\cot (N\theta)<0$. So the trace is not positive semi-definite, and we will not obtain a subfactor planar algebra. Therefore  $\displaystyle q=e^{\frac{i\pi}{2N+2}}$, for some $N\in \mathbb{N}^+$.
\end{proof}

When $\displaystyle q=e^{\frac{i\pi}{2N+2}}$, $N\in \mathbb{N}^+$, let us define an involution $*$ as an anti-linear map on the universal planar algebra $\mathscr{C}_{\bullet}'$ generated by $R$, in which $*$ fixes $R$ and reflects planar tangles vertically. Note that the action of $*$ on the relations of $R$ is stable, so $*$ is well-defined on $\mathscr{C}_{\bullet}$ and $(\mathscr{C}/\text{Ker})_{\bullet}$. Therefore $(\mathscr{C}/\text{Ker})_{\bullet}$ becomes a planar *-algebra.

Let us determine the semisimple quotient $(\mathscr{C}/\text{Ker})_{\bullet}$. Since we do not have a presumed candidate for the semisimple quotient, we need to construct the matrix units of $(\mathscr{C}/\text{Ker})_{\bullet}$ in $\mathscr{C}_{\bullet}$. Since $\mathscr{C}_{\bullet}$ is no longer semisimple, we need to show that the required matrix units are still well-defined while passing from the field $\mathbb{C}(q)$ to $\mathbb{C}$.
Unlike the semisimple case, the basic construction and Wenzl's formula do not always work and the complement of the support of the basic construction ideal $s_m$ is not defined.
We give an alternating Wenzl's formula for $(\mathscr{C}/\text{Ker})_{\bullet}$ and an alternating definition of $s_m$ in $\mathscr{C}_{\bullet}$ to construct the matrix units for $(\mathscr{C}/\text{Ker})_{\bullet}$.
We also construct $\text{Ker}$. Then we can determine the semisimple quotient $(\mathscr{C}/\text{Ker})_{\bullet}$.

Recall that $\tilde{y}_\lambda$ is defined as $s_{|\lambda|}y_\lambda$ over $\mathbb{C}(q)$. If $\tilde{y}_\lambda$ is well-defined over $\mathbb{C}$, then we have the trace formula \ref{trace formula},
$$tr(y_\lambda)=\prod_{c\in\lambda} \cot(h(c)\theta).$$

The $(1,1)$ cell of a Young diagram $\lambda$, denoted by $c_\lambda$, achieves the maximal hook length $h(c_{\lambda})$. Thus
$$\left\{\begin{aligned}
tr(y_\lambda)&>0, \text{~when~} h(c_\lambda)\leq N;\\
tr(y_\lambda)&=0, \text{~when~} h(c_\lambda)= N+1.
\end{aligned}
\right.$$

\begin{notation}[Truncated Weyl Chambers]
Take
\begin{align*}
Y(N)&=\{\lambda~|~h(c_\lambda)\leq N\}; \\
B(N)&=\{\kappa ~|~ \kappa>\lambda, ~\lambda\in Y(N), ~\kappa \notin Y(N) \}.
\end{align*}
Let us define $YL(N)$ to be the sub lattice of Young's lattice $YL$ consisting of $Y(N)$, and $YL(N)_{\bullet}$ to be the string algebra of $YL(N)$ starting from $\emptyset$.
\end{notation}
(Notation \ref{Notation:Young diagrams} for $\kappa>\lambda$: The Young diagram $\kappa$ is obtained by adding one cell to $\lambda$.)

For example, $YL(4)$ is given in Figure \ref{Figure:Pinciplegraph2}.
\begin{figure}[h]
  $$\grc{principalgraph4young}.$$
  \caption{The sub lattice $YL(4)$ of Young's lattice.}\label{Figure:Pinciplegraph2}
\end{figure}

Let $H_{\bullet}$ be the Hecke algebra generated by $\gra{b1}$ over $\mathbb{C}$. By the construction in Section \ref{hecke}, for any $\mu,\lambda\in Y(N)\cup B(N)$, such that $\mu<\lambda$, we have the well-defined idempotents $y_\mu$, $y_\lambda$, and morphisms $\rho_{\mu<\lambda}$ from $y_\mu\otimes 1$ to $y_\lambda$; $\rho_{\lambda>\mu}$ from $y_{\lambda}$ to $y_\mu\otimes 1$. Moreover $y_\mu^*=y_\mu$, $y_\lambda^*=y_\lambda$ and $\rho_{\mu<\lambda}^*=\rho_{\lambda>\mu}$. Then we have the branching formula for $y_{\mu}$, $\mu\in Y(N)$,
$$y_{\mu}\otimes 1=\sum_{\lambda>\mu} \rho_{\mu<\lambda}\rho_{\lambda>\mu}.$$

Now let us construct $\tilde{y}_\lambda$ (over $\mathbb{C}$) inductively, for $\lambda\in Y(N)\cup B(N)$.
Here we cannot use the former construction of $\tilde{y}_\lambda$, which was done over $\mathbb{C}(q)$, since $s_m$ is not defined over $\mathbb{C}$ yet.

Set up $\tilde{y}_\emptyset=\emptyset$.
Suppose $\mu\in Y(N)$ and $\tilde{y}_{\lambda}$ is constructed. For $\kappa\in Y(N)\cup B(N)$, $\kappa>\mu$, let us define $\tilde{y}_\kappa$ as
\begin{align*}
\tilde{y}_\kappa&=\rho_{\kappa>\mu} \left(
\tilde{y}_\mu \otimes 1 -\sum_{\lambda<\mu} \frac{<\lambda>}{<\mu>} (\tilde{y}_\mu\otimes 1)(\rho'_{\mu>\lambda} \otimes 1) (\tilde{y}_{\lambda}\otimes \cap) (\tilde{y}_{\lambda}\otimes \cup) (\rho'_{\lambda<\mu} \otimes 1) (\tilde{y}_\mu \otimes 1) \right) \rho_{\mu<\kappa}.
\end{align*}
Recall that $\rho$ and $\rho'$ are normalization of $\dot{\rho}$ over $\mathbb{C}(q)$ and $\mathbb{C}$ respectively. So
\begin{align*}
\tilde{y}_\kappa&=\rho_{\kappa>\mu} \left(
\tilde{y}_\mu \otimes 1 -\sum_{\lambda<\mu} \frac{<\lambda>}{<\mu>} (\tilde{y}_\mu\otimes 1)(\rho_{\mu>\lambda} \otimes 1) (\tilde{y}_{\lambda}\otimes \cap) (\tilde{y}_{\lambda}\otimes \cup) (\rho_{\lambda<\mu} \otimes 1) (\tilde{y}_\mu \otimes 1) \right) \rho_{\mu<\kappa},
\end{align*}
which is also defined over $\mathbb{C}(q)$. By Wenzl's formula \ref{wenzlformula}, we have $\tilde{y}_\kappa=s_my_\kappa$ over $\mathbb{C}(q)$. Therefore the definition of $\tilde{y}_{\kappa}$ over $\mathbb{C}$ is independent of the choice of $\mu$.

We have constructed $\tilde{y}_\lambda$, for $\lambda\in Y(N)\cup B(N)$. Thus Wenzl's formula \ref{wenzlformula} holds for $\tilde{y}_{\mu}$, $\mu\in Y(N)$, over $\mathbb{C}$ as follows
\begin{align*}
\tilde{y}_\mu \otimes 1 &=\sum_{\lambda<\mu} \frac{<\lambda>}{<\mu>} (\tilde{y}_\mu\otimes 1)(\rho'_{\mu>\lambda} \otimes 1) (\tilde{y}_{\lambda}\otimes \cap) (\tilde{y}_{\lambda}\otimes \cup) (\rho'_{\lambda<\mu} \otimes 1) (\tilde{y}_\mu \otimes 1) & \nonumber\\
                        &+\sum_{\lambda>\mu} (\tilde{y}_\mu\otimes 1) \rho'_{\mu<\lambda} \tilde{y}_{\lambda} \rho'_{\lambda>\mu} (\tilde{y}_\mu\otimes 1).
\end{align*}

\begin{lemma}\label{trace0}
For a spherical planar algebra $\mathscr{C}_{\bullet}$, if $p$ is a trace zero minimal idempotent in $\mathscr{C}_m$, then $p$ is in the kernel of the partition function of $\mathscr{C}_{\bullet}$.
\end{lemma}

\begin{proof}
By spherical isotopy, any closed diagram containing $p$ is of the form $tr(px)$ for some $x$ in $\mathscr{C}_m$.
By assumption $p$ is a trace zero minimal idempotent, so $tr(px)=0$. Therefore $p$ is in the kernel of the partition function of $\mathscr{C}_{\bullet}$.
\end{proof}

Note that $h(c_\kappa)= N+1$, for any $\kappa\in B(N)$. So $tr(y_{\kappa})=0$.
By Lemma \ref{trace0}, we have $y_{\kappa}\in\text{Ker}$.
Therefore in $(\mathscr{C}/\text{Ker})_\bullet$, Wenzl's formula for $\tilde{y}_{\mu}$, $\mu\in Y(N)$,  is given by
\begin{align}
\tilde{y}_\mu \otimes 1 &=\sum_{\lambda<\mu} \frac{<\lambda>}{<\mu>} (\tilde{y}_\mu\otimes 1)(\rho'_{\mu>\lambda} \otimes 1) (\tilde{y}_{\lambda}\otimes \cap) (\tilde{y}_{\lambda}\otimes \cup) (\rho'_{\lambda<\mu} \otimes 1) (\tilde{y}_\mu \otimes 1) & \nonumber\\
                        &+\sum_{\lambda>\mu,\lambda\in Y(N)} (\tilde{y}_\mu\otimes 1) \rho'_{\mu<\lambda} \tilde{y}_{\lambda} \rho'_{\lambda>\mu} (\tilde{y}_\mu\otimes 1). \label{positive wenzl formula}
\end{align}

Now let us construct the matrix units of $(\mathscr{C}/\text{Ker})_{\bullet}$ and show that it is a subfactor planar algebra.

\begin{theorem}[Positivity]\label{Positivity}
When $\displaystyle q=e^{\frac{i\pi}{2N+2}}$, $N\geq1$, $(\mathscr{C}/\text{Ker})_{\bullet}$ is a subfactor planar algebra, denoted by $\mathscr{C}^{N}_{\bullet}$.
Its principal graph is $YL(N)$. See Figure \ref{Figure:principal graphs}.
\end{theorem}

\begin{figure}[h]
$$\grc{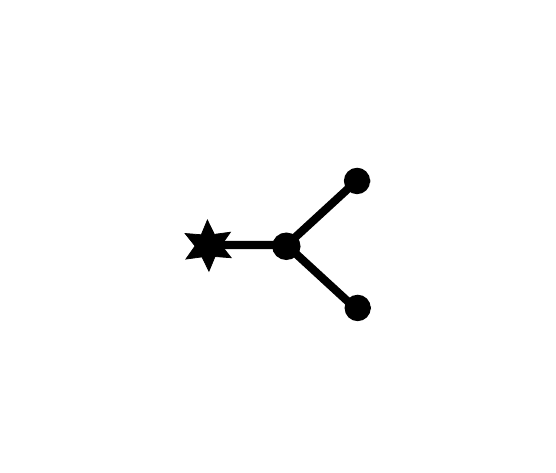}\quad \grc{principalgraph4}\quad \grc{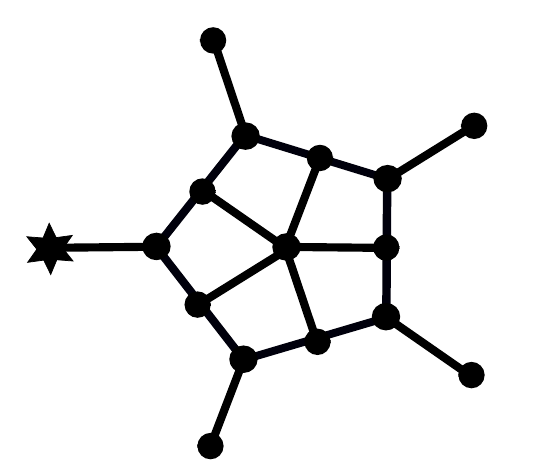} \cdots.$$
\caption{Principal graphs $YL(N)$, for $N=2,3,4,\cdots$}\label{Figure:principal graphs}
\end{figure}

\begin{remark}
When $N=1$, the planar algebra has index 1.
When $N=2$, the planar algebra is the group subfactor planar algebra $\mathbb{Z}_3$.
It is exactly the one in the classification result in \cite{BJL}, which is not in the two families, Bisch-Jones planar algebras and BMW planar algebras. When $N=3$, the subfactor planar algebra was constructed in \cite{LMP13}, i.e., the example (4) in Section \ref{subsection example}.
\end{remark}

\begin{remark}
For each $q=e^{\frac{i\pi}{2N+2}}$, we obtained a pair of subfactor planar algebras, due to the complex conjugate choice for the generator and relations in Definition \ref{Def:Centralizer algebra}.
\end{remark}

\begin{remark}
There are two different ways to identify the group subfactor planar algebra $\mathbb{Z}_3$ as an unshaded planar algebra. The two unshaded ones are complex conjugates of each other.
\end{remark}

\begin{proof}
Let $\text{Path}(m)$ be the set of all length $m$ paths $t$ in $YL(N)$ starting from $\emptyset$.
For $t\in \text{Path}(m)$ from $\emptyset$ to $\lambda$,
take $t'$ to be the first length $(m-1)$ sub path of $t$ from $\emptyset$ to $\mu$.
Let us define $\tilde{P}^{\pm}_t$ inductively as follows,
\begin{align*}
P^{\pm}_\emptyset&=\emptyset; &\\
\tilde{P}^+_t&=(\tilde{P}^+_{t'}\otimes 1)\rho'_{\mu < \lambda} \tilde{y}_{\lambda}, &\text{ when } \mu<\lambda;\\
\tilde{P}^+_t&=\sqrt{\frac{<\lambda>}{<\mu>}}(\tilde{P}^+_{t'}\otimes 1)(\rho'_{\mu>\lambda}\otimes 1) (\tilde{y}_{\lambda}\otimes \cap), &\text{ when } \mu>\lambda;\\
\tilde{P}^-_t&=\sqrt{\frac{<\lambda>}{<\mu>}}(\tilde{P}^-_{t'}\otimes \cup)(\rho'_{\mu < \lambda}\otimes 1) (\tilde{y}_{\lambda}\otimes 1), &\text{ when } \mu<\lambda;\\
\tilde{P}^-_t&=\tilde{P}^+_{t'}\rho'_{\mu>\lambda} (\tilde{y}_{\lambda}\otimes 1), &\text{ when } \mu>\lambda.
\end{align*}

By definitions, we have $y_{\lambda}^*=y_{\lambda}$ and $(\tilde{P}^+_t)^*=\tilde{P}^-_t$. By Theorem \ref{P=YL}, the map $\omega_m: YL(N)_m \to \mathscr{C}_m$ as a linear extension of
\begin{align*}
\omega_m(t\tau^{-1})=\tilde{P}^+_t \tilde{P}^{-}_\tau,
\end{align*}
is an injective *-homomorphism.
Recall that $tr(y_{\lambda})>0$, for any $\lambda\in Y(N)$, so $\omega_m$ is still injective passing to quotient $(\mathscr{C}/\text{Ker})_m$.

Applying Wenzl's formula (\ref{positive wenzl formula}) to the identity $1_m$ of $(\mathscr{C}/\text{Ker})_m$, we have
$$1_m=\sum_{t\in \text{Path}(m)}\tilde{P}^+_t\tilde{P}^-_t.$$
For an $m$-box $x$, if $t,\tau \in \text{Path}(m)$ are paths from $\emptyset$ to different vertices, then $\tilde{P}^-_t x \tilde{P}^+_\tau=0$ by Theorem \ref{P=YL}.
If $t,\tau \in \text{Path}(m)$ are paths from $\emptyset$ to $\mu$, then $tr(\tilde{P}^+_t\tilde{P}^-_\tau \tilde{P}^+_\tau \tilde{P}^-_t)=<\mu>\neq 0$. Take
$$x_{t,\tau}=\frac{tr(\tilde{P}^+_t\tilde{P}^-_t x \tilde{P}^+_\tau\tilde{P}^-_\tau \tilde{P}^+_\tau \tilde{P}^-_t )}{tr(\tilde{P}^+_t\tilde{P}^-_\tau \tilde{P}^+_\tau \tilde{P}^-_t)}.$$
By Theorem \ref{P=YL}, we have
$$\tilde{P}^+_t\tilde{P}^-_t x \tilde{P}^+_\tau\tilde{P}^-_\tau=x_{t,\tau}\tilde{P}^+_t\tilde{P}^-_\tau,$$
Let $\text{Pair}(m)$ be the set of all pairs of paths $(t,\tau)$ in $\text{Path}(m)$ from $\emptyset$ to the same vertex. Then
$$x=\sum_{(t,\tau) \in \text{Pair}(m)}x_{t,\tau} \tilde{P}^+_t\tilde{P}^-_\tau.$$
Therefore $\omega_m$ is onto $(\mathscr{C}/\text{Ker})_m$.

Since $\omega_m: YL(N)_m \to (\mathscr{C}/\text{Ker})_m$ is a *-isomorphism and the trace is positive definite, we have that $(\mathscr{C}/\text{Ker})_{\bullet}$ is a subfactor planar algebra. Moreover, its principal graph is $YL(N)$.
\end{proof}

\begin{corollary}
For each $m$, we have $\mathscr{C}_m=YL(N)_m\oplus \text{Ker}_m$, where $\text{Ker}_m$ is the two sided ideal of the algebra $\mathscr{C}_m$ generated by the trace zero minimal idempotents $\{\tilde{y}_{\lambda}\}_{\lambda\in B(N), |\lambda|\leq m}$.
\end{corollary}

\begin{proof}
Note that $\text{Ker}_m \subset \text{Ker}$ and the decomposition $$1_m=\sum_{t\in \text{Path}(m)}\tilde{P}^+_t\tilde{P}^-_t$$ also holds in $\mathscr{C}_m/\text{Ker}_m$, so $\mathscr{C}_m=YL(N)_m\oplus \text{Ker}_m$.
\end{proof}

\begin{remark}
Our strategy of decomposing the non-semisimple algebra $\mathscr{C}_m$ into a direct sum of a semisimple algebra $(\mathscr{C}/\text{Ker})_m$ and an ideal $\text{Ker}_m$ also works for other cases, such as Temperley-Lieb-Jones planar algebras, BMW planar algebras, Bisch-Jones planar algebras etc.

Usually the (planar) algebra $\mathscr{C}_{\bullet}$ given by generators and relations is semisimple over the field of rational functions in some parameters. However, it may not be semisimple over $\mathbb{C}$ when the parameters are scalars, in particular roots of unity.
First we construct the matrix units for the algebra over rational functions and identify them as loops of a (directed) graph $\Gamma$ starting from a distinguished vertex $\emptyset$.
Then we find out the subgraph $Y$ such that the statistical dimensions of vertices in $Y$ are non-zero and the statistical dimensions of vertices in the boundary $B$ of $Y$ are zero.
We need to check that \underline{the matrix units for the string algebra and the trace zero idempotents are well-defined} over the field $\mathbb{C}$.
Then we have the decomposition of $\mathscr{C}_{\bullet}$ over the field $\mathbb{C}$ as a direct sum of the string algebra of $Y$ and an ideal generated by trace zero idempotents corresponding to vertices in $B$.
\end{remark}

\begin{proposition}
When $\displaystyle q=e^{\frac{ik\pi}{2N+2}}$, $(k,2N+2)=1$, the quotient $(\mathscr{C}/\text{Ker})_{\bullet}$ is a spherical semisimple planar algebra, but not a subfactor planar algebra. The fusion category associated with the planar algebra is spherical, but not unitary. Moreover, the simple objects are given by $Y(N)$.
\end{proposition}
\begin{proof}
The argument is similar to the case for $q=e^{\frac{i\pi}{2N+2}}$.
\end{proof}

\section{Dihedral group symmetries}\label{section dihedral}
For $N\in\mathbb{N}^+$, $\theta=\frac{\pi}{2N+2}$, $q=e^{i\theta}$,
we have constructed the unshaded subfactor planar algebra $\mathscr{C}^{N}_{\bullet}=(\mathscr{C}/\text{Ker})_{\bullet}$. Its principal graph is $YL(N)$.
We are going to prove that the automorphism group of the graph $YL(N)$ is the dihedral group $D_{2(N+1)}$. From the $\mathbb{Z}_2$ symmetry, we construct another sequence of subfactor planar algebras. From the $\mathbb{Z}_{N+1}$ symmetry, we obtain at least one more subfactor for each odd ordered subgroup of $\mathbb{Z}_{N+1}$.

While considering $\mathscr{C}^{N}_{\bullet}$ as a unitary fusion category, its simple objects are given by $Y(N)$. The dimension of the object $\lambda\in Y(N)$ is given in Lemma \ref{trace formula}.
Let $G$ be the set of invertible objects, i.e. $G=\{\lambda\in Y(N) ~|~ <\lambda>=1\}$. Then $G$ forms a group under $\otimes$. Moreover, $G$ is a subgroup of the automorphism group $\text{Aut}(YL(N))$ of the graph $YL(N)$.

\begin{proposition}
Let $r_0=\emptyset$ and $r_k$, $1\leq k\leq N$, be the Young diagram with $k$ rows and each row has $N+1-k$ cells.  Then $G=\{r_k~|~0\leq k\leq N\}$.
\end{proposition}

\begin{proof}
Note that $\emptyset$ is in $G$ and it is a univalent vertex in $YL(N)$. So each vertex in $G$ is univalent in $YL(N)$. Then for any vertex $\lambda$ in $G$, $\lambda\neq \emptyset$, and any $\kappa>\lambda$, we have $\kappa\in B(N)$.
Thus the Young diagram $\lambda$ is a square with $k$ rows and $N+1-k$ columns, for some $1\leq k\leq N$, denoted by $r_k$.
Conversely applying the trace formula in Lemma \ref{trace formula}, we have that $<r_k>=1$ by the central symmetry of the Young diagram $r_k$ and the fact $\cot(n\theta)\cot((N+1-k)\theta)=1$.
\end{proof}

Since $\mathscr{C}^{N}_{\bullet}$ is a quotient of $\mathscr{C}_{\bullet}$, we keep the notations $\alpha=\gra{b+}$, $\alpha_i$, $H_{\bullet}$, $y_{\lambda}$ and $\tilde{y}_{\lambda}$ for $\mathscr{C}^{N}_{\bullet}$.
Let $s_m$ be the complement of the support of the basic construction ideal of $\mathscr{E}_m$, $m\geq0$. Then $\overline{s_m}=s_m$ and $s_{|\lambda|}y_{\lambda}=\tilde{y}_{\lambda}$, for any $\lambda\in Y(N)$.

\begin{proposition}\label{Prop:rr}
In $(\mathscr{C}^N)_l$, we have $\overline{\tilde{y}_{[l]}}=\tilde{y}_{[1^l]}$, for $0\leq l\leq N$.
Thus $r_N\otimes r_1=r_0$ in $G$.
\end{proposition}

\begin{proof}
Recall that $\alpha=\gra{b+}$ is the generator of the Hecke algebra. We construct the symmetrizers and anti-symmetrizers $f_l$ and $g_l$ as in Equations (\ref{fl}), (\ref{gl}).
Since $\tilde{y}_{[l]}=s_mf_l$ and $\tilde{y}_{[1^l]}=s_mg_l$.
multiplying $s_l$ to Equations \ref{fl1}, $\ref{gl}$, we have that
\begin{align}
\tilde{y}_{[l]}&=1\otimes \tilde{y}_{[l-1]}-\frac{[l-1]}{[l]}(1\otimes \tilde{y}_{[l-1]})s_2(q-\alpha_1)(1\otimes \tilde{y}_{[l-1]}); \label{equbraidf3}\\
\tilde{y}_{[1^{l}]}&=\tilde{y}_{[1^{l-1}]}-\frac{[l-1]}{[l]}(\tilde{y}_{[1^{l-1}]})s_2(q^{-1}+\alpha_{l-1})(\tilde{y}_{[1^{l-1}]}). \label{equbraidf4}
\end{align}

Note that $\overline{s_2(q-\alpha)}=s_2(q+\alpha)$ in $(\mathscr{C}^N)_2$.
Take the contragredient of Equation \ref{equbraidf3} in $(\mathscr{C}^N)_l$, we have
\begin{align}
\overline{\tilde{y}_{[l]}}&=\overline{\tilde{y}_{[l-1]}}-\frac{[l-1]}{[l]}(\overline{\tilde{y}_{[l-1]}})s_2(q^{-1}+\alpha_{l-1})(\overline{\tilde{y}_{[l-1]}})
\label{equbraidf5}.
\end{align}
Since $\overline{\tilde{y}_{[1]}}=1=\tilde{y}_{[1^{l}]}$, by recursive formulas \ref{equbraidf4}, \ref{equbraidf5}, we have
$$\overline{\tilde{y}_{[l]}}=\tilde{y}_{[1^l]}.$$
In particular, $\overline{\tilde{y}_{[N]}}=\tilde{y}_{[1^N]}$.
Thus $r_N\otimes r_1=r_0$ in $G$.
\end{proof}

\begin{proposition}\label{group fusion}
For $N\geq 1$, we have $G=\mathbb{Z}_{N+1}$ and $r_k \otimes r_1=r_{k+1}$, for $0\leq k\leq N$, where $r_{N+1}=r_{0}$.
\end{proposition}

\begin{proof}
Let $d(v,w)$ be the distance of vertices $v$ and $w$ in the graph $YL(N)$.
Then $r_k\otimes (\cdot)$ as an automorphism of $YL(N)$ preserves $d$, for $0\leq k\leq N$.

Recall that $r_0=\emptyset$, so $d(r_0,r_l)=|r_l|=(N+1-l)l$.
Then
\begin{align*}
d(r_0,r_l)
\left\{
\begin{aligned}
&=N &&\text{~for~} l=1, N;\\
&>N &&\text{~for~} 1< l< N.
\end{aligned}
\right.
\end{align*}
Therefore
\begin{align*}
d(r_k, r_k\otimes r_l)
\left\{
\begin{aligned}
&=N &&\text{~for~} l=1, N;\\
&>N &&\text{~for~} 1< l< N.
\end{aligned}
\right.
\end{align*}
There is a length $N$ path from $r_k$ to $r_{k+1}$ by removing the last column then adding one row.
So
$$d(r_k,r_{k+1})=N.$$
Therefore the automorphism $\cdot\otimes r_1$ maps a distance-$N$ pair $(r_k,r_{k+1})$ to another distance-$N$ pair $(r_{k'},r_{k'\pm1})$.
Note that $r_0\otimes r_1=r_1$ and $r_N\otimes r_1=r_0$ by Proposition \ref{Prop:rr}, so $r_k \otimes r_1=r_{k+1}$, for $0\leq k\leq N$.
\end{proof}

\begin{notation}
We define the $\mathbb{Z}_2$ automorphism $\Omega$ on the universal planar algebra $\mathscr{C}_{\bullet}'$ generated by $R$ by mapping the generator $R$ to $-R$. The action of $\Omega$ on the relations of $R$ is stable, so $\Omega$ is well-defined on the quotient $\mathscr{C}_{\bullet}$ and $\mathscr{C}^{N}$. Therefore $\Omega$ induces an $\mathbb{Z}_2$ automorphism on the principal graph $YL(N)$, still denoted by $\Omega$.
\end{notation}

\begin{proposition}
The induced $\mathbb{Z}_2$ automorphism $\Omega$ on Young diagrams is the reflection of Young diagrams by the diagonal.
\end{proposition}

\begin{proof}
If we switch the symmetrizers and antisymmetrizers in the construction of $y_{\lambda}$ in Formula \ref{equyoungidempotent}, then we obtain a minimal projection, which is equivalent to $y_{\Omega(\lambda)}$, where $\Omega(\lambda)$ is the reflection of the Young diagram $\lambda$ by the diagonal.

Note that $\Omega(s_m)=s_m$ and $\Omega(s_2(q-\sigma))=s_2(q^{-1}+\sigma)$.
By the recursive formulas (\ref{fl}) and (\ref{gl}), we have $\Omega(s_lf^{(l)})=\Omega(s_lg^{(l)})$.
Therefore $\Omega(\tilde{y}_{\lambda})$ is equivalent to $\tilde{y}_{\Omega(\lambda)}$.
\end{proof}

In particular, $\Omega(r_k)=r_{N+1-k}$.
Then $\Omega (r_k \otimes \Omega(\lambda))=r_{N+1-k} \otimes \lambda$.
So $G$ and $\{\Omega\}$ generates the Dihedral group $\text{D}_{2(N+1)}$ in $\text{Aut}(YL(N))$. The Dihedral Symmetries of $YL(N)$ was discovered by Suter in \cite{Sut02}. In our case, it is realized as the invertible objects and automorphisms of $\mathscr{C}^{N}$.

Furthermore, we have the following
\begin{proposition}\label{Prop:symmetry}
Suppose $\Gamma$ is a sub lattice of the Young lattice $TL$, such that for any $\lambda\in\Gamma$ and $\mu<\lambda$, we have $\mu\in\Gamma$. Then any automorphism of the graph $\Gamma$ fixing $\emptyset$ is either the identity or the reflection by the diagonal. Consequently
$$\text{Aut}(YL(N))=\text{D}_{2(N+1)}.$$
\end{proposition}

\begin{proof}
For a vertex $\lambda\in \Gamma$, we define the set $\lambda_<:=\{\mu ~|~ \mu<\lambda\}$.
Suppose $\lambda\neq \kappa$ and $\lambda_<=\kappa_<$, then $|\lambda|=|\kappa|$. Take $c_1$ in $\lambda\setminus\kappa$ and $c_2$ in $\kappa\setminus \lambda$. Then the only element in $\lambda_<=\kappa_<$ is $\lambda\setminus{\{c_1\}}$, and $\lambda\setminus{\{c_1\}}=\kappa\setminus{\{c_2\}}$. So $\{\lambda,\kappa\}=\{[2],[1^2]\}$.
Therefore,
\begin{align}
|\lambda|\geq3 \quad \& \quad \lambda_<=\kappa_< &\Longrightarrow \lambda=\kappa. \label{equd1}
\end{align}

Note that the distance from $\emptyset$ to $\lambda$ is $|\lambda|$.
If an automorphism $\Delta$ of the graph $\Gamma$ fixes $\emptyset$, then $|\Delta(\lambda)|=|\lambda|$.
Thus $\Delta([1])=[1]$, and
\begin{align}
\Delta(\lambda)_<&=\Delta(\lambda_<). \label{equd2}
\end{align}

If $\Delta([2])=[2]$ and $\Delta([1^2])=[1^2]$, then $\Delta$ is the identity map on Young diagrams with at most two cells. By \ref{equd1} and \ref{equd2}, $\Delta$ is the identity map on $\Gamma$.

If $\Delta([2])=[1^2]$ and $\Delta([1^2])=[2]$, then $\Delta=\Omega$ on Young diagrams with at most two cells. By \ref{equd1} and \ref{equd2}, $\Delta=\Omega$ on $\Gamma$.

When $\Gamma=YL(N)$, the automorphism $\Delta$ fixes the set of univalent vertices $Y(N)$. Note that $G$ acts transitively on $Y(N)$, so $\text{Aut}(YL(N))=\text{D}_{2(N+1)}$.

\end{proof}

\begin{corollary}\label{fusion rule for 1^n}
In particular, we have the fusion rule for $\mu\otimes [1^N]$ due to the $\text{D}_{2(N+1)}$ automorphism of $YL(N)$ constructed in \cite{Sut02}.
More precisely, the Young diagram $\mu\otimes [1^N]$ is obtained from $\mu$ by removing the first row of $\mu$ and adding one column with $N-k$ cells on the left, where $k$ is the number of cells in the first row of $\mu$.
\end{corollary}

From the $Z_2$ automorphism $\Omega$ of $\mathscr{C}^{N}_{\bullet}$, we obtain another subfactor planar algebra $(\mathscr{C}^{N}_{\bullet})^{\Omega}$ as the fixed point algebra. This process is also known as an orbifold construction or equivariantization. The fusion rules of equivariantizations of fusion categories are given in \cite{BurNat13}.
Thus we can derive the principal graph $YL(N)^{\Omega}$ of $(\mathscr{C}^{N}_{\bullet})^{\Omega}$ from the principal graph $YL(N)$ of $\mathscr{C}^{N}_{\bullet}$ as follows.

For a vertex $\lambda\in YL(N)$,
\begin{itemize}
\item[(1)] if $\Omega(\lambda)=\lambda$, then it splits into two vertices $\lambda_0$ and $\lambda_1$ in $YL(N)^{\Omega}$.
\item[(2)] If $\Omega(\lambda)\neq\lambda$, then $\lambda$ and $\Omega(\lambda)$ combine as one vertex $(\lambda,\Omega(\lambda))$ in $YL(N)^{\Omega}$.
\end{itemize}

For an edge between $\mu$ and $\lambda$ in $YL(N)$,
\begin{itemize}
\item[(3)]  if $\Omega(\mu)=\mu$ and $\Omega(\lambda)=\lambda$, then there is an edge between $\mu_k$ and $\lambda_k$, for $k=0,1$.
\item[(4)] If $\Omega(\mu)\neq\mu$ and $\Omega(\lambda)=\lambda$, then there is an edge between $(\mu,\Omega(\mu))$ and $\lambda_k$, for $k=0,1$.
\item[(5)]  If $\Omega(\mu)\neq\mu$ and $\Omega(\lambda)\neq\lambda$, then there is an edge between $(\mu,\Omega(\mu))$ and $(\lambda,\Omega(\lambda))$.
\end{itemize}

The Young diagrams invariant under $\Omega$ are the ones in the middle of the graph $YL(N)$. So $TL(N)^\Omega$ is the bottom half of $YL(N)$ with one more copy of the vertices in the middle and adjacent edges. Therefore we have the following:

\begin{theorem}\label{Thm:subfactors1}
We obtain a sequence of subfactor planar algebras $(\mathscr{C}^{N}_{\bullet})^{\Omega}$ from the $\mathbb{Z}_2$ action.
The principal graphs $YL(N)^{\Omega}$, for $N=2,3,4, \cdots$, are given by
$$\grc{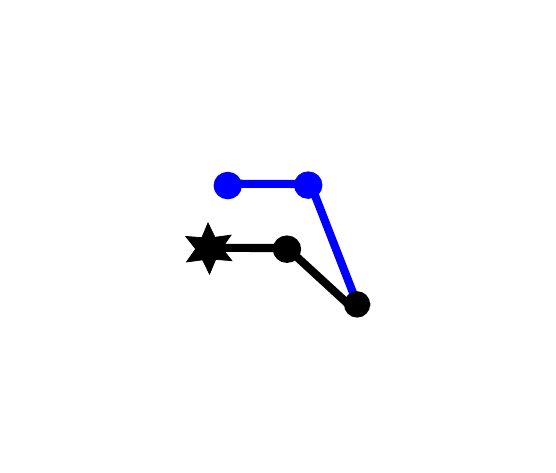}\quad \grc{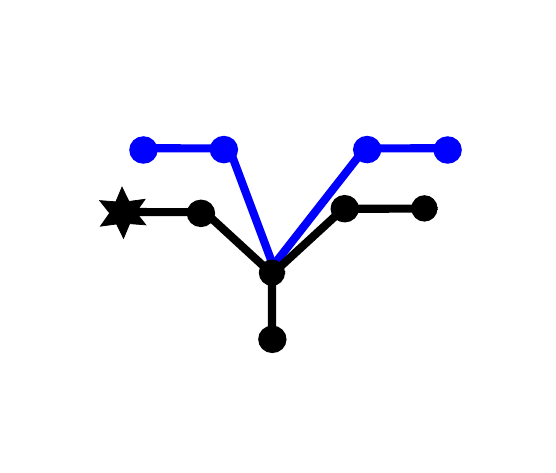}\quad \grc{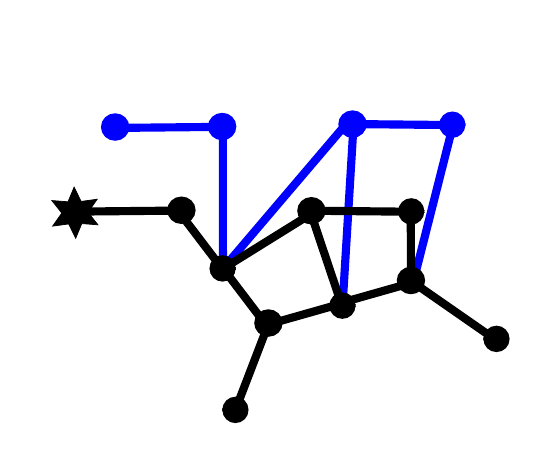} \cdots$$
\end{theorem}

When $N=3$, it is a near-group subfactor planar algebra. (Its even part is a near-group fusion category.) It is proved in \cite{LMP13} that its invertible objects forms the group $\mathbb{Z}_4$. This near-group subfactor planar algebra was first constructed by Izumi in \cite{Izu93}.
Therefore we obtain a sequence of (complex conjugate pairs of) subfactor planar algebras which is an extension of the near-group subfactor planar algebra for $\mathbb{Z}_4$.

We also obtain some subfactors from the $\mathbb{Z}_{N+1}$ symmetry. Take the stabilizer group of $\lambda$, $G_{\lambda}=\{g\in \mathbb{Z}_{N+1} ~|~ g\otimes \lambda=\lambda\}$. Then the irreducible summands of $\lambda\otimes \overline{\lambda}$ has exactly one $g$, for $g\in G_{\lambda}$. Let $\mathcal{N}\subset\mathcal{M}$ be the reduced subfactor of $\lambda$. Then it has an intermediate subfactor $\mathcal{P}$ and $\mathcal{N}\subset\mathcal{P}$ is the group subfactor $G_{\lambda}$. Therefore we obtain a subfactor $\mathcal{P}\subset\mathcal{M}$ with index $\displaystyle \frac{<\lambda>^2}{|G_{\lambda}|}.$

Let $\lambda_{N,m}$ be the following Young diagram,
$$\grd{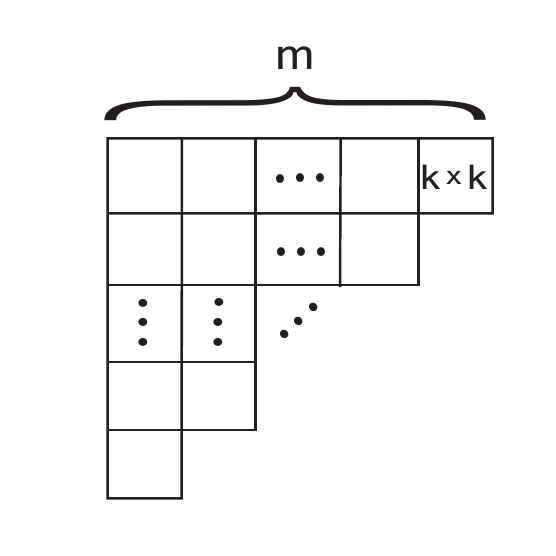},$$
where $(2m-1)k=N+1$.
This triangle has $\frac{m(m+1)}{2}$ blocks and each block is a square with $k\times k$ cells.
Note that $G_{\lambda}=\mathbb{Z}_{2m-1}$. Therefore
\begin{theorem}\label{Thm:subfactors2}
For each $N$ and each odd ordered subgroup $\mathbb{Z}_{2m-1}$ of $\mathbb{Z}_{N+1}$, we obtain a subfactor with index $\displaystyle \frac{<\lambda_{N,m}>^2}{2m-1}.$
\end{theorem}

\section{The centralizer algebra of Quantum subgroups}\label{section quantum subgroups}
Etingof, Nikshych, and Ostrik introduced fusion categories in \cite{ENO}.
In this section, we construct three-parameter families of unitary fusion categories $\mathscr{C}^{N,k,l}$. In particular, $\mathscr{C}^{N,1,0}$ is the bimodule category associated with the unshaded subfactor planar algebra $\mathscr{C}^{N}_{\bullet}$; $\mathscr{C}^{N,0,1}$ and $\mathscr{C}^{N,1,1}$ are representation categories \cite{Ost03} of (exceptional) subgroups of quantum $SU(N)_{N+2}$ and $SU(N+2)_{N}$ respectively. These representation categories are called modules or module categories of quantum subgroups by Onceanu \cite{Ocn00} and Ostrik \cite{Ost03}. From this construction, we see that $\mathscr{C}_{\bullet}$ is the centralizer algebra of each family of quantum subgroups. With a flavor reminiscent of Schur-Weyl duality, we study these module categories by the representations of $\mathscr{C}_{\bullet}$ in Sections \ref{subsection matrix units}, \ref{subsection positivity}. We also obtain a closed-form of the quantum dimensions of these representations.

When $\displaystyle q=e^{\frac{i\pi}{2N+2}}$, we have constructed the subfactor planar algebra $\mathscr{C}^{N}_{\bullet}=\mathscr{C}_{\bullet}/\text{Ker}$.
The subalgebra $H^N_{\bullet}$ of $\mathscr{C}^{N}_{\bullet}$ generated by shifts of $\gra{b1}$ is a Hecke algebra which is the centralizer algebra for quantum $SU(N)_{N+2}$.
Recall that the representation category of quantum $SU(N)_{N+2}$ is obtained from the Hecke algebra \emph{modulo} a $\mathbb{Z}_{N}$ periodicity given by the $N$th antisymmetrizer $g^{(N)}$.
We are going to show that the $\mathbb{Z}_{N}$ periodicity extends to $\mathscr{C}^{N}_{\bullet}$. Modulo the $\mathbb{Z}_{N}$ periodicity, we obtain the unitary fusion category $\mathscr{C}^{N,0,1}$  and it contains the representation category of quantum $SU(N)_{N+2}$ as a subcategory. Therefore it is the module category of an subgroup of quantum $SU(N)_{N+2}$.

The subalgebra of $\mathscr{C}^{N}_{\bullet}$ generated by shifts of $\gra{b2}$ is also a Hecke algebra, which is the centralizer algebra for quantum $SU(N+2)_{N}$. We are going to show that the corresponding $\mathbb{Z}_{N+2}$ periodicity also extends to $\mathscr{C}^{N}_{\bullet}$. Modulo the $\mathbb{Z}_{N+2}$ periodicity, we obtain $\mathscr{C}^{N,1,1}$ as the module category of an subgroup of quantum $SU(N+2)_{N}$.

We conjecture that the two families of unitary fusion categories $\mathscr{C}^{N,0,1}$ and $\mathscr{C}^{N,1,1}$ are isomorphic to the bimodule categories defined by Xu in \cite{Xu98} for conformal inclusions $\displaystyle SU(N)_{N+2}\subset SU(\frac{N(N+1)}{2})_{1}$ and $\displaystyle SU(N+2)_{N}\subset SU(\frac{(N+2)(N+1)}{2})_{1}$ respectively.

\begin{remark}
While checking Ocneanu's list in \cite{Ocn00} with Noah Snyder, we realized that that the zero-graded part of the subgroup $E_9$ of $SU(3)$ is a near-group category with simple objects $1, g, g^2, X$, such that
$X\otimes X=\oplus_{k=0}^2 g^k \oplus 6 X$. This example is particularly interesting, because 6 is a non-trivial multiple of the order of the group $\mathbb{Z}_3$.
\end{remark}

\begin{definition}
For an unshaded subfactor planar algebra $\mathscr{S}_{\bullet}$, we call a trace one projection $g$ in $\mathscr{S}_m$ a $\mathbb{Z}_m$ grading operator, if there is a partial isometry $u$ from $g\otimes 1$ onto $1\otimes g$, such that for any $x\in \mathscr{S}_k$
\begin{align}
\grc{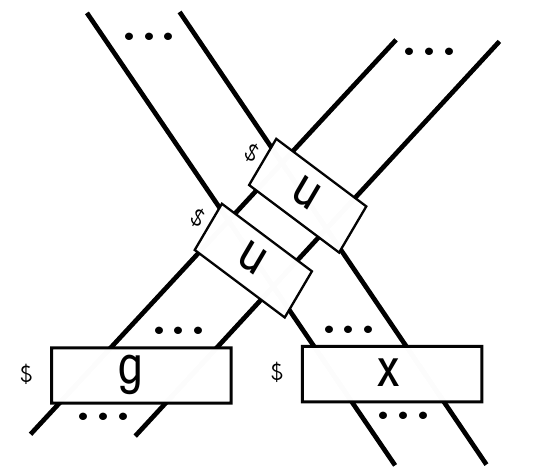}&=\grc{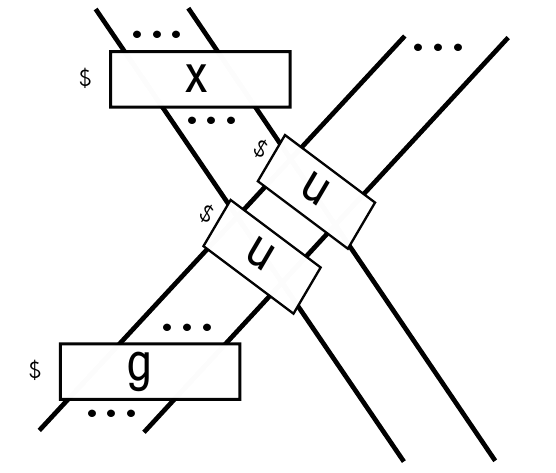}.  \label{braiding}
\end{align}
\end{definition}
The Jones projection $\displaystyle e=\frac{1}{\delta}\gra{22}$ is a $\mathbb{Z}_2$ grading operator.

\begin{proposition}\label{Prop:tensorgrading}
The tensor product of grading operators is a grading operator.
\end{proposition}

\begin{proof}
Suppose $g_i$ is a $\mathbb{Z}_{m_i}$ grading operator, and $u_i$ is the partial isometry from $g_i\otimes 1$ onto $1\otimes g_i$, for $i=1,2$, so that Equation \ref{braiding} holds.
Then $g_1\otimes g_2$ is a $\mathbb{Z}_{m_1+m_2}$ grading operator, and $(g_1 \otimes u_2)(u_1\otimes g_2)$ is the partial isometry from $g_1\otimes g_2\otimes 1$ onto $1\otimes g_1\otimes g_2$, so that Equation \ref{braiding} holds.
\end{proof}

\begin{definition}
For a grading operator $g$, we call the pair $(g,u)$ a commutative grading, if the following equation holds.
\begin{align}
\grc{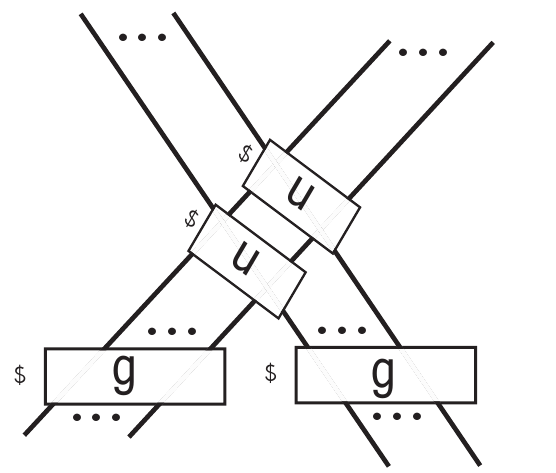}&=\grc{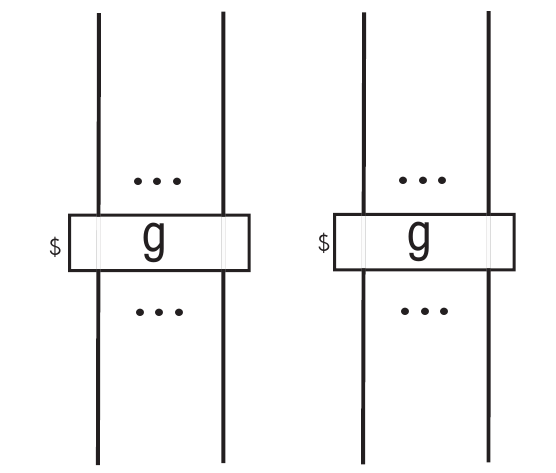}.  \label{commuting}
\end{align}
\end{definition}

If $h$ is a minimal projection equivalent to $g$ in $\mathscr{S}_{m}$, then $h$ is also a grading operator.
The above definitions only depends on the equivalence class of $g$.

\begin{proposition}
For any $\mathbb{Z}_m$ grading operator $g$ in $\mathscr{S}_{\bullet}$, there are $m$ commutative gradings.
\end{proposition}
\begin{proof}
The left side of Equation \ref{commuting} is a partial isometry from $g\otimes g$ onto $g\otimes g$.
Since $g\otimes g$ is a minimal projection in $\mathscr{S}_{2m}$, we can modify the isometry $u$ by a phase, such that Equation \ref{commuting} holds. There are $m$ choices of the phase corresponding to the $m$th roots of unity.
\end{proof}

\begin{definition}
A $\mathbb{Z}_m$ grading operator $g$ is said to have a periodicity $k$, if $k$ is the smallest positive integer, such that $g^{\otimes k} $ is equivalent to $ e^{\otimes \frac{mk}{2}}$ in $\mathscr{S}_{km}$.
\end{definition}

The Jones projection $e$ is a $\mathbb{Z}_2$ grading operator with periodicity 1.

A unshaded subfactor planar algebra $\mathscr{S}_{\bullet}$ is a $\mathbb{N}\cup\{0\}$ graded monoidal category with the usual tensor functor and the usual multiplication in Notation \ref{Notation:diagrams}.
If $g$ is a $\mathbb{Z}_m$ grading operator and $(g,u)$ is a commutative grading, then one can consider $A=\oplus_{k=0}^\infty g^{\otimes k}$ as a commutative algebra with a half-braiding defined in Equation \ref{braiding} by $(g,u)$. We obtain a $\mathbb{Z}_m$ graded monoidal category denoted by $\mathscr{S}/(g,u)$ which is the module category over $A$. If $(g,u')$ is another commutative grading, then $\mathscr{S}/(g,u')$ can be derived from $\mathscr{S}/(g,u)$ using the gauge transformation by the character of $\mathbb{Z}_m$. Therefore, we only need to consider one $(g,u)$, and simply denote $\mathscr{S}/(g,u)$ by $\mathscr{S}/g$.

From the pivotal, spherical and positive properties of a planar algebra, one can show that $\mathscr{S}/g$ is pivotal and spherical. Its simple objects have positive quantum dimensions. Furthermore, if $\mathscr{S}_{\bullet}$ has finite depth and $g$ has a finite periodicity, then $\mathscr{S}/g$ is a unitary fusion category. See Appendix \ref{Appendix:modulo grading operators} for the above construction of the unitary fusion category $\mathscr{S}/g$.

\begin{definition}
For a finite depth unshaded subfactor planar algebra $\mathscr{S}_{\bullet}$ and a $\mathbb{Z}_m$ grading operator $g$ with a finite periodicity, we construct a $\mathbb{Z}_m$ graded unitary fusion category as $\mathscr{S}_{\bullet}$ modulo a commutative grading $g$, denoted by $\mathscr{S}/g$.
\end{definition}

Let the $Ver$ be the set of vertices of the principal graph of the unshaded subfactor planar algebra $\mathscr{S}_{\bullet}$, then the (equivalent classes of) simple objects of $\mathscr{S}_{\bullet}$ as a $\mathbb{N}\cup\{0\}$ graded monoidal category are indexed by $\lambda\otimes e^{\otimes k}$ for $\lambda\in Ver$, $k\geq0$.
The simple objects of $\mathscr{S}/g$ are indexed by of $\lambda\otimes e^{\otimes k}$ modulo $g$. If $X$ is the object in $\mathscr{S}/g$ corresponding to the 1-box identity in $\mathscr{S}_{\bullet}$, then one can derive the branching formula of $\mathscr{S}/g$ with respect to $X$ from the principal graph of $\mathscr{S}_{\bullet}$.

Now let us construct grading operators in $\mathscr{C}^{N}_{\bullet}$ and related unitary fusion categories.
Recall that $g^{(N)}$ is the $N$th antisymmetrizer of the Hecke algebra $H^N_{\bullet}$.
The trace of $g^{(N)}=y_{[1^N]}$ is one. It has a trace one subprojection $\tilde{y}_{[1^N]}$. Thus $\tilde{y}_{[1^N]}=g^{(N)}$.

\begin{proposition}
The tensor product $(g^{(N)})^{\otimes k}\otimes e^{\otimes l}$ is a grading operator in $\mathscr{C}^{N}_{Nk+2l}$ with periodicity $\frac{N+1}{(N+1,k)}$, for $N\geq1,$ and $k,l\geq0$.
\end{proposition}
\begin{proof}
Take $U=(g^{(N)}\otimes1) \alpha_N\alpha_{N-1}\cdots\alpha_1$. Then $U$ is a partial isometry from $g^{(N)}\otimes1$ onto $1\otimes g^{(N)}$ by type III Reidemester moves for $\alpha$.
By Proposition \ref{flat}, Equation \ref{braiding} holds for any $x$.
Thus $g^{(N)}$ is a $\mathbb{Z}_N$ grading operator.
By Proposition \ref{group fusion}, the periodicity of $g^{(N)}$ is $N+1$.
Recall that $e$ is a grading operator. By Proposition \ref{Prop:tensorgrading}, the tensor product  $(g^{(N)})^{\otimes k}\otimes e^{\otimes l}$ is a grading operator.
By Proposition \ref{group fusion}, the periodicity of the grading operator $(g^{(N)})^{\otimes k}\otimes e^{\otimes l}$ is $\frac{N+1}{(N+1,k)}$.
\end{proof}

\begin{theorem}\label{Thm:UFC}
The unshaded finite depth subfactor planar algebra $\mathscr{C}^{N}_{\bullet}$ modulo the grading operator $(g^{(N)})^{\otimes k}\otimes e^{\otimes l}$ is a $\mathbb{Z}_{kN+2l}$ graded unitary fusion category, denoted by $\mathscr{C}^{N,k,l}$.
\end{theorem}

\begin{proof}
Follows from the construction in Appendix \ref{Appendix:modulo grading operators}.
\end{proof}

One obtains 3D TQFT from each unitary fusion category $\mathscr{C}^{N,k,l}$ by Turaev-Viro construction \cite{TurVir92}.
\begin{question}
Whether these 3D TQFT can be parameterized by three parameters?
\end{question}

\begin{theorem}\label{Cor:quantumsubgroup1}
The unitary fusion category $\mathscr{C}^{N,0,1}$ is the module category of a subgroups of quantum $SU(N)_{N+2}$.
Therefore $\mathscr{C}_{\bullet}$ is the centralizer algebra for these quantum subgroups.
\end{theorem}

\begin{proof}
Note that the unitary fusion category $\mathscr{C}^{N,0,1}$ is $\mathscr{C}^{N}/g^{(N)}$ which contains the representation category of quantum $SU(N)_{N+2}$ as a subcategory. Therefore it is the module category of a subgroups of quantum $SU(N)_{N+2}$ in the sense of Ocneanu \cite{Ocn00}, or in the sense of Ostrik by Theorem 1 in \cite{Ost03}.
\end{proof}

Recall that the subalgebra of $\mathscr{C}^{N}_{\bullet}$ generated by shifts of $\gra{b2}$ is also a Hecke algebra, which is the centralizer algebra for quantum $SU(N+2)_{N}$.
Let us construct the antisymmetrizers $h^{(l)}$, $1\leq l\leq N+2$ from $\beta_i$ as follows,
$$h^{(l)}=h^{(l-1)}-\frac{[l-1]}{[l]}h^{(l-1)}(q^{-1}+\beta_i)h^{(l-1)},$$
where $h^{(1)}=1$. In particular, $h^{(N+2)}$ is a trace one projection. By Proposition \ref{flat}, we have the following:

\begin{proposition}
The operator $h^{(N+2)}$ is $\mathbb{Z}_{N+2}$ grading operator for $\mathscr{C}^{N}_{\bullet}$.
\end{proposition}

\begin{proposition}
The minimal projection $g^{(N)}\otimes e$ is equivalent to $h^{(N+2)}$ in $\mathscr{C}^{N}_{N+2}$.
\end{proposition}

\begin{proof}

Let $\Phi$ be the trace preserving condition expectation from $\mathscr{C}^{N}_{N+2}$ to $\mathscr{C}^{N}_{N}$, i.e. adding two caps on the right of a $N+2$ box.
Then it is also a trace preserving condition expectation on the Hecke algebra and $\Phi(h^{(N+2)})=\frac{tr(h^{(N+2)})}{tr(h^{(N)})}h^{(N)}$.

Recall that $s_m$ is the complement of the support of the basic construction ideal of $\mathscr{C}^{N}_m$, so $s_m\alpha_i=s_m\beta_i$.
By the inductive construction of the antisymmetrizer, we have $s_mg^{(l)}=s_mh^{(l)},$ for $1\leq l\leq N$. Recall that $s_mg^{(N)}=g^{(N)}$, so
\begin{align*}
\Phi(h^{(N+2)}(g^{(N)}\otimes 1\otimes 1) )
=&\Phi(h^{(N+2)}g^{(N)})\\
=&\frac{tr(h^{(N+2)})}{tr(h^{(N)})}h^{(N)}g^{(N)}\\
=&\frac{tr(h^{(N+2)})}{tr(h^{(N)})}h^{(N)}s_mg^{(N)}\\
=&\frac{tr(h^{(N+2)})}{tr(h^{(N)})}g^{(N)}\\
\neq&0.
\end{align*}
Therefore the trace one projection $h^{(N+2)}$ is subequivalent to $g^{(N)}\otimes 1\otimes 1$.
Note that
$$1\otimes 1=e+\tilde{y}_{[11]}+\tilde{y}_{[1^2]}.$$
When $N\geq3$, $g^{(N)}\otimes 1\otimes 1$ has only one trace one subprojection $g^{(N)}\otimes e$. Thus
the minimal projection $g^{(N)}\otimes e$ is equivalent to $h^{(N+2)}$ in $\mathscr{C}^{N}_{N+2}$.

When $N=1$, $\mathscr{C}^1$ has index one. The statement is true.

When $N=2$, $\mathscr{C}^1$ is the group subfactor planar algebra $\mathbb{Z}_3$. The 2-box minimal projections are given by $e$, $g^{(2)}$, $h^{(2)}$. Thus $g^{(2)}\otimes e$ is equivalent to $h^{(2)}\otimes h^{(2)}$.
Applying two left caps or two right caps to $h^{(4)}$, we obtain a scalar multiple of $h^{(2)}$. Thus $h^{(4)}$ is a sub projection of $h^{(2)}\otimes h^{(2)}$. Both of them have trace one, so $h^{(4)}=h^{(2)}\otimes h^{(2)}$, which is equivalent to $g^{(2)}\otimes e$.
\end{proof}

\begin{theorem}\label{Cor:quantumsubgroup2}
The unitary fusion category $\mathscr{C}^{N,1,1}$ is the module category of a subgroups of quantum $SU(N+2)_{N}$.
Therefore $\mathscr{C}_{\bullet}$ is the centralizer algebra for these quantum subgroups.
\end{theorem}

\begin{proof}
Note that the unitary fusion category $\mathscr{C}^{N,1,1}$ is isomorphic to $\mathscr{C}^{N}/h^{N+2}$ which contains the representation category of quantum $SU(N)_{N+2}$ as a subcategory. Therefore it is the module category of a subgroups of quantum $SU(N)_{N+2}$ in the sense of Ocneanu \cite{Ocn00}, or in the sense of Ostrik by Theorem 1 in \cite{Ost03}.
\end{proof}

Let $g$ be the simple object in $\mathscr{C}^{N,k,l}$ corresponding to the minimal projection $g^{N}$. We have known the principal graph $YL(N)$ of $\mathscr{C}^{N}_{\bullet}$ and the action of $g$ by tensor product on the principal graph (Corollary \ref{fusion rule for 1^n}). Therefore we have \underline{the branching formula for all $\mathscr{C}^{N,k,l}$} with respect to the generating simple object, which is indexed by the Young diagram [1] with one cell.

Let us compute the branching formula for the case $\mathscr{C}^{3,1,0}$ as an example.

When $N=3$, the principal graph of $\mathscr{C}^{3}_{\bullet}$ is
$$\grd{principalgraph4young}.$$
(This is the branching formula of $\mathscr{C}^{3,0,1}$.)

The simple objects of $\mathscr{C}^{N}_{\bullet}$ as a $\mathbb{N}\cup\{0\}$ graded monoidal category are indexed by $\lambda\otimes e^{\otimes k}$ for all Young diagrams $\lambda$ and $k\geq0$.

The grading operator $g=g^{(3)}$ is indexed by the Young diagram $[1^3]=\gra{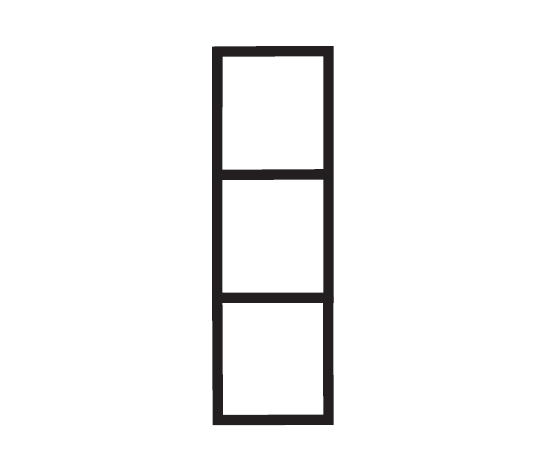}$. Its action on $Y(3)$ is given in Figure \ref{Figure:action on Y3}.

\begin{figure}[h]
$$\grd{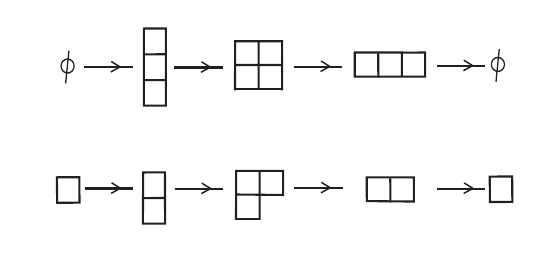}.$$
\caption{The action of the grading operator $g$ on $Y(3)$}\label{Figure:action on Y3}
\end{figure}

Let us fix representatives $\emptyset$ and $[1]$ for the two orbits. The periodicity of $g$ is 4. So $e^{\otimes 6}$ is equivalent to $g^{\otimes 4}$ in $\mathscr{C}^{3}_{12}$, equivalent to $\emptyset$ in  $\mathscr{C}^{3,1,0}$.
Therefore the simple objects of $\mathscr{C}^{3,k,l}$ are indexed by $ e^{\otimes t}$ and $[1] \otimes e^{\otimes t}$, for $\displaystyle 0\leq t< 6$.

The fusion rule is derived from the principal graph as follows.
\begin{align*}
e^{\otimes 6}& \sim g^{\otimes 4}\sim_{g} \emptyset;\\
[1]\otimes e&\sim e\otimes [1];\\
[1]\otimes [1]&\sim e\oplus [2] \oplus [1^2]\\
&\sim_g e\oplus ([2]\otimes g) \oplus ([1^2] \otimes g^{\otimes 3})\\
&\sim e\oplus ([1]\otimes e^{\otimes 2}) \oplus ([1] \otimes e^{\otimes 5}),
\end{align*}
where $\sim$ means two projections (or objects) are equivalent in $\mathscr{C}^{3}_{\bullet}$, and $\sim_g$ means two projections (or objects) are equivalent in $\mathscr{C}^{3,1,0}$.
Thus the $\mathbb{Z}_3$ graded branching formula of $\mathscr{C}^{3,1,0}$ is given in Figure \ref{Figure:Branchingformula2}. (This is Figure \ref{Figure:Branchingformula}.)
\begin{figure}[h]
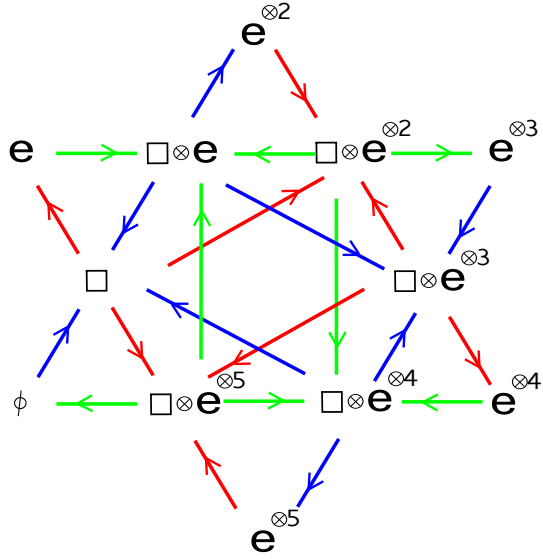

$$\grf{colorgraph3}.$$
  \caption{The branching formula for $\mathscr{C}^{3,1,0}$} \label{Figure:Branchingformula2}
\end{figure}

Ocneanu classified subgroups of quantum $SU(3)$ and $SU(4)$ in \cite{Ocn00}. He also gave their branching formulas.
Our branching formulas for $\mathscr{C}^{N,0,1}$, $N=3,4$, correspond to exceptional subgroups of quantum $SU(3)_5$ and $SU(4)_6$ in Ocneanu's list. Our branching formulas for $\mathscr{C}^{N,0,1}$, $N\geq5$, are new.

Xu constructed bimodule categories (or unitary fusion categories) from conformal inclusions in \cite{Xu98}, namely the $\alpha$-induction.
These bimodule categories contains the representation categories of quantum groups as a subcategory, if the the conformal inclusions are given by corresponding Lie groups. Therefore these bimodule categories are defined as representation categories of quantum subgroups.
In this construction, he could compute the branching formula for subgroups of some small rank quantum groups.
The branching formula for $\mathscr{C}^{3,1,0}$ in Figure \ref{Figure:Branchingformula2} is the same as the one computed by Xu for the conformal inclusion $\displaystyle SU(3)_{5}\subset SU(6)_{1}$.
We conjecture from this observation:
\begin{conjecture}
  The unitary fusion category $\mathscr{C}^{N,1,0}$ is isomorphic to Xu's bimudule category for the conformal inclusion $\displaystyle SU(N)_{N+2}\subset SU(\frac{N(N+1)}{2})_{1}$.
\end{conjecture}
The conjecture is true for $N=2,3,4$ by Ocneanu's classification result.

Recall that $\mathscr{C}^{N,1,1}$ is the unitary fusion category $\mathscr{C}^{N}/h$, where $h\sim g\otimes e$.
When $N=3$, the simple objects of $\mathscr{C}^{3,1,1}$ are given by $e^{\otimes j}$, $[1]\otimes e^{\otimes j}$, for $0\leq j<10$. Moreover, the fusion rule is given by
\begin{align*}
e^{\otimes 10}&\sim \emptyset\\
[1]\otimes e&\sim e\otimes [1]\\
[1]\otimes [1]&\sim e\oplus ([1]\otimes e^{\otimes 3}) \oplus ([1] \otimes e^{\otimes 8})
\end{align*}
This fusion rule is the same as the one computed by Xu for the conformal inclusion $SU(5)_3 \subset SU(10)_1$.
We conjecture from this observation:
\begin{conjecture}
   The unitary fusion category $\mathscr{C}^{N,1,1}$ is isomorphic to Xu's bimodule category for the conformal inclusion $\displaystyle SU(N+2)_{N}\subset SU(\frac{(N+2)(N+1)}{2})_{1}$.
\end{conjecture}
The conjecture is ture for $N=2$ by Ocneanu's classification result.

Our branching formula for $\mathscr{C}^{N,1,1}$ is new for $N\geq4$. Our unitary fusion categories $\mathscr{C}^{N,k,l}$ and branching formulas for other parameters $N,k,l$ are new.

\newpage
\appendix

\section{An algebraic presentation}\label{Appendix:algebraic presentation}
\begin{theorem}\label{Thm:algebraic presentation}
The $q$-parameterized Yang-Baxter relation planar algebra $\mathscr{C}_{\bullet}$ constructed in Section \ref{subsection consistency} has the following algebraic presentation. ($\alpha=\gra{b1}$, $h=\gra{22}$)
\begin{align*}
&\alpha_i-\alpha_i^{-1}=(q-q^{-1})\\
&\alpha_i\alpha_j=\alpha_j\alpha_i, ~\forall~ |i-j|\geq2\\
&\alpha_i\alpha_{i+1}\alpha_i=\alpha_{i+1}\alpha_i\alpha_{i+1}\\
&h_i^2=\frac{i(q+q^{-1})}{q-q^{-1}}h_i\\
&h_ih_j=h_jh_i, ~\forall~ |i-j|\geq2\\
&h_ih_{i\pm1}h_i=h_i\\
&\alpha_ih_i=h_i\alpha_i=qh_i\\
&\alpha_ih_j=h_j\alpha_i ~\forall~ |i-j|\geq2\\
&\alpha_i\alpha_{i+1}h_i=h_{i+1}\alpha_i\alpha_{i+1}=ih_{i+1}h_i\\
&h_i\alpha_{i+1}\alpha_i=\alpha_{i+1}\alpha_ih_{i+1}=-ih_ih_{i+1}\\
&\alpha_ih_{i\pm1}\alpha_{i\pm1}^{-1}=\alpha_{i\pm1}^{-1}h_{i}\alpha_{i\pm1}\\
&h_ih_{i\pm1}\alpha_i=h_i\alpha_{i\pm1}^{-1}\\
&\alpha_ih_{i\pm1}h_i=\alpha_{i\pm1}h_i\\
&h_i\alpha_{i\pm1}h_i=iq^{-1}h_i\\
\end{align*}
\end{theorem}

\begin{proof}
Let $A$ be the filtered algebra defined by the above presentation.
It is easy to check, these algebraic relations hold for $\mathscr{C}_{\bullet}$. By Proposition \ref{algebragenerated}, $\mathscr{C}_{\bullet}$ is a quotient of $A$ as an algebra.
Note that the construction of matrix units of $\mathscr{C}_{\bullet}$ only use these algebraic relations. Thus the matrix unit of $\mathscr{C}_{\bullet}$ is also a matrix unit of $\mathscr{A}$. Therefore $\mathscr{C}_{\bullet}=A$.
\end{proof}

\begin{remark}
If we consider $\alpha$ and $ih$ as generators, then the centralizer algebra $\mathscr{C}_{\bullet}$ is also defined on the field $\mathbb{Z}(q)$.
\end{remark}

\section{Wenzl's formula}\label{Appendix:Wenzl's formula}

We give a proof of the general Wenzl's formula in Theorem \ref{P=YL}.

\begin{proof}
Take
\begin{align}\label{wenzl1}
x &= \tilde{y}_\mu \otimes 1 -\sum_{\lambda<\mu} \frac{<\lambda>}{<\mu>} (\tilde{y}_\mu\otimes 1)(\rho_{\mu>\lambda} \otimes 1) (\tilde{y}_{\lambda}\otimes \cap) (\tilde{y}_{\lambda}\otimes \cup) (\rho_{\lambda<\mu} \otimes 1) (\tilde{y}_\mu \otimes 1)
\end{align}
and a length $(|\mu|+1)$ path $t$ from $\emptyset$ to $\lambda'$, $|\lambda'|<|\mu|$.

If $\lambda'<\mu$ does not hold, then
\begin{align*}
(\tilde{y}_{\lambda}\otimes \cup) (\rho_{\lambda<\mu} \otimes 1) (\tilde{y}_\mu \otimes 1)\tilde{P}^+(t)&=0,~\forall~\lambda<\mu,
\end{align*}
since it is a morphism from $\tilde{y}_{\lambda}$ to $\tilde{y}_\lambda'$.
By the Frobenius reciprocity, $(\tilde{y}_{\mu}\otimes 1) \tilde{P}^+(t)=0$, since it is a morphism from $\tilde{y}_{\lambda}\otimes 1$ to $\tilde{y}_\lambda'$.
Therefore $x\tilde{P}^+(t)=0$.

If $\lambda'<\mu$, then
\begin{align*}
(\tilde{y}_\mu \otimes 1) \tilde{P}^+(t)=c (\tilde{y}_\mu \otimes 1) (\rho_{\mu\to\lambda'}\otimes 1) (y_{\lambda'}\otimes \cap),
\end{align*}
for some constant $c$, because both sides are morphisms from $\tilde{y}_{\lambda}\otimes 1$ to $\tilde{\lambda'}$.
Thus
\begin{align*}
&(\tilde{y}_{\lambda'}\otimes \cup) (\rho_{\lambda'<\mu} \otimes 1) (\tilde{y}_\mu \otimes 1)\tilde{P}^+(t)\\
=&c(\tilde{y}_{\lambda'}\otimes \cup)(\rho_{\lambda'<\mu} \otimes 1) (\tilde{y}_\mu \otimes 1) (\rho_{\mu\to\lambda'}\otimes 1) (y_{\lambda'}\otimes \cap)\\
=&\frac{c<\mu>}{<\lambda'>}\tilde{y}_{\lambda'}.
\end{align*}
Moreover,
\begin{align*}
 (\tilde{y}_{\lambda}\otimes \cup) (\rho_{\lambda<\mu} \otimes 1) (\tilde{y}_\mu \otimes 1)\tilde{P}^+(t)&=0,~\text{when}~\lambda\neq \lambda',
\end{align*}
since it is a morphism from $\tilde{y}_{\lambda}$ to $\tilde{\lambda'}$.
Therefore
\begin{align*}
x\tilde{P}^+(t)&=c (\tilde{y}_\mu \otimes 1) (\rho_{\mu\to\lambda'}\otimes 1) (y_{\lambda'}\otimes \cap)-c (\tilde{y}_\mu \otimes 1) (\rho_{\mu\to\lambda'}\otimes 1) (y_{\lambda'}\otimes \cap)=0.
\end{align*}

Recall that $IYL_{|\mu|+1}\cong \mathscr{I}_{|\mu|+1}$, so $xz=0$, for any $z\in \mathscr{I}_{|\mu|+1}$. Thus $xs_{|\mu|+1}=x$.
Note that $s_{|\mu|+1}$ is central and $(\tilde{y}_{\lambda}\otimes \cup)s_{|\mu|+1}=0$, by Equation (\ref{wenzl1}), we have
\begin{align}\label{wenzl2}
x=xs_{|\mu|+1}=(\tilde{y}_{\mu}\otimes 1) s_{|\mu|+1}.
\end{align}

On the other hand,
\begin{align*}
&(\tilde{y}_{\mu}\otimes1) s_{|\mu|+1} &\\
=&(y_{\mu}s_{|\mu|}\otimes 1) s_{|\mu|+1} &\\
=&(y_{\mu}\otimes 1) s_{|\mu|+1} &\\
=& \sum_{\lambda>\mu} (y_{\mu}\otimes1)\rho_{\mu\to \lambda}y_{\lambda}\rho_{\lambda \to \mu} (y_{\mu}\otimes1) s_{|\mu|+1} & \text{Branching formula (\ref{branching formula}),}\\
=& \sum_{\lambda>\mu} (\tilde{y}_\mu\otimes 1) \rho_{\mu<\lambda} \tilde{y}_{\lambda} \rho_{\lambda>\mu} (\tilde{y}_\mu\otimes 1). & \numberthis \label{wenzl3}
\end{align*}
By Equations (\ref{wenzl1}), (\ref{wenzl2}), (\ref{wenzl3}), we obtain Wenzl's formula.
\end{proof}

\section{Proofs of Theorems and Lemmas}
\subsection{Proof of Theorem \ref{classification1}:}\label{Appendix:classification1}

\begin{proof}
There are two different ways to evaluate the 3-box $\gra{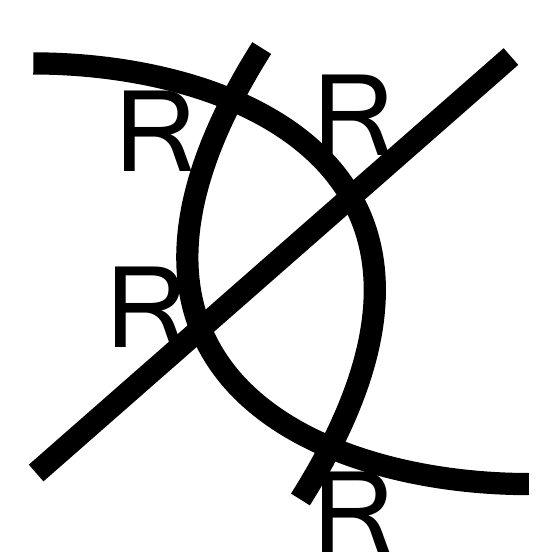}$ as a linear sum over the basis.
Replacing $\gra{15}$ by $\gra{16}$ and lower terms, we have
\begin{align*}
&\gra{17}=\grb{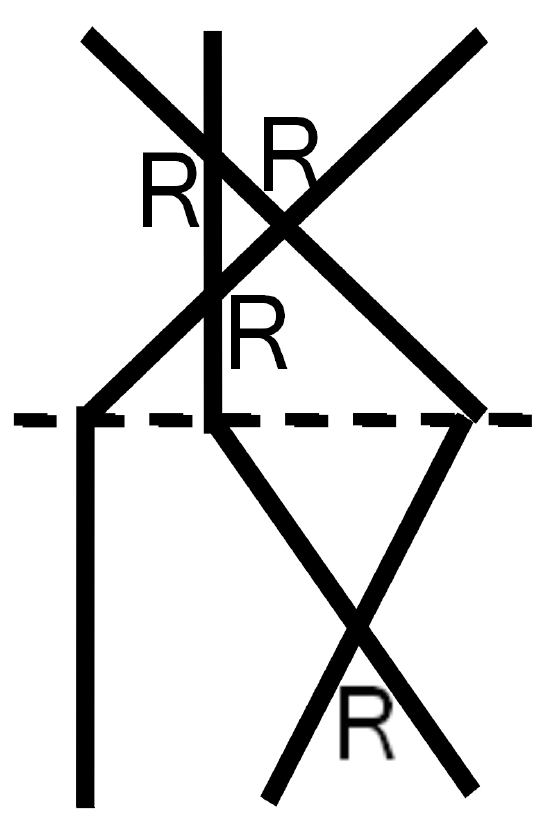}\\
                 &=B\gra{7}+C\gra{9}+C\gra{11}\\
                 &+D\gra{14}+D(a'\gra{7}+\gra{2}-\frac{1}{\delta}\gra{4})+D\gra{12}\\
                 &+E(a'\gra{9}+\gra{3}-\frac{1}{\delta}\gra{1})+E(a'\gra{11}+\gra{5}-\frac{1}{\delta}\gra{1})\\
                 &+F(a'\gra{12}+\gra{8}-\frac{1}{\delta}\gra{10})+F\gra{16}+F(a'\gra{14}+\gra{6}-\frac{1}{\delta}\gra{10})\\
                 &+G(a'\gra{16}+\gra{13}-\frac{1}{\delta}(a'\gra{10}+\gra{4}-\frac{1}{\delta}\gra{1})).
\end{align*}
Replacing $\gra{16}$ by $\gra{15}$ and lower terms, we have
\begin{align*}
&-G\gra{17}=-G\grb{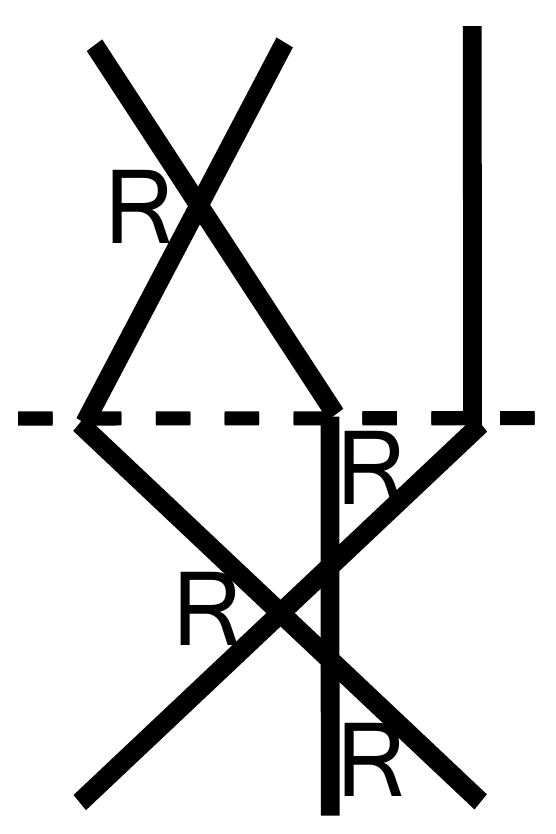}\\
&=-(a\gra{15}+\gra{13}-\frac{1}{\delta}(a\gra{7}+\gra{4}-\frac{1}{\delta}\gra{2}))\\
&+A\gra{10}+C\gra{6}+C\gra{8}\\
&+D(a\gra{6}+\gra{3}-\frac{1}{\delta}\gra{2})+D(a\gra{8}+\gra{5}-\frac{1}{\delta}\gra{2})\\
&+E\gra{14}+E\gra{12}+E(a\gra{10}+\gra{1}-\frac{1}{\delta}\gra{4})\\
&+F(a\gra{12}+\gra{11}-\frac{1}{\delta}\gra{7})+F\gra{15}+F(a\gra{14}+\gra{9}-\frac{1}{\delta}\gra{7}).\\
\end{align*}
Therefore
\begin{align*}
&(a-F)\gra{15}\\
&=(E-2GE\frac{1}{\delta}+G^2\frac{1}{\delta^2})\gra{1}+(-\frac{1}{\delta^2}-2D\frac{1}{\delta}+GD)\gra{2}\\
&+(D+GE)(\gra{3}+\gra{5})+(\frac{1}{\delta}-E\frac{1}{\delta}-GD\frac{1}{\delta}-G^2\frac{1}{\delta})\gra{4}\\
&+(C+Da+GF)(\gra{6}+\gra{8})+(a\frac{1}{\delta}-2F\frac{1}{\delta}+GB+GDa')\gra{7}\\
&+(F+GC+GEa')(\gra{9}+\gra{11})+(A+Ea-2GF\frac{1}{\delta}-G^2a'\frac{1}{\delta})\gra{10}\\
&+(E+Fa+GD+GFa')(\gra{12}+\gra{14})+(-1+G^2)\gra{13}+(GF+G^2a')\gra{16}.\\
\end{align*}
Comparing the coefficients of the basis,
we have the following equations.
\begin{align}
&\label{g16}
(a-F)G=GF+G^2a' &\gra{16}\\
&\label{g12}
(a-F)F=E+Fa+GD+GFa'=(-1+G^2) &\gra{12},\gra{13},\gra{14}\\
&\label{g9}
(a-F)E=F+GC+GEa'=A+Ea-2GF\frac{1}{\delta}-G^2a'\frac{1}{\delta} &\gra{9},\gra{10},\gra{11}\\
&\label{g6}
(a-F)D=C+Da+GF=a\frac{1}{\delta}-2F\frac{1}{\delta}+GB+GDa' &\gra{6},\gra{7},\gra{8}\\
&\label{g3}
(a-F)C=D+GE=\frac{1}{\delta}-E\frac{1}{\delta}-GD\frac{1}{\delta}-G^2\frac{1}{\delta} & \gra{3},\gra{4},\gra{5}\\
&\label{g2}
(a-F)B=-\frac{1}{\delta^2}-2D\frac{1}{\delta}+GD &\gra{2}\\
&\label{g1}
(a-F)A=(E-2GE\frac{1}{\delta}+G^2\frac{1}{\delta^2}) &\gra{1}
\end{align}

Case 1: If $F=0$, then equation (\ref{g12}) implies
$$G^2=1, E+GD=0.$$
By equation (\ref{g16}), we have
$$a'=Ga.$$
Applying $F=0, a'=Ga$ to the first equality of equation (\ref{g9}), we have
$$C=0.$$
Applying $F=0, a'=Ga, G^2=1$ to the second equality of equation (\ref{g9}), we have
$$A=\frac{Ga}{\delta}.$$
Applying $F=0, a'=Ga, G^2=1$ to the second equality of equation (\ref{g6}), we have
$$B=-\frac{Ga}{\delta}.$$
Applying $B=-\frac{Ga}{\delta}$ to equation (\ref{g2}), we have
$$(G\delta^2-2\delta)D=1-Ga^2\delta.$$
We have solved $A,B,C,D,E,F,G$ in term of $a$ and $\delta$ (and $D$).

Case 2: If $F\neq0 $, then equation (\ref{g12}) implies
$$a=F+\frac{G^2-1}{F}.$$
Note that $G\neq0$, since $\mathscr{C}_{\bullet}$ has a Yang-Baxter.
Substituting $a$ in equation (\ref{g16}), we have
$$ a'=\frac{\frac{G^2-1}{F}-F}{G}.$$
Substituting $a,a'$ in the first equalities of equation (\ref{g12}), (\ref{g9}), (\ref{g6}), we have

$$\left\{
\begin{array}{lllll}
GC&&-FE&=&-F\\
C&+(F+\frac{G^2-1}{F})D&&=&\frac{G^2-1}{F}-FG\\
&GD&+E&=&1-G^2\\
\end{array}
\right.$$
Let us consider $F,G$ as constants and $C,D,E$ as variables, then the determinant of the coefficient matrix on the left side is
$$\begin{vmatrix}
G & 0               &-F \\
1 &F+\frac{G^2-1}{F}& 0 \\
0 &G                & 1 \\
\end{vmatrix}
=\frac{G^2-1}{F}.$$

If $G^2-1\neq 0$, then we have the unique solution
$$\left\{
\begin{array}{lll}
C&=&-F-FG\\
F&=&1\\
E&=&1-G-G^2\\
\end{array}
\right.$$
Plugging the solution into the second equality of equation (\ref{g3}), we have
$$1+(1-G-G^2)G=\frac{1}{\delta}(1-(1-G-G^2)-G-G^2).$$
This implies $(1+G)^2(1-G)=0$. So $G=\pm1$, and $G^2-1=0$, contradicting to the assumption.

If $G^2-1=0$, then $G=\pm1$, $a=F$, $a'=-GF$, and
$$\left\{
\begin{array}{lllll}
GC&&-FE&=&-F\\
C&+FD&&=&-FG\\
&GD&+E&=&0\\
\end{array}
\right.$$
So
$$E=-GD, C=-F(G+D).$$
By equation $\ref{g2}$, we have
$(G\delta^2-2\delta)D=1$. So $G\delta^2-2\delta\neq 0$ and
$$D=\frac{1}{G\delta^2-2\delta}.$$
From the second equality of equation (\ref{g9}), (\ref{g6}), we have
$$A=B=(\frac{1}{\delta}+D)GF.$$
We have solved $A,B,C,D,E,F,G$ in term of $a$ and $\delta$.
\begin{align}
\left\{
  \begin{aligned}
    G&=\pm1\\
    A&=B=Ga(\frac{1}{\delta}+D)\\
    C&=-a(G+D)\\
    D&=-GE=\frac{1}{G\delta^2-2\delta}\\
    F&=a\\
    a'&=-Ga\\
  \end{aligned}
\right. \label{equ solution}
\end{align}
Adding a cap to the right of the following equation
\begin{align*}
\gra{15}&=A\gra{1}+B\gra{2}+C(\gra{3}+\gra{4}+\gra{5})\\
&+D(\gra{6}+\gra{7}+\gra{8})+E(\gra{9}+\gra{10}+\gra{11})\\
&+F(\gra{12}+\gra{13}+\gra{14})+G\gra{16},\\
\end{align*}
we get
\begin{align*}
0&=A \gra{21}+B\delta\gra{22}+C\delta\gra{21}+2C\gra{22}+D\delta\gra{23}+2E\gra{23}+F(a\gra{23}+\gra{21}-\frac{1}{\delta}\gra{22})+\\
&+Ga(a'\gra{23}+\gra{22}-\frac{1}{\delta}\gra{21})-G\frac{1}{\delta}\gra{23}\\
&=(A+C\delta+F-Gaa'\frac{1}{\delta})\gra{21}+(B\delta+2C-F\frac{1}{\delta}+Ga)\gra{22}+\\
&+(D\delta+2E+Fa+Gaa'-G\frac{1}{\delta})\gra{23}.
\end{align*}
Therefore
\begin{align*}
0&=A+C\delta+F-Gaa'\frac{1}{\delta}\\
&=Ga(\frac{1}{\delta}+D)-a(G+D)\delta+a+a^2\frac{1}{\delta}. && \text{by the solution \ref{equ solution}.}
\end{align*}
Recall that $a=F\neq0$, so
\begin{equation*}
a=-G(1+\delta D)+(G+D)\delta^2-1
\end{equation*}
If we replace the generator $R$ by $-R$ and repeat the above arguments, then $a,\delta,A,B,C,D,E,F,G$ are replaced by $-a,\delta,-A,-B,-C,D,E,-F,G$. So we have
\begin{equation*}
-a=-G(1+\delta D)+(G+D)\delta^2-1
\end{equation*}
Thus $a=0$, contradicting to $a=F\neq0$. Therefore there is no solution for Case 2.
\end{proof}

\subsection{Proof of Lemma \ref{solution of YBE 1}}\label{Appendix:solution of YBE 1}
Recall that
$R_U=a_1\gra{21}+a_2\gra{22}+a_3\gra{23}, a_3\neq0,$
we have
\begin{align*}
R_U(1\otimes R_U)R_U&=a_1a_1a_1\gra{3}+a_1a_1a_2\gra{2}+a_1a_1a_3\gra{6}\\
&+a_1a_2a_1\gra{1}+a_1a_2a_2\gra{5}+a_1a_2a_3\gra{11}\\
&+a_1Ga_3a_1\gra{9}+a_1Ga_3a_2\gra{8}+a_1Ga_3a_3\gra{13}\\
&+a_2a_1a_1\gra{2}+a_2a_1a_2\delta\gra{2}+a_2a_1a_3 0\\
&+a_2a_2a_1\gra{4}+a_2a_2a_2\gra{2}+a_2a_2a_3\gra{7}\\
&+a_2Ga_3a_1\gra{7}+a_2Ga_3a_2 0+a_2Ga_3a_3(a\gra{7}+\gra{4}-\frac{1}{\delta}\gra{2})\\
&+a_3a_1a_1\gra{6}+a_3a_1a_2 0+a_3a_1a_3(a\gra{6}+\gra{3}-\frac{1}{\delta}\gra{2})\\
&+a_3a_2a_1\gra{10}+a_3a_2a_2\gra{8}+a_3a_2a_3\gra{12}\\
&+a_3Ga_3a_1\gra{14}+a_3Ga_3a_2(a\gra{8}+\gra{5}-\frac{1}{\delta}\gra{2})+a_3Ga_3a_3\gra{15}
\end{align*}
and
\begin{align*}
(1\otimes R_U)R_U(1\otimes R_U)&=a_1a_1a_1\gra{3}+a_1a_1a_2\gra{1}+a_1a_1Ga_3\gra{9}\\
&+a_1a_2a_1\gra{2}+a_1a_2a_2\gra{4}+a_1a_2Ga_3\gra{7}\\
&+a_1a_3a_1\gra{6}+a_1a_3a_2\gra{10}+a_1a_3Ga_3\gra{14}\\
&+a_2a_1a_1\gra{1}+a_2a_1a_2\delta\gra{1}+a_2a_1Ga_3 0\\
&+a_2a_2a_1\gra{5}+a_2a_2a_2\gra{1}+a_2a_2Ga_3\gra{11}\\
&+a_2a_3a_1\gra{11}+a_2a_3a_2 0+a_2a_3Ga_3(a'\gra{11}+\gra{5}-\frac{1}{\delta}\gra{1})\\
&+Ga_3a_1a_1\gra{9}+Ga_3a_1a_2 0+Ga_3a_1Ga_3(a'\gra{9}+\gra{3}-\frac{1}{\delta}\gra{1})\\
&+Ga_3a_2a_1\gra{8}+Ga_3a_2a_2\gra{10}+Ga_3a_2Ga_3\gra{12}\\
&+Ga_3a_3a_1\gra{13}+Ga_3a_3a_2(a'\gra{10}+\gra{4}-\frac{1}{\delta}\gra{1})+Ga_3a_3Ga_3\gra{16}.
\end{align*}
Replacing $\gra{15}$ by $\gra{16}$ and lower terms, then comparing the coefficients, we have
$$R_U(1\otimes R_U)R_U=(1\otimes R_U)R_U(1\otimes R_U) \iff$$
\begin{align}
&a_3Ga_3a_3G=Ga_3a_3Ga_3 & \gra{15} \label{gg15}\\
&a_3Ga_3a_3F+a_3Ga_3a_1=a_1a_3Ga_3 & \gra{14} \label{gg14}\\
&a_3Ga_3a_3F+a_1Ga_3a_3=Ga_3a_3a_1 & \gra{13} \label{gg13} \\
&a_3Ga_3a_3F+a_3a_2a_3=Ga_3a_2Ga_3 & \gra{12} \label{gg12}\\
&a_3Ga_3a_3E+a_1a_2a_3=a_2a_2Ga_3+a_2a_3a_1+a_2a_3Ga_3a' & \gra{11} \label{gg11}\\
&a_3Ga_3a_3E+a_3a_2a_1=a_1a_3a_2+Ga_3a_2a_2+Ga_3a_3a_2a' & \gra{10} \label{gg10}\\
&a_3Ga_3a_3E+a_1Ga_3a_1=a_1a_1Ga_3+Ga_3a_1a_1+Ga_3a_1Ga_3a' & \gra{9} \label{gg9}\\
&a_3Ga_3a_3D+a_1Ga_3a_2+a_3a_2a_2+a_3Ga_3a_2a=Ga_3a_2a_1 & \gra{8} \label{gg8}\\
&a_3Ga_3a_3D+a_2a_2a_3+a_2Ga_3a_1+a_2Ga_3a_3a=a_1a_2Ga_3 & \gra{7} \label{gg7}\\
&a_3Ga_3a_3D+a_1a_1a_3+a_3a_1a_1+a_3a_1a_3a=a_1a_3a_1 & \gra{6} \label{gg6}\\
&a_3Ga_3a_3C+a_1a_2a_2+a_3Ga_3a_2=a_2a_2a_1+a_2a_3Ga_3 & \gra{5} \label{gg5}\\
&a_3Ga_3a_3C+a_2a_2a_1+a_2Ga_3a_3=a_1a_2a_2+Ga_3a_3a_2 & \gra{4} \label{gg4}\\
&a_3Ga_3a_3C+a_1a_1a_1+a_3a_1a_3=a_1a_1a_1+Ga_3a_1Ga_3 & \gra{3} \label{gg3}
\end{align}
\begin{equation}
a_3Ga_3a_3B+a_1a_1a_2+a_2a_1a_1+a_2a_1a_2\delta+a_2a_2a_2-\frac{1}{\delta}a_2Ga_3a_3-\frac{1}{\delta}a_3a_1a_3-\frac{1}{\delta}a_3Ga_3a_2=a_1a_2a_1  \gra{2} \label{gg2}
\end{equation}
\begin{equation}
a_3Ga_3a_3A+a_1a_2a_1=a_1a_1a_2+a_2a_1a_1+a_2a_1a_2\delta+a_2a_2a_2-\frac{1}{\delta}a_2a_3Ga_3-\frac{1}{\delta}Ga_3a_1Ga_3-\frac{1}{\delta}Ga_3a_3a_2  \gra{1} \label{gg1}
\end{equation}

Note that $(\ref{gg14}) \iff (\ref{gg13})$; $(\ref{gg11}) \iff (\ref{gg10})$; $(\ref{gg8}) \iff (\ref{gg7})$; $(\ref{gg5}) \iff (\ref{gg4})$.

Equation (\ref{gg15}) always holds.

Since $F=0$, Equation (\ref{gg14}), (\ref{gg12}) hold.

By Equation (\ref{sumproduct}), $z_1$ and $z_2G$ are solutions of $z^2+az-E=0$.
Since $a_1/a_3=z_1$, $a_2/a_3=z_2$, we have
\begin{align*}
(a_2G)^2+aa_2a_3G-a_3^2E&=0\\
a_1^2+aa_1a_3-a_3^2E&=0
\end{align*}
Moreover, $a'=Ga$, so Equation (\ref{gg11}), (\ref{gg9}) hold.

Since $E=-GD$, Equation (\ref{gg8}), (\ref{gg6}) follow from (\ref{gg11}), (\ref{gg9}).

Since $C=0$, Equation (\ref{gg5}), (\ref{gg3}) hold.

Note that
\begin{align*}
&GB+z_1^2z_2+z_1z_2^2\delta+z_2^3-\frac{1}{\delta}z_2G-\frac{1}{\delta}z_1-\frac{1}{\delta}z_2G\\
=&-G\frac{a}{\delta}+z_1^2z_2+z_1z_2^2\delta+z_2^3-\frac{1}{\delta}z_2G-\frac{1}{\delta}z_1-\frac{1}{\delta}z_2G &&\text{since } B=-A=-G\frac{a}{\delta},\\
=&z_1^2z_2+z_1z_2^2\delta+z_2^3-\frac{1}{\delta}z_2G &&\text{by Equation (\ref{sumproduct}),}\\
=&(-a)^2+(\delta-2G)(-E)-\frac{G}{\delta} &&\text{by Equation (\ref{sumproduct}),}\\
=&0  &&\text{since } E=-GD, ~(G\delta^2-2\delta)D=1-Ga^2\delta.
\end{align*}
So Equation (\ref{gg2}) holds.

Since $\displaystyle B=-A$, Equation (\ref{gg1}) follows from Equation (\ref{gg2}).

Therefore $$R_U(1\otimes R_U)R_U=(1\otimes R_U)R_U(1\otimes R_U).$$
\subsection{Proof of Theorem \ref{class-1}}\label{Appendix:class-1}

Up to the complex conjugate, we only need to consider the case for $G=i$.

\begin{proof}
There are two different ways to evaluate the 3-box $\gra{17}$ as a linear sum over the basis.
Replacing $\gra{15}$ by $\gra{16}$ and lower terms, we have
\begin{align*}
\gra{17}&=\grb{18}\\
                 &=B\gra{7}+C\gra{9}-C\gra{11}\\
                 &+D\gra{14}+D(-\gra{2}+\frac{1}{\delta}\gra{4})+D\gra{12}\\
                 &+E(-\gra{3}+\frac{1}{\delta}\gra{1})+E(\gra{5}-\frac{1}{\delta}\gra{1})\\
                 &+F(-\gra{8}+\frac{1}{\delta}\gra{10})+F\gra{16}+F(-\gra{6}+\frac{1}{\delta}\gra{10})\\
                 &+G(-\gra{13}+\frac{1}{\delta}(\gra{4}-\frac{1}{\delta}\gra{1})).\\
\end{align*}
Replacing $\gra{16}$ by $\gra{15}$ and lower terms, we have
\begin{align*}
-G\gra{17}
&=-G\grb{19}\\
&=-(-\gra{13}+\frac{1}{\delta}(\gra{4}-\frac{1}{\delta}\gra{2}))\\
&-A\gra{10}+C\gra{6}-C\gra{8}\\
&+D(\gra{3}-\frac{1}{\delta}\gra{2})+D(-\gra{5}+\frac{1}{\delta}\gra{2})\\
&+E\gra{14}+E\gra{12}+E(-\gra{1}+\frac{1}{\delta}\gra{4})\\
&+F(\gra{11}+\frac{1}{\delta}\gra{7})-F\gra{15}+F(\gra{9}-\frac{1}{\delta}\gra{7}).\\
\end{align*}
Therefore
\begin{align*}
F\gra{15}
&=(-E+G^2\frac{1}{\delta^2})\gra{1}+(\frac{1}{\delta^2}+GD)\gra{2}\\
&+(D+GE)\gra{3}+(-D-GE)\gra{5}+(-\frac{1}{\delta}+E\frac{1}{\delta}-GD\frac{1}{\delta}-G^2\frac{1}{\delta})\gra{4}\\
&+(C+GF)\gra{6}+(-C+GF)\gra{8}-GB\gra{7}\\
&+(F-GC)\gra{9}+(F+GC)\gra{11}-A\gra{10}\\
&+(E-GD)(\gra{12}+\gra{14})+(-1-G^2)\gra{13}-GF\gra{16}.\\
\end{align*}

Comparing the coefficients of $\gra{16}$, we have
$$FG=-GF.$$
Note that $\mathscr{C}_{\bullet}$ is a Yang-Baxter relation planar algebra, so $G\neq0$. Then
$$F=0.$$
Comparing the coefficients of other diagrams,
we have
$$G^2=-1,A=0,B=0,C=0,D=-\frac{1}{G\delta^2}, E=-\frac{1}{\delta^2}.$$
Then $G=\pm i$.
\end{proof}

\subsection{Proof of Lemma \ref{solution of YBE -1}}\label{Appendix:solution of YBE -1}

\begin{proof}
Note that
\begin{align*}
ABA&=a_1b_1a_1\gra{3}+a_1b_1a_2\gra{2}+a_1b_1a_3\gra{6}\\
&+a_1b_2a_1\gra{1}+a_1b_2a_2\gra{5}+a_1b_2a_3\gra{11}\\
&-a_1b_3a_1\gra{9}+a_1b_3a_2\gra{8}+a_1b_3a_3\gra{13}\\
&+a_2b_1a_1\gra{2}+a_2b_1a_2\delta\gra{2}+a_2b_1a_3 0\\
&+a_2b_2a_1\gra{4}+a_2b_2a_2\gra{2}-a_2b_2a_3\gra{7}\\
&-a_2b_3a_1\gra{7}+a_2b_3a_2 0+a_2b_3a_3(\gra{4}-\frac{1}{\delta}\gra{2})\\
&+a_3b_1a_1\gra{6}+a_3b_1a_2 0+a_3b_1a_3(\gra{3}-\frac{1}{\delta}\gra{2})\\
&-a_3b_2a_1\gra{10}-a_3b_2a_2\gra{8}+a_3b_2a_3\gra{12}\\
&-a_3b_3a_1\gra{14}+a_3b_3a_2(-\gra{5}+\frac{1}{\delta}\gra{2})-a_3b_3a_3\gra{15},
\end{align*}
and
\begin{align*}
BAB&=b_1a_1b_1\gra{3}+b_1a_1b_2\gra{1}-b_1a_1b_3\gra{9}\\
&+b_1a_2b_1\gra{2}+b_1a_2b_2\gra{4}-b_1a_2b_3\gra{7}\\
&+b_1a_3b_1\gra{6}-b_1a_3b_2\gra{10}-b_1a_3b_3\gra{14}\\
&+b_2a_1b_1\gra{1}+b_2a_1b_2\delta\gra{1}+b_2a_1b_3 0\\
&+b_2a_2b_1\gra{5}+b_2a_2b_2\gra{1}+b_2a_2b_3\gra{11}\\
&+b_2a_3b_1\gra{11}+b_2a_3b_2 0+b_2a_3b_3(-\gra{5}+\frac{1}{\delta}\gra{1})\\
&-b_3a_1b_1\gra{9}+b_3a_1b_2 0+b_3a_1b_3(-\gra{3}+\frac{1}{\delta}\gra{1})\\
&+b_3a_2b_1\gra{8}+b_3a_2b_2\gra{10}-b_3a_2b_3\gra{12}\\
&+b_3a_3b_1\gra{13}+b_3a_3b_2(\gra{4}-\frac{1}{\delta}\gra{1})-b_3a_3b_3\gra{16}.
\end{align*}
If $\dim(\mathscr{C}_3)=15$, then the 15 diagrams in the above 16 diagrams excluding $\gra{15}$ forms a basis.
Replacing $\gra{15}$ by $\gra{16}$ and lower terms and comparing the coefficients of the basis, we have
$$ABA=BAB \iff$$
\begin{align}
&a_3b_3a_3i=b_3a_3b_3 & \gra{16} \label{hh15}\\
&a_3b_3a_1=b_1a_3b_3 & \gra{14} \label{hh14}\\
&a_1b_3a_3=b_3a_3b_1 & \gra{13} \label{hh13} \\
&-a_3b_2a_3=b_3a_2b_3 & \gra{12} \label{hh12}\\
&a_3b_3a_3\frac{1}{\delta^2}+a_1b_2a_3=b_2a_2b_3+b_2a_3b_1 & \gra{11} \label{hh11}\\
&a_3b_3a_3\frac{1}{\delta^2}-a_3b_2a_1=-b_1a_3b_2+b_3a_2b_2 & \gra{10} \label{hh10}\\
&a_3b_3a_3\frac{1}{\delta^2}-a_1b_3a_1=-b_1a_1b_3-b_3a_1b_1 & \gra{9} \label{hh9}\\
&a_3b_3a_3\frac{-i}{\delta^2}+a_1b_3a_2-a_3b_2a_2=b_3a_2b_1 & \gra{8} \label{hh8}\\
&a_3b_3a_3\frac{-i}{\delta^2}-a_2b_2a_3-a_2b_3a_1=-b_1a_2b_3 & \gra{7} \label{hh7}\\
&a_3b_3a_3\frac{-i}{\delta^2}+a_1b_1a_3+a_3b_1a_1=b_1a_3b_1 & \gra{6} \label{hh6}\\
&a_1b_2a_2-a_3b_3a_2=b_2a_2b_1-b_2a_3b_3 & \gra{5} \label{hh5}\\
&a_2b_2a_1+a_2b_3a_3=b_1a_2b_2+b_3a_3b_2 & \gra{4} \label{hh4}\\
&a_1b_1a_1+a_3b_1a_3=b_1a_1b_1-b_3a_1b_3 & \gra{3} \label{hh3}\\
&a_1b_1a_2+a_2b_1a_1+a_2b_1a_2\delta+a_2b_2a_2-\frac{1}{\delta}a_2b_3a_3-\frac{1}{\delta}a_3b_1a_3+\frac{1}{\delta}a_3b_3a_2=b_1a_2b_1 & \gra{2} \label{hh2}\\
&a_1b_2a_1=b_1a_1b_2+b_2a_1b_1+b_2a_1b_2\delta+b_2a_2b_2+\frac{1}{\delta}b_2a_3b_3+\frac{1}{\delta}b_3a_1b_3-\frac{1}{\delta}b_3a_3b_2 & \gra{1} \label{hh1}
\end{align}

Note that $a_3\neq 0$, $b_3\neq0$,
by equations (\ref{hh15}), (\ref{hh14}), (\ref{hh12}), we have
$$b_3=ia_3, a_1=b_1, a_2=b_2.$$
Then by equations (\ref{hh11}), (\ref{hh9}), we have
$$a_2^2=\frac{a_3^2}{\delta^2}, a_1^2=-\frac{a_3^2}{\delta^2}.$$
One can check that the other the equations hold under these conditions.
\end{proof}

\subsection{Proof of Proposition \ref{taubd}}\label{Appendix:taubd}

\begin{proof}
Recall that $s_n$ is the complement of the support of the basic construction ideal of $\mathscr{C}_{n}$.
Since $s_2\alpha=s_2\beta$, we have
\begin{align}\label{equtau1}
\varepsilon_n s_n &= \tau_n s_n.
\end{align}
Then
\begin{align*}
&\tilde{y}_{\lambda} \rho_{\lambda>\mu} (\tilde{y}_\mu\otimes 1)\tau_{n+1}\\
=&\tilde{y}_{\lambda} \rho_{\lambda>\mu} \tau_{n+1}(\tilde{y}_\mu\otimes 1) &&\text{by Proposition \ref{flat}},\\
=&\tilde{y}_{\lambda} \rho_{\lambda>\mu} \tau_{n+1}s_{n+1}(\tilde{y}_\mu\otimes 1) && \text{since } \tilde{y}=\tilde{y}s_{n+1},\\
=&\tilde{y}_{\lambda} \rho_{\lambda>\mu} \varepsilon_{n+1}s_{n+1}(\tilde{y}_\mu\otimes 1) && \text{by Equation (\ref{equtau1})}, \\
=&b_{\lambda-\mu}\tilde{y}_{\lambda} \rho_{\lambda>\mu} (\tilde{y}_\mu\otimes 1) && \text{by Proposition \ref{taub}}.\\
\end{align*}

Note that $(\tilde{y}_{\lambda}\otimes \cup) (\rho_{\lambda<\mu} \otimes 1) (\tilde{y}_\mu \otimes 1)\tau_{n+1}=(\tilde{y}_{\lambda}\otimes \cup) (\rho_{\lambda<\mu} \otimes 1) \tau_{n+1} (\tilde{y}_\mu \otimes 1)$, which is an intertwiner from $\lambda$ to $\mu \otimes 1$. Moreover, the intertwiner space in $\mathscr{C}_{\bullet}$ is one dimensional. So Equation (\ref{eqd}) holds for some coefficient. Furthermore, the coefficient $-b_{\mu-\lambda}$ is determined by computing the inner product as follows.

Take $V=(\tilde{y}_{\lambda}\otimes 1) \rho_{\lambda<\mu} \tilde{y}_\mu$ and $W=\tilde{y}_\mu \rho_{\mu>\lambda} (\tilde{y}_{\lambda}\otimes 1)$. Then
\begin{align*}
&tr_{n+1}\left((\tilde{y}_\mu \otimes 1) (\rho_{\mu>\lambda} \otimes 1)(\tilde{y}_{\lambda}\otimes \cap)(\tilde{y}_{\lambda}\otimes \cup) (\rho_{\lambda<\mu} \otimes 1) (\tilde{y}_\mu \otimes 1)\tau_{n+1} \right)\\
=&tr_{n+1}(Wh_{n}V\tau_{n+1})\\
=&\grd{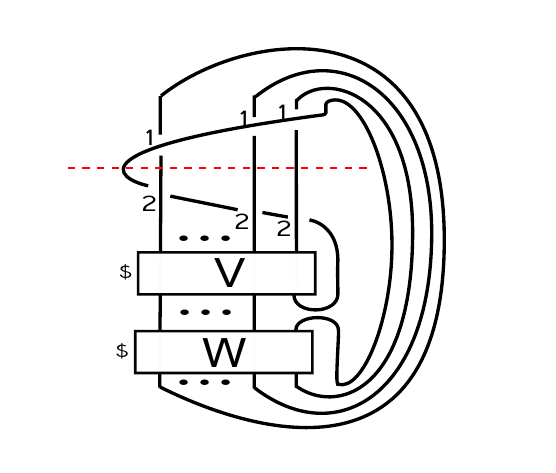}\\
=&\grd{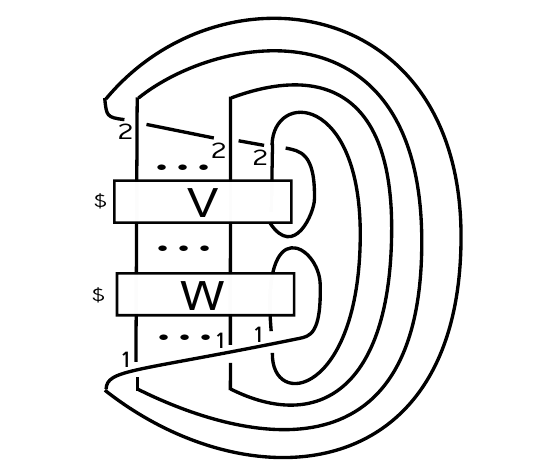} \text{by isotopy,}\\
=&\grd{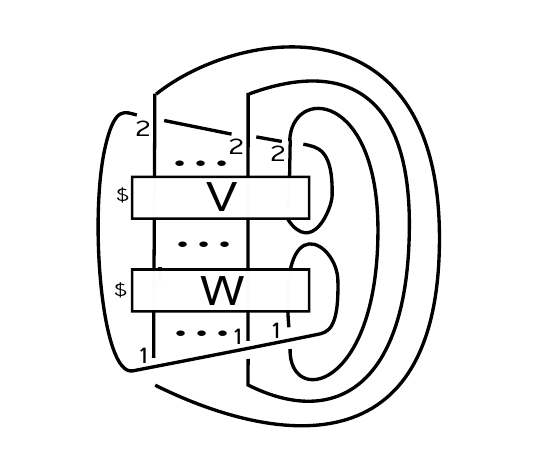} && \text{by sphericality,}\\
=&\grd{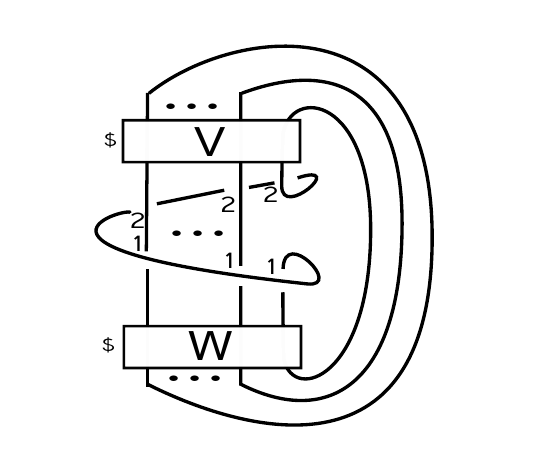} && \text{by Proposition \ref{flat},}\\
=&-\grd{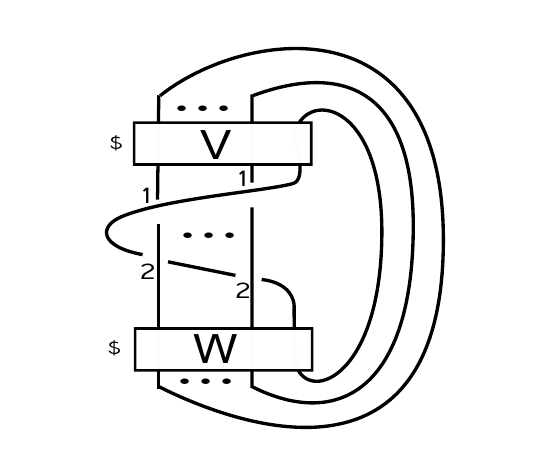} && \text{by Proposition \ref{braidfourier},} \\
=&-b_{\mu-\lambda}tr_{n}(WV) &&\text{by Equation (\ref{eqb}),} \\
=&-b_{\mu-\lambda}tr_{n+1}(Wh_{n}V)\\
=&-b_{\mu-\lambda}tr_{n+1}\left((\tilde{y}_\mu \otimes 1) (\rho_{\mu>\lambda} \otimes 1)(\tilde{y}_{\lambda}\otimes \cap)(\tilde{y}_{\lambda}\otimes \cup) (\rho_{\lambda<\mu} \otimes 1) (\tilde{y}_\mu \otimes 1) \right).\\
\end{align*}

\end{proof}

\subsection{Proof of Lemma \ref{traceZ}}\label{Appendix:traceZ}

\begin{proof}
  By Wenzl's formula (\ref{wenzlformula}), we have
\begin{align*}
\tilde{y}_\mu \otimes 1 =\sum_{\lambda<\mu,\lambda>\mu} p_\lambda,
\end{align*}
where
$$
p_\lambda=\left\{
\begin{array}{ll}
\frac{<\lambda>}{<\mu>} (\tilde{y}_\mu\otimes 1)(\rho_{\mu>\lambda} \otimes 1) (\tilde{y}_{\lambda}\otimes \cap) (\tilde{y}_{\lambda}\otimes \cup) (\rho_{\lambda<\mu} \otimes 1) (\tilde{y}_\mu \otimes 1), & \lambda<\mu;\\
(\tilde{y}_\mu\otimes 1) \rho_{\mu<\lambda} \tilde{y}_{\lambda} \rho_{\lambda>\mu} (\tilde{y}_\mu\otimes 1), &\lambda>\mu.
\end{array}
\right.
$$
Then $p_{\lambda}$ is an idempotent in $\mathscr{C}_{n+1}$ with trace $<\lambda>$. Moreover, by Proposition \ref{taubd},
$$
\tau_{n+1}p_\lambda=\left\{
\begin{array}{ll}
-b_{\mu-\lambda}p_\lambda & \lambda<\mu;\\
b_{\lambda-\mu}p_\lambda &\lambda>\mu.
\end{array}
\right.
$$

By definitions, we have
\begin{align*}
Z(\mu,u)\tilde{y}_{\mu}=&Z_{n+1}(u)\tilde{y}_{\mu} \\
=&\sum_{i\geq0}Z_{n+1}^{(i)}\tilde{y}_{\mu}u^{-i}\\
=&\sum_{i\geq0}\Phi_{n+1}(\tau_{n+1}^i)\tilde{y}_{\mu}u^{-i}\\
=&\sum_{i\geq0}\Phi_{n+1}(\tau_{n+1}^i(\tilde{y}_{\mu}\otimes 1))u^{-i}\\
=&\sum_{i\geq0}\Phi_{n+1}(\tau_{n+1}^i(\sum_{\lambda<\mu,\lambda>\mu} p_\lambda))u^{-i}\\
=&\sum_{i\geq0}\Phi_{n+1}(\sum_{\lambda<\mu} (-b_{\mu-\lambda})^ip_\lambda+\sum_{\lambda>\mu} b_{\lambda-\mu}^ip_\lambda)u^{-i}\\
=&\sum_{i\geq0}\left(\sum_{\lambda<\mu} (-b_{\mu-\lambda})^i\frac{<\lambda>}{<\mu>}\tilde{y}_{\mu}+\sum_{\lambda>\mu} b_{\lambda-\mu}^i\frac{<\lambda>}{<\mu>}\tilde{y}_{\mu}\right)u^{-i}\\
=&\left(\sum_{\lambda<\mu} \frac{u}{u+b_{\mu-\lambda}}\frac{<\lambda>}{<\mu>}+\sum_{\lambda>\mu} \frac{u}{u-b_{\lambda-\mu}}\frac{<\lambda>}{<\mu>}\right)\tilde{y}_{\mu} && \text{Fubini's theorem}\\
\end{align*}
Therefore
$$\frac{Z(\mu,u)}{u}=\sum_{\lambda<\mu} \frac{1}{u+b_{\mu-\lambda}}\frac{<\lambda>}{<\mu>}+\sum_{\lambda>\mu} \frac{1}{u-b_{\lambda-\mu}}\frac{<\lambda>}{<\mu>}.$$
Recall that $b_{c}=q^{2\text{cn}(c)}$, so $\{-b_{\mu-\lambda}\}_{\lambda<\mu}$ and $\{b_{\lambda-\mu}\}_{\lambda>\mu}$ are distinct. Therefore
$$\frac{<\lambda>}{<\mu>}=\text{res}_{u=b_{\lambda-\mu}}\frac{Z(\mu,u)}{u}, \text{ for } \lambda>\mu$$
and
$$\frac{<\lambda>}{<\mu>}=\text{res}_{u=-b_{\mu-\lambda}}\frac{Z(\mu,u)}{u}, \text{ for } \lambda<\mu.$$
\end{proof}

\subsection{Proof of Lemma \ref{ZZ}}\label{Appendix:ZZ}
Before proving Lemma \ref{ZZ}, let us prove some technical results.

\begin{lemma}
For $n\geq1$, we have
\begin{align}
\beta_{n}^{-1}\tau_{n+1}&=\tau_{n}\alpha_{n} \label{non1};\\
\tau_{n+1}\alpha_{n}^{-1}&=\beta_{n}\tau_{n} \label{non2};\\
h_{n}\tau_{n+1}&=-h_{n} \tau_{n} \label{non3};\\
\tau_{n+1}h_{n}&=-\tau_{n} h_{n}\label{non4};\\
\tau_{n}\tau_{n+1}&=\tau_{n+1}\tau_{n}\label{non5};\\
h_{n}(u-\tau_{n+1})^{-1}&=h_{n}(u+\tau_{n})^{-1} \label{non6};\\
(u-\tau_{n+1})^{-1}h_{n}&=(u+\tau_{n})^{-1}h_{n} \label{non7};\\
\beta^{-1}-\alpha&=-(q-q^{-1})\gra{21}+i(q-q^{-1})\gra{22}\label{non8};\\
\beta-\alpha^{-1}&=(q-q^{-1})\gra{21}+i(q-q^{-1})\gra{22}\label{non9};\\
\Phi_{n+1}(\beta_{n}\frac{1}{u-\tau_{n}}\beta_{n}^{-1})&=\frac{Z_{n}}{u}\label{non10}.
\end{align}
\end{lemma}

Recall that we identify an $n$-box as an $(n+1)$-box by adding a through string to the right.

\begin{proof}
Equation (\ref{non1}) follows from
$$\grc{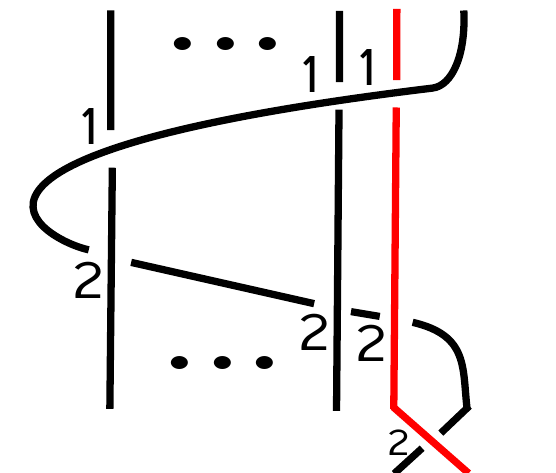}=\grc{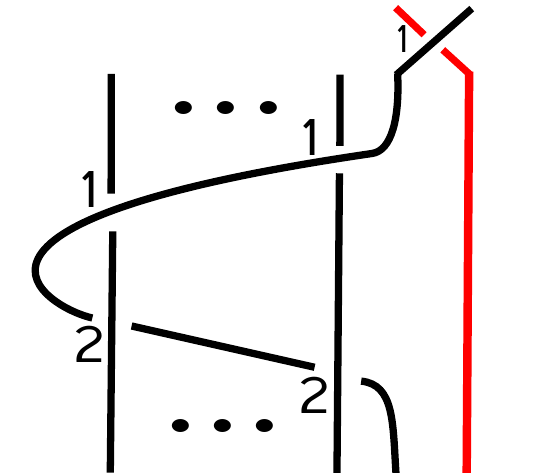}.$$
Equation (\ref{non2}) follows from
$$\grc{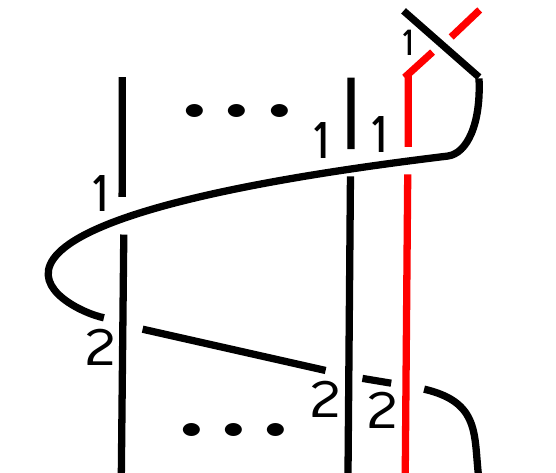}=\grc{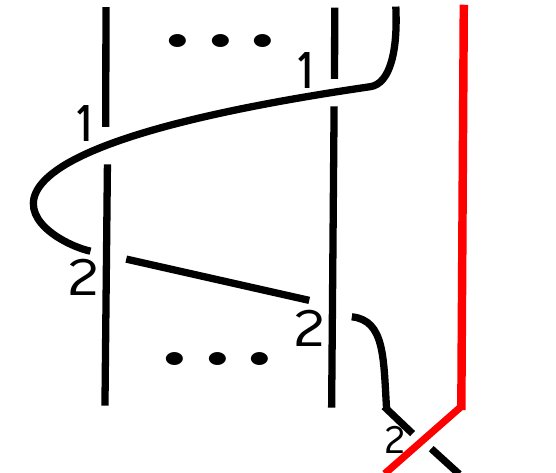}.$$
Equation (\ref{non3}) follows from
\begin{align*}
\grc{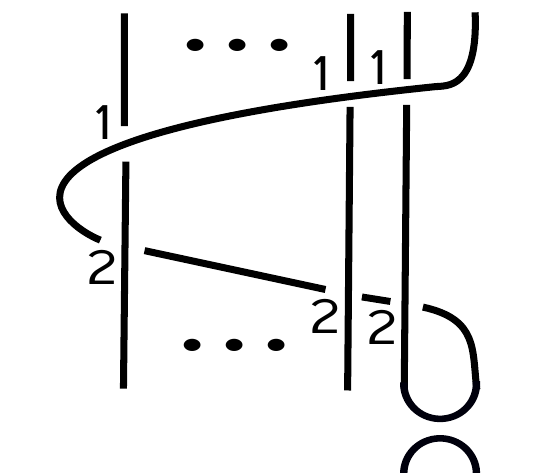}&=ir\grc{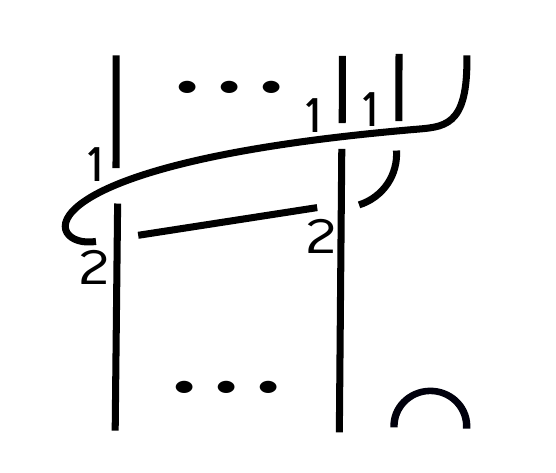} \\
&=ir\grc{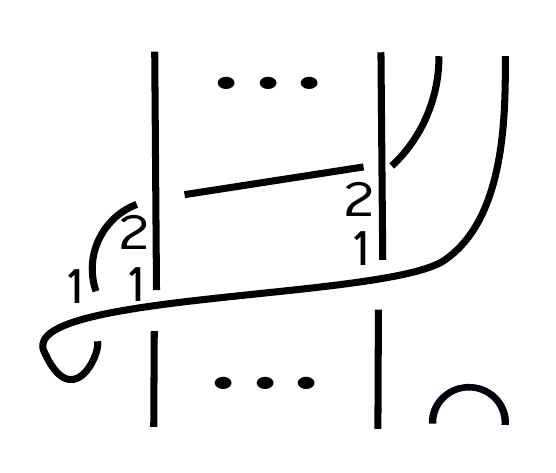} &&\text{by Proposition \ref{flat},}\\
&=-\grc{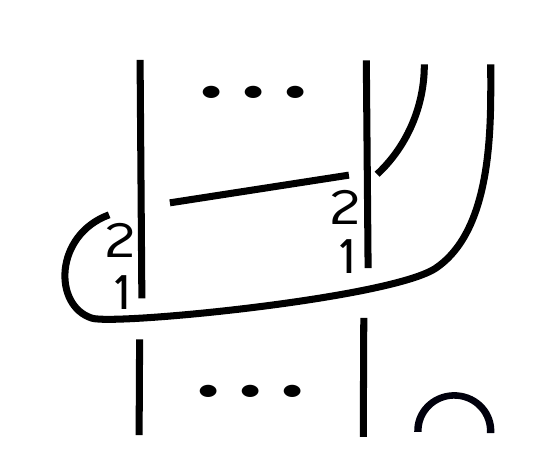} \\
&=-\grc{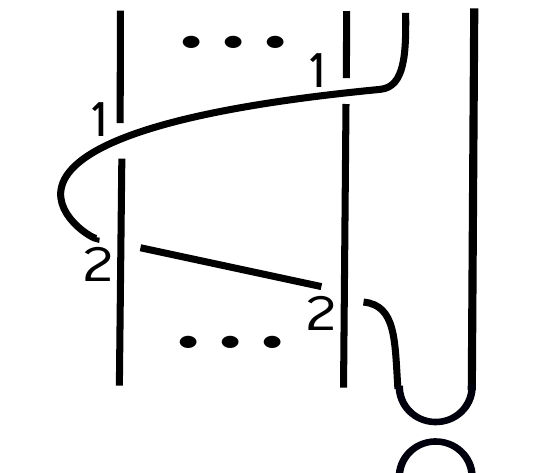} &&\text{by Proposition \ref{braidfourier}.}\\
\end{align*}

Similarly we have equation (\ref{non4}).

Equation (\ref{non5}) follows from Proposition \ref{flat},

By Equations (\ref{non3}), (\ref{non5}), we have
$h_{n}\tau_{n+1}^k=h_{n}(-\tau_{n})^k$.
So Equation (\ref{non6}) holds.

Similarly by Equations (\ref{non4}), (\ref{non5}), Equation (\ref{non7}) holds.

Equation (\ref{non8}), (\ref{non9}) follow from the definitions.

By Proposition \ref{braidfourier}, we have
$$\grb{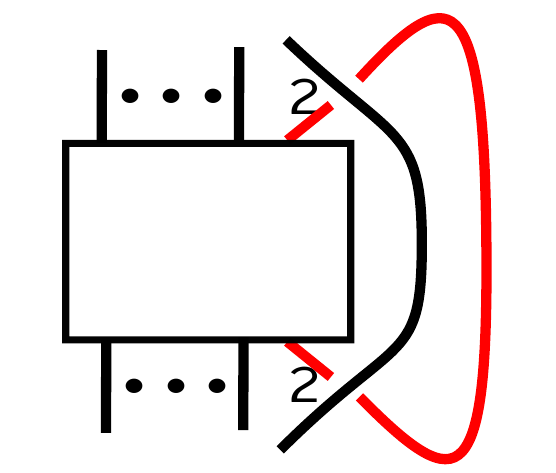}=\grb{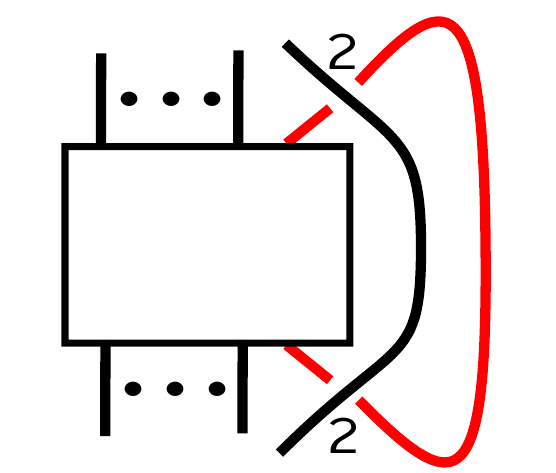}=\grb{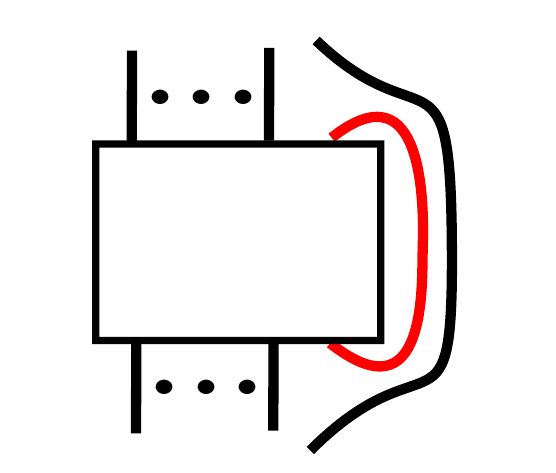}.$$
So by Equation (\ref{nonphi}),
$$\Phi_{n+1}(\beta_{n}\frac{1}{u-\tau_{n}}\beta_{n}^{-1})=\Phi_{n}(\frac{1}{u-\tau_{n}})=\frac{Z_{n}}{u}.$$
\end{proof}

\begin{proof}[\textbf{Proof of Lemma \ref{ZZ}}]
By Equation (\ref{non1}), we have
\begin{align*}
\beta_{n}^{-1}(u-\tau_{n+1})=(u-\tau_{n})\beta_{n}^{-1}+\tau_{n}(\beta_{n}^{-1}-\alpha_{n}).
\end{align*}
So
\begin{align}
\frac{1}{u-\tau_{n}}\beta_{n}^{-1}=\beta_{n}^{-1}\frac{1}{u-\tau_{n+1}}+\frac{\tau_{n}}{u-\tau_{n}}(\beta_{n}^{-1}-\alpha_{n})\frac{1}{u-\tau_{n+1}}.
\label{non11}
\end{align}
Therefore
$$\beta_{n}\frac{1}{u-\tau_{n}}\beta_{n}^{-1}=\frac{1}{u-\tau_{n+1}}+\beta_{n}\frac{\tau_{n}}{u-\tau_{n}}(\beta_{n}^{-1}-\alpha_{n})\frac{1}{u-\tau_{n+1}}.$$
Applying Equation (\ref{non8}), (\ref{non5}), (\ref{non6}) to the right side, we have
\begin{align}
\beta_{n}\frac{1}{u-\tau_{n}}\beta_{n}^{-1}
&=\frac{1}{u-\tau_{n+1}}-(q-q^{-1})\beta_{n}\frac{\tau_{n}}{u-\tau_{n}}\frac{1}{u-\tau_{n+1}}+i(q-q^{-1})\beta_{n}\frac{\tau_{n}}{u-\tau_{n}}h_{n}\frac{1}{u-\tau_{n+1}} \nonumber\\
&=\frac{1}{u-\tau_{n+1}}-(q-q^{-1})\beta_{n}\frac{1}{u-\tau_{n+1}}\frac{\tau_{n}}{u-\tau_{n}}+i(q-q^{-1})\beta_{n}\frac{\tau_{n}}{u-\tau_{n}}h_{n}\frac{1}{u+\tau_{n}} \label{eq1}.
\end{align}

By Equations (\ref{non11}), (\ref{non8}), (\ref{non6}),  we have
\begin{align}
\beta_{n}\frac{1}{u-\tau_{n+1}}&=(\beta_{n}-\beta_{n}^{-1})\frac{1}{u-\tau_{n+1}}+\beta_{n}^{-1}\frac{1}{u-\tau_{n+1}} \nonumber\\
&=(q-q^{-1})\frac{1}{u-\tau_{n+1}}+\frac{1}{u-\tau_{n}}\beta_{n}^{-1}-\frac{\tau_{n}}{u-\tau_{n}}(\beta_{n}^{-1}-\alpha_{n})\frac{1}{u-\tau_{n+1}} \nonumber\\
&=(q-q^{-1})\frac{1}{u-\tau_{n+1}}+\frac{1}{u-\tau_{n}}\beta_{n}^{-1} \nonumber\\
&+(q-q^{-1})\frac{\tau_{n}}{u-\tau_{n}}\frac{1}{u-\tau_{n+1}}-i(q-q^{-1})\frac{\tau_{n}}{u-\tau_{n}}h_{n}\frac{1}{u+\tau_{n}} \label{eq2}.
\end{align}

By Equation (\ref{non2}), we have
$$(u-\tau_{n+1})\beta_{n}=\beta_{n}(u-\tau_{n})-\tau_{n+1}(\beta_{n}-\alpha_{n}^{-1}).$$
So
$$\beta_{n}\frac{1}{u-\tau_{n}}=\frac{1}{u-\tau_{n+1}}\beta_{n}-\frac{\tau_{n+1}}{u-\tau_{n+1}}(\beta_{n}-\alpha_{n}^{-1})\frac{1}{u-\tau_{n}}.$$
Therefore
$$\beta_{n}\frac{\tau_{n}}{u-\tau_{n}}=\frac{\tau_{n+1}}{u-\tau_{n+1}}\beta_{n}-\frac{u\tau_{n+1}}{u-\tau_{n+1}}(\beta_{n}-\alpha_{n}^{-1})\frac{1}{u-\tau_{n}}.$$
Note that $\beta_{n}h_{n}=-q^{-1}h_{n}$, so
$$\beta_{n}\frac{\tau_{n}}{u-\tau_{n}}h_{n}=-q^{-1}\frac{\tau_{n+1}}{u-\tau_{n+1}}h_{n}-\frac{u\tau_{n+1}}{u-\tau_{n+1}}(\beta_{n}-\alpha_{n}^{-1})\frac{1}{u-\tau_{n}}h_{n}.$$
By Equations (\ref{non9}), (\ref{non5}), (\ref{non4}), (\ref{non7}), (\ref{nonphi}), we have
\begin{align}
\beta_{n}\frac{\tau_{n}}{u-\tau_{n}}h_{n}
=q^{-1}\frac{\tau_{n}}{u+\tau_{n}}h_{n}+(q-q^{-1})\frac{u\tau_{n}}{(u-\tau_{n})(u+\tau_{n})}h_{n}+i(q-q^{-1})\frac{u\tau_{n}}{u+\tau_{n}}\frac{Z_{n}}{u}h_{n} \label{eq3}.
\end{align}

Applying Equation (\ref{eq2}), (\ref{eq3}) to the right side of (\ref{eq1}), and applying $\Phi_{n+1}$ on both sides, we have
\begin{align*}
&\Phi_{n+1}(\beta_{n}\frac{1}{u-\tau_{n}}\beta_{n}^{-1})\\
=&\Phi_{n+1}(\frac{1}{u-\tau_{n+1}})\\
-&(q-q^{-1})^2\Phi_{n+1}(\frac{1}{u-\tau_{n+1}})\frac{\tau_{n}}{u-\tau_{n}}-(q-q^{-1})\frac{1}{u-\tau_{n}}\Phi_{n+1}(\beta_{n}^{-1})\frac{\tau_{n}}{u-\tau_{n}}\\
-&(q-q^{-1})^2\frac{\tau_{n}}{u-\tau_{n}}\Phi_{n+1}(\frac{1}{u-\tau_{n+1}})\frac{\tau_{n}}{u-\tau_{n}}+i(q-q^{-1})^2\frac{\tau_{n}}{u-\tau_{n}}\Phi_{n+1}(h_{n})\frac{1}{u+\tau_{n}}\frac{\tau_{n}}{u-\tau_{n}}\\
+&i(q-q^{-1})q^{-1}\frac{\tau_{n}}{u+\tau_{n}}\Phi_{n+1}(h_{n})\frac{1}{u+\tau_{n}}+i(q-q^{-1})^2\frac{u\tau_{n}}{(u-\tau_{n})(u+\tau_{n})}\Phi_{n+1}(h_{n})\frac{1}{u+\tau_{n}}\\
-&(q-q^{-1})^2\frac{u\tau_{n}}{u+\tau_{n}}\frac{Z_{n}}{u}\Phi_{n+1}(h_{n})\frac{1}{u+\tau_{n}}.
\end{align*}
By Proposition \ref{flat}, $\tau_{n}$, $Z_{n}$, $Z_{n+1}$ commutes with each other. By Equations (\ref{non10}), (\ref{nonphi}), we have
\begin{align*}
\frac{Z_{n}}{u}&=\frac{Z_{n+1}}{u}-(q-q^{-1})^2\frac{Z_{n+1}}{u}\frac{\tau_{n}}{u-\tau_{n}}-iq^{-1}(q-q^{-1})\frac{\tau_{n}}{(u-\tau_{n})^2}\\
&-(q-q^{-1})^2\frac{Z_{n+1}}{u}\frac{\tau_{n}^2}{(u-\tau_{n})^2}+i(q-q^{-1})^2\frac{\tau_{n}^2}{(u-\tau_{n})^2(u+\tau_{n})}\\
&+i(q-q^{-1})q^{-1}\frac{\tau_{n}}{(u+\tau_{n})^2}+i(q-q^{-1})^2\frac{u\tau_{n}}{(u-\tau_{n})(u+\tau_{n})^2}\\
&-(q-q^{-1})^2\frac{Z_{n}}{u}\frac{u\tau_{n}}{(u+\tau_{n})^2}.
\end{align*}
Recall that $\displaystyle \delta=\frac{i(q+q^{-1})}{q-q^{-1}}$.
The above equation can be simplified as
$$\frac{Z_{n}-\frac{\delta}{2}}{u}\left(1+(q-q^{-1})^2\frac{u\tau_{n}}{(u+\tau_{n})^2}\right)
=\frac{Z_{n+1}-\frac{\delta}{2}}{u}\left(1-(q-q^{-1})^2\frac{u\tau_{n}}{(u-\tau_{n})^2}\right).$$
Therefore
$$Z_{n+1}-\frac{\delta}{2}=(Z_{n}-\frac{\delta}{2})\frac{(u-\tau_{n})^2(u+q^{-2}\tau_{n})(u+q^2\tau_{n})}{(u+\tau_{n})^2(u-q^{-2}\tau_{n})(u-q^2\tau_{n})}.$$
\end{proof}

\section{Planar algebras modulo grading operators}\label{Appendix:modulo grading operators}
Suppose $\mathscr{S}_{\bullet}$ is an unshaded finite depth subfactor planar algebra, $g\in\mathscr{S}_{N}$ is a grading operator with periodicity $P$, and $(g,u)$ is a commutative grading. Let us construct the $\mathbb{Z}_{N}$ graded unitary fusion category $\mathscr{S}/g$.

Let us define the grading of $x\in \mathscr{S}_{m}$ by $m$, denoted by $|x|$.
Let $Ver$ be the set of vertices of the principal graph of $\mathscr{S}_{\bullet}$. For each vertex $\lambda\in Ver$, the distance between $\lambda$ and the distinguished vertex $\emptyset$ is denoted by $|\lambda|$. Take a representative $\tilde{y}_{\lambda}$, which is a minimal projection in $\mathscr{S}_{|\lambda|}$. For convenience, we can choose representatives such that the contragredient of $\tilde{y}_{\lambda}$ is $\tilde{y}_{\Omega(\lambda)}$ for a $\mathbb{Z}_2$ action $\Omega$ on $Ver$.

Note that the equivalence classes of minimal projections of $\mathscr{S}_{m}$ have representatives given by minimal projections $\tilde{y}_{\lambda} \otimes e^{\otimes k}$, for all $\lambda\in Ver$, $k\geq 0$, such that $|\lambda|+2k=m$.
Take
$$Ver_g=\{\tilde{y}_{\lambda}\otimes e^{\otimes k}, \lambda\in Ver, ~k\geq0 ~|~  \tilde{y}_\lambda e^{\otimes k} \nsim \tilde{y}_\mu e^{\otimes l}\otimes g \text{ in } \mathscr{S}_{|\lambda|+2k}, ~\forall~ \mu\in Ver, l\geq0\}.$$

Let us consider $\mathscr{S}_{\bullet}$ as a $\mathbb{N}\cup\{0\}$ graded semisimple strict monoidal category with simple objects $y_{\lambda} \otimes e^{\otimes k}\otimes g^{\otimes l}$ graded by $|\lambda|+2k+Nl$, for all $\lambda\in Ver$, $k<\frac{NP}{2}$, $l\geq0$.

Recall that $(g,u)$ is a commutative grading.
We simplify Equations (\ref{braiding}) and (\ref{commuting}) by the following notations,
$$\graa{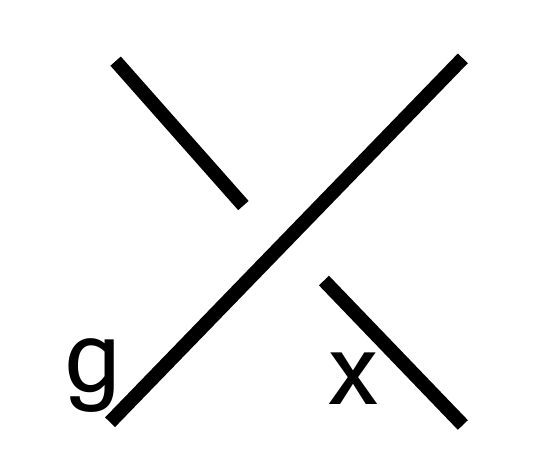}=\graa{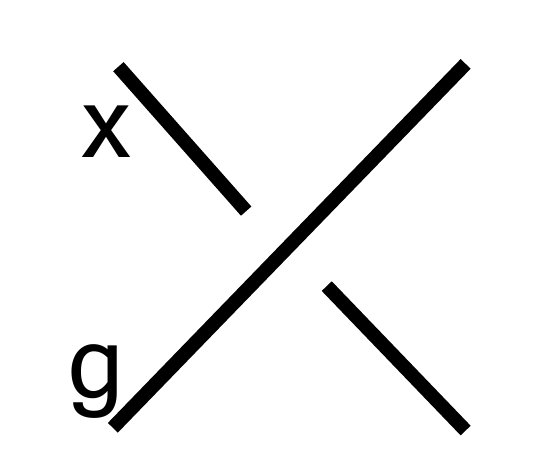}, \quad \graa{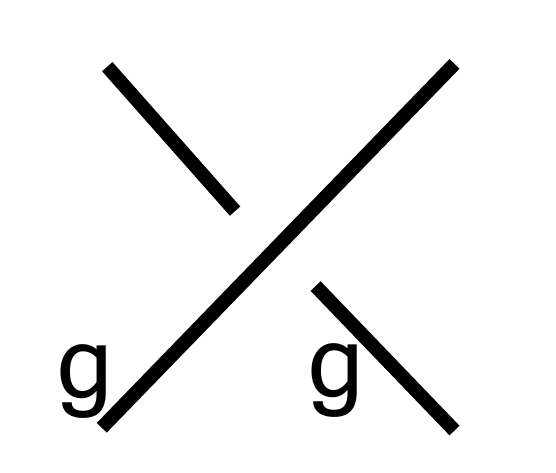}=\graa{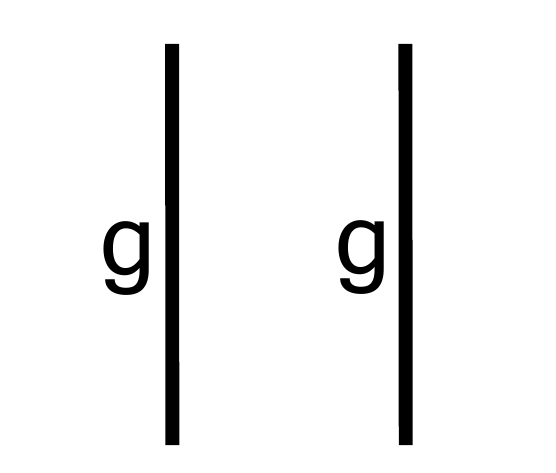}.$$
For objects $Y_k$,$1\leq k\leq 3$, let us define $\iota_l:~\hom(Y_1\otimes Y_2, Y_3)\rightarrow \hom((Y_1\otimes g)\otimes Y_2, Y_3\otimes g)$ as
$$\iota_l(\grb{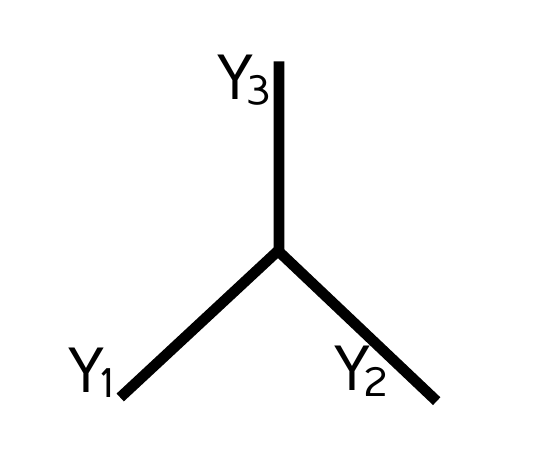})=\grb{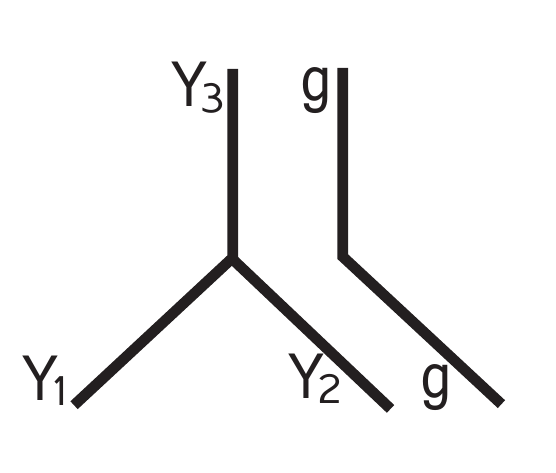},$$
and $\iota_r:~\hom(Y_1\otimes Y_2, Y_3)\rightarrow \hom(Y_1\otimes (Y_2\otimes g), Y_3\otimes g)$ as
$$\iota_r(\grb{gradehom})=\grb{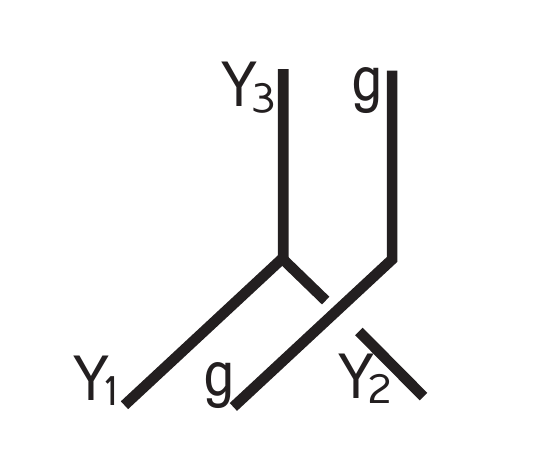}.$$
Then $\iota_l\iota_r=\iota_r\iota_l$.
Recall that $g$ is a trace one projection, thus
$$\graa{grade8}=\graa{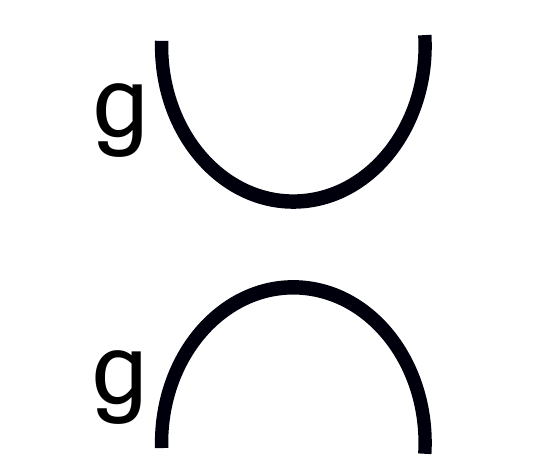}.$$
By this relation, we have that both $\iota_l$ and $\iota_r$ are invertible by capping off the $g$ string.

We define a relation $\sim$ for objects and morphisms of the $\mathbb{N}\cup\{0\}$ graded tensor category $\mathscr{S}/g$ as follows. For an object $Y$ and a morphism $x\in\hom(Y_1\otimes Y_2, Y_3)$,
\begin{align*}
&Y\sim Y\otimes g^l &&\text{for any object $Y$;}\\
&\iota_l^{k_1}\iota_r^{k_2}(x)\sim\iota_r^{k_3}\iota_r^{k_4}(x), &&\text{for any morphism $x$ and $k_j\geq0$, $1\leq 4$.}
\end{align*}
Since both $\iota_l$ and $\iota_r$ are invertible, we have that $\sim$ is an equivalence relation.
Modulo the equivalence relation, by Equations (\ref{braiding}) and (\ref{commuting}), the $6j$-symbol is well defined on the quotient. Moreover, the pentagon condition holds. Therefore the quotient is a $\mathbb{Z}_N$ graded semisimple strict monoidal category, denoted by $\mathscr{S}/g$.

The simple objects of $\mathscr{S}/g$ are given by $Ver_g$ and the simple object $\tilde{y}_{\lambda} \otimes e^{\otimes k}$ is graded by $|\lambda|+2k=m$ modulo $N$.

Since $g$ has periodicity $P$, we have a non-zero morphism $v$ from $g^{\otimes N+1}$ to $e^{\otimes \frac{NP}{2}}$.
For a simple object $\tilde{y}_{\lambda} \otimes e^{\otimes k}$, we construct the dual object as $\tilde{y}_{\Omega(\lambda)} \otimes e^{\otimes l}$, such that $2|\lambda|+2k+2l=N(N+1)$ or $2N(N+1)$. The evaluation map and the coevaluation map (up to a scalar) are given by
\begin{align*}
\grd{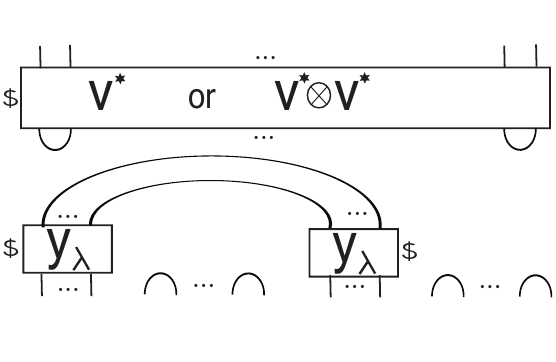} \quad\quad\quad\quad \grd{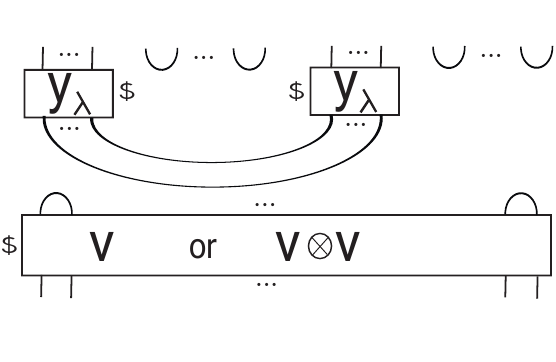}.
\end{align*}
Thus $\mathscr{S}/g$ is a fusion category. The pivotal property and spherical property follows from the corresponding properties of planar algebras. Since $\mathscr{S}_{\bullet}$ has finite depth and $g$ has finite periodicity, $Ver_g$ is a finite set. Since $\mathscr{S}_{\bullet}$ is a subfactor planar algebra, $\mathscr{S}/g$ is a unitary fusion category.

\begin{remark}
The grading operator $g$ action on $Ver$ as a $\mathbb{Z}_P$ automorphism. Take a representative $[\lambda]_g$ for each orbit. It is easy to see that $[\lambda]_g \otimes e^{\otimes k}$, $0\leq k< \frac{NP}{2}$
also give representatives for $Ver_g$.
\end{remark}

  \bibliography{bibliography}

\newcommand{\etalchar}[1]{$^{#1}$}
\providecommand{\bysame}{\leavevmode\hbox to3em{\hrulefill}\thinspace}
\providecommand{\MR}{\relax\ifhmode\unskip\space\fi MR }
\providecommand{\MRhref}[2]{%
  \href{http://www.ams.org/mathscinet-getitem?mr=#1}{#2}
}
\providecommand{\href}[2]{#2}
\begin{thebibliography}{FYH{\etalchar{+}}85}

\bibitem[Ais97]{Ais97}
AK~Aiston, \emph{A skein theoretic proof of the hook formula for quantum
  dimension}, available at q-alg/9711019 (1997).

\bibitem[Ale23]{Ale23}
James~Wadell Alexander, \emph{A lemma on systems of knotted curves},
  Proceedings of the National Academy of Sciences of the United States of
  America \textbf{9} (1923), no.~3, 93.

\bibitem[AM98]{Ais98}
AK~Aiston and HR~Morton, \emph{Idempotents of hecke algebras of type a},
  Journal of Knot Theory and Its Ramifications \textbf{7} (1998), no.~04,
  463--487.

\bibitem[Ati88]{Ati88}
M.~F. Atiyah, \emph{Topological quantum field theory}, Publications
  Math{\'e}matiques de l'IH{\'E}S \textbf{68} (1988), 175--186.

\bibitem[Bax07]{Bax07}
R.~J. Baxter, \emph{Exactly solved models in statistical mechanics}, Courier
  Corporation, 2007.

\bibitem[BB01]{BelBla}
A.~Beliakova and C.~Blanchet, \emph{Skein construction of idempotents in
  birman-murakami-wenzl algebras}, Math. Ann. \textbf{321} (2001), 347¨C373.

\bibitem[BJ97]{BisJonFC}
D.~Bisch and V.~F.~R. Jones, \emph{Algebras associated to intermediate
  subfactors}, Invent. Math. \textbf{128} (1997), 89--157.

\bibitem[BJ00]{BisJon00}
\bysame, \emph{Singly generated planar algebras of small dimension}, Duke Math.
  J. \textbf{101(1)} (2000), 41--76.

\bibitem[BJ03]{BisJon03}
\bysame, \emph{Singly generated planar algebras of small dimension, part {II}},
  Adv. Math. \textbf{175} (2003), 297--318.

\bibitem[BJL]{BJL}
D.~Bisch, V.~F.~R. Jones, and Z.~Liu, \emph{Singly generated planar algebras of
  small dimension, part {III}}, arXiv:1410.2876 To appear Trans. AMS.

\bibitem[Bla00]{Bla00}
C.~Blanchet, \emph{Hecke algebras, modular categories and 3-manifolds quantum
  invariants}, Topology \textbf{39} (2000), no.~1, 193--223.

\bibitem[BMPS12]{BMPS}
S.~Bigelow, S.~Morrison, E.~Peters, and N.~Snyder, \emph{Constructing the
  extended {H}aagerup planar algebra}, Acta Math. (2012), 29--82.

\bibitem[BN13]{BurNat13}
S.~Burciu and S.~Natale, \emph{Fusion rules of equivariantizations of fusion
  categories}, Journal of Mathematical Physics \textbf{54} (2013), no.~1,
  013--511.

\bibitem[Bra37]{Bra37}
R.~Brauer, \emph{On algebras which are connected with the semisimple continuous
  groups}, Ann. Math. (1937), 857--872.

\bibitem[BW89]{BirWen}
J.~Birman and H.~Wenzl, \emph{Braids, link polynomials and a new algebra},
  Trans. AMS \textbf{313(1)} (1989), 249--273.

\bibitem[Dri86]{Dri86}
V.~G. Drinfeld, \emph{Quantum groups}, Zapiski Nauchnykh Seminarov POMI
  \textbf{155} (1986), 18--49.

\bibitem[EK98]{EvaKaw98}
DE~Evans and Y.~Kawahigashi, \emph{Quantum symmetries on operator algebras},
  vol. 147, Clarendon Press Oxford, 1998.

\bibitem[ENO05]{ENO}
P.~Etingof, D.~Nikshych, and V.~Ostrik, \emph{On fusion categories}, Annals of
  Mathematics (2005), 581--642.

\bibitem[FYH{\etalchar{+}}85]{Homfly}
P.~Freyd, D.~Yetter, J.~Hoste, R.~Lickorish, K.~Millett, and A.~Ocneanu,
  \emph{A new polynomial invariant of knots and links}, Bulletin of the AMS
  \textbf{12} (1985), no.~2, 239--246.

\bibitem[GdlHJ89]{GHJ}
F.~Goodman, P.~de~la Harpe, and V.F.R. Jones, \emph{Coxeter graphs and towers
  of algebras}, vol.~14, Springer-Verlag, MSRI publications, 1989.

\bibitem[GJS10]{GJS}
A.~Guionnet, V.~F.~R. Jones, and D.~Shlyakhtenko, \emph{Random matrices, free
  probability, planar algebras and subfactors}, Quanta of maths:
  Non-commutative Geometry Conference in Honor of Alain Connes, Clay Math.
  Proc. \textbf{11} (2010), 201--240.

\bibitem[Haa94]{Haa94}
U.~Haagerup, \emph{Principal graphs of subfactors in the index range $4 <
  [{M}\colon {N}] < 3+\sqrt{2}$}, Subfactors (Kyuzeso,1993), World Sci. Publ.,
  River Edge, NJ, 1994, pp.~1--38.

\bibitem[Izu93]{Izu93}
M.~Izumi, \emph{Subalgebras of infinite {C}*-algebras with finite {W}atatani
  indices {I}. {C}untz algebras}, Comm. Math. Phys. \textbf{155} (1993), no.~1,
  157--182.

\bibitem[Jim85]{Jim85}
M.~Jimbo, \emph{A q-analogue of {U(gl(N+1))} and the {Y}ang-{B}axter equation},
  Letters in Mathematical Physics \textbf{10} (1985), no.~1, 63--69.

\bibitem[JLW16]{JLW16}
C.~Jiang, Z.~Liu, and J.~Wu, \emph{Noncommutative uncertainty principles}, J.
  Funct. Anal. \textbf{270} (2016), no.~1, 264--311.

\bibitem[JMS14]{JMS}
V.~F.~R. Jones, S.~Morrison, and N.~Snyder, \emph{The classification of
  subfactors of index at most 5}, Bulletin of the AMS \textbf{51} (2014),
  no.~2, 277--327.

\bibitem[Jon83]{Jon83}
V.~F.~R. Jones, \emph{Index for subfactors}, Invent. Math. \textbf{72} (1983),
  1--25.

\bibitem[Jon85]{Jon85}
V.~F.~R. Jones, \emph{A polynomial invariant for knots via von neumann
  algebras}, Mathematical Sciences Research Institute, 1985.

\bibitem[Jon87]{Jon87}
V.~F.~R. Jones, \emph{Hecke algebra representations of braid groups and link
  polynomials}, Ann. of Math \textbf{126} (1987), no.~2, 335--388.

\bibitem[Jon98]{JonPA}
\bysame, \emph{Planar algebras, {I}}, New Zealand J. Math. arXiv:9909027
  (1998).

\bibitem[Jon12]{Jon12}
\bysame, \emph{Quadratic tangles in planar algebras}, Duke Math. J.
  \textbf{161} (2012), no.~12, 2257--2295.

\bibitem[JP11]{JonPen}
V.~F.~R. Jones and D.~Penneys, \emph{The embedding theorem for finite depth
  subfactor planar algebras}, Quantum Topol. \textbf{2} (2011), 301--337.

\bibitem[Kau90]{Kau90}
L.H. Kauffman, \emph{An invariant of regular isotopy}, Trans. AMS \textbf{318}
  (1990), 417--471.

\bibitem[Liu]{Liuex}
Z.~Liu, \emph{Exchange relation planar algebras of small rank},
  arXiv:1308.5656v2 to appear Trans. AMS.

\bibitem[LMP13]{LMP13}
Z.~Liu, S.~Morrison, and D.~Penneys, \emph{1-supertransitive subfactors with
  index at most 6 1/5}, Comm. Math. Phys. \textbf{334} (2013), no.~2, 889--922.

\bibitem[LP]{LiuPen}
Z.~Liu and D.~Penneys, \emph{The generator conjecture for $3^{G}$ subfactor
  planar algebras}, http://arxiv.org/pdf/1507.04794.pdf To appear in the
  Proceedings in honor of Vaughan F. R. Jones' 60th birthday conferences.

\bibitem[MPS10]{MPSD2n}
S.~Morrison, E.~Peters, and N.~Snyder, \emph{Skein theory for the ${D}_{2n}$
  planar algebras}, Journal of Pure and Applied Algebra \textbf{214} (2010),
  117--139.

\bibitem[MPS15]{MPS15}
S.~Morrison, E.~Peters, and N.~Snyder, \emph{Categories generated by a
  trivalent vertex}, arXiv:1501.06869 (2015).

\bibitem[Mur87]{Mur87}
J.~Murakami, \emph{The kauffman polynomial of links and representation theory},
  Osaka J. Math. \textbf{24(4)} (1987), 745--758.

\bibitem[MW10]{MorWal}
S.~Morrison and K.~Walker, \emph{The graph planar algebra embedding theorem},
  2010, http://tqft.net/gpa.

\bibitem[Naz96]{Naz96}
M.~Nazarov, \emph{Young's orthogonal form for brauer's centralizer algebra},
  Journal of Algebra \textbf{182} (1996), no.~3, 664--693.

\bibitem[Ocn88]{Ocn88}
A.~Ocneanu, \emph{Quantized groups, string algebras and {G}alois theory for
  algebras}, Operator algebras and applications, Vol.\ 2, London Math. Soc.
  Lecture Note Ser., vol. 136, Cambridge Univ. Press, Cambridge, 1988,
  pp.~119--172.

\bibitem[Ocn00]{Ocn00}
Adrian Ocneanu, \emph{The classification of subgroups of quantum {SU(N)}}.

\bibitem[Ons44]{Ons44}
L.~Onsager, \emph{Crystal statistics. i. a two-dimensional model with an
  order-disorder transition}, Physical Review \textbf{65(3-4)} (1944), 177.

\bibitem[Ost03]{Ost03}
V.~Ostrik, \emph{Module categories, weak hopf algebras and modular invariants},
  Transformation Groups \textbf{8(2)} (2003), 177--206.

\bibitem[PAY]{Per06}
JHH Perk and H~Au-Yang, \emph{Yang-baxter equations}, arXiv preprint
  math-ph/0606053, 2006.

\bibitem[Pet10]{Pet10}
E.~Peters, \emph{A planar algebra construction of the {H}aagerup subfactor},
  International Journal of Mathematics \textbf{21} (2010), 987--1045.

\bibitem[Pop86]{Pop86}
S.~Popa, \emph{Correspondences}, Prepr. ser. in mathematics \textbf{13(56)}
  (1986).

\bibitem[Pop94]{Pop94}
\bysame, \emph{Classification of amenable subfactors of type {II}}, Acta Math.
  \textbf{172} (1994), 352--445.

\bibitem[Pop95]{Pop95}
\bysame, \emph{An axiomatization of the lattice of higher relative commutants},
  Invent. Math. \textbf{120} (1995), 237--252.

\bibitem[PT88]{PT88}
JH~Przytycki and P.~Traczyk, \emph{Invariants of links of conway type}, Kobe
  Journal of Mathematics \textbf{4} (1988), no.~2, 115--139.

\bibitem[Ren]{Ren}
Y.~Ren, \emph{A skein theory for one way yang-baxter relation}, In preparation.

\bibitem[RT91]{ResTur91}
N.~Reshetikhin and V.~G. Turaev, \emph{Invariants of 3-manifolds via link
  polynomials and quantum groups}, Invent. Math. \textbf{103} (1991), no.~1,
  547--597.

\bibitem[Saw95]{Saw95}
S.~Sawin, \emph{Subfactors constructed from quantum groups}, Amer. J. Math.
  \textbf{117} (1995), 1349--1369.

\bibitem[Sch27]{Sch27}
Issai Schur, \emph{{\"U}ber die rationalen darstellungen der allgemeinen
  linearen gruppe}, 1927.

\bibitem[Sut]{Sut12}
R.~Suter, \emph{Young's lattice and dihedral symmetries revisited: {M}\"obius
  strips and metric geometry}, arXiv preprint arXiv:1212.4463, 2012.

\bibitem[Sut02]{Sut02}
\bysame, \emph{Young¡¯s lattice and dihedral symmetries}, Europ. J.
  Combinatorics \textbf{23} (2002), no.~2, 233--238.

\bibitem[Thu]{Thu04}
D.~Thurston, \emph{From dominoes to hexagons}, arXiv:0405482.

\bibitem[TL71]{TemLie71}
Harold~NV Temperley and Elliott~H Lieb, \emph{Relations between
  the'percolation'and'colouring'problem and other graph-theoretical problems
  associated with regular planar lattices: some exact results for
  the'percolation'problem}, Proceedings of the Royal Society of London A:
  Mathematical, Physical and Engineering Sciences, vol. 322, The Royal Society,
  1971, pp.~251--280.

\bibitem[TV92]{TurVir92}
VG~Turaev and OY~Viro, \emph{State sum invariants of 3-manifolds and quantum
  6j-symbols}, Topology \textbf{31} (1992), no.~4, 865--902.

\bibitem[Wen88]{Wen88}
H.~Wenzl, \emph{On the structure of brauer's centralizer algebras}, Ann. of
  Math. \textbf{128} (1988), 173--193.

\bibitem[Wen90]{Wen90}
\bysame, \emph{Quantum groups and subfactors of type {B}, {C} and {D}}, Comm.
  Math. Phys \textbf{133} (1990), 383--433.

\bibitem[Wey46]{Wey46}
H.~Weyl, \emph{The classical groups, their invariants and representations},
  Princeton University Press, 1946.

\bibitem[Wit88]{Wit88}
E.~Witten, \emph{Topological quantum field theory}, Comm. Math. Phys.
  \textbf{117} (1988), no.~3, 353--386.

\bibitem[Xu98a]{Xu98S}
F.~Xu, \emph{Standard $\lambda$-lattices from quantum groups}, Invent. Math.
  \textbf{134} (1998), no.~3, 455--487.

\bibitem[Xu98b]{Xu98}
Feng Xu, \emph{New braided endomorphisms from conformal inclusions}, Comm.
  Math. Phys. \textbf{192} (1998), no.~2, 349--403.

\bibitem[Yan67]{Yan67}
C.~N. Yang, \emph{Some exact results for the many-body problem in one dimension
  with repulsive delta-function interaction}, Physical Review Letters
  \textbf{19} (1967), no.~23, 1312.

\bibitem[Yok97]{Yok97}
Y.~Yokota, \emph{Skeins and quantum {SU(N)} invariants of 3-manifolds}, Math.
  Ann. \textbf{307} (1997), no.~1, 109--138.

\end{thebibliography}
  \bibliographystyle{amsalpha}

\end{document}